\documentclass[a4paper,oneside,10pt,reqno,final]{amsart}

\usepackage{amsmath, amssymb, amsthm}

\usepackage{enumitem}
\usepackage[matrix,arrow]{xy}
\usepackage[mathscr]{eucal}
\usepackage{showkeys}
\usepackage{version,color}
\usepackage[l2tabu,orthodox]{nag}
\usepackage[all, warning]{onlyamsmath}
\usepackage[bookmarks=false,draft=false,breaklinks,colorlinks]{hyperref}

\numberwithin{equation}{section}
\numberwithin{figure}{section}
\allowdisplaybreaks

\newenvironment{ack}%
{\par \vspace{\baselineskip}%
 \noindent \textbf{Acknowledgements.}}%
{\par \vspace{\baselineskip}}%

\theoremstyle{definition}
\newtheorem{thm}{Theorem}[subsection]
\newtheorem{prp}[thm]{Proposition}
\newtheorem{lem}[thm]{Lemma}
\newtheorem{cor}[thm]{Corollary}
\newtheorem{dfn}[thm]{Definition}
\newtheorem{fct}[thm]{Fact}
\newtheorem{ntn}[thm]{Notation}
\newtheorem{rmk}[thm]{Remark}
\newtheorem{eg}[thm]{Example}
\newtheorem*{thm*}{Theorem}
\newtheorem*{prp*}{Proposition}
\newtheorem*{lem*}{Lemma}
\newtheorem*{cor*}{Corollary}
\newtheorem*{dfn*}{Definition}
\newtheorem*{ntn*}{Notation}
\newtheorem*{fct*}{Fact}
\newtheorem*{rmk*}{Remark}
\newtheorem*{eg*}{Example}
\newcommand{\ol}{\overline}
\newcommand{\ul}{\underline}
\newcommand{\wh}{\widehat}
\newcommand{\wt}{\widetilde}

\newcommand{\xr}{\xrightarrow}
\newcommand{\xrr}[1]{\xrightarrow{\, #1 \,}}
\newcommand{\inj}{\hookrightarrow}
\newcommand{\surj}{\twoheadrightarrow}
\newcommand{\longto}{\longrightarrow}
\newcommand{\simto}{\xr{\sim}}
\newcommand{\longsimto}{\xrr{\sim}}
\newcommand{\longinj}{\lhook\joinrel\longrightarrow}
\newcommand{\longsurj}{\relbar\joinrel\twoheadrightarrow}
\newcommand{\qiseq}{\underset{\mathrm{qis}}{\simeq}}
\newcommand{\dqiseq}{\, \qiseq \, }
\newcommand{\wtc}{\mathbin{\wt{\circ}}}

\newcommand{\tsum}{\textstyle \sum}
\newcommand{\tprd}{\textstyle \prod}
\newcommand{\tboplus}{\textstyle \bigoplus}
\newcommand{\tbotimes}{\textstyle \bigotimes}
\newcommand{\Wedge}{\textstyle \bigwedge}

\newcommand{\bl}{\bullet}
\newcommand{\ve}{\varepsilon}
\newcommand{\hy}{\mathchar`-}
\newcommand{\gq}{/ \kern-0.3em /}
\newcommand{\sr}{/ \kern-0.3em / \kern-0.3em /}
\newcommand{\vac}{\left|0\right>}
\newcommand{\one}{\mathbf{1}}
\newcommand{\itp}{\mathbin{\lrcorner}}
\newcommand{\sinf}{\tfrac{\infty}{2}}
\DeclareMathOperator{\colim}{colim}
\newcommand{\Higgs}{\mathcal{M}_{\mathrm{Higgs}}}

\newcommand{\bbk}{\Bbbk}
\newcommand{\bbA}{\mathbb{A}}

\newcommand{\bbL}{\mathbb{L}}
\newcommand{\bbN}{\mathbb{N}}
\newcommand{\bbP}{\mathbb{P}}
\newcommand{\bbQ}{\mathbb{Q}}
\newcommand{\bbR}{\mathbb{R}}
\newcommand{\bbC}{\mathbb{C}}
\newcommand{\bbT}{\mathbb{T}}

\newcommand{\bbZ}{\mathbb{Z}}
\newcommand{\bfV}{\mathbf{V}}

\newcommand{\calD}{\mathcal{D}}

\newcommand{\calS}{\mathcal{S}}
\newcommand{\calT}{\mathcal{T}}

\newcommand{\frkg}{\mathfrak{g}}

\newcommand{\frkl}{\mathfrak{l}}

\newcommand{\frkS}{\mathfrak{S}}

\newcommand{\shM}{\mathscr{M}}

\newcommand{\shO}{\mathscr{O}}

\DeclareMathOperator{\shHom}{\mathscr{H} \kern-0.1em \mathit{om}}

\newcommand{\ch}{\mathrm{ch}}
\newcommand{\cl}{\mathrm{cl}}
\newcommand{\et}{\mathrm{\acute{e}t}}
\newcommand{\fp}{\mathrm{fp}}
\newcommand{\ft}{\mathrm{ft}}
\newcommand{\lf}{\mathrm{lf}}
\newcommand{\nd}{\mathrm{nd}}
\newcommand{\op}{\mathrm{op}}
\newcommand{\dR}{\mathrm{dR}}
\newcommand{\DS}{\mathrm{DS}}
\newcommand{\dec}{\mathrm{dec}}
\newcommand{\odd}{\mathrm{odd}}
\newcommand{\ord}{\mathrm{ord}}
\newcommand{\qis}{\mathrm{qis}}
\newcommand{\tgr}{\mathrm{gr}}
\newcommand{\tCE}{\mathrm{CE}}
\newcommand{\tVA}{\mathrm{VA}}
\newcommand{\cois}{\mathrm{co}}
\newcommand{\even}{\mathrm{even}}

\newcommand{\tbar}{\mathrm{bar}}
\newcommand{\trst}{\mathrm{rst}}

\newcommand{\tKos}{\mathrm{Kos}}
\newcommand{\tLie}{\mathrm{Lie}}
\newcommand{\tBRST}{\mathrm{BRST}}

\DeclareMathOperator{\ad}{ad}

\DeclareMathOperator{\gr}{gr}
\DeclareMathOperator{\id}{id}
\DeclareMathOperator{\pr}{pr}

\DeclareMathOperator{\Cl}{Cl}
\DeclareMathOperator{\Eq}{Eq}

\DeclareMathOperator{\CE}{CE}
\DeclareMathOperator{\sgn}{sgn}
\DeclareMathOperator{\chg}{charge}

\DeclareMathOperator{\cCl}{\ol{Cl}}
\DeclareMathOperator{\chCl}{Cl_{ch}}
\DeclareMathOperator{\coCl}{Cl_{co}}
\DeclareMathOperator{\Der}{Der}
\DeclareMathOperator{\End}{End}
\DeclareMathOperator{\Hom}{Hom}
\DeclareMathOperator{\Img}{Im}
\DeclareMathOperator{\Ker}{Ker}
\DeclareMathOperator{\Kos}{Kos}
\DeclareMathOperator{\Lie}{Lie}
\DeclareMathOperator{\Map}{Map}
\DeclareMathOperator{\Mor}{Mor}
\DeclareMathOperator{\Pol}{Pol}
\DeclareMathOperator{\Ran}{Ran}
\DeclareMathOperator{\Res}{Res}
\DeclareMathOperator{\Sym}{Sym}
\DeclareMathOperator{\Tot}{Tot}

\DeclareMathOperator{\BRST}{BRST}
\DeclareMathOperator{\Cone}{Cone}
\DeclareMathOperator{\Pois}{Pois}

\DeclareMathOperator{\Spec}{Spec}
\DeclareMathOperator{\uHom}{\ul{Hom}}
\DeclareMathOperator{\uEnd}{\ul{End}}
\DeclareMathOperator{\uMap}{\ul{Map}}
\DeclareMathOperator{\Specm}{Specm}
\DeclareMathOperator{\cBRST}{BRST_{cl}}
\DeclareMathOperator{\dBRST}{BRST^{\prime}_{cl}}
\DeclareMathOperator{\coBRST}{BRST_{co}}

\newcommand{\Bo}{\mathsf{Bo}_2}
\newcommand{\HS}{\mathsf{HS}}

\newcommand{\KL}{\mathsf{KL}}
\newcommand{\MT}{\mathsf{MT}}
\newcommand{\chMT}{\mathsf{MT}_{\ch}}
\newcommand{\coMT}{\mathsf{MT}_{\cois}}
\newcommand{\VA}{\mathsf{VA}}

\newcommand{\Sch}{\mathsf{Sch}}
\newcommand{\dgKL}{\mathsf{dgKL}}

\newcommand{\cVec}{\mathsf{Vec}}
\newcommand{\sVec}{\mathsf{sVec}}
\newcommand{\gVec}{\mathsf{gVec}}
\newcommand{\dgVec}{\mathsf{dgVec}}
\newcommand{\ca}{\mathsf{Comu}}
\newcommand{\dga}{\mathsf{dguAlg}}
\newcommand{\cdga}{\mathsf{dguCom}}
\newcommand{\dgla}{\mathsf{dgLie}}
\newcommand{\dgpa}{\mathsf{dgPA}}
\newcommand{\dgVA}{\mathsf{dgVA}}
\newcommand{\dgVL}{\mathsf{dgVL}}
\newcommand{\dgVP}{\mathsf{dgVP}}

\newcommand{\PMod}[1]{#1\hy\mathsf{PMod}}
\newcommand{\VMod}[1]{#1\hy\mathsf{VMod}}
\newcommand{\VPMod}[1]{#1\hy\mathsf{VPMod}}
\newcommand{\dgMod}[1]{#1 \hy \mathsf{dgMod}}
\newcommand{\dgVMod}[1]{#1 \hy \mathsf{dgVMod}}
\newcommand{\dgVPMod}[1]{#1\hy\mathsf{dgVPMod}}
\newcommand{\eLQCoh}[1]{\mathop{\mathsf{L}_\mathsf{QCoh}^{#1}}}
\DeclareMathOperator{\DMod}{\mathsf{DMod}}
\DeclareMathOperator{\QCoh}{\mathsf{QCoh}}
\DeclareMathOperator{\DQCoh}{\mathsf{D_{QCoh}}}
\DeclareMathOperator{\LQCoh}{\mathsf{L_{QCoh}}}
\DeclareMathOperator{\PQCoh}{\mathsf{PoisQCoh}}
\DeclareMathOperator{\LPQCoh}{\mathsf{L_{PoisQCoh}}}
\newcommand{\ddgVec}{\mathbf{dgVec}}

\newcommand{\ddgla}{\mathbf{dgLie}}

\newcommand{\dgPMod}[1]{#1 \hy \mathsf{dgPMod}}
\newcommand{\ddgMod}[1]{#1 \hy \mathbf{dgMod}}

\newcommand{\ddgVPMod}[1]{#1 \hy \mathbf{dgVPMod}}

\newcommand{\h}{\mathrm{h}}
\newcommand{\iC}{\EuScript{C}}

\newcommand{\iH}{\EuScript{H}}
\newcommand{\iS}{\EuScript{S}}

\newcommand{\idSt}{\mathbf{dSt}}
\newcommand{\idSch}{\mathbf{dSch}}
\newcommand{\idAff}{\mathbf{dAff}}

\newcommand{\icdga}{\mathbf{dguCom}}

\newcommand{\abs}[1]{\left| #1 \right|}
\newcommand{\sabs}[1]{\bigl| #1 \bigr|}
\newcommand{\lsp}[1]{\left< #1 \right>_{\mathrm{lin}}}
\newcommand{\rst}[2]{\left. #1 \right|_{#2}}
\newcommand{\pair}[1]{\left< #1 \right>}

\begin{document}

\title{Derived gluing construction of chiral algebras}
\author{Shintarou Yanagida}
\address{Graduate School of Mathematics, Nagoya University. 
Furocho, Chikusaku, Nagoya, Japan, 464-8602.}
\email{yanagida@math.nagoya-u.ac.jp}
%\subjclass[2010]{16S99, 18E30}
\thanks{This work is supported by JSPS KAKENHI Grant Number 19K03399.}
%\date{December 1--, 2019}
\date{April 21, 2020}

\begin{abstract}
We discuss the gluing construction of class $\mathcal{S}$ chiral algebras
in derived setting.
The gluing construction in non-derived setting was introduced by Arakawa to
construct a family of vertex algebras of which the associated varieties give 
genus zero Moore-Tachikawa symplectic varieties. 
Motivated by the higher genus case, we introduce a dg vertex algebra version
$\mathsf{MT}_{\mathrm{ch}}$ of the category of Moore-Tachikawa 
symplectic varieties, where a morphism is given by a dg vertex algebra 
equipped with action of the universal affine vertex algebra,
and composition of morphisms is given by the BRST reduction.
We also show that the procedure taking the associated scheme of 
gives a functor from $\mathsf{MT}_{\mathrm{ch}}$ to the category 
$\mathsf{MT}$ of derived Moore-Tachikawa varieties,
which would imply compatibility of gluing constructions in both categories.
\end{abstract}

\maketitle

{\footnotesize \tableofcontents}

%%%%%%%%%%%%%%%%%%%%%%%%%%%%%%%%%%%%%%%%%%%%%%%%%%%%%%%%%%%%%%%%%%%%%%%%%%%%%%%%
%%%%%%%%%%%%%%%%%%%%%%%%%%%%%%%%%%%%%%%%%%%%%%%%%%%%%%%%%%%%%%%%%%%%%%%%%%%%%%%%
%%%%%%%%%%%%%%%%%%%%%%%%%%%%%%%%%%%%%%%%%%%%%%%%%%%%%%%%%%%%%%%%%%%%%%%%%%%%%%%%
\setcounter{section}{-1}
\section{Introduction}
\label{s:intro}

%%%%%%%%%%%%%%%%%%%%%%%%%%%%%%%%%%%%%%%%%%%%%%%%%%%%%%%%%%%%%%%%%%%%%%%%%%%%%%%%
%%%%%%%%%%%%%%%%%%%%%%%%%%%%%%%%%%%%%%%%%%%%%%%%%%%%%%%%%%%%%%%%%%%%%%%%%%%%%%%%
\subsection{Backgrounds}

Let $G$ be a simply connected semisimple linear 
algebraic group over the complex field $\bbC$.
In \cite{A18}, Arakawa introduced a family of vertex algebras
\begin{equation}\label{eq:0:chS}
 \bfV^{\calS}_{G,b} \quad (b \in \bbZ_{\ge 1})
\end{equation}
called  the \emph{genus zero chiral algebras of class $\calS$}.
They are designed as ``chiral quantization" of genus zero 
\emph{Moore-Tachikawa symplectic varieties},
and the design results in the \emph{associativity} property of the family.
The associativity is encoded by BRST reduction of vertex algebras,
and this article concerns a technical issue on the reduction
when one wants to study higher genus cases.

%%%%%%%%%%%%%%%%%%%%%%%%%%%%%%%%%%%%%%%%%%%%%%%%%%%%%%%%%%%%%%%%%%%%%%%%%%%%%%%%
\subsubsection{Moore-Tachikawa symplectic varieites}
\label{sss:0:MT}

In order to explain the detail, 
let us give a short recollection on Moore-Tachikawa symplectic varieties.
See also \cite[\S 1]{A18} for a brief explanation.

In \cite{MT}, concerning expected properties of Higgs branches of 
Sicilian theories, Moore and Tachikawa proposed a two-dimensional topological 
field theory which target holomorphic symplectic varieties.
Mathematically speaking, they conjectured for each semisimple algebraic group $G$
over $\bbC$ the existence of a symmetric monoidal functor 
\[
 \eta_G: \Bo \longto \HS
\]
{}from the category $\Bo$ of 2-bordisms to the category $\HS$ of 
holomorphic symplectic varieties
which satisfies some axioms.

The source category $\Bo$ is a familiar one in the context of 
topological field theory.
We use the following description of $\Bo$:
\begin{itemize}[nosep]
\item 
An object is a closed oriented one-dimensional manifold, 
or a disjoint union of $S^1$'s.

\item
A morphism from $B_1$ to $B_2$ is the diffeomorphism class of an oriented 
two-dimensional manifold $\Sigma$ with boundary $(-B_1)\sqcup B_2$.
We denote by $\Sigma_{g,b}$ the class of an oriented surface of 
genus $g$ with $b$ boundary components.
In particular, 
the tube $\Sigma_{0,2}$ represents the identity morphism $\id_{S^1}$.

\item
Composition of morphisms is given by gluing.
In particular, we have 
\begin{equation}\label{eq:0:glue}
 \Sigma_{0,b'} \circ \Sigma_{0,b} =  \Sigma_{0,b+b'-2}.
\end{equation}
\end{itemize}
Disjoint union $\sqcup$ gives $\Bo$ a symmetric monoidal structure.

The target category $\HS$ is described as follows.
\begin{itemize}[nosep]
\item 
An object is a semisimple algebraic group over $\bbC$.

\item
A morphism from $G_1$ to $G_2$ is a possibly singular symplectic variety $X$
over $\bbC$ with a $\bbC^{\times}$-action scaling the symplectic form 
by weight $2$ together with Hamiltonian action of $G_1 \times G_2$ 
satisfying some regularity condition.
The identity $\id_G \in \Hom_{\HS}(G,G)$ is the cotangent bundle $T^* G$
with the left and right multiplication of $G$.

\item
Composition of  $X \in \Hom_{\HS}(G_1,G_2)$ and
$X' \in \Hom_{\HS}(G_2,G_3)$ is given by the Hamiltonian reduction
of the product with respect to the diagonal $G_2$-action:
\[
 X' \circ X := (X^{\op} \times X') \gq \Delta(G_2) = \mu^{-1}(0)/\Delta(G_2).
\]
Here $X^{\op}$ denotes the symplectic variety $X$ with the opposite symplectic
structure, and the morphism $\mu: X^{\op} \times X' \to \frkg_2^*$ is
the momentum map $\mu(x,y):=-\mu_X(x)+\mu_{X'}(y)$ with $\mu_X$ and $\mu_{X'}$ 
the $\frkg_2^*$-component of the momentum map $X \to \frkg_1^* \times \frkg_2^*$
and $X' \to \frkg_2^* \times \frkg_3^*$ respectively.
The reduction doesn't touch the $G_1$- and $G_3$-actions, so that we have 
$X' \circ X \in \Hom_{\HS}(G_1,G_3)$.
\end{itemize}
The cartesian product of groups and varieties gives $\HS$ 
a symmetric monoidal structure.

Let us denote 
\[
 W_G^b := \eta_G(\Sigma_{0,b}),
\]
where $\Sigma_{0,b}$ is an oriented surface of genus $0$ 
with $b$ boundary components.
It encodes the genus zero part of the functor $\eta_G$.
The functor $\eta_G$ should satisfy
$\eta_G(S^1)=G$ and $W_G^2=T^* G$.
We refer \cite[\S 3]{MT} for the full axiom of $\eta_G$.

In \cite{BFN}, concerning mathematical construction of the Coulomb branches of 
three-dimensional supersymmetric gauge theory via 
perverse sheaves on affine Grassmannians, 
Braverman, Finkelberg and Nakajima constructed 
the genus zero Moore-Tachikawa varieties $W_G^b := \eta_G(\Sigma_{0,b})$ 
in a uniform way, and showed 
%a symmetric monoidal functor 
%\[
% \eta_G: \catB^{g=0}_2 \longto \HS
%\] 
%where $\catB^{g=0}_2 \subset \Bo$ denotes the subcategory 
%In other words, 
%under the identification $\Ob(\catB^{g=0}_2)=\bbN:=\{0,1,2,\ldots\}$
%and $\Hom_{\catB^{g=0}_2}(m,n)=\{\abs{m-n}\}$, 
%for each surface $\Sigma_{0,b}$ of genus $0$ with $b$ boundaries,
%they constructed a symplectic variety $W_G^b := \eta_G(\Sigma_{0,b})$
%equipped with Hamiltonian action of $G^{\times b}$,
%satisfying the following conditions.
\begin{equation}\label{eq:0:MT}
 W_G^1 \simeq G \times \EuScript{S}, \quad 
 W_G^2 \simeq T^* G, \quad 
 W_G^{b'} \circ W_G^b \simeq W_G^{b+b'-2},
\end{equation}
where $\EuScript{S} \subset \frkg^*$ denotes the Slodowy slice.
The third isomorphism reflects the gluing \eqref{eq:0:glue} in $\Bo$,
and we call it the \emph{gluing condition}.
In particular, the genus zero part of the functor $\eta_G$ is established.

%%%%%%%%%%%%%%%%%%%%%%%%%%%%%%%%%%%%%%%%%%%%%%%%%%%%%%%%%%%%%%%%%%%%%%%%%%%%%%%%
\subsubsection{The 4d/2d duality and chiral algebras of class $\calS$}
\label{sss:0:42}

Next we explain the 4d/2d duality.%, which is the motivation of \cite{A18}.
In \cite{BLLPRvR}, Beem, Lemos, Liendo, Peelaers, Rastelli and van Rees proposed
a ``functorial" construction 
\begin{equation}\label{eq:0:42}
 \calT \longmapsto V_{\calT}
\end{equation}
of conformal vertex algebras $V_{\calT}$ from a four-dimensional $\mathcal{N}=2$
superconformal field theory (4d SCFT for short) $\calT$.
Among such four-dimensional theories, 
we have the \emph{theory $\calT^{\calS}_{G,\Sigma}$ of class $\calS$} 
attached to a complex semisimple algebraic group $G$ 
and a punctured Riemann surface $\Sigma$.
We will not go into the detail and 
refer the exposition \cite{Ta} for mathematicians.
The vertex algebra obtained from $\calT^{\calS}_{G,\Sigma}$ 
by the above ``functor" is called the \emph{chiral algebras of class $\calS$}.
We denote it by $V^{\calS}_{G,\Sigma}$.

A clue to identify the vertex algebra $V^{\calS}_{G,\Sigma}$ is to consider 
an ``invariant" of the physical theory $\calT^{\calS}_{G,\Sigma}$.
The attachment $\calT \mapsto V_{\calT}$ is one of such invariants.
Another interesting invariant is the hyperk\"ahler manifold $\Higgs(\calT)$
called the \emph{Higgs branch} of $\calT$.
See \cite{Ta} for more information.
In \cite{BR}, Beem and Rastelli conjectured that for any 4d SCFT $\calT$
there is an isomorphism 
\begin{equation}\label{eq:0:BR}
 \Higgs(\calT) \stackrel{?}{\simeq} \Specm(R_{V_\calT})
\end{equation}
of holomorphic symplectic varieties, 
where the right hand side denotes the \emph{associated variety}
of the vertex algebra $V_\calT$. 
See \S \ref{sss:0:ch} for more information, 
and Definition \ref{dfn:li:C2} for the precise definition of $R_{V_\calT}$.

Now we can explain the result in \cite{A18}.
Arakawa considered the following ``chiral quantization" $\eta_G^{\tVA}$ of 
the Moore-Tachikawa functor $\eta_G$.
Let us denote by $V_{k}(\frkg)$ the universal affine vertex algebra at level $k$
for $\frkg=\Lie(G)$, and by $h^\vee$ the dual Coxeter number of $\frkg$
(see \S \ref{sss:li:uava} for the detail).
He considered a symmetric monoidal functor
\[
 \eta_G^{\tVA}: \Bo \longto \VA,
\]
where the category $\VA$ is roughly explained as follows.
See \S \ref{ss:ch:gl} for the detail.
\begin{itemize}[nosep]
\item
An object is a semisimple algebraic group.
\item
A morphism from $G_1$ to $G_2$ is a vertex algebra $V$ equipped with 
a vertex algebra morphism 
$V_{-h_1^{\vee}}(\frkg_1) \otimes V_{-h_2^{\vee}}(\frkg_2) \to V$ 
and satisfying some conditions.

\item
Composition of 
$V \in \Hom_{\VA}(G_1,G_2)$ and $V' \in \Hom_{\VA}(G_2,G_3)$
is given by the relative BRST reduction:
\begin{equation}\label{eq:0:BRST}
 V' \circ V := H^{\sinf+0}(\wh{\frkg}_{-2 h^{\vee}_2},\frkg_2, V^{\op} \otimes V').
\end{equation}
\end{itemize}
The condition on $\eta_G^{\VA}$ is that for any surface 
$\Sigma \in \Mor(\Bo)$, the vertex algebra $\eta_G^{\tVA}(\Sigma)$ should
coincide with $V^{\calS}_{\Sigma}$, the chiral algebra of class $\calS$. 

Arakawa constructed in \cite{A18} the genus zero part 
of the functor $\eta_G^{\VA}$.
In other words, he constructed the image 
\[
 \bfV^{\calS}_{G,b} := \eta_G^{\VA}(\Sigma_{0,b})
\]
of the genus $0$ surface $\Sigma_{0,b}$ with $b$ boundaries,
and checked a ``chiral quantization" of \eqref{eq:0:MT}:
\begin{equation}\label{eq:0:chgl}
 \bfV^{\calS}_{G,1} \simeq H^0_{\DS}(\calD^{\ch}_G), \quad
 \bfV^{\calS}_{G,2} \simeq \calD^{\ch}_G, \quad
 \bfV^{\calS}_{G,b'} \circ \bfV^{\calS}_{G,b} \simeq \bfV^{\calS}_{G,b+b'-2},
\end{equation}
Here $\calD^{\ch}_G$ denotes the algebra of chiral differential operators 
on $G$ at the critical level, 
and $H^0_{\DS}$ denotes the quantum Drinfeld-Sokolov reduction.
The third relation is called the \emph{associativity} in \cite{A18}.
Moreover Arakawa showed
\begin{equation}\label{eq:0:RV=W}
 R_{\bfV^{\calS}_{G,b}} = \bbC[W^{b}_{G}],
\end{equation}
where the right hand side denotes the coordinate ring.
Thus he solved the conjecture \eqref{eq:0:BR} for 
the genus zero class $\calS$ theories $\calT=\calT^{\calS}_{\Sigma_{0,b}}$.

%equipped with vertex algebra homomorphism 
%$\otimes_{i=1}^b V^{\kappa_c}(\frkg) \to \bfV^{\mathcal{S}}_{G,b}$ from 
%the $b$-th tensor product of the universal affine vertex algebra 

%%%%%%%%%%%%%%%%%%%%%%%%%%%%%%%%%%%%%%%%%%%%%%%%%%%%%%%%%%%%%%%%%%%%%%%%%%%%%%%%
\subsubsection{Chiral quantization}
\label{sss:0:ch}

So far we have used the word ``\emph{chiral quantization}" several times.
Let us clarify what it means.
Since it is also a good place to recall the \emph{associated schemes} of vertex
algebras, let us explain the chiralization of a Poisson structure.
\begin{enumerate}[nosep]
\item
We start with a \emph{Poisson algebra} $R$, i.e., 
a commutative algebra with a Poisson bracket.
It corresponds to an affine Poisson scheme $\Spec(R)$.

\item
Recall the \emph{arc space}, or the \emph{$\infty$-jet space}, 
$J_\infty(Y)$ of a scheme $Y$.
It is a scheme having the universal property
\[
 \Hom_{\Sch}(\Spec(A),J_\infty(Y))=\Hom_{\Sch}(\Spec(A[[t]]),Y)
\]
for any commutative algebra $A$,
where $\Sch$ denotes the category of schemes.
See \S \ref{ss:jet:ord} for the precise statement.

\item
By \cite[\S 2.3]{A15}, the coordinate ring $J_\infty(R)$ of the arc space
$J_\infty(\Spec(R))$ has a structure of \emph{vertex Poisson algebra}
induced naturally from the original Poisson algebra structure on $R$.
See \S \ref{ss:li:dgvp} for the detail.

\item
By \cite{Li}, for a vertex algebra $V$, the associated graded space $\gr^F V$
with respect to the \emph{Li filtration} $F^{\bl}V$ 
has a structure of vertex Poisson algebra.
The component 
\[
 R_V :=  F^0 V/F^1 V
\]
is a Poisson algebra,
which is called \emph{Zhu's $C_2$-algebra} \cite{Z}.
The corresponding affine Poisson scheme $X_V:=\Spec(R_V)$
is called the \emph{associated scheme} of $V$.
We have a natural surjective morphism $J_\infty(R_V) \surj \gr V$
of vertex Poisson algebras.
See \S \ref{sss:li:li} for the detail.
\end{enumerate}

Now we cite from \cite[Definition 2.1]{A18} the terminology:
Let $V$ be a vertex algebra.
If Zhu's $C_2$-algebra $R_V$ is isomorphic to a Poisson algebra $R$, then we call
$V$ a \emph{chiral quantization} of the affine Poisson scheme $X=\Spec(R)$.
If moreover $V$ is separated (Definition \ref{dfn:li:sep}), $R_V$ is reduced
and the surjection $J_{\infty}(R_V) \surj \gr^F V$ is an isomorphism,
then $V$ is called a \emph{strict chiral quantization} of $X$.

Using the terminology above, we can restate the relation \eqref{eq:0:RV=W} as:
The genus zero chiral algebra $\bfV^{\calS}_{G,b}$ of class $\calS$ is 
a strict chiral quantization of 
the genus zero Moore-Tachikawa varieties $W^{b}_{G}$.
We also have a functorial picture
\[
 \bfV^{\calS}_{G,b'} \circ \bfV^{\calS}_{G,b} 
 \underset{\eqref{eq:0:chgl}}{\simeq} \bfV^{\calS}_{G,b+b'-2}
 \, \xrightarrow{R_{(-)}} \, 
 W_G^{b'} \circ W_G^b \underset{\eqref{eq:0:MT}}{\simeq} W_G^{b+b'-2}, 
\]
and in this sense we can say that the relations \eqref{eq:0:chgl}
are chiral quantization of \eqref{eq:0:MT}.
We call this picture the \emph{compatibility of gluing constructions}.

It is then natural to expect that the procedure $R_{(-)}$ 
taking Zhu's $C_2$-algebra makes the following diagram commutative:
\begin{equation}\label{diag:0:com}
 \xymatrix{
  \Bo \ar@{=}[d] \ar[rr]^{\eta_G^{\VA}} & & \VA \ar[d]^{R_{(-)}} \\
  \Bo \ar[rr]_{\eta_G} & & \HS}
\end{equation}
We can say that \cite{A18} established this commutativity 
restricting to the genus zero part of $\Bo$.
His work can be regarded as a part of the mathematical formulation of 
the 4d/2d duality ``functor" \eqref{eq:0:42}.

%%%%%%%%%%%%%%%%%%%%%%%%%%%%%%%%%%%%%%%%%%%%%%%%%%%%%%%%%%%%%%%%%%%%%%%%%%%%%%%%
%%%%%%%%%%%%%%%%%%%%%%%%%%%%%%%%%%%%%%%%%%%%%%%%%%%%%%%%%%%%%%%%%%%%%%%%%%%%%%%%
\subsection{Derived gluing of vertex algebras --- Organization of the text}
\label{ss:0:org}

The proof in \cite{A18} of the compatibility of gluing constructions for 
genus zero chiral algebras of class $\calS$ is based on cohomology vanishing 
in the BRST reduction \eqref{eq:0:BRST}.
As mentioned in the footnote in \cite[p.3]{A18},
there is a subtlety on this point in higher genus case.
The issue is that the momentum map associated to the Hamiltonian action 
can be non-flat in higher genus case,
for which we don't have a clean cohomology vanishing.
As a result, we don't know explicit description of 
the chiral algebra $\bfV^{\calS}_{\Sigma}$ of class $\calS$ 
for a higher genus surface $\Sigma$ at this moment.

It was lucky for the author to take a lecture series by Arakawa 
on the article \cite{A18} in the end of November, 2019.
In the lecture Arakawa suggested to use \emph{derived symplectic geometry} to 
overcome this difficulty.
The aim of this article is to give a first step to fulfill his suggestion.

The main materials in this text are as follows:
\begin{itemize}[nosep]
\item
We introduce a differential graded vertex algebra analogue 
\[
 \chMT
\]
of the category $\MT$ of derived Moore-Tachikawa varieties.
A morphism in $\chMT$ is a differential graded  vertex algebra $V$,
and we denote by $V' \wtc V$ a composition of morphisms.
\item 
The \emph{compatibility of derived gluing constructions}
\[
 R_{(V' \wtc V)} \simeq R_{V'} \wtc R_V
\]
for \emph{differential graded} vertex algebras $V$ and $V'$ of certain type.
Here $R_V$ denotes \emph{Zhu's $C_2$-algebra} of $V$ as for the ordinary 
vertex algebra, and the symbol $\simeq$ denotes quasi-isomorphism as 
\emph{homotopy Poisson algebras}.
Equivalently, the functor $R_{(-)}$ induces
\[
 R_{(-)}: \chMT \longto \MT.
\]
\end{itemize} 
 
Below we give an overview of this text 
along the line of the arguments in \S \ref{sss:0:ch}.
Let us start with the derived version of Poisson algebras.
Hereafter we use the word ``dg" to mean ``differential graded".

%%%%%%%%%%%%%%%%%%%%%%%%%%%%%%%%%%%%%%%%%%%%%%%%%%%%%%%%%%%%%%%%%%%%%%%%%%%%%%%%
\subsubsection{Derived symplectic and Poisson geometry}
\label{sss:0:sp}

The idea of using derived symplectic geometry to realize Moore-Tachikawa varieties
is, as far as the author understands, originally due to Calaque 
\cite[Concluding remarks]{C1}, \cite[Example 3.5]{C2}.
He introduced a derived version of the category $\MT$ where composition of 
morphisms is given by the derived intersection of Lagrangian structures.
This approach enables us to consider Hamiltonian reduction for 
non-flat momentum maps.
See \S \ref{ss:pr:MT} for more information.

In this text we use affine \emph{derived Poisson geometry} 
instead of derived symplectic geometry.
Derived Poisson geometry was introduced in \cite{CPTVV}
as a natural Poisson analogue of derived symplectic geometry \cite{PTVV}.
As in the ordinary Poisson and symplectic structures,
we have an equivalence of non-degenerate Poisson and symplectic structures.
See \cite{C1, C2} for more information on derived symplectic geometry
in the present context, and \cite{S2} for a review of derived Poisson geometry.
The reason for us to use derived Poisson geometry is that Poisson 
structures appear naturally if we regard vertex algebras as chiral quantization,
as we saw in \S \ref{sss:0:ch}.

%This problem is obviously related to a derived approach to vertex algebra,
%and on the latter there are much literature.
%However, as far as we understand, an appropriate language is not built to 
%tackle our problem.
%For example, it is natural to consider the dg version of 
%Beilinson-Drinfeld chiral algebras \cite{BD} and to build 
%the theory of associated varieties in the dg context.
%But in \cite[3.3.13]{BD}, it is noted that treating the model structure on 
%the category of dg chiral algebras has several difficulties, 
%and this approach seems to be a hard one.
The aim of the beginning \S \ref{s:sp} is to give 
a recollection on shifted Poisson structures.
Since what we need in this text is an affine version, we mainly treat
those structures on commutative dg algebras.
The main object is \emph{$\bbP_n$-algebra} (Definition \ref{dfn:sp:sp}).
The case $n=1$ corresponds to dg (non-shifted) Poisson algebra.
Essentially we only need this $n=1$ case, but in order to introduce 
\emph{coisotropic structures} to define \emph{derived Hamiltonian reduction}
later, we treat general shifted Poisson structures.

In the course of preparations in \S\S \ref{ss:sp:dg}--\S \ref{ss:sp:dga},
we explain notations on dg objects and algebraic structures on them.
In particular, we denote by
\[
 \CE(\frkl,M)
\]
the \emph{Chevalley-Eilenberg complex} for a dg Lie algebra $\frkl$ 
and a dg $\frkl$-module $M$.
See \S \ref{sss:dga:CE}, Definition \ref{dfn:dga:CE} for the detail.
Recall also that the correspondence $\CE(\frkl,-)$ is functorial,
which will be used repeatedly in the following explanation.

%Nevertheless, we give a short account on derived algebraic geometry
%in \S \ref{ss:sp:dag} for later use and completeness of the text.
%later use in arc spaces of derived schemes.

We will also use the \emph{Kirillov-Kostant Poisson algebra}:
For a dg Lie algebra $\frkl$, the symmetric algebra $\Sym(\frkl)$ has a 
structure of dg Poisson algebra whose Poisson bracket is uniquely determined by
the Lie bracket of $\frkl$ and the Leibniz rule.
See \S \ref{sss:sp:KK} for the detail.

%%%%%%%%%%%%%%%%%%%%%%%%%%%%%%%%%%%%%%%%%%%%%%%%%%%%%%%%%%%%%%%%%%%%%%%%%%%%%%%%
\subsubsection{Derived Hamiltonian reduction of shifted Poisson algebra}

We next consider a derived analogue of 
Hamiltonian reduction of shifted Poisson algebras.
For the reduction of ordinary Poisson algebras, we refer \cite[Chapter 5]{LPV}.

As mentioned in the previous \S \ref{sss:0:sp}, 
Calaque \cite{C2} formulated the gluing of Moore-Tachikawa varieties
as derived intersection of Lagrangians in derived symplectic schemes 
with Hamiltonian group action.
The corresponding procedure in the shifted Poisson structure 
is given by derived intersection of \emph{coisotropic structures}
introduced in \cite{CPTVV}.
In this text we use an equivalent but alternative approach of Safronov \cite{S},
which will be reviewed in \S \ref{s:pr}.
This approach has an advantage in the point that the connection to 
the \emph{classical BRST complex} is clear.

The coisotropic structure and Hamiltonian reduction is explained in \S \ref{ss:pr:pr}.
Let $\frkl$ be a dg Lie algebra and $R$ a dg Poisson algebra.
Noticing that $R$ can be regarded as a dg Lie algebra, 
we call a morphism $\mu: R \to \frkl$ of dg Lie algebras a \emph{momentum map}
(see Definition \ref{dfn:sp:Ham} and Remark \ref{rmk:sp:mu}).
For such $\mu$, the morphism $\CE(\frkl,\mu)$ is coisotropic
in the sense of Definition \ref{dfn:pr:cois}.
Using coisotropic morphisms, we define 
\begin{equation}\label{eq:0:dHr}
 R \gq^{\bbL}_{\mu} \Sym(\frkl) := \CE(\frkl,\bbk) 
 \tbotimes^{\bbL}_{\CE(\frkl,\Sym \frkl)} \CE(\frkl,R),
\end{equation}
and call it the \emph{derived Hamiltonian reduction} of $R$ 
with respect to the momentum map $\mu$.
It has a structure of \emph{$\wh{\bbP}_1$-algebra}, which can be restated as
homotopy Poisson algebra (Definition \ref{dfn:sp:whP}).

The classical BRST complex, 
originally defined in non-dg setting by Kostant and Sternberg \cite{KS}, 
will be introduced in \S \ref{ss:pr:BRST}.
For the triple $(\frkl, R, \mu)$ as above,
we define the \emph{classical BRST complex} $\cBRST(\frkl,R,\mu)$ by 
\[
 \cBRST(\frkl,R,\mu) := (\CE(\frkl,\Kos(\frkl,R,\mu)),d_{\cl}),
\]
where $\Kos(\frkl,R,\mu)$ denotes the \emph{Koszul complex} of $\mu^{-1}(0)$,
and the differential $d_{\tBRST}$ is given by the BRST charge.
See Definition \ref{dfn:pr:cBRST} for the detail.

The connection between the derived Hamiltonian reduction is explained in 
Proposition \ref{prp:pr:HBp}:
For a finite-dimensional Lie algebra $\frkl$, we have an quasi-isomorphism
\begin{equation}\label{eq:0:HB}
 R \gq^{\bbL}_{\mu} \Sym(\frkl) \dqiseq \cBRST(\frkg,R,\mu)
\end{equation}
of homotopy Poisson algebras.

In \S \ref{ss:pr:MT}, we introduce 
the \emph{category $\MT$ of derived Moore-Tachikawa varieties}, 
which is essentially the same with the category introduced by Calaque \cite{C2}
mentioned before.
An object is a semisimple algebraic group $G$, 
which is identified with its Lie algebra $\frkg:=\Lie(G)$.
A morphism from $G_1$ to $G_2$ is a non-degenerate $\wh{\bbP}_1$-algebra $R$
with a momentum map $\mu_R=\mu_R^1+\mu_R^2: \frkg_1 \oplus \frkg_2 \to R$.
Composition of $R \in \Hom_{\MT}(G_1,G_2)$ and $R' \in \Hom_{\MT}(G_2,G_3)$
is given by
\begin{equation}\label{eq:0:dgl}
 R' \wtc R := \cBRST \bigl(\frkg_2,R^{\op} \otimes R', \mu \bigr)
 \simeq \bigl(R^{\op} \otimes R'\bigr) \gq^{\bbL}_{\mu} \Sym(\frkg_2)
\end{equation}
with $\mu:=-\mu_R^2+\mu_{R'}^1$.
This is the \emph{derived gluing} of Moore-Tachikawa varieties.
See Definition \ref{dfn:pr:MT} and Remark \ref{rmk:pr:MT} for the detail.

%%%%%%%%%%%%%%%%%%%%%%%%%%%%%%%%%%%%%%%%%%%%%%%%%%%%%%%%%%%%%%%%%%%%%%%%%%%%%%%%
\subsubsection{Jet and arc spaces for derived schemes, 
and dg vertex Poisson algebras}

As explained in \S \ref{sss:0:ch}, in order to go intro the vertex world, 
we consider \emph{arc space} of Poisson schemes.
Thus, for the purpose of this text, we need a derived analogue of the theory of
jet and arc spaces, which is a kind of exercise of derived algebraic geometry.
Since there seems no explicit literature,
we will give an explicit account it in \S \ref{s:jet}.

For the explanation below, let us fix some notations.
Let $R$ be a commutative dg algebra.
Then there is a commutative dg algebra $J_\infty(R)$ such that 
$\Hom(R,A[[t]]) \simeq \Hom(J_\infty(R),A)$ for any commutative dg algebra $A$
(see Lemma \ref{lem:jet:daff} for the precise statement).
We call it the \emph{arc space} of $R$ (precisely speaking, 
we should call it the derived coordinate ring of the arc space).
We have a natural embedding $R \inj J_\infty(R)$.

In the context of derived gluing, we should consider 
a dg Poisson algebra $R$ and its arc space $J_\infty(R)$.
The arc space inherits the Poisson structure of $R$,
which should be of infinite-dimensional nature.
As explained in \S \ref{sss:0:ch}, for a non-dg Poisson algebra $R$,
$J_\infty(R)$ has a structure of \emph{vertex Poisson algebra}
by the work of Arakawa \cite[\S 2.3]{A12}.
Thus we should study a dg version of vertex Poisson algebra.
Such a notion is in fact included in the theory of \emph{coisson algebra}
introduced by Beilinson and Drinfeld \cite[2.6]{BD}.
In this text we explain a special case of their theory in \S \ref{s:li}.

We will explain standard notions on vertex algebras, vertex Poisson algebras
and \emph{Li's canonical filtration} in \S \ref{s:li}.
Along the way, we also introduce dg versions of these ``vertex" notions.
Since these dg versions may not be standard, 
we list up the references of definitions:
\begin{itemize}[nosep]
\item 
\emph{Dg vertex algebras}: Definition \ref{dfn:li:dgva}.

\item
\emph{Dg modules over a dg vertex algebra}: Definition \ref{dfn:li:dgMod}.

\item
\emph{Dg vertex Poisson algebras}: Definition \ref{dfn:li:dgvp}.

\item
\emph{Li filtration of a dg vertex algebra}: Definition \ref{dfn:li:li}.
\end{itemize}

In this introduction, the following statement is sufficient:
For a dg Poisson algebra $R$, the arc space $J_\infty(R)$ has a unique structure
of dg vertex Poisson algebra such that 
$a_{(n)}b= \delta_{n,0} \{a,b\}_R$ for $a,b \in R \subset J_\infty(R)$.
In particular, for a dg Lie algebra $\frkl$,
the arc space $J_\infty(\Sym(\frkl))$ of the Kirillov-Kostant Poisson algebra 
(see the end of \S \ref{sss:0:sp}) is a dg vertex Poisson algebra.

%%%%%%%%%%%%%%%%%%%%%%%%%%%%%%%%%%%%%%%%%%%%%%%%%%%%%%%%%%%%%%%%%%%%%%%%%%%%%%%%
\subsubsection{Derived gluing of vertex Poisson algebras}

In \S \ref{s:co}, we introduce a vertex Poisson analogue 
of the derived gluing \eqref{eq:0:dgl} and define the category $\coMT$
of ``coisson Moore-Tachikawa varieties".
%Hamiltonian reduction \eqref{eq:0:dHr}
%and the relation to the \emph{coisson BRST complex}.

Let $\frkl$ be a dg Lie algebra.
In the vertex Poisson world, the corresponding notion of momentum map is 
given by a morphism  $\mu_{\cois}: J_{\infty}(\Sym(\frkl)) \to P$ of 
dg vertex Poisson algebras, called a \emph{coisson momentum map}.
The coisson BRST complex is defined to be 
\[
 \BRST_{\cois}(J_\infty(\frkl),P,\mu_{\cois}) := 
 (P \otimes \coCl(\frkl), d_{\cois}),
\]
where $\coCl(\frkl)$ denotes 
the \emph{Clifford vertex Poisson algebra} (Definition \ref{dfn:li:coCl}).
See Definition \ref{dfn:co:coBRST} for the detail of th coisson BRST complex.

%we set
%\[
% P \gq_{\mu}^{\bbL} J_\infty(\Sym(\frkl)) := 
% \CE(\frkl[[t]],\bbk) 
% \tbotimes^{\bbL}_{\CE(\frkl[[t]],J_\infty(\Sym(\frkl)))}
% \CE(\frkl[[t]],P)
%\]
%and call it the \emph{coisson Hamiltonian reduction}.
%For these two objects we have the same relation as \eqref{eq:0:HB}:
%\begin{equation}\label{eq:0:cHB}
% P \gq^{\bbL}_{\mu} J_\infty(\Sym(\frkl)) \dqiseq 
% \BRST_{\cois}(J_\infty(\frkl),P,\mu).
%\end{equation}
%See Proposition \ref{prp:co:HBcp} for the precise statement.

A vertex Poisson analogue of the gluing \eqref{eq:0:dgl} is defined 
in the following way: Let $\frkg_1$, $\frkg_2$ and $\frkg_3$ be the Lie algebras
of semisimple Lie groups $G_1$, $G_2$ and $G_3$.
Let also  $\mu_P: J_{\infty}(\Sym(\frkg_1 \oplus \frkg_2)) \to P$, 
$\mu_{P'}: J_{\infty}(\Sym(\frkg_2 \oplus \frkg_3)) \to P'$
be morphisms of dg vertex Poisson algebras.
Then we define the \emph{coisson gluing} to be 
\[
 P' \wtc P := 
 \BRST_{\cois}(J_\infty(\frkg_2),P^{\op} \otimes P',\mu)
 %\bigl(P^{\op} \otimes P'\bigr) \gq^{\bbL}_{\mu} J_\infty(\Sym(\frkg_2)).
\]
%It is a dg vertex Poisson algebra equipped with a morphism
%$J_{\infty}(\Sym(\frkg_1 \oplus \frkg_3)) \to P' \circ P$.
Here $P^{\op}$ denotes the \emph{opposite} dg vertex Poisson algebra of $P$.
See Definition \ref{dfn:co:glue} for the detail.

Using this operation, we introduce the category $\coMT$ 
(Definition \ref{dfn:co:coMT}).
An object is the same as that of the category $\MT$,
and a morphisms from $\frkg_1$ to $\frkg_2$ is a dg vertex Poisson algebra $P$
equipped with a morphism $J_\infty(\Sym(\frkg_1 \oplus \frkg_2)) \to P$
and satisfying some finiteness condition.
Composition of morphism is given by the coisson gluing.

For a dg vertex Poisson algebra $P$, 
we have a dg Poisson algebra $R^{\cois}_P := P/\Img(T)$,
where $T$ denotes the translation of $P$.
See Definition \ref{dfn:li:dgapa} for the detail.
The main statement in \S \ref{s:co} is that 
this construction induces a functor
\[
 R^{\cois}_{(-)}: \coMT \longto \MT.
\]
See Proposition \ref{prp:co:glue} and Theorem \ref{thm:co:glue} for the detail.

%%%%%%%%%%%%%%%%%%%%%%%%%%%%%%%%%%%%%%%%%%%%%%%%%%%%%%%%%%%%%%%%%%%%%%%%%%%%%%%%
\subsubsection{Derived gluing of dg vertex algebras}

In the final \S \ref{s:ch} we study the gluing for dg vertex algebras.
%The line of the argument is a similar one as those in the previous parts.

We denote by $V_k(\frkg)$ the universal affine vertex algebra at level $k$
for a finite dimensional semisimple Lie algebra $\frkg$ as in \S \ref{sss:0:42}.
In the vertex world, the corresponding notion of momentum map is 
a morphism $\mu_V: V_k(\frkg) \to V$ of dg vertex algebras, 
called a \emph{chiral momentum map}.
See Remark \ref{rmk:ch:chmm} for 
the difference of the terminology of \cite{A18} and ours.
Such a datum $(V,\mu_V)$ will be called a dg vertex algebra object 
in $\dgVMod{V_k(\frkg)}$.

For a dg vertex algebra object $(V,\mu_V)$ in $\dgVMod{V_k(\frkg)}$, 
we have the \emph{BRST complex}
\[
 \BRST(\wh{\frkg}_{k},V,\mu) := (V \otimes \Wedge^{\sinf}(\frkg),d_{\cl}),
\]
where $\Wedge^{\sinf}(\frkg)$ denotes the \emph{free fermionic vertex algebra}.
See Definition \ref{dfn:ch:BRST} for the detail.

For another dg vertex algebra object $(V',\mu_{V'})$ in $\dgVMod{V_l(\frkg)}$, 
we define the \emph{chiral gluing} $V' \wtc V \in \dgVMod{V_{k+l}(\frkg)}$ to be
\[
 V^{\op} \wtc V' := \BRST(\wh{\frkg}_{k+l},V \otimes V',\mu)
\]
with $\mu: V^{\op} \otimes V'$ defined by $\mu(a,b) := -\mu_V(a) +\mu_{V'}(b)$.
%Recall also the description of the Li filtration $F^{\bl}V^k(\frkg)$,
%Zhu's $C_2$-algebra $R_{V^k(\frkg)}$ in \S \ref{sss:li:vpa:uava}.

Using the chiral gluing, we introduce the category $\chMT$ 
which is a vertex algebra analogue of the category $\MT$.
An object is the same as that of $\MT$,
and a morphisms from $\frkg_1$ to $\frkg_2$ is a dg vertex algebra $V$
equipped with a chiral momentum map $V_k(\frkg_1) \otimes V_l(\frkg_2) \to V$
and satisfying some finiteness condition.
Composition of morphism is given by the chiral gluing.
See Definition \ref{dfn:ch:chMT} for the precise description.

We then have a functor $R_{(-)}: \chMT \to \MT$ 
induced by the procedure $V \mapsto R_V$ taking Zhu's $C_2$-algebra.
We also have a functor $\gr^F: \chMT \to \coMT$ induced by the procedure
$V \mapsto \gr^F V$ taking the associated graded space of the Li filtration.
They sit in a commutative diagram
\[
 \xymatrix{ \chMT \ar[r]^{\gr^F} \ar[d]_{R}& \coMT \ar[d]^{R^{\cois}} \\ 
              \MT \ar@{=}[r] & \MT}
\]
It would imply compatibility of gluing constructions in \eqref{diag:0:com}
(replacing $\VA$ by $\chMT$ and $\HS$ by $\MT$).
See Theorem \ref{thm:ch:main} for the detail.

%%%%%%%%%%%%%%%%%%%%%%%%%%%%%%%%%%%%%%%%%%%%%%%%%%%%%%%%%%%%%%%%%%%%%%%%%%%%%%%%
%%%%%%%%%%%%%%%%%%%%%%%%%%%%%%%%%%%%%%%%%%%%%%%%%%%%%%%%%%%%%%%%%%%%%%%%%%%%%%%%
\subsection{Global notation}
\label{ss:0:ntn}

Here is a list of global notations.
\begin{enumerate}[nosep]
\item
$\delta_{m,n}$ denotes the Kronecker delta.

\item
The symbol $\bbN$ denotes the set of non-negative integers.

\item
The words \emph{ring} and \emph{algebra} mean associative ones 
unless otherwise stated.

\item
The word \emph{dg} means differential graded.

%\item
%For a field $\bbk$, we denote by $\chr \bbk$ its characteristic.

%\item
%The symbol $\bbQ$ denotes the rational number field,
%and $\bbC$ denotes the complex number field.
%In most cases we denote categories and $\infty$-categories in sans-serif letters,
%and sometimes put the subscript $\infty$ to indicate an $\infty$-category.
%For example, in \S \ref{sss:derived-ring} 
%we will denote by $\sCom$ the category of simplicial commutative rings, 
%and by $\isCom$ the $\infty$-category of simplicial commutative rings.

\item
\label{i:ntn:inf}
On $\infty$-categories.
\begin{enumerate}[nosep,label=(\roman*)]
\item 
We follow \cite{Lu1} for the terminology on $\infty$-categories.
In particular, an \emph{$\infty$-category} is a simplicial set satisfying 
the weak Kan condition \cite[Definition 1.1.2.4]{Lu1}.

\item
For an $\infty$-category $\iC$, 
we write $x \in \iC$ to mean that $x$ is an \emph{object} of $\iC$,
i.e., a vertex of the simplicial set $\iC$.

\item
The \emph{homotopy category} \cite[\S 1.1.4]{Lu1} 
of an $\infty$-category $\iC$ is denoted by $\h \iC$.

\item
The $\infty$-category $\iS$ of \emph{spaces} \cite[Definition 1.2.16.1]{Lu1} is
defined to be the simplicial nerve of the simplicial category of Kan complexes.
Its object will be called a space.

\item
We denote $\iH := \h \iS$.
We have the notion of \emph{homotopy group} $\pi_n X$ for $X \in \iH$.

\item
For objects $x,y$ of an $\infty$-category $\iC$, 
we denote by $\Map_{\iC}(x,y) \in \iH$ the \emph{mapping space} from 
$x$ to $y$ \cite[Definition 1.2.2.1]{Lu1}.
It is the homotopy type of the space representing the maps $x \to y$ 
in the homotopy category of the simplicial category attached to $\iC$.
We also denote $\Hom_{\iC}(x,y) := \pi_0 \Map_{\iC}(x,y)$.

\item
For an $\infty$-category $\iC$ and $x,y,z \in \iC$, 
we denote the \emph{pullback} by $x \times_y z$
%, and the \emph{pushout} $x \coprod_y z$ 
if they exist \cite[\S 4.4.2]{Lu1}.

\item
For an $\infty$-category $\iC$, its \emph{opposite $\infty$-category} 
\cite[\S 1.2.1]{Lu1} is denoted by $\iC^{\op}$.

\item
An ordinary category is identified with its nerve
and regarded as an $\infty$-category.
We use sans-serif font to denote ordinary categories
such as the category $\dgVec$ of complexes.

\item
\label{i:ntn:inf:dg}
A dg category is identified 
with its differential graded nerve \cite[\S 1.3.1]{Lu2}
and regarded as an $\infty$-category.
We use bold letters to denote dg categories
such as the dg category $\ddgVec$ of complexes.
\end{enumerate}
\end{enumerate}

\begin{ack}
The author thanks Tomoyuki Arakawa for his detailed explanation on \cite{A18}
in the intensive course at Nagoya University in November, 2019.
%This article is a ``report" of the course.
\end{ack}

\newpage

%%%%%%%%%%%%%%%%%%%%%%%%%%%%%%%%%%%%%%%%%%%%%%%%%%%%%%%%%%%%%%%%%%%%%%%%%%%%%%%%
%%%%%%%%%%%%%%%%%%%%%%%%%%%%%%%%%%%%%%%%%%%%%%%%%%%%%%%%%%%%%%%%%%%%%%%%%%%%%%%%
%%%%%%%%%%%%%%%%%%%%%%%%%%%%%%%%%%%%%%%%%%%%%%%%%%%%%%%%%%%%%%%%%%%%%%%%%%%%%%%%
\section{Poisson algebras in dg setting}
\label{s:sp}

We work over a fixed field $\bbk$ of characteristics $0$.

%%%%%%%%%%%%%%%%%%%%%%%%%%%%%%%%%%%%%%%%%%%%%%%%%%%%%%%%%%%%%%%%%%%%%%%%%%%%%%%%
%%%%%%%%%%%%%%%%%%%%%%%%%%%%%%%%%%%%%%%%%%%%%%%%%%%%%%%%%%%%%%%%%%%%%%%%%%%%%%%%
\subsection{Dg convention}
\label{ss:sp:dg}

For the definiteness, let us start with our dg convention.

%%%%%%%%%%%%%%%%%%%%%%%%%%%%%%%%%%%%%%%%%%%%%%%%%%%%%%%%%%%%%%%%%%%%%%%%%%%%%%%%
\subsubsection{Graded and dg linear spaces}

Let $\cVec$ be the category of linear spaces and linear maps over $\bbk$.
We denote by $\Hom_{\bbk}(-,-) = \Hom_{\cVec}(-,-)$ 
the linear space of linear maps, 
by $V^*:= \Hom_{\bbk}(V,\bbk)$ the linear dual of $V \in \cVec$,
and by $\otimes_{\bbk}$ the tensor product of linear spaces.
The braiding (or the commutativity) isomorphism on the tensor product 
is an isomorphism  $V \otimes W \simto W \otimes V$
in $\cVec$ given by $v \otimes w \mapsto w \otimes v$.
These give $\cVec$ a structure of $\bbk$-linear unital symmetric monoidal category
with the unit $\bbk$.

\begin{dfn}\label{dfn:dg:gVec}
\begin{enumerate}[nosep]
\item 
A \emph{graded linear space} is a linear space 
%$V=V_{\ol{0}} \oplus V_{\ol{1}}$ 
equipped with an extra $\bbZ$-grading.
We express the $\bbZ$-grading by superscript as 
$V^{\bl}=\bigoplus_{n \in \bbZ} V^n$.
An element $v \in V^n$ for some $n \in \bbZ$ will be called \emph{homogeneous},
and for such an element we denote $\abs{v}:=n$.

\item
\label{i:dg:gVec:mor}
A \emph{morphism} of graded linear spaces is a homogeneous linear map of degree $0$.
For such a morphism $f: V \to W$, 
we denote $f^n := \rst{f}{V^n}: V^n \to W^n$.
We denote by $\gVec$ the category of graded linear space and their morphisms.
%Thus we have 
%\[
% \Hom_{\gVec}(V,W)=\Hom_{\bbk}(V,W)^0.
%\]
%Thus we have $f=\sum_{n \in \bbZ}f^n$.

%\item
%For $n \in \bbZ$ and a graded linear spaces $V$ and $W$, 
%we denote by $\Hom_{\bbk}(V,W)^n$
%the linear space of homogeneous linear maps $f$ of degree $n$,
%i.e., $f(V^m) \subset W^{m+n}$ for any $m \in \bbZ$.
%We also define a graded linear space $\uHom(V,W)$ by 
%\[
% \uHom(V,W) := \bigoplus_{n \in \bbZ}\Hom_{\bbk}(V,W)^n
%\]
%and call it the \emph{internal hom}.

\item 
%\label{i:dg:gVs:[1]}
For a graded linear space $V$, we denote $V[1]$
the graded linear space with $V[1]^n := V^{n+1}$ for each $n \in \bbZ$.
\end{enumerate}
\end{dfn}

\begin{dfn}\label{dfn:dg:dgVec}
\begin{enumerate}[nosep]
\item 
A \emph{dg linear space} or a \emph{complex} is a pair $(V,d)$ consisting of 
a graded vector space $V$ and a morphism $d: V \to V[1]$ of graded vector spaces
satisfying $d^2=0$.
The morphism $d$ is called the \emph{differential},
and the $\bbZ$-grading on $V$ is called the \emph{degree} 
or the \emph{cohomological degree}.

\item
A \emph{morphism} of complexes is a morphism of graded linear spaces
which respects the differentials.
In particular, a morphism $f: V \to W$ of complexes preserves 
the $\bbZ$-grading: $f(V^n) \subset W^n$ for any $n \in \bbZ$.

\item
We denote by $\dgVec$ the category of complexes and their morphisms.
\end{enumerate}
\end{dfn}

%Let us spell out the definition of the differential.
%The differential $d$ of a complex $(V,d)$ is a homogeneous linear 
%endomorphism of cohomological degree $1$ satisfying $d^2=0$..
\begin{rmk*}
\begin{enumerate}[nosep]
\item 
We denote a complex simply by $V = (V,d)$ if no confusion may occur.
We also denote $(V^{\bl},d_V)$ to emphasize the $\bbZ$-grading 
and that the differential is attached to $V$.

\item
As for the differential $d$ of a complex $V$, 
we have $d^n: V^n \to V^{n+1}$ and $d^{n+1} d^n = 0$
using the notation in Definition \ref{dfn:dg:gVec} \eqref{i:dg:gVec:mor},

\item
Hereafter we regard a graded linear space as a complex with trivial differential.
Thus we regard $\gVec \subset \dgVec$ as a full subcategory.
\end{enumerate}
\end{rmk*}

A category enriched over $\dgVec$ will be called a \emph{dg category}.
The complexes of morphisms in a dg category 
are called the \emph{hom complexes}.
A typical example of a dg category is the dg category of complexes.
Let us recall the precise definition.

\begin{dfn}\label{dfn:dg:uHom}
\begin{enumerate}
\item \label{i:dg:uHom} 
For $V, W \in \dgVec$, we define $\uHom(V,W)$ to be the complex of which
\begin{itemize}[nosep]
\item
the underlying graded linear space is given by
\[
 \uHom(V,W) = \bigoplus_{n \in \bbZ} \uHom(V,W)^n, \quad 
 \uHom(V,W)^n := \prod_{i \in \bbZ}\Hom_{\bbk}(V^i,W^{i+n}), 
\]
\item
and the differential is given by $d f:= d_W \circ f - (-1)^n f \circ d_V$
for $f \in \uHom(V,W)^n$.
\end{itemize}
We call $\uHom(V,W)$ the \emph{internal hom} in $\dgVec$.
An element of $\uHom(V,W)^n$ will be called 
a \emph{homogeneous linear map of cohomological degree $n$}.

\item
We denote by $\ddgVec$ the dg category of which the objects are complexes
and the hom complexes are given by $\uHom(-,-)$.
\end{enumerate}
\end{dfn}

Let us recall the standard symmetric monoidal structure on $\dgVec$.

\begin{dfn}\label{dfn:dg:dgVmon}
Let $V,W \in \dgVec$. %be $\bbZ$-graded linear superspaces.
\begin{enumerate}[nosep]
\item
\label{i:dg:gVmon:tens}
The \emph{tensor product} $V \otimes W$ in $\dgVec$ is a complex of which
\begin{itemize}[nosep]
\item 
the underlying linear space is $V \otimes_{\bbk} W$,
\item
the $\bbZ$-grading is given by 
\[
 (V \otimes W)^n := \tboplus_{r+s=n} V^r \otimes_{\bbk} W^s,
\]
\item
and the differential is given by 
\[
 d_{V \otimes W}(v \otimes w) := d_V v \otimes w + (-1)^{\abs{v}} v \otimes d_W w
\]
for any homogeneous $v \in V$ and any $w \in W$.
\end{itemize}

\item
For homogeneous linear maps $f: V \to V'$ and $g: W \to W'$ of complexes,
we define the linear map $f \otimes g: V \otimes W \to V' \otimes W'$ by
\[
 (f \otimes g)(v \otimes w) := (-1)^{\abs{g}\abs{v}} f(v) \otimes g(w)
\]
for homogeneous $v \in V$ and $w \in W$.
This rule will be called the \emph{Koszul sign rule}.
In particular, we have a tensor product for a morphism of complexes.

\item
\label{i:dg:dgVmon:braid}
and the \emph{braiding isomorphism} in $\dgVec$ is defined to be 
\[
 V \otimes W \longsimto W \otimes V, \quad 
 v \otimes w \longmapsto (-1)^{\abs{v}\abs{w}} w \otimes v
\]
for homogeneous $v \in V$ and $w \in W$.
\end{enumerate}
\end{dfn}

These give the category $\dgVec$ a structure of unital symmetric monoidal
category with the unit $\bbk$, which is denoted by $\dgVec^{\otimes}$.
We also have the notion of a \emph{monoidal dg category},
and the dg category $\ddgVec$ of complexes has a structure of 
unital symmetric monoidal dg category, which is denoted by $\ddgVec^{\otimes}$.

Let us recall the shift functor on $\dgVec$:

\begin{dfn*}
\begin{enumerate}[nosep]
\item 
We denote by $\bbk[1]$ the complex whose underlying graded linear space is 
$(\bbk[1])^n = \delta_{n,0} \bbk$ and whose differential is $0$.

\item
The \emph{shift functor} $[1]$ on $\dgVec$ is given by 
$V \mapsto V[1] := \bbk[1] \otimes V$. 
In particular, we have $d^n_{V[1]}= -d^{n+1}_V$ for any $n \in \bbZ$.
The inverse is denoted by $[-1]$, and for $m \in \bbZ$ 
the $m$-th repetition is denoted by $[m]$.
\end{enumerate}
\end{dfn*}

Here are some basic constructions of complexes.

\begin{eg}\label{eg:dg:TS}
Let $V$ be a complex.
\begin{enumerate}[nosep]
\item \label{i:dg:TS:T}
We define the complex $T(V)$ by 
\[
 T(V) := \tboplus_{p \in \bbN} V^{\otimes p}
\]
and call it the \emph{tensor space} of $V$.
In particular, the $\bbZ$-grading of 
$v_1 \otimes \cdots \otimes v_p \in V^{\otimes p}$ 
is given by $\abs{v_1 \otimes \cdots \otimes v_p} := \sum_{i=1}^p \abs{v_i}$.
We call the $\bbN$-grading given by the tensor power $p$ the \emph{weight}.

\item
We define the \emph{symmetric tensor space} $\Sym(V)$  of $V$ to be the complex
\[
 \Sym(V) := \tboplus_{p \in \bbN} V^{\otimes p}/\frkS_p.
\]
where the $p$-th symmetric group $\frkS_p$ acts on $V^{\otimes p}$ by permutation.
Similarly as in \eqref{i:dg:TS:T},
we call the $\bbN$-grading given by the tensor power $p$ the \emph{weight}.
\end{enumerate}
\end{eg}

For later reference, we give:

\begin{eg}\label{eg:dg:dgW}
\begin{enumerate}[nosep]
\item
\label{i:dg:S1=W:1}
Let $\bbk[1]$ be the shifted one-dimensional linear space
regarded as a complex with trivial differential.
Consider the shifted symmetric tensor space $\Sym(\bbk[1])$.
Its weight $n$ component is $(\bbk[1])^{\otimes n}/\frkS_n$,
where the $n$-th symmetric group $\frkS_n$ acts by permutation.
Using the braiding isomorphism in $\dgVec$ 
(Definition \ref{dfn:dg:dgVmon} \eqref{i:dg:dgVmon:braid}),
we can identify
\[
 (\bbk[1])^{\otimes n} \simeq \text{the signature representation}
\]
as $\frkS_n$-representations.
Thus $\Sym(\bbk[1])$ can be regarded as 
a sequence of signature representations of symmetric groups.

\item
Let $V$ be a complex and 
consider the shifted symmetric tensor space $\Sym(V[1])$.
By the consideration in \eqref{i:dg:S1=W:1}, we can identify 
\[
 (V[1])^{\otimes p} \simeq \Wedge^p(V)
\]
as linear spaces, 
where the right hand side denotes the exterior product space.
Abusing terminology, we call the complex $\Sym(V[1])$
the \emph{exterior product space} of $V$.
\end{enumerate}
\end{eg}

We close this part by recalling the cohomology and the cone.

\begin{dfn}\label{dfn:dg:coh}
\begin{enumerate}[nosep]
\item 
The \emph{cohomology} $H(V,d)$ of a complex $(V,d)$
is defined to be the graded linear space given by
$H^n(V,d) := \Ker(d: V^n \to V^{n+1})/\Img(d: V^{n-1} \to V^{n})$.
We denote $H^{\bl}(V,d)$ to emphasize the $\bbZ$-grading.
If confusion will not occur, we simply denote $H(V):=H(V,d)$.

\item
A complex $V$ is called \emph{acyclic} if $H(V)=0$.

\item
For a morphism $f:V \to W$ in $\dgVec$,
the induced morphism $H(V) \to H(W)$ in $\gVec$ is denoted by $H(f)$.
Thus, we have a functor $H: \dgVec \to \gVec$.

\item
A morphism $f: V \to W$ is called a \emph{quasi-isomorphism}
if the morphism $H(f): H(V) \to H(W)$ is an isomorphism in $\gVec$.
\end{enumerate}
\end{dfn}

\begin{dfn}\label{dfn:dg:cone}
For a morphism $f: V \to W$ of complexes, 
we define the \emph{mapping cone} of $f$ to be the complex
\[
 \Cone(f) := V[1] \oplus W, \quad 
 d_{\Cone(f)} := 
 \begin{bmatrix}d_{V[1]} & 0 \\ f[1] & d_{W}\end{bmatrix}.
\]
\end{dfn}

Let us also recall the following standard fact:

\begin{fct}\label{fct:dg:qis}
Let $f: V \to W$ be a morphism in $\dgVec$.
Then $f$ is a quasi-isomorphism 
if and only if the mapping cone $\Cone(f)$ is acyclic.

In particular, for a complex $V$, 
the mapping cone $\Cone(\id_V)$ of the identity $\id_V: V \to V$ is acyclic.
\end{fct}

%%%%%%%%%%%%%%%%%%%%%%%%%%%%%%%%%%%%%%%%%%%%%%%%%%%%%%%%%%%%%%%%%%%%%%%%%%%%%%%%
%%%%%%%%%%%%%%%%%%%%%%%%%%%%%%%%%%%%%%%%%%%%%%%%%%%%%%%%%%%%%%%%%%%%%%%%%%%%%%%%
\subsection{Algebraic structures in dg setting}
\label{ss:sp:dga}

We collect here the terminology of algebraic structures in the dg setting.
The most appropriate language here is the theory of operad,
but we avoid to use it since we don't need such a full generality. 
We continue to work over a field $\bbk$ with characteristics $0$.

Recall the symmetric monoidal structures $\otimes$ on 
$\dgVec$ (Definition \ref{dfn:dg:dgVmon}).

%%%%%%%%%%%%%%%%%%%%%%%%%%%%%%%%%%%%%%%%%%%%%%%%%%%%%%%%%%%%%%%%%%%%%%%%%%%%%%%%
\subsubsection{Dg algebras}

We start with ring objects in $\dgVec$.

\begin{dfn}\label{dfn:dga:dga}
\begin{enumerate}[nosep]
\item 
A \emph{dg algebra} (\emph{dga} for short) is a unital associative ring object 
in the  monoidal category $\dgVec$. 
In other words, it is a triple $(A, \cdot, u)$ consisting of 
\begin{itemize}
\item 
a complex $A=(A,d_A)$,
\item
a morphism $\cdot: A \otimes A \to A$ in $\dgVec$ 
called the \emph{multiplication}, and 
\item
a morphism $u: \bbk \to A$ in $\dgVec$ called the \emph{unit} (\emph{morphism})
\end{itemize}
satisfying the standard axioms of associativity and unitality.
We often omit the symbol $\cdot$ of the multiplication 
and denote $a b := a \cdot b$ for $a,b \in A$.
We also denote the image of the unit $1_{\bbk}$ of the field $\bbk$ under 
the unit morphism by $1_A :=u(1_{\bbk}) \in A$ and call it 
the \emph{unit} (\emph{element}) of $A$.

\item
A \emph{morphism} of dgas is a morphism in $\dgVec$ 
which respects the ring structures.
We denote by $\dga$ the category of dgas and their morphisms
(the letter $\mathsf{u}$ indicates that we consider unital objects).

\item
The \emph{commutator} on a dga $A$ is denoted by $[-,-]$.
It is defined by 
\[
 [a,b] := a b - (-1)^{\abs{a} \abs{b}} b a
\]
for homogeneous elements $a,b \in A$ 
and is extended by linearity.
\end{enumerate}
\end{dfn}

\begin{rmk*}
\begin{enumerate}[nosep]
\item 
We include unital condition in the definition.

\item
Let us recall a more down-to-earth definition of a dga.
Let $A$ be a dga in the above sense.
Since the multiplication $\cdot: A \otimes A \to A$ is a morphism in $\dgVec$,
it commutes with the differentials $d_{A \otimes A}$ and $d_A$.
Recalling the differential of the tensor product of complexes
(Definition \ref{dfn:dg:dgVmon}), we have 
$d_{A \otimes A}(a b)=(d_A a) \otimes b + (-1)^{\abs{a}}a \otimes d_A b$
for $a,b \in A$.
Then the commutativity of $\cdot$ and the differentials can be expressed as 
\begin{equation}\label{eq:dga:Leib}
 d_A(a b) = (d_A a)b+(-1)^{\abs{a}}a(d_A b).
\end{equation}
This is nothing but the \emph{Leibniz rule},
and the differential $d_A$ is a derivation with respect to the multiplication
(see Definition \ref{dfn:dga:der} below).

Thus, we can restate the definition as:
A dga is a triple  $(A, \cdot, 1_A)$ consisting of 
\begin{itemize}[nosep]
\item 
a complex $A=(A,d_A)$,
\item
an associative multiplication $\cdot: A \times A \to A$
which preserves the $\bbZ$-grading and 
satisfies the Leibniz rule \eqref{eq:dga:Leib}, and 
\item
a unit element $1_A \in A^0$ with respect to the multiplication $\cdot$.
\end{itemize}
\end{enumerate}
\end{rmk*}

\begin{eg}\label{eg:dga:dga}
Let $V$ be a complex.
\begin{enumerate}[nosep]
\item
\label{i:dga:dga:T}
The \emph{tensor algebra} $T(V)$ is a typical example of a dga.
The underlying complex is the tensor space 
(Example \ref{eg:dg:TS}),
and the multiplication is given by the concatenation of the tensor:
$(v_1 \otimes \cdots \otimes v_p) \cdot (w_1 \otimes \cdots \otimes w_q) := 
 v_1 \otimes \cdots \otimes v_p \otimes w_1 \otimes \cdots \otimes w_q$.

\item 
\label{i:dga:dga:E}
The \emph{endomorphism algebra} $\uEnd(V)$ is another typical example. 
The underlying graded linear space is 
$\bigoplus_{n \in \bbZ} \End_{\bbk}(V)^n=\bigoplus_{n \in \bbZ} \Hom_{\bbk}(V,V)^n$,
the differential is $d_{\uEnd(V)}f=d_V \circ f-(-1)^{\abs{f}}f \circ d_V$
(see Definition \ref{dfn:dg:uHom} \eqref{i:dg:uHom}),
and the multiplication is given by the composition.
In particular, we have the commutator $[-,-]$ on $\uEnd(V)$.
%, which will be used freely in the explanation of 
%dg vertex algebras (\S \ref{ss:co;dva}).
\end{enumerate}
\end{eg}

Let us also recall the monoidal structure on $\dga$.

\begin{dfn}\label{dfn:dga:dgamon}
The \emph{tensor product} $A \otimes B$ of dgas $A$ and $B$ 
is defined to be the graded algebra whose underlying graded linear space is 
$A \otimes B \in \dgVec$ and whose multiplication is given by 
\[
 (a_1 \otimes b_1) \cdot (a_2 \otimes b_2) := 
 (-1)^{\abs{a_2} \abs{b_1}} a_1 a_2 \otimes b_1 b_2
\]
for homogeneous elements.
The tensor product of morphisms of dgas are defined by the same formula.

This tensor product gives $\dga$ a structure of a monoidal category.
\end{dfn}

We also have the notion of modules over a dga:

\begin{dfn*}%\label{ntn:dgMod}
Let $A$ be a dga.
\begin{enumerate}[nosep]
\item 
A \emph{left dg $A$-module} $M$ is a complex equipped with 
a morphism $.: A \otimes M \to M$ in $\dgVec$ called 
the (left) \emph{action} satisfying the standard axioms of 
associativity and unitality.

\item
A \emph{morphism} $f: M \to N$ of left dg $A$-modules $M$ and $N$ 
is a morphism in $\dgVec$ which respects the $A$-module structures.

\item
We denote the category of left dg $A$-modules by $\dgMod{A}$.
\end{enumerate}
\end{dfn*}

Explicitly, a left dg $A$-module $M$ is a complex equipped with 
a homogeneous bilinear map $.: A \times M \to M$ of cohomological degree 0
satisfying the Leibniz rule 
\[
 d_M(a.m)=d_A(a).m+(-1)^{\abs{a}}a.d_M(m)
\]
for $a \in A$ and homogeneous $m \in M$.

For left dg $A$-modules $M$ and $N$, we define the complex $\uHom_A(M,N)$ by
\begin{align*}
&\uHom_A(M,N) := \tboplus_{n \in \bbZ}\Hom_A(M,N[n]),  \\
&\Hom_A(M,N[n]):= \{f \in \Hom(M,N[n]) \mid 
 \text{$f$ respects dg $A$-module structures of $M$ and $N[n]$}\}.
\end{align*}
Thus $\uHom_A(M,N)$ is a subcomplex of 
$\uHom(M,N)=\bigoplus_{n \in \bbZ}\uHom(M,N[n])$.
Here we used the fact that 
the shift $M[n]$ is naturally a left dg $A$-module.
The complex $\uHom_A(M,N)$ is naturally a left dg $A$-module.
Thus we have the following definition:

\begin{dfn}\label{dfn:dga:uHom_A}
Let $A$ be a dga.
\begin{enumerate}[nosep]
\item 
For left dg $A$-modules $M,N$, 
we call $\uHom_A(M,N)$ the \emph{dg $A$-module of morphisms of $A$-modules}.
\item
We denote by $\ddgMod{A}$ the resulting dg category of left dg $A$-modules.
\end{enumerate}
\end{dfn}

We also have the notions of \emph{right dg $A$-modules}, \emph{dg $A$-bimodules}
and \emph{left/right/both-side dg ideals}.
The details are omitted.

%%%%%%%%%%%%%%%%%%%%%%%%%%%%%%%%%%%%%%%%%%%%%%%%%%%%%%%%%%%%%%%%%%%%%%%%%%%%%%%%
\subsubsection{Commutative dg algebras}

Next we introduce commutative algebra objects in $\dgVec$.
%, for the use in the derived algebraic geometry, let us give 

\begin{dfn*}%\label{dfn:dgs:csa}
\begin{enumerate}[nosep]
\item
%A \emph{commutative graded algebra} (resp.\ \emph{commutative dga}) 
%is a graded algebra (resp.\ dga) whose commutator always vanishes.
A \emph{commutative dga} (\emph{cdga} for short)
is a dga whose commutator always vanishes.

\item
We denote by $\cdga$ the subcategory of $\dga$ spanned by cdgas
(the letter $\mathsf{u}$ indicates that we consider unital objects). 
\end{enumerate}
\end{dfn*}

\begin{eg}\label{eg:dga:SW}
Let $V$ be a complex.
\begin{enumerate}[nosep]
\item 
\label{i:dg:SW:S}
Recall the symmetric tensor space $\Sym(V)$ (Example \ref{eg:dg:TS}).
It is a cdga with the multiplication induced by that on 
the tensor algebra $T(V)$ (Example \ref{eg:dga:dga} \eqref{i:dga:dga:T}).
We denote the resulting cdga by the same symbol $\Sym(V)$
and call it the \emph{symmetric tensor algebra} of $V$.

\item
The exterior product space $\Sym(V[1])$ (Example \ref{eg:dg:dgW})
is also a cdga with the multiplication 
induced by that on the tensor algebra $T(V)$.
We denote the resulting cdga by the same symbol $\Sym(V[1])$
and call it the \emph{exterior algebra} of $V$.
\end{enumerate}
\end{eg}

The category $\cdga$ inherits the monoidal structure of $\dga$ 
(Definition \ref{dfn:dga:dgamon}),
and the braiding isomorphisms on $\dgVec$ (Definition \ref{dfn:dg:dgVmon})
makes $\cdga$ a symmetric monoidal category.

We also have the notion of modules over cdgas.

\begin{dfn*}%\label{ntn:dgMod}
Let $A$ be a cdga.
\begin{enumerate}[nosep]
\item 
A \emph{dg $A$-module} $M$ is a left dg $A$-module
where $A$ is regarded as a dga.
We denote the category of dg $A$-modules by $\dgMod{A}$.

\item
The category $\dgMod{A}$ has a structure of a symmetric monoidal category.
We denote the tensor product by $\otimes_A$.
\end{enumerate}
\end{dfn*}

We also have the dg category of dg $A$-modules,
whose hom complex from $M$ to $N$ is defined to be the subcomplex 
$\uHom_A(M,N) \subset \uHom(M,N)$ of morphisms 
respecting the dg $A$-module structures.

\begin{ntn}\label{ntn:dga:uHom_A}
For a cdga $A$,
we denote by $\ddgMod{A}$ the dg category of dg $A$-modules,
and denote by $\uHom_A(-,-)$ the hom complex in $\ddgMod{A}$.
\end{ntn}

Let us introduce some dg modules over cdgas
which will appear repeatedly in the following subsections.

\begin{eg}\label{eg:dga:TS}
Let $A$ be a cdga and $M$ be a dg $A$-module.
\begin{enumerate}[nosep]
\item 
We denote the iterated tensor products of $M$ as 
$M^{\otimes_A 2} := M \otimes_A M$,
$M^{\otimes_A 3} := M \otimes_A M \otimes_A M$ and so on.
We define the dg $A$-module $T_A(M)$ by 
\[
 T_A(M) := \tboplus_{p \in \bbN}M^{\otimes_A p}
\]
and call it the \emph{tensor algebra} of $M$ over $A$.
\item
We define the dg $A$-module $\Sym_A(M)$ by 
\[
 \Sym_A(M) := \tboplus_{p \in \bbN} M^{\otimes_A p}/\frkS_n,
\]
where the $p$-th symmetric group $\frkS_p$ acts on $M^{\otimes_A p}$
by permutation.
We call it the \emph{symmetric tensor algebra} of $M$ over $A$.
\end{enumerate}
These dg $A$-modules have the extra $\bbN$-grading given by the tensor power $p$.
We call this extra grading the \emph{weight} grading, 
following the terminology in Example \ref{eg:dg:TS}.
\end{eg}

%%%%%%%%%%%%%%%%%%%%%%%%%%%%%%%%%%%%%%%%%%%%%%%%%%%%%%%%%%%%%%%%%%%%%%%%%%%%%%%%
\subsubsection{Dg Lie algebras}

We finally introduce Lie algebra objects.
% for the explanation of Poisson algebras.

\begin{dfn}\label{dfn:dga:dgla}
\begin{enumerate}[nosep]
\item
A \emph{graded Lie algebra} is a Lie algebra object 
in the symmetric monoidal category $\gVec$.
In other words, it is a graded linear space $\frkl$ 
together with a bilinear map $[-.-]: \frkl \times \frkl \to \frkl$ 
called the \emph{Lie bracket} satisfying
\begin{enumerate}[nosep, label=(\roman*)]
\item 
$[x,y]=(-1)^{1+\abs{x} \abs{y}}[y,x]$ for homogeneous $x,y \in \frkl$, and 
\item
$(-1)^{\abs{z}\abs{x}}[x,[y,z]]+(-1)^{\abs{x}\abs{y}}[y,[z,x]]
+(-1)^{\abs{y}\abs{z}}[z,[x,y]]=0$ 
for homogeneous $a,b,c \in \frkl$.
\end{enumerate}
%We have a natural notion of morphisms of graded Lie algebras, and 
%denote by $\gla$ the category of graded Lie algebras and their morphisms.

\item
A \emph{dg Lie algebra} is a Lie algebra object 
in the symmetric monoidal category $\dgVec$.
In other words, it is a triple $(\frkl,d,[-,-])$ consisting of 
a graded linear space $\frkl$, a morphism $d \in \Hom_{\gVec}(\frkl,\frkl[1])$, 
and a bilinear map $[-.-]: \frkl \times \frkl \to \frkl$ such that 
\begin{enumerate}[nosep, label=(\roman*)]
\item
$(\frkl,d)$ is a complex, 
\item 
$(\frkl,[-,-])$ is a graded Lie algebra, and 
\item
$d [x,y]=[d x,y]+(-1)^{\abs{x}}[x, d y]$ for homogeneous $x,y \in \frkl$.
\end{enumerate}
We have a natural notion of morphisms of dg Lie algebras, and 
denote by $\dgla$ the category of dg Lie algebras and their morphisms.
\end{enumerate}
\end{dfn}

Let us give an example of a dg Lie algebra arising from a dga.
Recall that we impose 

\begin{eg}\label{eg:dga:dgla:uEnd}
Let $A=(A^{\bl},d_A,\cdot)$ be a dga, 
and $[-,-]$ be the commutator in $A$ (Definition \ref{dfn:dga:dga}).
Then the triple $A_L := (A^{\bl},d_A,[-,-])$ is a dg Lie algebra.
The correspondence $A \mapsto A_L$ gives rise to a functor
\[
 (-)_L: \dga \longto \dgla.
\]
In particular, the endomorphism algebra $\uEnd(V)$ of a complex $V$
(Example \ref{eg:dga:dga} \eqref{i:dga:dga:E}) 
has a structure of dg Lie algebra.
\end{eg}

Let us now recall the notion of modules over a dg Lie algebra.

\begin{dfn*}
Let $\frkl$ be a dg Lie algebra.
\begin{enumerate}[nosep]
\item 
A \emph{dg $\frkl$-module} is a complex $M$ equipped with 
a morphism $\rho: \frkl \to \uEnd(M)$ of dg Lie algebras, where $\uEnd(M)$ 
is regarded as a dg Lie algebra by Example \ref{eg:dga:dgla:uEnd}.

\item
A \emph{morphism} of dg $\frkl$-modules is a morphism of complexes
which respects the dg $\frkl$-module structures.

\item
The category of dg $\frkl$-modules is denoted by $\dgMod{\frkl}$.
It is naturally a symmetric monoidal category. 
The tensor product is denoted by $\otimes_{\frkl}$.
\end{enumerate}
\end{dfn*}

For a dg $\frkl$-module $M$,
the action of $x \in \frkl$ on $m \in M$ 
is denoted by $x.m := \rho(x)(m)$.

For the explanation of another example of a dga, we introduce:

\begin{dfn}\label{dfn:dga:der}
Let $A$ be a dga and $M$ be a left $A$-module.
\begin{enumerate}[nosep]
\item 
Let $n \in \bbZ$.
A \emph{derivation of degree $n$ from $A$ to $M$}, 
or an \emph{$n$-derivation from $A$ to $M$},  
is a morphism $\theta \in \Hom_{\dgVec}(A,M[n])$ such that 
for any homogeneous $a,b \in A$ we have 
\[
 \theta(a b) = (-1)^{(\abs{a}+n)\abs{b}} b.\theta(a)+(-1)^{n \abs{a}} a.\theta(b).
\]
%A $0$-derivation will be simply called a \emph{derivation}.

\item
\label{i:dga:der:2}
For $n \in \bbZ$,
we denote by $\Der(A,M)^n$ the set of $n$-derivations from $A$ to $M$, 
which is naturally a linear space.
We define the subcomplex $\Der(A,M) \subset \uHom(A,M)$ to be 
\[
 \Der(A,M):=\tboplus_{n \in \bbZ} \Der(A,M)^n.
\]
The differential is given by 
$d_{\Der(A,M)}\theta=d_{\uHom(A,M)}\theta
 =d_M \theta-(-1)^{\abs{\theta}}\theta d_A$.
Then $\Der(A,M)$ is naturally a left dg $A$-module,
and we call it the \emph{left dg $A$-module of derivations}.

\item
\label{i:dga:der:3}
For $M=A$, we denote $\Der(A)^n:=\Der(A,A)^n$ and $\Der(A):=\Der(A,A)$.
An element of $\Der(A)^n$ is called an \emph{$n$-derivation on $A$}.
%A $0$-derivation is just called a \emph{derivation}.

%\item
%A graded algebra equipped with a $0$-derivation is called a \emph{differential algebra}.
\end{enumerate}
\end{dfn}

\begin{rmk*}
\begin{enumerate}
\item
In the item \eqref{i:dga:der:2} we can check 
$d_{\uHom(A,M)}\theta \in \Der(A,M)$ by
\begin{align*}
(d_{\uHom(A,M)}\theta)(a b)
&= d_M \theta(a b) - (-1)^{n}\theta(d_A(a b)) \\
&= d_M \bigl((-1)^{(\abs{a}+n)\abs{b}}b.\theta(a)+(-1)^{n\abs{a}}a.\theta{b}\bigr)
   - (-1)^{n} \theta\bigl((d_A a) b+(-1)^{\abs{a}}a(d_A b)\bigr) \\
&= (-1)^{(\abs{a}+n)\abs{b}}
   \bigl((d_A b).\theta(a)+(-1)^{\abs{b}}b.d_M\theta(a)\bigr) 
  +(-1)^{n\abs{a}}\bigl((d_A a).\theta(b)+(-1)^{\abs{a}} a.d_M \theta(b)\bigr) \\
&\phantom{=}
  -(-1)^{n} \bigl((-1)^{(\abs{a}+n+1)\abs{b}} b.\theta(d_A a)
                            +(-1)^{n(\abs{a}+1)}(d_A a).\theta(b)\bigr) \\
&\phantom{=}
 -(-1)^{\abs{a}+n}\bigl((-1)^{(\abs{a}+n)(\abs{b}+1)}(d_A b).\theta(a)
                                  +(-1)^{n\abs{a}}a.(\theta d_A b)\bigr) \\
&= (-1)^{(\abs{a}+n+1)\abs{b}}b.\bigl(d_M \theta(a)-(-1)^n \theta(d_A a)\bigr)
  +(-1)^{(n+1)\abs{a}} a.\bigl(d_M \theta(b)-(-1)^n \theta(d_A b)\bigr) \\
&= (-1)^{(\abs{a}+n+1)\abs{b}} b.(d_{\uHom(A,M)}\theta)(a)
  +(-1)^{(n+1)\abs{a}} a.(d_{\uHom(A,M)}\theta)(b)
\end{align*}
for $\theta \in \Der(A,M)^n$ and homogeneous $a,b \in A$.

\item 
For a cdga $A$, 
we can rephrase the definition in \eqref{i:dga:der:3} as follows.
An $n$-derivation $\theta$ on $A$ is a homogeneous linear endomorphism on $A$ 
of degree $n$ such that %$\theta(1)=0$ and 
$[\theta,a]=\theta(a)$ in the endomorphism algebra $\uEnd(A)$
(Example \ref{eg:dga:dga} \eqref{i:dga:dga:E}).
Here in the left hand side 
we regard $a \in A \subset \uEnd(A)$ as a multiplication operator from left.
\end{enumerate}
\end{rmk*}

Now we can explain another example of a dg Lie algebra.

\begin{eg}\label{eg:dga:T}
For a cdga $A$, consider the dg $A$-module $\Der(A)$ 
of derivations on $A$ (Definition \ref{dfn:dga:der} \eqref{i:dga:der:3}).
Also recall the dg Lie algebra $\uEnd(A)$, 
where we regard $A$ as a complex (Example \ref{eg:dga:dgla:uEnd}).
Since we have $\Der(A) \subset \uEnd(A)$ as complexes, 
we can consider the commutator $[\theta,\varphi]$ %and the differential 
%$d_{\Der(A)} \theta := d_{\uEnd(A)} \theta$ 
for $\theta,\varphi \in \Der(A)$.
Then we can check that $[\theta,\varphi] \in \Der(A)$, 
%and $d_{\Der(A)} \theta \in \Der(A)$, 
so that the triple $(\Der(A),d_{\Der(A)},[-,-])$ is a dg Lie algebra.
For later citation, we denote 
\[
 \bbT_A := (\Der(A),d_{\Der(A)},[-,-])
\]
and call $\bbT_A$ the \emph{dg Lie algebra of derivations on $A$}.
\end{eg}

%%%%%%%%%%%%%%%%%%%%%%%%%%%%%%%%%%%%%%%%%%%%%%%%%%%%%%%%%%%%%%%%%%%%%%%%%%%%%%%%
\subsubsection{Chevalley-Eilenberg complex}
\label{sss:dga:CE}

In this subsection we recall the Chevalley-Eilenberg complex
associated to a dg Lie algebra and its dg module.
The material is more or less standard.
We borrow some notations from \cite[1.4.5]{BD}.
Let us fix a dg Lie algebra $\frkl=(\frkl,d_{\frkl},[-,-]_{\frkl})$ in this part.

Recall the mapping cone of a morphism of complexes (Definition \ref{dfn:dg:coh}).
We denote by $\frkl_{\dagger}$ %$\Cone(\frkl)$ 
the mapping cone of the identify $\id_{\frkl}: \frkl \to \frkl$. 
Thus, it is a complex given by
\[
 \frkl_{\dagger} :=\Cone(\id_{\frkl}) = \frkl[1] \oplus \frkl, \quad 
 d = \begin{bmatrix}d_{\frkl[1]} & 0 \\ \id_{\frkl}[1] & d_{\frkl}\end{bmatrix}.
\]
It is acyclic by Fact \ref{fct:dg:qis}.

We denote an element of $\frkl_{\dagger}^n=\frkl^{n+1} \oplus \frkl^n$ 
as $(x,y)$ with $x \in \frkl^{n+1}$ and $y \in \frkl^n$.
Then we have 
\[
 d^n(x,y)=(-d_{\frkl}^{n+1} x,x+d_{\frkl}^n y) \in 
 \frkl^{n+2} \oplus \frkl^{n+1}=\frkl_{\dagger}^{n+1}.
\]
Using the same notation, 
we can endow $\frkl_{\dagger}$ with a dg Lie algebra structure by
\[
 [(x,y),(x',y')] := ([x,y']_{\frkl}+[y,x']_{\frkl},[y,y']_{\frkl}).
\]
We have an injective morphism of dg Lie algebras
\[
 \frkl \longinj \frkl_{\dagger}, \quad y \longmapsto (0,y).
\]
By this injection, the mapping cone $\frkl_{\dagger}$ is also a dg $\frkl$-module.

Next recall the \emph{enveloping algebra} $U(\frkl)$.
It is the quotient dga $T(\frkl)/I$ of the tensor algebra $T(\frkl)$ 
(Example \ref{eg:dga:dga} \eqref{i:dga:dga:T})
by the both-side dg ideal $I$ generated by terms of the form
$x \otimes y - (-1)^{\abs{x}\abs{y}}-[x,y]$ with $x,y \in \frkl$.
The correspondence $\frkl \mapsto U(\frkl)$ gives a functor 
\[
 U: \dgla \longto \dga,
\]
which is left adjoint to the functor 
$(-)_L:\dga \to \dgla$ in Example \ref{eg:dga:dgla:uEnd}.
Thus we have an adjunction $U: \dgla \rightleftarrows \dga: (-)_L$
of functors of categories.
Let us also recall that there is a functorial equivalence
\[
 \dgMod{\frkl} \longsimto \dgMod{U(\frkl)}
\]
between the category of dg $\frkl$-modules and that of dg $U(\frkl)$-modules.
Hereafter we identify these categories.

In particular, the universal enveloping algebra $U(\frkl_{\dagger})$
of the dg Lie algebra $\frkl_{\dagger}$ is a dga,
and it is also a dg $U(\frkl)$-module since $\frkl_{\dagger}$ is a dg $\frkl$-module.
Then we can obtain the following statements by direct calculations.

\begin{lem}\label{lem:dga:UCone}
Let $\frkl$ be a dg Lie algebra.
\begin{enumerate}[nosep]
\item 
\label{i:dga:UCone:UC=SU}
As a graded linear space, we have 
$U(\frkl_{\dagger}) \simeq \Sym(\frkl[1]) \otimes U(\frkl)$.

\item
\label{i:dga:UCone:dUC}
Under the isomorphism $\Sym(\frkl[1]) \simeq \Lambda(\frkl)$,
the differential $d_{U(\frkl_{\dagger})}$ is given by 
\begin{align*}
d_{U(\frkl_{\dagger})}(x_1\wedge \cdots \wedge x_p \otimes u)
&=\tsum_{i=1}^p (-1)^{\sum_{a=1}^{i-1}(\abs{x_p}+1)}
  x_1 \wedge \cdots \wedge d_{\frkl}x_i \wedge \cdots \wedge x_n \otimes u \\
&+\tsum_{1 \le i < j \le p}(-1)^{(\abs{x_i}+1)\sum_{a=i+1}^{j-1}(\abs{x_a}+1)}
  x_1 \wedge \cdots \wh{x_i} \cdots \wedge [x_i,x_j] \wedge \cdots \wedge x_n \otimes u \\
&+\tsum_{i=1}^p (-1)^{(\abs{x_i}+1)\sum_{a=i+1}^{p}(\abs{x_a}+1)}
  x_1 \wedge \cdots \wh{x_i} \cdots  \wedge x_n \otimes x_i u \\
&+(-1)^{\sum_{a=1}^{p}(\abs{x_a}+1)}x_1 \wedge \cdots \wedge x_n \otimes d_{U(\frkl)}u.
\end{align*}
for $x_i \in \frkl$ and $u \in U(\frkl)$.
\end{enumerate}
\end{lem}

Finally,  let us recall the dg $A$-module $\uHom_A(-,-)$ of morphisms 
for a dga $A$ (Definition \ref{dfn:dga:uHom_A}).

\begin{dfn}\label{dfn:dga:CE}
Let $\frkl$ be a dg Lie algebra.
\begin{enumerate}[nosep]
\item
For a dg $\frkl$-module $M$,
we define the \emph{Chevalley-Eilenberg} (\emph{cochain}) \emph{complex} 
$\CE(\frkl,M)$ to be the dg $\frkl$-module 
%cdga whose underlying graded linear space is given by
\[
 \CE(\frkl,M) := \uHom_{U(\frkl)}(U(\frkl_{\dagger}),M).
\]
The differential is denoted by $d_{\tCE}$.

\item
We denote the functor $M \mapsto \CE(\frkl,M)$ by
\[
 \CE(\frkl,-): \dgMod{\frkl} \longto \dgMod{\frkl}.
\]
\end{enumerate}
\end{dfn}

\begin{rmk}\label{rmk:dga:CE}
\begin{enumerate}[nosep]
\item
\label{i:dga:CE:CE=W}
By Lemma \ref{lem:dga:UCone} \eqref{i:dga:UCone:UC=SU}, we have 
\[
 \CE(\frkl,M) \simeq \uHom(\Sym(\frkl[1]),M)
\]
as a graded linear space. 
Thus, we have $\CE(\frkl,M) \simeq \Hom(\Wedge \frkl,M)$ as a linear space,
which is the standard definition in the literature.
%There is an $\bbN$-grading
%$\CE(\frkl,M)=\bigoplus_{n \in \bbN}\Hom(\Sym^n(\frkl[1]),M)$.

\item 
By definition, the differential $d_{\tCE}$ is given by
\[
 d_{\tCE}f = d_M f-(-1)^{\abs{f}}f d_{U(\frkl_{\dagger})}
\]
for a homogeneous element $f \in \Hom(\Wedge \frkl,M)$
of cohomological degree $\abs{f}$.
Denoting the weight decomposition by 
$f=\tsum_{n \in \bbN}f_p$, $f_p \in \Hom(\Wedge^p \frkl,M)$
and using the description 
$\CE(\frkl,M) \simeq \Hom(\Wedge \frkl,M)$ in \eqref{i:dga:CE:CE=W},
we can write down $d_{\tCE} f=\sum_{p\in\bbN}(d_{\tCE}f)_p$ as:
%See \cite[2.1, (11)]{S} for the detail.
\begin{align*}
  (d_{\tCE} f)_p (x_1 \wedge \cdots \wedge x_p) 
&= d_{M}f_p(x_1 \wedge \cdots \wedge x_p) \\
&+\tsum_{i=1}^p (-1)^{\abs{f}+\sum_{a=1}^{i-1} (\abs{x_a}+1)}  
 f_p(x_1 \wedge \cdots \wedge (d_{\frkl}x_i) \wedge \cdots \wedge x_p) \\
&+\tsum_{1 \le i<j \le p} 
  (-1)^{\abs{f}+ (\abs{x_i}+\abs{x_j})\sum_{a=1}^{i-1}(\abs{x_a}+1) 
        +(\abs{x_j}+1)\sum_{a=i+1}^{j-1}(\abs{x_a}+1)} \\
&\phantom{+\tsum_{1 \le i<j \le p}1}
  f_{p-1}([x_i,x_j] \wedge x_1 \wedge \cdots \wh{x_i} 
          \cdots \wh{x_j} \cdots \wedge x_p) \\
&+\tsum_{i=1}^p 
 (-1)^{(\abs{x_i}+1)\bigl(\abs{f}+\sum_{a=1}^{i-1}(\abs{x_a}+1)\bigr)} 
 x_i.f_{p-1}(x_1 \wedge \cdots \wh{x_i} \cdots \wedge x_p).
\end{align*}
In the non-dg case $\abs{f}=\abs{x_i}=0$ and $d_{\frkl}=d_M=0$,
we recover the original Chevalley-Eilenberg differential.
%and $x_1 \wedge \cdots \wedge x_p \in \Wedge^p \frkl$.

\item\label{i:CE:ism}
%\cite[1.4.10]{BD}
If $\frkl$ is of finite dimension, 
then we have an isomorphism 
\[
 \CE(\frkl,M) \simeq \Sym(\frkl^*[-1]) \otimes M
\] 
as a graded linear space. %cdgas.
Let us rewrite the differential $d_{\tCE}$ under this isomorphism.
Note that the graded Lie algebra $\frkl_\dagger$ 
(forget the differential) acts on $\Sym(\frkl^*[-1])$ in the way that 
$\frkl \subset \frkl_\dagger$ acts by the coadjoint action and 
$\frkl[1] \subset \frkl_\dagger$ acts by $-\{\cdot,\cdot\}$, where 
$\{\cdot,\cdot\}: \frkl^* \otimes \frkl \to \bbk$ denotes the canonical pairing.
Then there exits a unique differential $\delta_{\tCE}$ on $\Sym(\frkl^*[-1])$
such that the $\frkl_\dagger$-action is compatible.
Explicitly, $\rst{\delta_{\tCE}}{\frkl^*[-1]}$ is equal to the composition
$\frkl^*[-2] \xr{\nu} \frkl^*[-1] \otimes \frkl^*[-1] \xr{-\frac{1}{2}\cdot}
 \Sym^2(\frkl^*[-1])$,
where $\nu$ denotes the dual of the Lie bracket $[\cdot,\cdot]_{\frkl}$.
Finally, $d_{\tCE}$ is equivalent to 
$\delta_{\tCE} \otimes \id + f + \id \otimes d_M$,
where $f: M \to \frkl^* \otimes M$ denotes the dual of the $\frkl$-action.
\end{enumerate}
\end{rmk}

The Chevalley-Eilenberg complex has an extra structure.
For the explanation, we need:

\begin{dfn}\label{dfn:dga:shuffle}
Let $p,q \in \bbN$ and consider the symmetric group $\frkS_{p+q}$.
%\begin{enumerate}[nosep]
%\item 
%For $\sigma \in \frkS_p$, we denote by $\ell(\sigma)$ the length of presentation 
%with respect to the Coxeter elements, so that $(-1)^{\sigma}=\sgn(\sigma)$.
%
%\item
We denote by $\frkS_{p,q} \subset \frkS_{p+q}$
the subset of \emph{$(p,q)$-shuffles}, i.e.,
permutations $\sigma$ such that $\sigma(1)<\cdots<\sigma(p)$
and $\sigma(p+1)<\cdots<\sigma(p+q)$.
%\end{enumerate}
\end{dfn}

\begin{dfn}\label{dfn:dga:cup}
Let $\frkl$ be a dg Lie algebra.
\begin{enumerate}[nosep]
\item 
For dg $\frkl$-modules $M,N$,
we define the morphism 
$\cup: \CE(\frkl,M) \otimes \CE(\frkl,N) \to \CE(\frkl,M \otimes N)$
of dg $\frkl$-modules by
\begin{align*}
 (f \cup g)(x_1 \wedge \cdots \wedge x_{p+q}) := 
 \sum_{\sigma \in \frkS_{p,q}}
 \sgn(\sigma) (-1)^{\ve+\ve_1}
 f(x_{\sigma(1)} \wedge \cdots \wedge x_{\sigma(p)}) \otimes 
 g(x_{\sigma(p+1)} \wedge \cdots \wedge x_{\sigma(p+q)}), 
\end{align*}
for $f \in \CE(\frkl,M)(m)$ and $g \in \CE(\frkl,N)(n)$.
Here $\ve$ denotes the sign of the braiding isomorphisms 
for the permutation $\sigma$ of $x_i$'s, and 
$\ve_1 := p \abs{g}+\tsum_{i=1}^p \abs{x_{\sigma(i)}}(q+\abs{g})$.
We call the operation $\cup$ the \emph{cup product}.

\item
Let $A$ be a commutative ring object in $\dgMod{\frkl}$,
i.e., a cdga which is also a dg $\frkl$-module and 
the multiplication is a morphism in $\dgMod{\frkl}$.
Consider the composition
\[
 \CE(\frkl,A) \otimes \CE(\frkl,A) \xrr{\cup} \CE(\frkl,A \otimes A)
 \xrr{\CE(\frkl,\cdot)} \CE(\frkl,A),
\]
where the first $\cup$ denotes the cup product 
and the second $\CE(\frkl,\cdot)$ is the image of the multiplication of $A$
under the functor $\CE(\frkl,-)$ (Definition \ref{dfn:dga:CE}).
We denote this composition by the same symbol as
\[
 \cup: \CE(\frkl,A) \otimes \CE(\frkl,A) \longto \CE(\frkl,A).
\]
\end{enumerate}
\end{dfn}

Now we can check the following classical result.

\begin{lem}\label{lem:dga:CEa}
Let $A$ be a commutative ring object in $\dgMod{\frkl}$.
Then the Chevalley-Eilenberg complex $\CE(\frkl,A)$ with the cup product 
$\cup: \CE(\frkl,A) \otimes \CE(\frkl,A) \to \CE(\frkl,A)$ is a cdga.
We call it the \emph{Chevalley-Eilenberg cdga}.
\end{lem}

%%%%%%%%%%%%%%%%%%%%%%%%%%%%%%%%%%%%%%%%%%%%%%%%%%%%%%%%%%%%%%%%%%%%%%%%%%%%%%%%
%%%%%%%%%%%%%%%%%%%%%%%%%%%%%%%%%%%%%%%%%%%%%%%%%%%%%%%%%%%%%%%%%%%%%%%%%%%%%%%%
\subsection{Shifted Poisson algebras}

In this subsection we recollect basics on shifted Poisson structures 
following \cite{CPTVV, M, S}.
We continue to work over a field $\bbk$ with characteristics $0$.

%%%%%%%%%%%%%%%%%%%%%%%%%%%%%%%%%%%%%%%%%%%%%%%%%%%%%%%%%%%%%%%%%%%%%%%%%%%%%%%%
\subsubsection{Definition}

%Recall the shift functor $[1]$ on the category $\dgVec$ of complexes.
%We start with 

\begin{dfn}\label{dfn:sp:sp}
Let $n \in \bbZ$.
A \emph{$\bbP_n$-algebra in $\dgVec$} is a data $(R,\cdot,\{-,-\})$ consisting of 
\begin{itemize}[nosep]
\item
a cdga $(R,\cdot)$ and 
\item
a morphism $\{-,-\}: R \otimes R \longto R[1-n]$ in $\dgVec$,
called the  \emph{$n$-Poisson bracket} of $R$
\end{itemize}
which satisfies the following conditions.
\begin{enumerate}[nosep, label=(\roman*)]
\item
$\{-,-\}$ gives a structure of dg Lie algebra on $R[n-1]$ 
(Definition \ref{dfn:dga:dgla}).

\item
The \emph{Leibniz rule}
\[
 \{f,g \cdot h\} = \{f,g\} \cdot h + (-1)^{\abs{g} \abs{h}}\{f,h\} \cdot g
\]
holds for homogeneous $f,g,h \in R$.
\end{enumerate}
We often omit to mention the category $\dgVec$ and just call it a $\bbP_n$-algebra.
We also call it a \emph{dg $n$-Poisson algebra}.
In the case $n=1$, we just call it a \emph{dg Poisson algebra}.
\end{dfn}

\begin{rmk}\label{rmk:sp:sp}
\begin{enumerate}[nosep]
\item
If we replace $\dgVec$ by $\cVec$, then a $\bbP_1$-algebra in $\cVec$
is nothing but a (commutative) \emph{Poisson algebra} in the ordinary sense.

\item
A $\bbP_2$-algebra (in $\gVec$) is nothing but a Gerstenhaber algebra.
See \cite[13.3.10--13.3.15]{LV} for the detail.

\item 
The symbol $\bbP_n$ indicates that we can give a definition by a dg operad 
on the symmetric monoidal dg category $\ddgVec$ of complexes.
We refer \cite{M} for the detail.

%\item
%The object introduced in \cite[1.4.3]{CPTVV} is 
%a \emph{graded} dg shifted Poisson structure.
%Let us consider the category $\dgVec^{\gr}$ of dg vector spaces 
%with an extra $\bbZ$-grading.
%Its object is denoted by $\{R(p)\}_{p \in \bbZ}$ with $R(p) \in \dgVec$.
%A graded $n$-Poisson cdga is an object 
%$\{R(p)\}_{p \in \bbZ} \in \dgVec^{\gr}$ together with 
%a family of multiplications $\cdot: R(p) \otimes R(q) \to R(p+q)$
%and a family of morphisms $\{-,-\}: R(p) \otimes R(q) \to R(p+q-1)[1-n]$
%satisfying the similar conditions as non-graded shifted Poisson structures.
%The extra $\bbZ$-grading indicated by $(p)$ is called the \emph{weight}.
\end{enumerate}
\end{rmk}

Hereafter we often omit the symbol $\cdot$ and denote $r s := r \cdot s$,
and also denote just by $R$ the Poisson algebra $(R,\cdot,\{-,-\})$.

Let us also introduce the category of $\bbP_n$-algebras.

\begin{dfn*}
\begin{enumerate}[nosep]
\item 
A \emph{morphism} of $\bbP_n$-algebras is defined to be a morphism 
in $\cdga$ which respects the Poisson brackets.

\item
We define the \emph{tensor product} of $\bbP_n$-algebras 
$(R,\{-,-\}_R)$ and $(S,\{-,-\}_S)$ to be the one where
\begin{itemize}[nosep]
\item 
the underlying cdga is the tensor product 
$R \otimes S$ in $\cdga$, and 

\item
the Poison bracket $\{-,-\}$ is given by
\[
 \{r \otimes s, r' \otimes s'\} :=  
 \{r,r'\}_R \otimes (s s') + (-1)^{\abs{r'}\abs{s}} (r r') \otimes \{s,s'\}_S
\]
for homogeneous $r,r' \in R$ and $s,s' \in S$.
\end{itemize}

\item
We denote by $\bbP_n\hy\dgVec$ the category of $\bbP_n$-algebras 
and their morphisms.
It is a symmetric monoidal category with respect to the tensor product.
In the case $n=1$, we also denote $\dgpa := \bbP_1\hy\dgVec$.
\end{enumerate}
\end{dfn*}

%Let $R$ be a commutative dg superalgebra.
%The totality of dg $n$-Poisson structures on $R$ forms a full subspace
%of the mapping space $\Map_{\idgsVec}(R \otimes R,R[1-n])$ by definition.
%
%\begin{ntn*}
%We denote the corresponding $\infty$-category of dg $n$-Poisson structures on $R$ 
%by $\Pois(R,n-1)$.
%\end{ntn*}
%
%See also \cite[1.2, 1.4.2]{CPTVV},
%where they considered the model category of dg $n$-Poisson structures first,
%and then construct the $\infty$-category by localizing along weak equivalences.

Next we introduce notations for Poisson modules over $\bbP_n$-algebras.

\begin{dfn}\label{dfn:sp:mod}
\begin{enumerate}[nosep]
\item
\label{i:sp:mod:1}
Let $R$ be a $\bbP_n$-algebra.
A \emph{dg Poisson $R$-module $M$} is a complex $M$ equipped with two morphisms 
\[
 .: R \otimes M \longto M, \quad  \{-,-\}: R \otimes M \longto M
\]
in $\dgVec$ such that 
\begin{enumerate}[nosep,label=(\roman*)]
\item 
the morphism $.$ is a dg $R$-module structure where we regard $R \in \cdga$,
\item
$\{-,-\}$ is an $R[n-1]$-module structure where we regard $R[n-1] \in \dgla$, 
and
\item
\label{i:sp:mod:3}
the \emph{Leibniz rule}  
\[
 \{r,s.m\}=\{r,s\}.m+(-1)^{\abs{r}\abs{s}}s.\{r,m\}, \quad  
 \{r \cdot s,m\}=r.\{s,m\}+(-1)^{\abs{r}\abs{s}}s.\{r,m\}
\]
hold for any homogeneous $r,s \in R$ and $m \in M$.
\end{enumerate}
We denote by $\dgPMod{R}$ the category of dg Poisson $R$-modules.

\item
For a $\bbP_1$-Poisson algebra $R$ in $\cVec$, i.e.,
a Poisson algebra in the ordinary sense, 
we define a \emph{Poisson $R$-module} to be 
a linear space $M$ equipped with $.$ and $\{-,-\}$ 
with the same conditions as in \eqref{i:sp:mod:1}.
We denote by $\PMod{R}$ the category of Poisson $R$-modules.
\end{enumerate}
\end{dfn}

%%%%%%%%%%%%%%%%%%%%%%%%%%%%%%%%%%%%%%%%%%%%%%%%%%%%%%%%%%%%%%%%%%%%%%%%%%%%%%%%
\subsubsection{Shifted polyvectors}

Let us explain the \emph{space of shifted polyvectors} with the \emph{Schouten bracket},
which is a standard example of a shifted Poisson algebra.
See \cite[3.3.2]{LPV} for the non-dg case.
We need some preliminaries.

%\begin{dfn}\label{dfn:sp:der}
%Let $n \in \bbZ$, $A$ be a cdga,
%and $M$ be a dg $A$-module.
%\begin{enumerate}[nosep]
%\item 
%A \emph{derivation of degree $n$}, or an \emph{$n$-derivation}, from $A$ to $M$
%is an element  $\theta$ of $\Hom_{\dgVec{A}}(A,M[n])$ such that 
%for any homogeneous $a,b \in A$
%we have 
%\[
% \theta(a b) = \theta(a).b+(-1)^{n \abs{a}} a.\theta(b).
%\]
%Here we denoted by $a.m = m.a \in M$ the action of $a \in A$ on $m \in M$.
%
%\item
%We denote by $\Der(A,M)^n$ the set of $n$-derivations from $A$ to $M$,
%which is naturally an $A$-module.
%We define the \emph{dg $A$-module $\Der(A,M)$ of derivations} by 
%\[
% \Der(A,M) := \tboplus_{n \in \bbZ} \Der(A,M)^n.
%\]
%\end{enumerate}
%\end{dfn}

Let $A$ be a cdga.
For a dg $A$-module $M$, we have the dg $A$-module $\Der(A,M)$ 
of derivations (Definition \ref{dfn:dga:der} \eqref{i:dga:der:2}).
The functor 
\[
 \Der(A,-): \dgMod{A} \longto \dgMod{A}, \quad M \longmapsto \Der(A,M) 
\]
%sending a dg $A$-module 
commutes with limits so that it is corepresentable by a dg $A$-module.

\begin{ntn}\label{ntn:sp:Omega}
Let $A$ be a cdga.
We denote the dg $A$-module corepresenting the functor $\Der(A,-)$ by $\Omega_A^1$, 
and call it the \emph{module of K\"ahler differentials over $A$}.
\end{ntn}

We have an explicit description of $\Omega_A^1$.
It is a dg $A$-module 
\begin{itemize}[nosep]
\item 
generated over $A$ by the symbols $d a$ for each $a \in A$
with cohomological degree $\abs{d a}:= \abs{a}$, and
\item
the defining relation is 
\[
 d(a b) = (-1)^{\abs{a}\abs{b}} b.(d a) + (-1)^{\abs{a}} a.(d b).
\]
\end{itemize}
Let us give a simple example for later citation.

\begin{eg}\label{eg:sp:OmSymV}
Let $V$ be a complex, and $\Sym(V)$ be the symmetric tensor algebra 
(Example \ref{eg:dga:SW} \eqref{i:dg:SW:S}). 
Then the module $\Omega_{\Sym(V)}^1$ of K\"ahler differentials is 
a free dg $\Sym(V)$-module generated by $V$.
\end{eg}

The universality of corepresenting object says that 
for a dg $A$-module $M$ we have a functorial isomorphism 
\[
 \uHom_A(\Omega_A^1,M) \longsimto \Der(A,M).
\]
In particular, there is a morphism $d: A \to \Omega_A^1$ in $\dgMod{A}$,
and the above functorial isomorphism is given by 
$\alpha \mapsto \alpha \circ d$.
The morphism $d: A \to \Omega_A^1$ is nothing but 
the correspondence $a \mapsto d a$.

%By the symmetric monoidal structure on the category $\dgMod{A}$ of dg $A$-modules,
%we also have the notion of \emph{symmetric tensor product} of dg $A$-modules. 
%
%\begin{ntn*}
%For a dg $A$-module $M$, we denote by $\Sym_A(M)$ the symmetric tensor product,
%which is again a dg $A$-module.
%\end{ntn*}

We need one more terminology.

\begin{dfn}\label{dfn:sp:wt}
A \emph{graded dg linear space}, or a \emph{graded complex} $M$ 
is a complex with an extra $\bbZ$-grading.
We call the extra $\bbZ$-grading the \emph{weight} of $M$ and denote by
$M=\bigoplus_{p \in \bbZ} M(p)$.
We denote the cohomological grading by $M(p)=\bigoplus_{n \in \bbZ}M(p)^n$.
Thus we have 
\[
 M=\tboplus_{p \in \bbZ} M(p) = 
 \tboplus_{p \in \bbZ}\tboplus_{n \in \bbZ}M(p)^n.
\]
\end{dfn}

\begin{eg*}
We have already introduced complexes with extra $\bbN$-gradings.
\begin{enumerate}[nosep]
\item
Let $V$ be a complex. 
Then the tensor space (algebra) $T(V)$ and the symmetric tensor space (algebra)
$\Sym(V)$ in Example \ref{eg:dg:TS} are graded complexes.
Explicitly, we have 
\[
 T(V)(p)=V^{\otimes p}, \quad 
 \Sym(V)(p) = V^{\otimes p}/\frkS_p \quad (p \in \bbN).
\] 

\item
Let $A$ be a cdga and $M$ be a dg $A$-module.
Then the tensor algebra $T_A(M)$ and the symmetric tensor algebra $\Sym_A(M)$
in Example \ref{eg:dga:TS} are graded complexes.
%Note that $\Sym_A(M)$ is a graded complex with 
Explicitly, we have 
\[
 T_A(M)(p)=M^{\otimes_A p}, \quad 
 \Sym_A(M)(p) = M^{\otimes_A p}/\frkS_p \quad (p \in \bbN).
\]
%The weights $p$ are non-negative in this case.
\end{enumerate}
\end{eg*}

Now we can introduce \emph{shifted polyvector fields} over a cdga.
Recall the hom complex $\uHom_A(M,N) \in \dgMod{A}$ for dg $A$-modules $M$ and $N$
(Notation \ref{ntn:dga:uHom_A}). 

\begin{dfn*}
Let $n \in \bbZ$ and $A$ be a cdga.
We define the graded complex $\Pol(A,n)$ of \emph{$n$-shifted polyvector fields} by
\[
 \Pol(A,n) := \uHom_{A}(\Sym_A(\Omega_A^1[n+1]),A)
\]
with the weight grading 
$\Pol(A,n)(p) := \uHom_A(\Sym_A(\Omega_A^1[n+1])(p),A)$.
An element of $\Pol(A,n)$ will be called a \emph{polyvector}.
\end{dfn*}

Let us explain the Poisson algebra structure on $\Pol(A,n)$.
For that, we use the following notation in \cite{S}:

\begin{ntn}\label{ntn:sp:itp}
Let $n$ and $A$ be as above.
We use the formal symbol $d_{\dR}$ 
sitting in the cohomological degree $-(n+1)$ in $\Omega_A^1[n+1]$ and denote 
\[
 d_{\dR} a \in (\Omega_A^1[n+1])^{m-n-1} = (\Omega_A^1)^m
\]
for $a \in A^m$.
We also denote by 
$\itp: \Pol(A,n) \otimes \Sym_A(\Omega^1[n+1]) \to A$ the natural pairing.
Then, for $v \in \Pol(A,n)(p)$ and $a_i \in A$ ($i=1,\ldots,p$), we set
\begin{align}\label{eq:sp:itp}
 v(a_1,\ldots,a_p) :=
 v \itp (d_{\dR} \otimes \cdots \otimes d_{\dR})(a_1 \otimes \cdots \otimes a_p).
\end{align}
Here the term 
$(d_{\dR} \otimes \cdots \otimes d_{\dR})(a_1 \otimes \cdots \otimes a_p)$
is regarded as a monomial in the cdga $\Sym_A^p(\Omega_A^1[n+1])$.
\end{ntn}

In \eqref{eq:sp:itp}, 
the sign rule for multiplication is given by Definition \ref{dfn:dga:dgamon}, 
and that for commutation relation is given by Definition \ref{dfn:dg:dgVmon}.
For example, using $n^2 \equiv n \pmod{2}$, we have
\[
 v(a_1,a_2,a_3,\ldots,a_p) = 
 (-1)^{\abs{a_1}\abs{a_2}+n+1} v(a_2,a_1,a_3,\ldots,a_p).
\]

\begin{fct}[{\cite[\S\S1.1--1.2]{S}}]\label{fct:sp:pv}
Let $n \in \bbZ$ and $A$ be a cdga.
\begin{enumerate}[nosep]
\item
\label{i:sp:pv:cup}
For $v \in \Pol(A,n)(p)$ and $w \in \Pol(A,n)(q)$, 
the following formula defines  $v \cdot w \in \Pol(A,n)(p+q)$. 
\[
 (v \cdot w)(a_1,\ldots,a_{p+q}) 
 := \tsum_{\sigma \in \frkS_{p,q}} \sgn(\sigma)^n (-1)^{\ve+\ol{\ve}} 
 v(a_{\sigma(1)},\ldots,a_{\sigma(p)}) w(a_{\sigma(p+1)},\ldots,a_{\sigma(p+q)}).
\]
Here $\frkS_{p,q} \subset \frkS_{p+q}$ denotes the $(p,q)$-shuffles 
(Definition \ref{dfn:dga:shuffle}),
%\cite[1.3.2]{LV},
and $\ve$ denotes the sign of the braiding isomorphisms 
for the permutation $\sigma$ of $a_i$'s.
The sign $\ol{\ve}$ is given by 
\[
 \ol{\ve}:= \abs{w} (n+1) p + \sum_{i=1}^k \abs{a_{\sigma(i)}}((n+1)q+\abs{w}).
\]

\item 
\label{i:sp:pv:Sch}
For $v \in \Pol(A,n)(p)$ and $w \in \Pol(A,n)(q)$, 
the following formula gives $[v,w]_S \in \Pol(A,n)(p+q-1)$. 
\begin{align*}
 [v,w]_S(a_1,\ldots,a_{p+q-1}) :=
 &\sum_{\sigma \in \frkS_{q,p-1}} \sgn(\sigma)^{n+1}(-1)^{\ve+\ve_1}
  v(w(a_{\sigma(1)},\ldots,a_{\sigma(q)}),a_{\sigma(q+1)},\ldots,a_{\sigma(p+q-1)}) \\
-&\sum_{\sigma \in \frkS_{p,q-1}}  \sgn(\sigma)^{n+1}(-1)^{\ve+\ve_2}
  w(v(a_{\sigma(1)},\ldots,a_{\sigma(p)}),a_{\sigma(p+1)},\ldots,a_{\sigma(p+q-1)}).
\end{align*}
Here $\frkS_{p,q}$ and $(-1)^{\ve}$ are the same as \eqref{i:sp:pv:cup},
and the sign $\ve_i$ are given by 
\begin{align*}
&\ve_1 := (n+1)(\abs{w}+q)(p+1)+(n+1)\abs{v}, \\
&\ve_2 := (\abs{v}-(n+1)p)(\abs{w}-(n+1)q)+(n+1)(p+1)(\abs{w}+1)+(n+1)\abs{v}.
\end{align*}
We call the operation $[-,-]_S$ the \emph{Schouten bracket}.
It is of cohomological degree $-n-1$.

\item
The multiplication and the Schouten bracket define a 
dg $(n+2)$-Poisson algebra structure on $\Pol(A,n)$.
\end{enumerate}
\end{fct}

Here is a dg version of the classical fact \cite[Proposition 3.5]{LPV}
on the Poisson structure and the Schouten bracket.
We omit the proof since it is essentially the same with the classical case.

\begin{lem}\label{lem:sp:PSvan}
Let $n\in \bbZ$, $A$ be a cdga, and 
$\pi \in \Pol(A,n-1)(2)^{n+1}$ be a polyvector field of weight $2$ 
(i.e., a bivector field) whose cohomological degree is $n+1$.
We define the bilinear map 
\[
 \{-,-\}_{\pi}: A \otimes A \longto A[1-n], \quad
 \{a,b\}_A := \pi(a,b),
\]
where we used Notation \ref{ntn:sp:itp}.
Then $\{-,-\}_{\pi}$ defines a $\bbP_n$-algebra structure on $A$ 
if and only if $[\pi,\pi]_S=0$.
\end{lem}

%%%%%%%%%%%%%%%%%%%%%%%%%%%%%%%%%%%%%%%%%%%%%%%%%%%%%%%%%%%%%%%%%%%%%%%%%%%%%%%%
\subsubsection{Kirillov-Kostant Poisson structure}
\label{sss:sp:KK}

In this part we recall the Kirillov-Kostant Poisson structure on a Lie algebra,
which gives an example of a Poisson algebra arising from Lie theory.

Let $\frkl$ be a dg Lie algebra over a field $\bbk$ with characteristics $0$.
We denote the Lie bracket by $[-,-]_{\frkl}$.
Then the symmetric algebra $\Sym(\frkl)=\Sym_{\bbk}(\frkl)$ of the 
underlying complex of $\frkl$ has a structure of dg Poisson algebra 
with the Poisson bracket $\{-,-\}_{\frkl}$ determined by 
\[
 \{x,y\}_{\frkl} := [x,y]_{\frkl} \quad 
 \text{for $x, y \in \frkl \subset \Sym(\frkl)$} 
\]
and by the Leibniz rule. 

\begin{ntn}\label{ntn:sp:KK}
For a dg Lie algebra $\frkl$,
the Poisson bracket $\{-,-\}_{\frkl}$ is called 
the \emph{Kirillov-Kostant Poisson bracket},
and the resulting dg Poisson algebra $(\Sym(\frkl),\{-,-\}_{\frkl})$ 
is called the \emph{Lie-Poisson algebra} of $\frkl$.
\end{ntn}

%Hereafter we regard the coordinate ring $\bbC[\frkl^*]$ as a Poisson algebra 
%by this Lie-Poisson structure $\{-,-\}_{\frkl}$.

Recall that for a cdga $R$,  
we denote by $\bbT_R = \Der(R)$ the dg Lie algebra of derivations on $R$
(Example \ref{eg:dga:T}).
If moreover $R$ is a dg Poisson algebra with the Poisson bracket $\{-,-\}_R$, 
then we have a linear map 
\[
 D: R \longto \Der(R), \quad D(r):= \{r,-\}_R.
\]

\begin{dfn}\label{dfn:sp:Ham}
Let $R$ be a dg Poisson algebra.
A \emph{Hamiltonian $\frkl$-action} on $R$ is 
a morphism $a: \frkl \to \Der(R)$ of dg Lie algebras 
equipped with a dg Lie algebra morphism $\mu: \frkl \to R$ 
such that $D \circ \mu = a$.
In this case, the morphism $\mu$ is called the \emph{momentum map}.
\end{dfn}

\begin{rmk}\label{rmk:sp:mu}
\begin{enumerate}
\item 
More explicitly, the condition is %that we have 
\begin{equation}\label{eq:sp:mu}
 \{\mu(x),r\}_R=a(x)(r)
\end{equation}
for any $x \in \frkl$ and $r \in R$.
We call it the \emph{momentum map equation}.

\item\label{i:sp:mu:2}
An example is $R=\bbk$ with the trivial Poisson bracket.
The morphism $a$ and the momentum map $\mu$ are both trivial.

\item\label{i:sp:mu:3}
We can replace the momentum map $\mu: \frkl \to R$
by the induced map $\Sym \frkl \to R$ of dg Poisson algebras.
We denote it by the same symbol $\mu$ and also call it the momentum map.
\end{enumerate}
\end{rmk}

%%%%%%%%%%%%%%%%%%%%%%%%%%%%%%%%%%%%%%%%%%%%%%%%%%%%%%%%%%%%%%%%%%%%%%%%%%%%%%%%
%%%%%%%%%%%%%%%%%%%%%%%%%%%%%%%%%%%%%%%%%%%%%%%%%%%%%%%%%%%%%%%%%%%%%%%%%%%%%%%%
\subsection{Shifted Poisson structures for derived stacks}
\label{ss:sp:dag}

In this subsection we give a brief recollection on the theory of 
shifted Poisson structures on derived stacks \cite{CPTVV}.

%%%%%%%%%%%%%%%%%%%%%%%%%%%%%%%%%%%%%%%%%%%%%%%%%%%%%%%%%%%%%%%%%%%%%%%%%%%%%%%%
\subsubsection{Affine derived Poisson schemes}
\label{sss:sp:daff}

Let us restate the notions on shifted Poisson algebras in terms of 
the language of \emph{affine derived schemes} in the sense of \cite{TVe,T14}.
The language of derived algebraic geometry will be introduced 
in the later \S \ref{sss:sp:dSt},
and the present part is a preparation for it.

Hereafter we use the language of $\infty$-categories.
See \S \ref{ss:0:ntn} for our terminology on $\infty$-categories.
In particular, we will identify a category with its nerve and a dg category with 
its dg nerve, so that we regard them as $\infty$-categories.

\begin{ntn}\label{ntn:sp:daff}
Here is a brief list of the notations on affine derived schemes.
\begin{enumerate}[nosep]
\item
In this part we assume that $\bbk$ contains $\bbQ$.
Although we can remove this assumption, we put it to make the text simple and 
compatible with the literatures of derived algebraic geometry.

\item
We denote by $\icdga$ the dg category of unital cdgas over $\bbk$,
and denote by $\icdga^{\le 0} \subset \icdga$ the full sub-dg-category 
spanned by objects concentrated in non-positive degrees.
As in \S \ref{ss:0:ntn} \eqref{i:ntn:inf} \ref{i:ntn:inf:dg},
we regard these dg categories as $\infty$-categories.

\item
\label{i:sp:daff:dt}
The $\infty$-categories $\icdga$ and $\icdga^{\le 0}$ are symmetric monoidal in 
the sense of \cite{Lu2} with the \emph{derived tensor product} $\otimes^{\bbL}$.
We also have a relative version: Given two morphisms $B \to A$ and $C \to A$
of cdgas, we have the derived tensor product $B \otimes^{\bbL}_A C$.
It is realized by the \emph{two-sided bar complex}. 
See Definition  \ref{dfn:pr:bcpx} for the detail.

\item
\label{i:sp:daff:daff}
The $\infty$-category $\idAff$ of \emph{affine derived schemes} over $\bbk$ is 
defined to be the opposite $\infty$-category of $\icdga^{\le 0}$.
For $R \in \icdga^{\le 0}$, 
we denote the corresponding affine derived scheme by $\Spec(R)$.
Conversely, for an affine derived scheme $X=\Spec(R)$, %we set $\bbk[X]:=R$ and 
we call $R$ the \emph{coordinate} (\emph{derived}) \emph{ring} of $X$.

\item
The derived tensor product $B \otimes^{\bbL}_A C$ in \eqref{i:sp:daff:dt} is
transfered to the \emph{derived fiber product} or \emph{derived intersection}
$\Spec(B) \times^{\bbL}_{\Spec(A)}\Spec(C)$ in the $\infty$-category $\idAff$
of affine derived schemes.

\item
For an affine derived scheme $X=\Spec(R)$, 
we denote the \emph{structure sheaf} by $\shO_X$.
It is a sheaf of unital cdgas on the \'etale $\infty$-topos on $X$.
See \cite[Chap.\ 2.2]{TVe} for the detail.

\item
\label{i:sp:daff:LQCoh}
For an affine derived scheme $X=\Spec(R)$,
we denote by $\LQCoh(X)$ the $\infty$-category of quasi-coherent sheaves of 
$\shO_X$-modules over $X$.
\end{enumerate}
\end{ntn}

The equivalence of the (ordinary) category of quasi-coherent sheaves 
over an affine scheme and the category of modules 
over the corresponding commutative ring is enhanced to

\begin{fct}\label{fct:Lqc=Mod}
For an affine derived scheme $X=\Spec(R)$, 
the $\infty$-category $\LQCoh(X)$ is equivalent to the $\infty$-category
$\ddgMod{R}$ associated to the dg category of dg $R$-modules.
\end{fct}

Moreover, $\LQCoh(X)$ is a stable $\infty$-category 
in the sense of \cite[Chap.\ 1]{Lu2}.
If $X$ is an affine (non-derived) scheme, 
then the homotopy category of $\LQCoh(X)$ is equivalent to
the derived category $\DQCoh(X)$ of quasi-coherent \'etale sheaves 
of $\shO_X$-modules on $X$.

Now the following definition should be a natural one.

\begin{dfn}\label{dfn:sp:adps}
\begin{enumerate}[nosep]
\item 
The $\infty$-category of \emph{affine derived $n$-Poisson schemes} is 
defined to be the opposite of the $\infty$-category $\bbP_n\hy\icdga^{\le 0}$
associated to the dg category of $\bbP_n$-algebras 
concentrated in non-positive cohomological degrees.

\item
\label{i:dp:adps:2}
An object of the above $\infty$-category is called 
an \emph{affine derived $n$-Poisson scheme}.
Following the notation of affine derived schemes, we denote by $X = \Spec(R)$
the object associated to $R \in \bbP_n\hy\icdga^{\le 0}$.
%We denote by $X=\Spec(R)$ the affine derived $n$-Poisson superscheme
%corresponding to $R \in \bbP_n\hy\icdgsa^{\le 0}$, and conversely 
%we denote $\bbk[X] := R$ and call it 
%the \emph{coordinate} (\emph{derived}) \emph{ring} of $X$.

\item
\label{i:sp:adps:3}
In the case $n=1$ of \eqref{i:dp:adps:2},
the object is called an \emph{affine derived Poisson scheme}.
\end{enumerate}
\end{dfn}

\begin{eg*}
Recall the Kirillov-Kostant Poisson structure on 
a dg Lie algebra $\frkl$ (Notation \ref{ntn:sp:KK}).
If $\frkl$ is concentrated in non-positive cohomological degrees,
then we may regard $\Sym(\frkl)$ as the coordinate ring 
of the affine derived scheme $\frkl^*$.
Thus we have the affine derived $1$-Poisson scheme $\frkl^*$.
\end{eg*}

Let $X=\Spec(R)$ be an affine derived $n$-Poisson scheme.
Then the structure sheaf $\shO_X$ is an \'etale sheaf 
both of cdgas and of $n$-shifted dg Lie algebras 
whose local sections satisfy the Leibniz rule.

We denote by $\LPQCoh(X)$ the $\infty$-category of \'etale sheaves on $X$ 
which are both quasi-coherent sheaves of dg $\shO_X$-modules and sheaves of 
$n$-shifted dg Lie superalgebra $\shO_X$-modules whose local sections satisfy 
the relation Definition \ref{dfn:sp:mod} \ref{i:sp:mod:3}.
Then Fact \ref{fct:Lqc=Mod} naturally yields

\begin{lem*}%\label{lem:sp:eq}
Let $X=\Spec(R)$ be an affine derived $n$-Poisson scheme.
Then there is an equivalence of $\infty$-categories 
\[
 \LPQCoh(X) \simeq \dgPMod{R}.
\]
\end{lem*}

%%%%%%%%%%%%%%%%%%%%%%%%%%%%%%%%%%%%%%%%%%%%%%%%%%%%%%%%%%%%%%%%%%%%%%%%%%%%%%%%
\subsubsection{Shifted Poisson structures and shifted symplectic structures}
\label{sss:sp:dSt}

At this stage we recall the shifted Poisson structure for derived stacks
introduced in \cite[3.1]{CPTVV}.
We use the language of \emph{derived stacks} developed in \cite{TVe}
and its $\infty$-categorical presentation in \cite{T14}.
See Notation \ref{ntn:sp:daff} for our notation on the affine derived schemes,
and also \S \ref{ss:0:ntn} for our terminology on $\infty$-categories.

\begin{ntn}\label{ss:sp:dSt}
\begin{enumerate}[nosep]
\item
We work over a field $\bbk$ which contains $\bbQ$. 

%\item
%We denote by $\idSt$ the $\infty$-category of derived stacks over $\bbC$.

\item
A \emph{geometric derived stack} means 
an $n$-geometric derived $D^-$-stack over $\bbk$ with some $n \in \bbZ_{\ge-1}$
in the sense of \cite[\S 1.3.3, \S 2.2.3]{TVe}.
In particular, the affine derived schemes 
(Notation \ref{ntn:sp:daff} \eqref{i:sp:daff:daff})
are $(-1)$-geometric derived stack.

\item
Algebraic stacks, algebraic spaces and schemes over $\bbk$ are regarded as 
geometric derived stacks as in the way of \cite[\S 2.1.2, \S 2.2.4]{TVe}.

\item 
For a geometric derived stack $X$, we denote by $\LQCoh(X)$ 
the derived $\infty$-category of quasi-coherent sheaves over $X$.
For an affine derived scheme $X=\Spec(R)$, it coincides with the one 
in Notation \ref{ntn:sp:daff} \eqref{i:sp:daff:LQCoh}.
%s homotopy category coincides with the derived 
%category of the abelian category $\QCoh(X)$of quasi-coherent $\shO_X$-modules.

\item
For a geometric derived stack $X$, the \emph{cotangent complex} 
\cite[Chap.\ 1.4]{TVe} is denoted by  $\bbL_X \in \LQCoh(X)$.
If $X$ is locally of finite presentation \cite[1.3.6]{TVe},
then $\bbL_X$ is dualizable \cite[\S 1.4.1]{TVe},
and we denote its dual by $\bbT_X := \bbL_X^{\vee} \in \LQCoh(X)$.
\end{enumerate}
\end{ntn}

%(or a noetherian commutative $\bbQ$-algebra as in \cite{CPTVV}).
%We continue to use the language of derived stacks. 
%We denote by $\icdga_k$ the $\infty$-category of commutative dg-algebras over $\bbk$,
%by $\icdga_k^{\le 0}$ its sub-$\infty$-category spanned by objects 
%concentrated in non-positive degrees,
%and by $\idAff_k$ the $\infty$-category of affine derived schemes over $\bbk$,
%or the opposite $\infty$-category of $\icdga_k^{\le 0}$.

Let $X$ be a geometric derived stack locally of finite presentation.
By \cite[3.1]{CPTVV}, we have a graded $\bbP_{n+1}$-algebra 
\[
 \Pol(X,n) = \bbR \Gamma(X,\Sym_{\shO_X}(\bbT_X[-n-1]))
\]
whose element is called a \emph{$n$-shifted polyvector} of $X$.
For an affine derived scheme $X=\Spec(R)$, 
the cdga $\Pol(X,n)$ is quasi-isomorphic to $\Sym_R(\bbT_R[-n-1])$,
where $\bbT_R$ is the dg Lie algebra of derivations on the cdga $R$
(Example \ref{eg:dga:T}).
%It is an $\bbP_{n+1}$-algebra object in the symmetric monoidal category 
%$\idgVec^{\tgr}$ of graded dg vector spaces 
%(Remark \ref{rmk:sp:dgps} (\ref{i;rmk:sp:dgps:2})),
%and the underlying graded dg vector space is given by 
%$\bigoplus_{p \in \bbN} \Gamma(X,\Sym^p_{\shO_X}(\bbT_X[-n-1]))$
%where $p$ indicates the weight. 
%One can see that $\Pol(X,n)[n]$ has a structure of graded dg Lie algebra.

\begin{dfn}[{\cite[Definition 3.1.1]{CPTVV}}]\label{dfn:sp:dSt}
The space of \emph{$n$-shifted Poisson structures} on $X$ is 
the object $\Pois(X,n) \in \iS$ given by 
\[
 \Pois(X,n) := \uMap_{\ddgla^{\tgr}}(\bbk(2)[-1],\Pol(X,n)[n+1]).
\]
Here $\bbk(2)$ indicates the shift of the weight grading 
(Definition \ref{dfn:sp:wt}),
and $\ddgla^{\tgr}$ denotes (the $\infty$-category associated to) 
the dg category of dg Lie algebras with extra weight gradings.
\end{dfn}

Finally we recall the relation between shifted symplectic and Poisson structures.

\begin{dfn*}
Let $X$ be a geometric derived stack locally of finite presentation.
An $n$-shifted Poisson structure $\pi$ on $X$ is \emph{non-degenerate} 
if the morphism $\pi^{\sharp}: \bbL_X \to \bbT_X[-n]$
induced by the underlying bivector field is an equivalence.
\end{dfn*}

By \cite[Theorem 3.2.4]{CPTVV}, there is an equivalence between 
the $\infty$-category of non-degenerate $n$-shifted Poisson structures
and that of \emph{shifted symplectic structures} in the sense of \cite{PTVV}
on a geometric derived stack locally of finite presentation.
We will not give the detail, and refer \cite{CPTVV} for the detail.
%In this text, we define a shifted symplectic structure by this equivalence.

In this text we only need affine versions of these notions.
Let us collect them in:

\begin{dfn}\label{dfn:sp:nd}
Let $R$ be a cdga.
\begin{enumerate}[nosep]
\item
\label{i:sp:nd:nd}
A $\bbP_n$-algebra structure on $R$ is \emph{non-degenerate}
if the morphism $\pi^{\sharp}: \bbL_R \to \bbT_R[-n]$
induced by the underlying bivector field is a quasi-isomorphism of cdgas.

\item
A $\wh{\bbP}_n$-algebra structure on $R$ is \emph{non-degenerate} if the induced
$\bbP_n$-algebra structure on the cohomology $H(R)$ 
(Remark \ref{rmk:sp:whP} \eqref{i:sp:whP:coh})
is non-degenerate in the sense of \eqref{i:sp:nd:nd}.

\item 
\label{i:sp:nd:symp}
An \emph{$(n-1)$-shifted symplectic structure} on $R$ is defined to be 
a non-degenerate $n$-shifted Poisson structure on $R$.
\end{enumerate}
\end{dfn}

%%%%%%%%%%%%%%%%%%%%%%%%%%%%%%%%%%%%%%%%%%%%%%%%%%%%%%%%%%%%%%%%%%%%%%%%%%%%%%%%
%%%%%%%%%%%%%%%%%%%%%%%%%%%%%%%%%%%%%%%%%%%%%%%%%%%%%%%%%%%%%%%%%%%%%%%%%%%%%%%%
%%%%%%%%%%%%%%%%%%%%%%%%%%%%%%%%%%%%%%%%%%%%%%%%%%%%%%%%%%%%%%%%%%%%%%%%%%%%%%%%
\section{Derived Hamiltonian reduction and classical BRST complex}
\label{s:pr}

In this section we review the work of Safronov \cite{S} on the reduction of 
shifted Poisson algebras and its relation to the classical BRST complex.
We work over a field $\bbk$ with characteristics $0$.

%%%%%%%%%%%%%%%%%%%%%%%%%%%%%%%%%%%%%%%%%%%%%%%%%%%%%%%%%%%%%%%%%%%%%%%%%%%%%%%%
%%%%%%%%%%%%%%%%%%%%%%%%%%%%%%%%%%%%%%%%%%%%%%%%%%%%%%%%%%%%%%%%%%%%%%%%%%%%%%%%
\subsection{Derived Hamiltonian reduction}
\label{ss:pr:pr}

In this subsection we cite from \cite[\S 1]{S}
the reduction of shifted Poisson structure 
in terms of derived coisotropic intersection.

%%%%%%%%%%%%%%%%%%%%%%%%%%%%%%%%%%%%%%%%%%%%%%%%%%%%%%%%%%%%%%%%%%%%%%%%%%%%%%%%
\subsubsection{Coisotropic structure}

In this part we explain the coisotropic structure for shifted Poisson algebras.
It is an enhancement of the notion of coisotropic subschemes of 
affine Poisson schemes.
and also gives the notion of modules over shifted Poisson algebras.
In the next \S \ref{sss:pr:ci},
we explain the construction of homotopy shifted Poisson algebras 
via coisotropic structures.

Let $R$ be a $\bbP_n$-algebra.
Then by Lemma \ref{lem:sp:PSvan}, we have a bivector field 
$\pi_R \in \Pol(R,n-1)(2)$ with the Schouten bracket $[\pi_R,\pi_R]_S=0$
which corresponds to the Poisson bracket of $R$.
%We can also see from Fact \ref{fct:sp:pv} that the operation $[\pi_R,-]_S$ 
%gives a linear map 
%\[
% [\pi_R,-]_S: \Pol(R,n-1)(p)^m \longto \Pol(R,n-1)(p+1)^{m+1}.
%\]
On the other hand, we have the dg $R$-module $\Omega_R^1$
of K\"ahler differentials (Notation \ref{ntn:sp:Omega}).
We denote its differential by $d_{\Omega_R^1}$.

\begin{dfn}\label{dfn:pr:Pc}
For a $\bbP_n$-algebra $R$, its \emph{Poisson center} is the $\bbP_{n+1}$-algebra
\[
 Z(R) = (\wh{\Pol}(R,n-1), d_{Z(R)}, \{\cdot,\cdot\}_{Z(R)})
\]
in $\dgVec$ where
\begin{itemize}[nosep]
\item 
the underlying linear space is defined to be the completion
\[
 \wh{\Pol}(R,n-1) := \uHom_R \bigl(\wh{\Sym}_R(\Omega_R^1[n]),R \bigr) 
 %= \prod_{p \in \bbN} \uHom_R \bigl(\Sym_R^p(\Omega_R^1[n]),R \bigr)
\]
of $\Pol(R,n-1)=\uHom_R(\Sym_R(\Omega_R^1[n]),R)$ 
with respect to the weight grading, and 

\item
the dg $(n+1)$-Poisson structure is induced by 
that on $\Pol(R,n-1)$ (Fact \ref{fct:sp:pv}).
In particular, the Lie bracket is the Schouten bracket.
%cohomological grading is induced by the cohomological grading, 
%\item
%the differential $d_{Z(R)}$ is the one induced by $d_{\Omega_R^1}$ and $[\pi_R,-]_S$, 
\end{itemize}
We also denote by 
\[
 p_R: Z(R) \longto R
\]
the morphism in $\cdga$ 
given by the projection to the weight $0$ part of polyvector fields.
\end{dfn}

We give an example of Poisson center which can be described explicitly.
Let $\frkl$ be a dg Lie algebra.
Recall the Kirillov-Kostant Poisson structure on $\Sym(\frkl)$ 
(Definition \ref{ntn:sp:KK}), which makes $\Sym(\frkl)$ a dg Poisson algebra.
Thus we can consider its Poisson center. 
%(Definition \ref{dfn:pr:Pc}).

\begin{lem}[{\cite[\S 2.1]{S}}]\label{lem:pr:Z=C}
The Poisson center $Z(\Sym(\frkl))$ of the Lie-Poisson algebra $\Sym(\frkl)$ is 
the dg $2$-Poisson algebra of which
\begin{itemize}[nosep]
\item 
the underlying cdga is  
the completion $\wh{\CE}(\frkl,R)$ of the Chevalley-Eilenberg cdga 
$\CE(\frkl,R)$ (Lemma \ref{lem:dga:CEa})
with respect to the weight grading, and

\item
the $2$-Poisson bracket is given by the Schouten bracket (Fact \ref{fct:sp:pv}).
\end{itemize}
\end{lem}

\begin{proof}
Denoting $R:=\Sym(\frkl)$, 
the module $\Omega_{R}^1$ of K\"ahler differentials 
is a free dg $R$-module generated by $\frkl$ (Example \ref{eg:sp:OmSymV}). 
Thus we have an isomorphism 
\[
 Z(R) = \wh{\Pol}(R,0) = \tprd_{p \in \bbN}\uHom_R(\Sym_R(\frkl[1])(p),R)
 \simeq \tprd_{p \in \bbN}\uHom(\Sym(\frkl[1])(p),R) \simeq \wh{\CE}(\frkl,R)
\]
of graded linear spaces.
The compatibility of the dg $2$-Poisson structure
can be checked by comparing that on $\CE(\frkl,R)$ 
(Remark \ref{rmk:dga:CE}, Definition \ref{dfn:dga:cup})
and that on $Z(R)$ (Fact \ref{fct:sp:pv}).
\end{proof}

Now we given the definition of coisotropic structure.

\begin{dfn}\label{dfn:pr:cois}
%Let $R$ be a $\bbP_{n+1}$-algebra, $M$ be a cdga
%and $f: A \to M$ be a morphism of cdgas.
%A \emph{coisotropic structure} on $f$ is a $\bbP_n$-algebra structure on $M$
%and a morphism $\wt{f}:R \to Z(M)$ of $\bbP_{n+1}$-algebras such that 
%$f = p_M \circ \wt{f}$.
Let $R$ be a $\bbP_{n+1}$-algebra and $M$ be a $\bbP_n$-algebra.
A \emph{coisotropic morphism} is a morphism $f: R \to M$ in $\cdga$
equipped with a morphism $\wt{f}:R \to Z(M)$ of $\bbP_{n+1}$-algebras such that 
$f = p_M \circ \wt{f}$ holds in $\cdga$.
\[
 \xymatrix{
  R \ar[rd]_{f} \ar@{.>}[r]^(0.4){\wt{f}} & Z(M) \ar@{->>}[d]^{p_M} \\
  & M}
\]
\end{dfn}

\begin{rmk}\label{rmk:pr:cois}
\begin{enumerate}[nosep]
\item 
We can rewrite the definition of a coisotropic morphism as a family
\[
 \left\{f_p: A \longto \uHom_M\bigl(\Sym(\Omega_M^1[n])(p),M\bigr) 
        \mid p \in \bbN\right\}
\]
with $f_0=f:A \to M$ which satisfies some compatibility conditions. 
See \cite[\S 1.3]{S} for the detail.

\item
Let us explain the origin of the name ``coisotropic morphism".
The compatibility conditions imply that 
the kernel of $f_0$ is closed under the Poisson bracket.
Thus, if $A$ and $M$ are (non-dg) commutative algebras,
then $\Spec(M) \to \Spec(A)$ is a coisotropic subscheme in the ordinary sense.
See \cite[Remark 1.9]{S} for the detail.
\end{enumerate}
\end{rmk}

%%%%%%%%%%%%%%%%%%%%%%%%%%%%%%%%%%%%%%%%%%%%%%%%%%%%%%%%%%%%%%%%%%%%%%%%%%%%%%%%
\subsubsection{Coisotropic intersection}
\label{sss:pr:ci}

We turn to a construction of homotopy shifted Poisson algebras 
by the ``derived" intersection of coisotropic structures,
It will be used for the Poisson reduction in the next \S \ref{sss:pr:pr}.

We start with the definition of homotopy shifted Poisson algebra.
Abstractly, we replace the Lie algebra structure 
by homotopy Lie algebra structure, i.e., $L_{\infty}$-algebra structure.
Explicitly, we have:

\begin{dfn}[{\cite[\S 1.2]{S}}]\label{dfn:sp:whP}
\begin{enumerate}[nosep]
\item 
A \emph{$\wh{\bbP}_n$-algebra in $\dgVec$}, or a \emph{homotopy $n$-Poisson algebra}
is a cdga $A=(A,d_A,\cdot)$ equipped with a family 
\[
 \left\{l_p: A^{\otimes p} \to A[(p-1)n-1] \mid p \in \bbN\right\}
\]
of morphisms in $\cdga$ satisfying the following relations.
\begin{enumerate}[nosep,label=(\roman*)]
\item
$l_1=d_A$, 

\item 
$l_p(a_1,\ldots,a_p)=
(-1)^{\abs{a_i}\abs{a_{i+1}}+n}l_p(a_1,\ldots,a_{i+1},a_i,\ldots,a_p)$.

\item 
$l_p(a_1,\ldots,a_p \cdot a_{p+1})= l_p(a_1,\ldots,a_p) \cdot a_{p+1}+
(-1)^{\abs{a_p}\abs{a_{p+1}}}l_p(a_1,\ldots,a_{p-1},a_{p+1}) \cdot a_p$.

\item
$\sum_{q=1}^p (-1)^{n q (p-q)} \sum_{\sigma \in \frkS_{q,p-q}} \sgn(\sigma)^n 
(-1)^{\ve} l_{p-q+1}(l_q(a_{\sigma(1)},\ldots,a_{\sigma(q)}),
 a_{\sigma(q+1)},\ldots,a_{\sigma(p)})=0$,\\
where $\frkS_{q,p-q}$ and $\ve$ are the same as Definition \ref{dfn:dga:cup}.
\end{enumerate}

\item
\label{i:sh:whP:op}
For a $\wh{\bbP}_n$-algebra $A$, we define the \emph{opposite algebra} 
$A^{\op}$ to be the $\wh{\bbP}_n$-algebra with 
the same underlying cdga and $l_p^{\op} := (-1)^{p+1}l_p$.
\end{enumerate}
\end{dfn}

\begin{rmk}\label{rmk:sp:whP}
\begin{enumerate}[nosep]
\item 
A $\bbP_n$-algebra is nothing but a $\wh{\bbP}_n$-algebra
with $l_p = 0$ for $p \ge 3$ and $l_2=\{-,-\}$.

\item
\label{i:sp:whP:coh}
For a $\wh{\bbP}_n$-algebra $A$, the underlying cohomology $H(A,d_A)$ 
carries a natural $\bbP_n$-algebra structure.

\item
As mentioned in Remark \ref{rmk:sp:sp}, there is a dg operad $\bbP_n$
such that a dg $n$-Poisson algebra is an algebra over $\bbP_n$ in $\dgVec$.
%Since it is Koszul, we can construct its Koszul resolution 
Similarly, we can construct an operad $\wh{\bbP}_n$ such that 
a homotopy $n$-Poisson algebra is nothing but 
an algebra over $\wh{\bbP}_n$.
%See \cite[\S 6.3]{LV} for the general theory,
\end{enumerate}
\end{rmk}

Finally we can explain:

\begin{fct}[{\cite[Theorem 1.18]{S}}]\label{fct:pr:ci}
Let $R$ be a $\bbP_{n+1}$-algebra, $M$ and $N$ be $\bbP_n$-algebras,
and $f:R \to M$ and $g:R \to N$ be coisotropic morphisms.
Then the two sided bar complex 
\[
 M \tbotimes_{f,R,g}^{\bbL} N = M \tbotimes^{\bbL}_R N
\]
has a structure of 
a $\wh{\bbP}_n$-algebra such that the projection 
$M^{\op} \otimes N \to M \otimes_R^{\bbL} N$
is a morphism of $\wh{\bbP}_n$-algebras.
\end{fct}

Let us explain the detail.
Recall first the two-sided bar complex:

\begin{dfn}\label{dfn:pr:bcpx}
Let $A$ be a dg algebra, $L$ be a left dg $A$-module, 
and $M$ be a right dg $A$-module $M$.
We consider the double complex 
\[
 B^{-p,q} :=\tboplus_{i+j+k=q} L^i \otimes (A[1]^{\otimes p})^{j} \otimes M^k 
 \quad (p \in \bbN, q \in \bbZ)
\] 
with the differentials
\begin{align*}
d_v^{-p,q}: B^{-p,q} \longto B^{-p,q+1}, \quad
d_h^{-p,q}: B^{-p,q} \longto B^{-p+1,q}.
\end{align*}
The vertical differential $d_v$ is the one on the tensor product 
$B \otimes T(A[1]) \otimes C$ of complexes (Definition \ref{dfn:dg:dgVmon}).
The horizontal differential $d_h$ is given by 
\[
 d_h^{-p,q}[l|a_1|\cdots | a_p | m] =
 [l.a_1|\cdots|a_p|m] + \tsum_{i=1}^{p-1} 
 (-1)^i [l | \cdots | a_i a_{i+1} | \cdots | m ]+ 
 (-1)^p [l | \cdots | a_p.m].
\]
Here we denoted by 
\[
 [l|a_1|\cdots|a_p|m] \in L \otimes A^{\otimes p} \otimes M
\]
with $l \in L$, $a_i \in A$ and $m \in M$.
Note that we have 
%this formula defines a map $B^{-p,q} \to B^{-p+1,q}$ correctly
%due to 
$(A[1])^{\otimes p}=A^{\otimes p}[p]$.
Now the complex $L \otimes_A^{\bbL} M$ is defined to be the total complex
\[
 L \otimes_A^{\bbL} M := (\Tot(B^{\bl,\bl}),(-1)^p d_v^{-p,q}+d_h^{-p,q})
\]
\end{dfn}

Next we cite some constructions from \cite{S}.
The fist one is:

\begin{fct}[{\cite[Proposition 1.14]{S}}]\label{fct:pr:Pbalg}
Let $A$ be a $\bbP_{n+1}$-algebra.
Then the graded linear space
\[
 T(A[1]) = \tboplus_{p \in \bbN}A^{\otimes p}[p]
\]
has a \emph{$\bbP_n$-bialgebra} structure $(d_{\tbar},\cdot,\Delta_{\dec},\{-,-\})$.
\end{fct}

In other words, 
$\bigl(T(A[1]),d_{\tbar},\cdot,\{-,-\}\bigr)$ is a $\bbP_n$-algebra and 
\[
 \Delta_{\dec}: T(A[1]) \otimes T(A[1]) \longto T(A[1])
\]
is a coassociative comultiplication which is a morphism of $\bbP_n$-algebras.

In order to write down the $\bbP_n$-bialgebra structure,  
let us denote an element of $A^{\otimes p}$ 
%in $T(A[1]) = \bigoplus_{p \in \bbN}A^{\otimes p}[p]$
by $[a_1|\cdots|a_p]$ with $a_i \in A$, following the notation in \cite{S}.
Then the structure is given by
\begin{itemize}[nosep]
\item 
The differential $d_{\tbar}$ is the bar differential:
\begin{align*}
d_{\tbar} [a_1|\cdots|a_p]=
&\tsum_{i=1}^p (-1)^{\sum_{j=1}^{i-1}\abs{a_j}+i-1}[a_1|\cdots|d_A a_i|\cdots|a_p] \\
+&\tsum_{i=1}^{p-1} (-1)^{\sum_{j=1}^{i}\abs{a_j}+i}
 [a_1|\cdots|a_i a_{i+1}|\cdots|a_p].
\end{align*}
Note that as a map on $T(A[1])$ it is indeed of cohomological degree $1$.

\item
The multiplication $\cdot$ is the shuffle product:
\begin{align}\label{eq:pr:shf}
[a_1|\cdots|a_p] \cdot [a_{p+1}|\cdots|a_{p+q}] = 
\tsum_{\sigma \in \frkS_{p,q}}(-1)^{\ve}[a_{\sigma(1)}|\cdots|a_{\sigma(p+q)}],
\end{align}
where $\frkS_{p,q}$ is the set of shuffles and 
$\ve$ is the same braiding sign as in Definition \ref{dfn:dga:cup}.

\item
The comultiplication $\Delta_{\dec}$ is the deconcatenation coproduct:
\begin{align*}
\Delta_{\dec}[a_1|\cdots|a_p]=
\tsum_{i=0}^p [a_1|\cdots|a_i] \otimes [a_{i+1}|\cdots|a_p].
\end{align*}

\item
The Lie bracket is 
\begin{align*}
\{[a_1|\cdots|a_p],[b_1|\cdots|b_q]\}
=\tsum_{i=1}^p \tsum_{j=1}^q 
 &(-1)^{\ve_1+\abs{a_i}+n+1}
 ([a_1|\cdots|a_{i-1}] \cdot [b_1|\cdots|b_{j-1}]) \wedge [\{a_i,b_j\}] \\
 &\wedge ([a_{i+1}|\cdots|a_p] \cdot [b_{j+1}|\cdots|b_q]),
\end{align*}
where $\wedge$ denotes the concatenation:
$[a_1|\cdots|a_i] \wedge [a_{i+1}|\cdots|a_p] = [a_1|\cdots|a_p]$.
\end{itemize}

\begin{rmk*}
The bialgebra structure coincides with the bialgebra 
$(T(V),\mu',\Delta)$ in \cite[\S 1.3.2]{LV},
which is the one of the two standard bialgebra structures 
on the tensor space $T(V)$.
\end{rmk*}

The second construction is:

\begin{fct}[{\cite[Proposition 1.17]{S}}]\label{fct:pr:lPc}
Let $A$ be a $\bbP_{n+1}$-algebra, $M$ be a $\bbP_n$-algebra 
and $f:A \to M$ be a coisotropic morphism.
Then the tensor product $T(A[1]) \otimes M$ of graded linear spaces 
has a \emph{left $\wh{\bbP}_n$-comodule} structure 
$(d_{\tbar},\cdot,\{l_k \mid k \in \bbN\},c)$ over 
the $\bbP_n$-bialgebra $T(A[1])$ in Fact \ref{fct:pr:Pbalg}.
\end{fct}

In other words,
$\bigl(T(A[1]) \otimes M,d_{\tbar},\cdot,\{l_k \mid k \in \bbN\}\bigr)$ 
is a $\wh{\bbP}_n$-bialgebra and
\[
 c: T(A[1]) \otimes M \longto T(A[1]) \otimes (T(A[1]) \otimes M)
\]
is a coassociative left coaction map 
which is also a morphism of $\wh{\bbP}_n$-algebras.

As in the explanation of the $\bbP_n$-bialgebra $T(A[1])$,
we denote by $[a_1|\cdots|a_p|m]$ an element of $A^{\otimes p} \otimes M$.
Then the left $\wh{\bbP}_n$-comodule structure is given as follows.
\begin{itemize}[nosep]
\item 
The differential $d_{\tbar}$ is the bar differential:
\begin{align*}
d_{\tbar} [a_1|\cdots|a_p|m]=
&\tsum_{i=1}^p (-1)^{\sum_{j=1}^{i-1}\abs{a_j}+i-1}
 [a_1|\cdots|d_A a_i|\cdots|a_p|m] 
+(-1)^{\sum_{j=1}^p \abs{a_j}+p} [a_1|\cdots|a_p|d_M m] \\
+&\tsum_{i=1}^{p-1} (-1)^{\sum_{j=1}^{i}\abs{a_j}+i}
  [a_1|\cdots|a_i a_{i+1}|\cdots|a_p]
+(-1)^{\sum_{j=1}^p \abs{a_j}+p} [a_1|\cdots|f(a_p)m].
\end{align*}

\item
The multiplication on the component $T(A[1])$ is the shuffle product 
\eqref{eq:pr:shf}, and 
the one on the component $M$ is that of the given cdga structure of $M$.

\item
As for the $L_\infty$-operations $l_p$'s, 
it is enough to give them for the arguments in $T(A[1])$ or in $M$ 
since we have $[a_1|\cdots|a_p|m]=[a_1|\cdots|a_p|1_M]\cdot[m]$,
They are given by
\begin{align}\label{eq:pr:lPc:Li}
\begin{split}
&l_{p+1}([a_1|\cdots|a_q|1_M],[m_1],\cdots,[m_p]) 
=(-1)^{(\sum_{i=1}^q \abs{a_i}+q)(1-n p)}
 [a_1|\cdots|a_{q-1}|f_p(a_q)(m_1,\ldots,m_p)], \\
&l_2([m_1],[m_2])=[\{m_1,m_2\}_M],
\end{split}
\end{align}
and all the other operations are defined to be zero.
Here $f_p: A \to \uHom_M(\Sym(\Omega^1_M[n])(p),M)$ is the map associated to
the coisotropic morphism explained in Remark \ref{rmk:pr:cois},
and we use the notation \eqref{eq:sp:itp} to get 
$f_p(a_q)(m_1,\ldots,m_p) \in M$.

\item
The coaction map $c$ is the one induced by the multiplication on $T(A[1])$.
\end{itemize}

Finally we can explain the outline of the proof of Fact \ref{fct:pr:ci}.
Let $A$ be a $\bbP_{n+1}$-algebra, $L,M$ be a $\bbP_n$-algebra, 
and $A \to L$, $A \to N$ be coisotropic morphisms.
We denote by $\wt{A}:=T(A[1])$ the $\bbP_n$-bialgebra (Fact \ref{fct:pr:Pbalg}),
and by $\wt{N}:=\wt{A} \otimes N$ the left $\wh{\bbP}_n$-comodule over $\wt{A}$
(Fact \ref{fct:pr:lPc}).
In an opposite way, we can construct a right $\wh{\bbP}_n$-comodule structure 
over $\wt{A}$ on $\wt{M}:=M \otimes \wt{A}$.

The tensor product $\wt{M} \otimes \wt{N}$ has a $\wh{\bbP}_n$-structure
induced by those on $\wt{M}$ and $\wt{N}$.
Now consider the cotensor product
\[
 \wt{M} \otimes^{\wt{A}} \wt{N} := \Eq(\wt{M} \otimes \wt{N} 
 \underset{\id \otimes c_N}{\overset{c_M \otimes \id}{\rightrightarrows}} 
 \wt{M} \otimes \wt{A} \otimes \wt{N}),
\]
where $c_M$ and $c_N$ denote coactions on $\wt{M}$ and $\wt{N}$ respectively,
and the equalizer means the strict equalizer in the category $\dgVec$.
Then we can check that the cotensor product is a $\wh{\bbP}_n$-subalgebra 
of $\wt{M} \otimes \wt{N}$.

It is now enough to construct an isomorphism
$\wt{M} \otimes^{\wt{A}} \wt{N} \simeq M \otimes^{\bbL}_A N$ of cdgas.
Note that the coassociativity of the comultiplication 
$\Delta: \wt{A} \to \wt{A} \otimes \wt{A}$ yields an isomorphism 
\[
 \wt{\Delta}: \wt{A} \longsimto \Eq(\wt{A} \otimes \wt{A}
 \underset{\id \otimes \Delta}{\overset{\Delta \otimes \id}{\rightrightarrows}} 
 \wt{A} \otimes \wt{A} \otimes \wt{A})
\]
of cdgas.
Thus we have the desired isomorphism of cdgas:
\[
  M \otimes_A^{\bbL} N = M \otimes \wt{A} \otimes N 
  \xrightarrow[\id \otimes \wt{\Delta} \otimes \id]{\sim}
  \Eq(M \otimes \wt{A}^{\otimes 2} \otimes N \rightrightarrows 
      M \otimes \wt{A}^{\otimes 3} \otimes N) = \wt{M} \otimes^{\wt{A}} \wt{N}.
\]

%%%%%%%%%%%%%%%%%%%%%%%%%%%%%%%%%%%%%%%%%%%%%%%%%%%%%%%%%%%%%%%%%%%%%%%%%%%%%%%%
\subsubsection{Poisson reduction with respect to momentum maps}
\label{sss:pr:pr}

In this part we explain the work of Safronov 
on the reduction of Poisson algebras following \cite[\S 2]{S}.

Let $R$ be a dg Poisson algebra with a Hamiltonian $\frkl$-action 
(Definition \ref{dfn:sp:Ham}),
and denote by $\mu: \frkl \to R$ the associated momentum map.
Then $R$ can be regarded as a dg $\frkl$-module, and the momentum map gives 
a morphism $\mu: \Sym(\frkl) \to R$ of dg Poisson algebras
(see Remark \ref{rmk:sp:mu} \eqref{i:sp:mu:3}).
Thus the functor $\CE(\frkl,-)$ yields a morphism 
\[
 \CE(\frkl,\mu): \CE(\frkl,\Sym(\frkl)) \to \CE(\frkl,R)
\]
in $\dgMod{\frkl}$.
Since $\Sym(\frkl)$ and $R$ are commutative ring objects in $\dgMod{\frkl}$,
$\CE(\frkl,-)$ yields the Chevalley-Eilenberg cdgas (Lemma \ref{lem:dga:CEa}),
and the map $\CE(\frkl,\mu)$ is also a morphism in $\cdga$.

If moreover $\frkl$ is of finite dimension and concentrated in non-positive 
cohomological degrees, then by Remark \ref{rmk:dga:CE} \eqref{i:CE:ism} we have 
\[
 \CE(\frkl,R) \simeq \Sym(\frkl^*[-1]) \otimes R
\]
as cdgas.
Thus, the $1$-Poisson structure on $R$ induces 
a $1$-Poisson structure on $\CE(\frkl,R)$.
On the other hand, as for the Poisson center $Z(\Sym(\frkl))$ 
of the dg Poisson algebra $\Sym(\frkl)$, we have 
\[
 Z(\Sym(\frkl)) \simeq \wh{\CE}(\frkl,\Sym(\frkl)) 
 = \uHom(\wh{\Sym}(\frkl[1]),\Sym(\frkl))
 = \uHom(\Sym(\frkl[1]),\Sym(\frkl)) = \CE(\frkl,\Sym(\frkl)).
\]
Here the first isomorphism is the one in Lemma \ref{lem:pr:Z=C},
and the second equality comes from 
$\wh{\Sym}(\frkl[1])=\wh{\Wedge}(\frkl)=\Wedge(\frkl)=\Sym(\frkl[1])$,
where we used the assumption on $\frkl$.

%is as a dg $2$-Poisson algebra 
The present situation is summarized as the following diagram in $\cdga$:
\[
 \xymatrix@C0.5em{
  Z(\Sym(\frkl)) \ar@{=}[r] & \CE(\frkl,\Sym(\frkl)) \ar[rrd]_{\CE(\frkl,\mu)} 
  & & Z(\CE(\frkl,R)) \ar@{->>}[d] \\
  & & & \CE(\frkl,R) \ar@{=}[r] & \Sym(\frkl^*[-1]) \otimes R
 }
\]

We now cite a key fact from \cite{S}.
Note that the cdga $\CE(\frkl,\Sym(\frkl))$ is generated by 
$\CE(\frkl,\bbk)$ and $\frkl$.

\begin{fct}[{\cite[Proposition 2.5]{S}}]\label{fct:pr:CEmu}
Let $\frkl$ be a dg Lie algebra which is of finite dimension and 
concentrated in non-positive cohomological degrees.
Let also $R$ be a dg Poisson algebra with a Hamiltonian $\frkl$-action, and denote
by $\mu: \Sym(\frkl) \to R$ the associated momentum map (of dg Poisson algebras).
Then the morphism
\[
 \CE(\frkl,\mu): \CE(\frkl,\Sym(\frkl)) \longto \CE(\frkl,R)
\]
of cdgas is a coisotropic morphism (Definition \ref{dfn:pr:cois}) with the lifting
\[
 \CE(\frkl,\Sym(\frkl)) \longto Z(\CE(\frkl,R)) \simeq 
 \CE(\frkl,\wh{\Sym}(\bbT_M[-1])\otimes\wh{\Sym}(\frkl))
\]
given by the natural embedding 
$\CE(\frkl,\bbk) \inj \CE(\frkl,\wh{\Sym}(\bbT_M[-1])\otimes\wh{\Sym}(\frkl))$
and $x \mapsto \mu(x)-x$ for $x \in \frkl$.
\end{fct}

In particular, we can apply this claim to the trivial $\frkl$-module $R=\bbk$ 
by Remark \ref{rmk:sp:mu} \eqref{i:sp:mu:2}, and find that the morphism 
\[
 \CE(\frkl,0): \CE(\frkl,\Sym(\frkl)) \to \CE(\frkl,\bbk)
\]
is coisotropic.
Given another dg Poisson algebra $R$ with momentum map $\mu$, 
we can then apply Fact \ref{fct:pr:ci} to the coisotropic morphisms
$\CE(\frkl,\mu)$ and $\CE(\frkl,0)$ to obtain
a $\wh{\bbP}_1$-algebra structure on the complex 
\[
 \CE(\frkl,\bbk) 
 \tbotimes^{\bbL}_{\CE(\frkl,0),\CE(\frkl,\Sym(\frkl)),\CE(\frkl,\mu)} 
 \CE(\frkl,R).
\]

\begin{dfn}\label{dfn:pr:dhr}
Let $\frkl$, $R$ and $\mu$ be as in Fact \ref{fct:pr:CEmu}.
The $\wh{\bbP}_1$-algebra obtained above is denoted by 
\[ 
 R \gq^{\bbL}_{\mu} \Sym(\frkl) := 
 \CE(\frkl,\bbk) \otimes^{\bbL}_{\CE(\frkl,\Sym(\frkl))} \CE(\frkl,R),
\]
and is called the \emph{derived Hamiltonian reduction} of $R$ 
with respect to the momentum map $\mu$.
\end{dfn}

Recall now the non-degeneracy of shifted Poisson structure 
(Definition \ref{dfn:sp:nd} \eqref{i:sp:nd:nd}).
We have a relative analogue:
A coisotropic morphism $f: A \to M$ from a $\bbP_{n+1}$-algebra $A$ to 
a $\bbP_n$-algebra $M$ is called \emph{non-degenerate}
if the Poisson structure on $A$ is non-degenerate
and the induced morphism $\bbL_{M/A} \to \bbT_A[-n]$ is an equivalence,
where $\bbL_{M/A}$ is the relative cotangent complex.

\begin{lem}\label{lem:pr:int}
Let $\frkl$, $R$ and $\mu: \Sym(\frkl) \to R$ be as in Fact \ref{fct:pr:CEmu},
and assume that the Poisson structure on $R$ is non-degenerate.
\begin{enumerate}[nosep]
\item 
\label{i:pr:int:mu}
The coisotropic morphism $\CE(\frkl,\mu):\CE(\frkl,\Sym(\frkl)) \to \CE(\frkl,R)$
is non-degenerate.

\item
\label{i:pr:int:nd}
The Poisson structure on the derived Hamiltonian reduction 
$R \gq^{\bbL}_{\mu} \Sym(\frkl)$ is non-degenerate.
\end{enumerate}
\end{lem}

\begin{proof}
\begin{enumerate}[nosep]
\item
The non-derived case is standard, and the same proof works.

\item
By \cite[Theorem 4.23]{MS}, a coisotropic structure is equivalent to 
a \emph{Lagrangian structure} between shifted symplectic structures.
The derived intersection of Lagrangians is also a Lagrangian
by \cite[Theorem 3.1]{C2}. 
Combining these results and \eqref{i:pr:int:mu}, we have the consequence.
\end{enumerate}
%On the other hand, 
%a non-degenerate coisotropic structure Lagrangian 
\end{proof}

%%%%%%%%%%%%%%%%%%%%%%%%%%%%%%%%%%%%%%%%%%%%%%%%%%%%%%%%%%%%%%%%%%%%%%%%%%%%%%%%
%%%%%%%%%%%%%%%%%%%%%%%%%%%%%%%%%%%%%%%%%%%%%%%%%%%%%%%%%%%%%%%%%%%%%%%%%%%%%%%%
\subsection{Classical BRST complex}
\label{ss:pr:BRST}

In this subsection we introduce the classical BRST complex 
for a dg Poisson algebra, and explain the relation to the Poisson reduction 
revealed by Safronov \cite[\S 2.3]{S}.
We continue to work over a field $\bbk$ with characteristics $0$.

We begin with the recollection on the \emph{classical Clifford algebra}.
First we give an implicit definition following \cite[1.4.18, Examples (iii)]{BD}.
\begin{enumerate}[nosep]
\item 
\label{i:pr:cCl}
Let $\frkl$ be a dg Lie algebra, and 
$\frkl^\flat=\frkl \oplus \bbk$ be a one-dimensional central extension. 
We define the \emph{twisted symmetric algebra} $\Sym^\flat(\frkl)$ to be 
the quotient of the symmetric algebra $\Sym(\frkl^\flat)$ by the ideal generated
by the difference of embeddings  
$\bbk = \Sym(\frkl^\flat)^0 \inj \Sym(\frkl^\flat)$ and
$\bbk \inj \frkl^\flat=\Sym(\frkl^\flat)^1 \inj \Sym(\frkl^\flat)$.
This quotient inherits the Poisson bracket from 
the Kirillov-Kostant bracket on $\Sym(\frkl^\flat)$ (\S \ref{sss:sp:KK}).
Thus we obtain a dg Poisson algebra $\Sym^\flat(\frkl)$.

\item
Let $U$ be a complex.
We denote by $U^* = \uHom(U,\bbk)$ the dual complex, and by 
$\pair{\cdot,\cdot}: U^* \otimes U \to \bbk$ the natural pairing.
Regard $U[1] \oplus U^*[-1]$ as a commutative dg Lie algebra,
and let $(U[1] \oplus U^*[-1])^\flat$ be the one-dimensional central extension
determined by the pairing $\pair{\cdot,\cdot}$.
Now we apply the construction \eqref{i:pr:cCl} to $(U[1] \oplus U^*[-1])^\flat$,
and obtain a dg Poisson algebra $\Sym^\flat(U[1] \oplus U^*[-1])$.
\end{enumerate}

\begin{dfn}\label{dfn:pr:cCl}
For a complex $U$, we denote the above dg Poisson algebra by
\[
 \cCl(U) := \Sym^\flat(U[1] \oplus U^*[-1])
\]
and call it the \emph{classical Clifford algebra}.
\end{dfn}

On the naming we refer Remark \ref{rmk:li:coCl}.

The classical Clifford algebra $\cCl(U)$ is explicitly described as follows.
\begin{itemize}[nosep]
\item
As a commutative dg algebra, we have 
$\cCl(U) = \Sym(U[1] \oplus U^*[-1])$.
%\item
%the cohomological grading is given by 
%\begin{align*}
%&\cCl^{\bl}(U) = \bigoplus_{n \in \bbZ} \cCl^{n}(U), \quad 
% \cCl^{n}(U) := \bigoplus_{i+j=n} \cCl^{i,j}(U), \\
%&\cCl^{i,j}(U) := 
% \begin{cases}
%  \Wedge^{-i} U \otimes_{\bbk} \Wedge^j U & (0 \le -i,j \le \dim U), \\
%  0 & (\text{otherwise}), 
% \end{cases}
%\end{align*}

\item
Taking a linear basis $\{u_i \mid i \in I\}$ of $U$ and its dual basis of $U^*$,
we denote the corresponding algebra generators by 
\[
 \ol{\psi}_i   \in U  [1 ], \quad 
 \ol{\psi}^*_i \in U^*[-1]  \quad (i \in I).
\] 
The Poisson bracket is determined by the Leibniz rule and 
\[
 \{\ol{\psi}_i,\ol{\psi}_j\} = 0 = \{\ol{\psi}^*_i,\ol{\psi}^*_j\}, \quad 
 \{\ol{\psi}_i,\ol{\psi}^*_j\} = \delta_{i,j} \quad (i,j \in I).
\]
%$\{x,y\}=\{f,g\}=0$ and $\{f,x\}=f(x)$ for $x,y \in U$ and $f,g \in U^*$.

\item
$\cCl(U)$ has an extra $\bbZ$-grading, called the \emph{charge grading},
given by 
\[
 \chg(\ol{\psi}_i) = -1, \quad \chg(\ol{\psi}^*_i) =1.
\]
\end{itemize}
%The charge grading of an element $x \in \cCl(\frkl)$ is denoted by $\chg x$.
%Thus we have $\chg a =0$ ($a \in A$), 

If $U$ is concentrated in cohomological degree $0$,
then we recover the standard definition
\[
 \cCl(U) \simeq \Wedge(U) \otimes \Wedge(U^*)
\]
as the tensor product of the exterior algebras.
The charge grading is given by
\[
 \cCl^n(U) = 
 \tboplus_{\substack{p, q \in \bbN \\ -p+q=n}}
 \Wedge^{p}(U) \otimes \Wedge^{q}(U^*).
\]

Next we introduce the classical BRST complex.
Following \cite[1.4.23--1.4.25]{BD},
we give a construction in the next lemma.
We can check the statements by direct computation, and omit the proofs.

\begin{lem}\label{lem:pr:cBRST}
Let $\frkl$ be a dg Lie algebra, and $R$ be a dg Poisson algebra 
equipped with a Hamiltonian $\frkl$-action (Definition \ref{dfn:sp:Ham}). 
The momentum map is denoted by $\mu: \frkl \to R$.
%(Remark \ref{rmk:sp:mu} \eqref{i:sp:mu:3}).
\begin{enumerate}[nosep]
\item
The adjoint action of $\frkl$ on itself yields a morphism of dg Lie algebras
\[
 \beta: \frkl \longto \frkl \otimes \frkl^* \longsimto \frkl[1] \otimes \frkl^*[-1] 
 \longinj \cCl(\frkl)^0.
\]

\item
Recall the contractible dg Lie algebra $\frkl_\dagger$ (\S \ref{sss:dga:CE}).
Let
\[
 \ell: \frkl_{\dagger} \longto \cCl(\frkl) \otimes R
\]
be the morphism of graded linear spaces given by
\[
 \ell^{0}:=1 \otimes \mu + \beta \otimes 1: \frkl \longto \cCl(\frkl)^0 \otimes R, 
 \quad
 \ell^{-1}: \frkl[1] \longinj \cCl(\frkl)^{-1} \longto \cCl(\frkl) \otimes R.
\]
Then $\ell$ is a morphism of graded Lie algebras.

\item
We define the following elements.
\begin{itemize}[nosep]
\item
$\wt{\mu} \in \frkl^* \otimes R \subset (\cCl(\frkl) \otimes R)^1$ 
is the element corresponding to $\mu$.
\item 
$\beta'  \in \frkl^* \otimes \cCl(\frkl)^0$  is the element corresponding to $\beta$.
\item
$\beta'' \in \cCl(\frkl)^1$ is the image of $\beta'$
by the product map $\frkl^*[-1] \otimes \cCl(\frkl) \to \cCl(\frkl)$.
\item
$\wt{\beta} \in (\cCl(\frkl) \otimes R)^1$ is the image of $\beta''$
by $\cCl(\frkl) \to \cCl(\frkl) \otimes R$.
\end{itemize}
Then the \emph{classical BRST charge} 
\[
 \ol{Q} := \wt{\mu}+\tfrac{1}{2}\wt{\beta} \in (\cCl(\frkl) \otimes R)^1
\]
satisfies $\{\ol{Q},\ol{Q}\}=0$,
where $\{-,-\}$ denotes the Poisson bracket of the tensor product
$\cCl(\frkl) \otimes R$ of graded Poisson algebras.
\end{enumerate}
\end{lem}

Now we can introduce:

\begin{dfn}\label{dfn:pr:cBRST}
The \emph{classical BRST complex} $\cBRST(\frkl,R,\mu)$ is the dg Poisson algebra 
which is the tensor product
\[
 \cBRST(\frkl,R,\mu) := \cCl(\frkl) \otimes R
\]
as a graded Poisson algebra and the differential is given by
$d_{\cl}:=d_{\cCl(\frkl) \otimes R}+\{\ol{Q},-\}$, where the first term is
the differential of the tensor product of complexes.
The cohomology is denoted by
\[
 H^{\sinf+\bl}_{\cl}(\frkl,R,\mu) := H^{\bl}\cBRST(\frkl,R,\mu),
\]
which is a graded Poisson algebra.
\end{dfn}

%We can easily check the following statements.
Let us recall some properties of the classical BRST complex
following \cite[1.4.26]{BD}.
We omit the proofs.

\begin{lem}\label{lem:pr:BRST2}
Let $(\frkl,R,\mu)$ be the same as in Lemma \ref{lem:pr:cBRST}.
We also use the symbol $\ell$ there.
\begin{enumerate}[nosep]
\item 
$\ell: \frkl_{\dagger} \to \cCl(\frkl) \otimes R$ 
is a morphism of dg Lie algebras.

\item
We have $[\ol{Q},\ell^{-1}]=\ell^0$.

\item
Let $J$ be the dg ideal of $\cBRST(\frkl,R,\mu)$ generated by $\ol{\psi}^*_i$'s.
Then the morphism 
$\Sym(\frkl^*[-1]) \otimes R \to \cCl(\frkl) \otimes R/J$
of graded algebras is surjective,
and the kernel is the ideal generated by $\mu(\frkl)$.

\item 
Denoting $\ol{R}:=R/\mu(\frkl)R$, 
we have $\cBRST(\frkl,R,\mu)/J \simeq \ol{R} \otimes \Sym(\frkl^*[-1])$ as cdgas.

\item
The differential on $\cBRST(\frkl,R,\mu)/J$ induced by $d_{\cl}$ 
coincides with the Chevalley-Eilenberg differential on 
$\ol{R} \otimes \Sym(\frkl^*[-1]) \simeq \CE(\frkl,\ol{R})$,
where we regard $\ol{R}$ as a $\frkl$-module in terms of the action induced by 
$\mu$ and the Poisson bracket.
Thus we have the surjection of cdgas
\[
 \cBRST(\frkl,R,\mu) \longsurj \CE(\frkl,\ol{R}).
\]

\item
On the zero-th cohomology we have a morphism
$H^{\sinf+0}_{\cl}(\frkl,R,\mu) \to \ol{R}^\frkl$
of commutative algebras,
where $\ol{R}^\frkl$ denotes the $\frkl$-invariant part of $\ol{R}$.
Moreover, it is a morphism of Poisson algebras, where $\ol{R}^\frkl$ is 
regarded as a Poisson algebra with the induced Poisson structure from $R$.
\end{enumerate}
\end{lem}

The following lemma can be checked by direct computation.
It recovers the original definition by Kostant and Sternberg \cite{KS}
in the finite-dimensional  case.

\begin{lem}\label{lem:pr:olQ}
%Let $\frkg$ be a \emph{finite-dimensional} dg Lie algebra, 
%and $R$ be a dg Poisson algebra with Hamiltonian $\frkg$-action.
%We denote by $\mu: \Sym(\frkg) \to R$ the momentum map.
%We use the notation $x_i$, $\ol{\psi}_i$ and $\ol{\psi}^*_i$ 
%($i \in I$, $\abs{I}=\dim \frkg$)
%in Definition \ref{dfn:pr:cCl} (replacing $\frkl$ there by $\frkg$).
Assume that the dg Lie algebra $\frkl$ is of finite dimension.
Choose a linear basis $\{x_i \mid i \in I\}$ of $\frkl$,
and let $\ol{\psi}_i \in \frkl[1] \subset \cCl(\frkl)$, 
$\ol{\psi}^*_i \in \frkl^*[-1] \subset \cCl(\frkl)$
be as before.
\begin{enumerate}[nosep]
\item
The classical BRST charge is expressed as 
\[
 \ol{Q} = 
 \sum_{i \in I} %\bigl(\mu(x_i)+c(x_i)\bigr)
 \mu(x_i) \otimes \ol{\psi}^*_i 
-\frac{1}{2} \sum_{i,j,k \in I} c_{i j}^k
  \otimes  \ol{\psi}^*_i \ol{\psi}^*_j \ol{\psi}_k.
\]
Here %$\{x_i \mid i \in I\}$ is a linear basis of $\frkl$ and 
$c_{i j}^k$ denotes the structure constant of $\frkl$: 
$[x_i,x_j]=\sum_{i,j,k \in I}c_{i j}^k x_k$.
%Note that $\ol{Q} \in C^{1}(\frkl^{cl},A)_{\ol{1}}$.

\item 
For $r \in R$ we have
$\{\ol{Q},r \otimes 1\}=\sum_i \{\mu(x_i),r\}_R \otimes\ol{\psi}^*_i$.
\item
For $\eta =\sum_k \eta_k \ol{\psi}_k^* \in \Wedge^1(\frkl^*)$ we have
$\{\ol{Q},1 \otimes \eta\}=
-\frac{1}{2}\sum_{i,j,k}\eta_k c^k_{i j}\ol{\psi}^*_i \ol{\psi}^*_j$,
%where we expand $\eta=\sum_k \eta_k \ol{\psi}_k^*$.
%with respect to the dual basis $\{x^*_k\}$ of $\{x_k\}$.
\item
For $y=\sum_i y_i \ol{\psi}_i \in \Wedge^1(\frkl)$ we have
$\{\ol{Q},1 \otimes y\}=\sum_i y_i \mu(x_i)+
\sum_{i,j,k} c^k_{i j} y_j \ol{\psi}^*_i \ol{\psi}_k$,
%where we expand $y=\sum_i y_i \ol{\psi}_i$. 
%\item
%We have $\{\ol{Q},\ol{Q}\}=0$ so that $d_{\ol{Q}}:=\{\ol{Q},-\}$
%satisfies $(d_{\ol{Q}})^2=0$.
\end{enumerate}
\end{lem}

Let us also recall the relation between the classical BRST complex 
and the \emph{Koszul complex}.

\begin{dfn}\label{dfn:pr:dBRST}
Let $\frkl,R,\mu$ be as in Definition \ref{lem:pr:cBRST}.
\begin{enumerate}[nosep]
\item 
\label{i:pr:dBRST:Kos}
The \emph{Koszul cdga} is the cdga 
\[
 \Kos(\frkl,R,\mu) := (\Sym(\frkl[1]) \otimes R,d,\cdot)
\]
of which
\begin{itemize}[nosep]
\item 
the commutative graded algebra structure $(\Sym(\frkl[1]) \otimes R,\cdot)$
is given by the tensor product of the graded algebras $\Sym(\frkl[1])$ and $R$
(Definition \ref{dfn:dga:dgamon} with forgetting the differentials), and

\item
the differential is $d=d_{\Sym(\frkl) \otimes R} + d_{\tKos}$, 
where $d_{\Sym(\frkl) \otimes R}$ is the tensor product differential
(Definition \ref{dfn:dg:dgVmon}),
and $d_{\tKos}$ is the \emph{Koszul differential} given by
\[
 d_{\tKos}(x_1 \wedge \cdots \wedge x_p \otimes r)
=\tsum_{i=1}^p (-1)^{\sum_{a=1}^{i-1}(\abs{x_a}+1)+\abs{x_i}
                     +\abs{x_i}\sum_{a=i+1}^p(\abs{x_a}+1)} 
 x_1 \wedge \cdots \wh{x_i} \cdots \wedge x_p \otimes \mu(x_i)r.
\]
\end{itemize}
It is a commutative ring object in $\dgMod{\frkl}$.

\item
\label{i:pr:dBRST}
Assume that $\frkl$ is finite dimensional. 
%and concentrated in non-positive cohomological degrees.
We define the dg Poisson algebra 
\[
 \dBRST(\frkl,R,\mu)
\]
as follows.
\begin{itemize}[nosep]
\item 
The underlying cdga is the Chevalley-Eilenberg cdga
\[
 \CE(\frkl,\Kos(\frkl,R,\mu)) 
= \uHom(\Sym(\frkl[1]),\Sym(\frkl[1])\otimes R),
\]
which is isomorphic to 
\[
 \Sym(\frkl^*[-1] \oplus \frkl[1]) \otimes R
\]
under the assumption on $\frkl$.

\item
The underlying graded Lie algebra is the tensor product  
$\Sym(\frkl^*[-1] \oplus \frkl[1]) \otimes R$ %of the given $R$ 
%and the Lie algebra $\Sym(\frkl^*[-1] \oplus \frkl[1])$ 
whose Lie bracket $\{-,-\}$ is given by the Leibniz rule and 
\[
 \{x,y\}=\{f,g\}=0, \quad  \{f,x\}=f(x) \quad (x,y \in \frkl, \ f,g \in \frkl^*).
\]
\end{itemize}
\end{enumerate}
\end{dfn}

\begin{rmk*}
In \eqref{i:pr:dBRST} the differential of the Chevalley-Eilenberg dga 
is indeed a derivation with respect to the Lie bracket 
due to the momentum map equation \eqref{eq:sp:mu}.
\end{rmk*}

Using the explicit form of the classical BRST charge (Lemma \ref{lem:pr:olQ}),
we can check:

\begin{lem*}
%\label{i:pr:cBRST:2}
%Our notation follows \cite[\S 5.1]{K} and \cite[\S 3]{A18}.
Let $\frkl,R,\mu$ be as in Definition \ref{dfn:pr:dBRST} \eqref{i:pr:dBRST}.
Then we have an isomorphism of dg Poisson algebras
\[
 \dBRST(\frkl,R,\mu) \simeq \cBRST(\frkl,R,\mu).
\]
\end{lem*}

Now we explain the work of Safronov \cite[\S 2.3]{S}.
Let $(\frkl,R,\mu)$ be as above.
Then we have the Poisson reduction  
$\CE(\frkl,\bbk) \otimes^{\bbL}_{\CE(\frkl,\Sym \frkl)} \CE(\frkl,R)$.
(Definition \ref{dfn:pr:dhr}).

\begin{fct}[{\cite[Corollary 2.7]{S}}]\label{fct:pr:HBc}
The derived Hamiltonian reduction 
$R \gq^{\bbL}_{\mu} \Sym(\frkl) = 
 \CE(\frkl,\bbk) \otimes^{\bbL}_{\CE(\frkl,\Sym \frkl)} \CE(\frkl,R)$
(Definition \ref{dfn:pr:dhr}) is quasi-isomorphic to
the classical BRST complex as cdgas:
\[
  R \gq^{\bbL}_{\mu} \Sym(\frkl) \dqiseq \cBRST(\frkl,R,\mu)
\] 
\end{fct}

For later use, we copy the proof in loc.\ cit.

\begin{proof}
We will construct the following quasi-isomorphisms of cdgas:
\begin{equation}\label{eq:pr:HBc}
 \cBRST(\frkl,R,\mu) \xrightarrow[\ \qis \ ]{\sim}
 \CE(\frkl,\bbk \otimes^{\bbL}_{\Sym(\frkl)} R) \xrightarrow[\ \qis \ ]{\sim}
 R \gq^{\bbL}_{\mu} \Sym(\frkl).
\end{equation}

Let us build the first quasi-isomorphism.
We have the following quasi-isomorphism of $\frkl$-modules:
\[
 \Sym(\frkl[1]) \otimes R  \underset{\qis}{\longsimto} 
 \bbk \otimes^{\bbL}_{\Sym \frkl} R, \quad
 x_1 \wedge \cdots \wedge x_p \otimes r \longmapsto \sum_{\sigma \in \Sigma_p}
 (-1)^{\ve}[x_{\sigma(1)} | \cdots | x_{\sigma(p)} | r].
\]
Here we used the notation for the two-sided bar complex 
in Definition \ref{dfn:pr:bcpx}, and the signature $\ve$ comes from 
the braiding isomorphisms for the permutation $\sigma$ of $x_i$'s.
By the functoriality of $\CE(\frkl,-)$, we have the desired
quasi-isomorphism of cdgas:
\[
 \cBRST(\frkl,R,\mu) \xrightarrow[\ \qis \ ]{\sim}
 \CE(\frkl,\bbk \otimes^{\bbL}_{\Sym(\frkl)} R).
\]

Next we consider the second quasi-isomorphism.
As we argued in Definition \ref{dfn:pr:dBRST} \eqref{i:pr:dBRST}, 
for a $\frkl$-module $M$  we have an isomorphism 
\[
 \CE(\frkl,M) = \uHom(\Sym(\frkl[1]),M) \simeq \Sym(\frkl^*[-1]) \otimes M
\]
of cdgas since $\frkl$ is finite dimensional.
Then we have 
\[
 R \gq^{\bbL}_{\mu} \Sym(\frkl) = 
 \CE(\frkl,\bbk) \otimes^{\bbL}_{\CE(\frkl,\Sym(\frkl))} \CE(\frkl,R)
 \simeq V \otimes W
\]
as cdgas, where we set
\[
 V := \bbk \otimes^{\bbL}_{\Sym(\frkl)} R, \quad 
 W := \Sym(\frkl^*[-1]) \otimes^{\bbL}_{\Sym(\frkl^*[-1])} \Sym(\frkl^*[-1]). 
\]
Recall that the derived tensor product is presented by the two-sided bar complex
(Definition \ref{dfn:pr:bcpx}): 
$A \otimes_B^{\bbL} C = A \otimes T(B) \otimes C$,
where $T(-)$ denotes the tensor algebra of a dg algebra 
(Example \ref{eg:dga:dga} \eqref{i:dga:dga:T}).
Let us denote $W':=\Sym(\frkl^*[-1])$.
Then the multiplication map gives 
$W \simeq W' \otimes T(W') \otimes W' \to W'$.
It has a splitting
$W' \to W' \otimes T(W') \otimes W'$, $x \mapsto x \otimes 1 \otimes 1$.
Thus we have a quasi-isomorphism
$V \otimes W' \xrightarrow[\qis]{\sim} V \otimes W$ of cdgas, 
which yields the desired quasi-isomorphism of cdgas:
\[
 R \gq^{\bbL}_{\mu} \Sym(\frkl) \xrightarrow[\ \qis \ ]{\sim} V \otimes W' 
 \simeq \CE(\frkl,\bbk \otimes^{\bbL}_{\Sym(\frkl)} R).
\]
\end{proof}

Recall that $R \gq^{\bbL}_{\mu} \Sym(\frkl)$ is a homotopy Poisson algebra,
and $\cBRST(\frkl,R,\mu)$ is a dg Poisson algebra.
The following statement is a slight generalization of \cite[Remark 2.8]{S}:

\begin{prp}\label{prp:pr:HBp}
Let $\frkl$ be a finite-dimensional Lie algebra,
and $R$ be a dg Poisson algebra with Hamiltonian $\frkg$-action.
Then the quasi-isomorphism 
$R \gq^{\bbL}_{\mu} \Sym(\frkg) \dqiseq \cBRST(\frkg,R,\mu)$
in Fact \ref{fct:pr:HBc} gives an equivalence as homotopy Poisson algebras.
\end{prp}

\begin{proof}
It is enough to check that the composition \eqref{eq:pr:HBc}
respects the Poisson bracket $\{\cdot,\cdot\}_{\tBRST}$ on the classical BRST
complex $\cBRST(\frkg,R,\mu)$ and the map $l_2$ of the $L_\infty$-structure on
the derived Hamiltonian reduction $R \gq^{\bbL}_{\mu} \Sym(\frkg)$.
%We denote $W' := \Sym(\frkg^*[-1])$ as before.

Recall that  we have 
\[
 \cBRST(\frkg,R,\mu) = \cCl(\frkg) \otimes R
\]
as graded Poisson algebras, where $\cCl(\frkg)$ denotes the classical 
Clifford algebra, 
%the Poisson bracket $\{\cdot,\cdot\}_{\tBRST}$ is described as in 
and the Poisson bracket can be written as  
$\{\cdot,\cdot\}_{\tBRST}=
 \{\cdot,\cdot\}_{\cCl}\cdot_M + \cdot_{\cCl} \{\cdot,\cdot\}_{\tBRST}$.

On the $L_\infty$-structure of $R \gq^{\bbL}_{\mu} \Sym(\frkg)$,
let us unwind the proof of Fact \ref{fct:pr:ci} explained in \S \ref{sss:pr:ci}.
Denoting 
\[
 A:=\CE(\frkg,\Sym(\frkg)), \ \wt{A}:=T(A[1]), \  
 L:=\CE(\frkg,\bbk), \ \wt{L}:=L \otimes \wt{A}, \ 
 M:=\CE(\frkg,R), \ \wt{M}:=\wt{A} \otimes M,
\] 
we have 
\[
 R \gq^{\bbL}_{\mu} \Sym(\frkg) \simeq \Eq(\wt{L} \otimes \wt{M}
 \rightrightarrows \wt{L} \otimes \wt{A} \otimes \wt{M}),
\]
and the $L_\infty$-structure comes from that on $\wt{L} \otimes \wt{M}$.
It is enough to compute $l_2$ on $\wt{L}$ and $\wt{M}$.
We only argue that on $\wt{M}=\wt{A} \otimes M$, 
which is given by \eqref{eq:pr:lPc:Li}.

The operation coming from $M$ is $\{\cdot,\cdot\}_M$, 
which is obviously compatible with the part of $\{\cdot,\cdot\}_{\tBRST}$. 
Unwinding the definitions, the operation coming from $\wt{A}$ is 
given by the first line of \eqref{eq:pr:lPc:Li} with $p=1$, 
$f_1: A \to \frkg^*[-1] \oplus \bbT_R[-1] \oplus \frkg$,
which is the first component of the lifting $A \to Z(M)$ of 
the coisotropic morphism $\CE(\frkg,\mu): A \to M$.
Under the identification $A=\CE(\frkg,\Sym(\frkg)) \simeq S^* \otimes \Sym(\frkg)$ 
with $S^*:=\Sym(\frkg^*[-1])$, the map $f_1$ is determined by
the image of $S^*$ and $\frkg$.
Recall the description of $f_1$ from Fact \ref{fct:pr:CEmu},
we have $f_1(a \otimes 1)=\pr_1(a)$ for $a \in S^*$, 
$\pr_1: S^* \surj \frkg^*[-1]$, and 
$f_1(1 \otimes x)=\mu(x) - x$ for $x \in \frkg$.
Using this description, 
we can check that under the composition \eqref{eq:pr:HBc}
the two Poisson structures are compatible.
\end{proof}

For later use, we also give a infinite-dimensional modification 
of Proposition \ref{prp:pr:HBp}.

\begin{dfn}\label{prp:pr:rst}
Let $V$ be a linear space equipped with a decomposition 
$V=\bigoplus_{\gamma \in \Gamma}V_\gamma$ by an abelian group $\Gamma$ 
such that each component $V_\gamma$ is finite-dimensional.
\begin{enumerate}[nosep]
\item 
We call such a decomposition a \emph{finite-dimensional $\Gamma$-decomposition}.
\item
We set $V^{\vee}:=\bigoplus_{\gamma \in \Gamma}V_\gamma^*$
and call it the \emph{restricted dual} of $V$.
\item
For another linear space $W$, we denote 
$\Hom^{\trst}(V,W):=\bigoplus_{\gamma \in \Gamma}\Hom(V_{\gamma},W)
\simeq V^{\vee} \otimes W$.
\end{enumerate}
\end{dfn}

\begin{prp}\label{prp:pr:HBinf}
Let $\frkg$ be a Lie algebra,
and $R$ be a Poisson algebra with a Hamiltonian $\frkg$-action.
Assume that $\frkg$ has a finite-dimensional $\Gamma$-decomposition
$\frkg=\bigoplus_{\gamma \in \Gamma} \frkg_\gamma$
and the $\frkg$-action on $M$ is locally finite.
Then replacing $\Hom$ by $\Hom~{\trst}$ in the definition of 
the Chevalley-Eilenberg complexes, we have an quasi-isomorphism
\[
 R \gq^{\bbL}_{\mu} \Sym(\frkg) \dqiseq \cBRST(\frkg,R,\mu)
\]
of homotopy Poisson algebras.
\end{prp}

\begin{proof}
In the proof of Fact \ref{fct:pr:HBc} (and that of Proposition \ref{prp:pr:HBp}),
we used the finite-dimensional assumption only at the isomorphism
$\CE(\frkg,-) = \uHom(\Sym(\frkg[1]),-) \simeq \Sym^*(\frkg[-1]) \otimes -$
for the Chevalley-Eilenberg complex.
Thus under the assumption and the replacement we can apply the same argument, 
and have the consequence from Proposition \ref{prp:pr:HBp}.
\end{proof}

\begin{rmk}\label{rmk:pr:HBp}
Let us give a connection to the argument in \cite[\S 3, p.\ 8]{A18}.
For simplicity we set $\bbk:=\bbC$.
We make the following assumptions:
\begin{enumerate}[nosep,label=(\roman*)]
\item
$\frkg$ is the Lie algebra of a linear algebraic group $G$.

\item 
$R$ is a Poisson algebra equipped with Hamiltonian $\frkg$-action.
We denote $X:= \Spec(R)$ and consider it as an affine Poisson scheme 
with Hamiltonian $G$-action.

\item
The momentum map $\mu_X: X \to \frkg^*$ corresponding to $\mu_R$
is flat as a morphism of schemes.
\end{enumerate} 
Then we have the non-derived Hamiltonian reduction $X \gq G$, 
which is isomorphic to $\mu^{-1}(0)/G$ as a scheme.
On the other hand, from Proposition \ref{prp:pr:HBp} we have an isomorphism 
%the cohomology of each side of the above equivalence vanishes except degree $0$
\[
 \cBRST(\frkg,R,\mu) \simeq R \gq_{\mu} \Sym(\frkg) = \bbC[\mu^{-1}(0)/G]
\] 
of graded Poisson algebras.
Here the symbol $\gq_{\mu}$ in the middle is the non-derived 
Hamiltonian reduction of Poisson algebra.
Thus, taking cohomology, we have 
\[
 H^\bl \cBRST(\frkg,R,\mu) \simeq 
 H^\bl \bbC[\mu^{-1}(0)/G] \simeq 
 \bbC[\mu^{-1}(0)] \otimes H^{\bl}(G,\bbC)
\]
and recover the computation of the classical BRST cohomology in \cite{K}.
\end{rmk}

%%%%%%%%%%%%%%%%%%%%%%%%%%%%%%%%%%%%%%%%%%%%%%%%%%%%%%%%%%%%%%%%%%%%%%%%%%%%%%%%
%%%%%%%%%%%%%%%%%%%%%%%%%%%%%%%%%%%%%%%%%%%%%%%%%%%%%%%%%%%%%%%%%%%%%%%%%%%%%%%%
\subsection{Moore-Tachikawa varieties in derived setting}
\label{ss:pr:MT}

The main purpose of this subsection is to recall 
the proposal of Calaque \cite{C2} 
which gives a derived geometric approach to Moore-Tachikawa varieties.
We make a slight modification and treat affine derived schemes only.

We work over $\bbC$ in this subsection.
As in the latter half of \S \ref{sss:sp:KK}, 
we consider simply connected semisimple algebraic groups over $\bbC$.
We identify such a group $G$ and its Lie algebra $\frkg:=\Lie(G)$.
Recall Definition \ref{dfn:sp:Ham} of Hamiltonian $\frkg$-action
and that we denote by $R^{\op}$ the opposite algebra 
of a $\wh{\bbP}_n$-algebra $R$ (Definition \ref{dfn:sp:whP} \eqref{i:sh:whP:op}).

%the momentum map Lagrangian structure for the momentum map 
%$\mu: X \to \frkg^*$ on a holomorphic symplectic manifold $(X,\omega_X)$
%with Hamiltonian $G$-action.
%It is a Lagrangian structure on 
%the morphism $[\mu]: [X/G] \to [\frkg^*/G]$ of quotient stacks.
%In particular, for any coadjoint orbit $\calO \subset \frkg^*$, 
%there is a Lagrangian structure on $[\calO/G] \to [\frkg^*/G]$.
%Then % by Lemma \ref{lem:dmt:LYM}, 
%the  fiber product
%\[
% [X \gq_{\calO} \,  G] := [X/G] \times_{[\frkg^*/G]} [\calO/G]. 
%\]
%in $\idSt$ has a $0$-symplectic form.

%\begin{dfn*}
%We call the $0$-symplectic derived stack $[X \gq_{\calO}\, G]$ 
%the \emph{derived symplectic reduction} with respect to $\mu$.
%\end{dfn*}

%If we define $\bbR \mu^{-1}(\calO) := X \times_{\frkg^*}\calO$, then we have
%\[
% [X \gq_{\calO} \,  G] \simeq [\bbR \mu^{-1}(\calO)/G].
%\]
%If $G$ acts nicely and $\calO$ is the orbit of a regular value of $\mu$,
%then the reduced derived stack coincides with the reduced scheme of $\mu^{-1}(\calO)/G$.

\begin{dfn}[{c.f.\ \cite[Example 3.5]{C2}}]\label{dfn:pr:MT}
We define the \emph{category $\MT$ of derived Moore-Tachikawa varieties} 
by the following description.
\begin{itemize}[nosep]
\item 
An object is a simply connected semi-simple algebraic group $G$ over $\bbC$.
We identify it with the associated Lie algebra $\frkg$.

\item
A morphisms from $\frkg_1$ to $\frkg_2$ is an equivalence classes of 
$\wh{\bbP}_1$-algebras $R$ in $\dgVec$
together with Hamiltonian ($\frkg_1 \oplus \frkg_2$)-action.  
The momentum map is denoted by 
$\mu_R=\mu_R^1+\mu_R^2: \frkg_1 \oplus \frkg_2 \to R$.
%We regard it as the momentum map Lagrangian structure on 
%$[\mu_X]: [X/G_1 \times G_2] \to [\frkg_1^* \oplus \frkg_2^*/G_1 \times G_2]$.

\item
The composition of $R \in \Hom_{\MT}(\frkg_1,\frkg_2)$ and 
$R' \in \Hom_{\MT}(\frkg_2,\frkg_3)$ is given by
\[
 R' \wtc R := 
 \bigl(R^{\op} \otimes R'\bigr) \gq^{\bbL}_{\mu} \Sym(\frkg_2)
\]
where the momentum map $\mu: \frkg_2 \to R^{\op} \otimes R'$ is given by 
$\mu := -\mu_{R}^2 \otimes 1 + 1 \otimes \mu_{R'}^1$.
We call $R' \wtc R$ the (\emph{derived}) \emph{gluing}.
\end{itemize}
We denote by $\MT^{\nd} \subset \MT$ the full subcategory
spanned by non-degenerate $\wh{\bbP}_1$-algebras $R$
(Definition \ref{dfn:sp:nd}).
\end{dfn}

\begin{rmk}\label{rmk:pr:MT}
\begin{enumerate}[nosep]
\item
The gluing $R' \wtc R$ has a Hamiltonian ($\frkg_1 \oplus \frkg_3$)-action 
with the momentum map $-\mu_{R}^1 \otimes 1+1 \otimes \mu_{R'}^2$.
Thus $\wtc$ is well-defined.
The associativity of composition holds since $\gq^{\bbL}_{\mu}$ is realized 
as a pushforward in the $\infty$-category $\cdga$.
Note also that if $R$ and $R'$ is non-degenerate, then $R' \wtc R$ is also 
non-degenerate by Lemma \ref{lem:pr:int} \eqref{i:pr:int:nd}. 
Thus $\MT^{\nd}$ is surely a full subcategory.

\item
As we recalled in \S \ref{sss:0:MT},
the original proposal of Moore and Tachikawa in \cite{MT} is given 
in terms of (non-derived) holomorphic symplectic varieties
with Hamiltonian actions and an additional $\bbC^*$-action.

\item
\label{i:pr:MT:C}
We explain the definition of the category by Calaque \cite[Example 3.5]{C2}.
Objects are the same, and a morphism from $G_1$ to $G_2$ is a derived 
$0$-symplectic scheme $(X,\omega_X)$ equipped with 
Hamiltonian $(G_1 \times G_2)$-action.
We denote by $\mu_X: X \to \frkg^*$ the corresponding momentum map.
The composition of $X \in \Hom(G_1,G_2)$ and $Y \in \Hom(G_2,G_3)$ is given by
\[
 [\bigl((X^{\op} \times Y) \gq^{\bbL} \Delta(G_2)\bigr)/(G_1 \times G_3)] \simeq
 [X^{\op}/(G_1 \times G_2)] \times^{\bbL}_{[\frkg_2^*/G_2]} [Y/(G_2 \times G_3)]
\]
with the following notations:
\begin{itemize}[nosep]
\item 
$X^{\op}$ denotes the derived scheme $X$ with the opposite $0$-symplectic
structure $-\omega_X$.
\item
For a derived scheme $Z$ with $G$-action, 
the symbol $[Z/G]$ denotes the quotient derived stack.
\item
In the left hand side $(X^{\op} \times Y) \gq^{\bbL} \Delta(G_2)$ 
denotes the derived symplectic reduction \cite[Example 3.4]{C2}
with respect to the diagonal $G_2$-action and 
the momentum map $-\mu_X+\mu_Y: X^{\op} \times Y \to \frkg_2^*$.
\item
The right hand side denotes the derived Lagrangian intersection
\cite[Example 3.2]{C2}.
\end{itemize}
%at the trivial orbit $\{0\} \subset \frkg^*$.

We can partially recover Calaque's definition from $\MT^{\nd}$ in an easy manner.
By Definition \ref{dfn:sp:nd} \eqref{i:sp:nd:symp} of shifted symplectic 
structure, we can restate that a morphism in $\MT$ is a cdga 
equipped with a $0$-symplectic structure and a Hamiltonian action.
Adding the condition that all the cdgas are concentrated in non-positive 
cohomological degree, we recover the description in \eqref{i:pr:MT:C}
by translating the construction on cdgas to affine derived schemes.
\end{enumerate}
\end{rmk}

%%%%%%%%%%%%%%%%%%%%%%%%%%%%%%%%%%%%%%%%%%%%%%%%%%%%%%%%%%%%%%%%%%%%%%%%%%%%%%%%
%%%%%%%%%%%%%%%%%%%%%%%%%%%%%%%%%%%%%%%%%%%%%%%%%%%%%%%%%%%%%%%%%%%%%%%%%%%%%%%%
\subsection{The case of $G$-equivariant Poisson algebras}
\label{ss:pr:G}

In this subsection we translate the arguments on classical BRST reduction 
discussed in the beginning of \cite[\S3]{A18} into our language.
We work over $\bbC$ here.

Let $G$ be a simply connected semisimple algebraic group.
We denote by $\frkg:=\Lie(G)$ the Lie algebra of $G$.
We regard the linear dual $\frkg^*$ as an affine scheme with the coordinate ring
$\bbC[\frkg^*] = \Sym(\frkg)$, 
so that the Lie-Poisson algebra structure makes $\frkg^*$ 
an \emph{affine Poisson scheme}.

For the Lie-Poisson algebra $\Sym(\frkg)$,
a Poisson $\Sym(\frkg)$-module $M$ in $\cVec$ (Definition \ref{dfn:sp:mod}) 
is nothing but a $\Sym(\frkg)$-module 
together with a morphism $\ad: \frkg \to \End(M)$ of Lie algebras
such that $\ad(x)(s.m)=\{x,s\}.m+s.\ad(x)(m)$ for any $x \in \frkg$,
$s \in \Sym(\frkg)$ and $m \in M$.
We can recast such a Poisson module in terms of sheaves over 
the affine Poisson scheme $\frkg^*$.

\begin{ntn}\label{ntn:pr:Gsch}
Let us consider a scheme $X$ over $\bbC$.
\begin{enumerate}[nosep]
\item 
If $X$ is a $G$-scheme, then 
we denote by $\QCoh^G(X)$ the category of $G$-equivariant quasi-coherent sheaves
of $\shO_{X}$-modules. 

\item
%For a Poisson algebra $R$, 
%let $X = \Spec(R)$ be the corresponding affine Poisson scheme. 
If $X$ is a Poisson scheme, then 
we denote by $\PQCoh(X)$ the category of sheaves on $X$ 
which are both quasi-coherent sheaves of $\shO_X$-modules 
and sheaves of Lie algebra $\shO_X$-modules whose local sections satisfy 
the relation in Definition \ref{dfn:sp:mod} \ref{i:sp:mod:3}.
\end{enumerate}
\end{ntn}

Regarding $\frkg^*$ as a $G$-scheme by the coadjoint action,
we find that an object of $\QCoh^G(\frkg^*)$ has a structure of a sheaf of 
Lie algebra $\shO_{\frkg^*}$-modules induced by the coadjoint action of $G$.
Thus we have an embedding 
\[
 \QCoh^G(\frkg^*) \longinj \PQCoh(\frkg^*).
\]
On the other hand, since $\frkg^* = \Spec(\bbC[\frkg^*])=\Spec(\Sym(\frkg))$ 
is an affine Poisson scheme,
we have an equivalence of categories
\[
 \PQCoh(\frkg^*) \longsimto \PMod{\Sym(\frkg)}.
\]

For a given $\shM \in \QCoh^G(\frkg^*)$, let us denote by $M$
the corresponding Poisson $\Sym(\frkg)$-module under the composition 
$\QCoh^G(\frkg^*) \inj \PQCoh(\frkg^*) \simeq \PMod{\Sym(\frkg)}$.
Then we can see that the Lie algebra morphism $\ad: \frkg \to \End(M)$ corresponds to 
the Lie algebra action of $\shO_{\frkg^*}$ on $\shM$ induced by the $G$-action.
In total, we have:

\begin{lem}\label{lem:MT:QCoh}
We have an equivalence of categories
\[
 \QCoh^G(\frkg^*) \simeq \PMod{\Sym(\frkg)}^{\lf},
\]
where $\PMod{\Sym(\frkg)}^{\lf}$ denotes the full subcategory of 
$\PMod{\Sym(\frkg)}$ spanned by those objects 
on which the adjoint action of $\frkg$ is locally finite.
\end{lem}

The category $\QCoh^G(\frkg^*)$ is a symmetric monoidal category 
by the tensor product $\otimes_{\shO_{\frkg^*}}$,
so that we have the notion of \emph{Poisson algebra objects} therein.
Using the equivalence in Lemma \ref{lem:MT:QCoh} 
and unwinding the definition, we recover the claim in \cite[\S 3, p.\ 8]{A18}:

\begin{lem}\label{lem:pr:pao}
A Poisson algebra object in the symmetric monoidal category $\QCoh^G(\frkg^*)$ is 
a Poisson algebra $R$ equipped with a morphism $\mu_R: \Sym(\frkg) \to R$ of 
Poisson algebras under which the adjoint action of $\frkg$ is locally finite.
\end{lem}

In particular, a Poisson algebra object $R$ is equipped with Hamiltonian 
$\frkg$-action, and the morphism $\mu_R$ is the corresponding momentum map.
Thus we can consider Hamiltonian reduction of Poisson algebra objects.

We have an obvious dg analogue of the above arguments.
Recall Notation \ref{ntn:sp:daff} \eqref{i:sp:daff:LQCoh}: 
We denote by $\LQCoh(\frkg^*)$ 
the $\infty$-category of quasi-coherent sheaves of $\shO_{\frkg^*}$-modules,
whose homotopy category is equivalent to the derived category $\DQCoh(\frkg^*)$. 
It is a symmetric monoidal category under the derived tensor product 
$\otimes^{\bbL}_{\shO_{\frkg^*}}$.
We also have the $\infty$-category $\eLQCoh{G}(\frkg^*)$ of $G$-equivariant sheaves,
which is also equivalent to $\LQCoh([\frkg^*/G])$ of sheaves over 
the quotient stack $[\frkg^*/G]$.
The $\infty$-category $\eLQCoh{G}(\frkg^*)$ is also symmetric monoidal, 
and we can consider Poisson algebra objects therein. 

As in the arguments for Lemma \ref{lem:pr:pao},
a Poisson algebra object in $\eLQCoh{G}(\frkg^*)$ is nothing but 
a dg Poisson algebra $R$ equipped with a morphism $\mu_R: \Sym(\frkg) \to R$ of 
dg Poisson algebras under which the adjoint action of $\frkg$ is locally finite.

Now let us consider Hamiltonian reduction of Poisson algebra objects.
We are interested in the derived gluing in the category $\MT$, 
or the composition of morphisms therein (Definition \ref{dfn:pr:MT}).
Thus, instead of considering the reduction of an object $R$,
let us consider the reduction of tensor product $R^{\op} \otimes R'$
% of two objects $R$ and $R'$,
where $R^{\op}$ denotes the Poisson algebra $(R,-\{\cdot,\cdot\}_R)$ 
equipped with the momentum map $-\mu_R$.
Then the derived Hamiltonian reduction corresponding to 
the composition of morphisms is described in the following way.

\begin{prp}\label{prp:pr:glue}
Let $R$ and $R'$ be Poisson algebra objects in $\eLQCoh{G}(\frkg^*)$.
We denote by $\mu_R$ and $\mu_{R'}$ the corresponding momentum maps respectively,
%We also denote by $R^{\op}$ the dg Poisson algebra with momentum map $-\mu_R$.
and define 
$\mu:=-\mu_R \otimes 1 + 1 \otimes \mu_{R'}: \Sym(\frkg) \to R^{\op} \otimes R'$. 
Then we have an quasi-isomorphism of homotopy Poisson algebras
\[
 \cBRST(\frkg,R^{\op}\otimes R',\mu) \dqiseq 
 R' \wtc R = (R^{\op} \otimes R') \gq^{\bbL}_{\mu} \Sym(\frkg).
\] 
%of the derived Hamiltonian reduction and the BRST reduction 
\end{prp}

\begin{proof}
This is a direct consequence of Proposition \ref{prp:pr:HBp}.
\end{proof}

\begin{rmk}\label{rmk:pr:glue}
Let us continue Remark \ref{rmk:pr:HBp} and give a connection to 
the argument in \cite[\S 3, pp.\ 8--9]{A18}.
We make the following assumptions on $R$:
\begin{enumerate}[nosep,label=(\roman*)]
\item 
$R$ is a Poisson algebra object in $\QCoh^{G}(\frkg^*)$.
We denote $X:=\Spec(R)$, 
and consider $X$ as a affine Poisson scheme with Hamiltonian $G$-action.

\item
There is a closed subscheme $S \subset X$ such that the action map gives 
an isomorphism $G \times S \simto X$.

\item
The momentum map $\mu_X: X \to \frkg^*$ corresponding to $\mu_R$
is flat as a morphism of schemes.
\end{enumerate} 
Let $R'$ be another Poisson object in $\QCoh^{G}(\frkg^*)$.
Denoting $X':= \Spec(R')$ the corresponding affine Poisson scheme
with the momentum map $\mu_{X'}: X' \to \frkg^*$ and $X^{\op}:=\Spec(R^{\op})$,
we have the affine Poisson scheme $X^{\op} \times X$ 
with the flat momentum map $\mu(x,x'):=-\mu_X(x)+\mu_{X'}(x')$.
Thus we have the non-derived Hamiltonian reduction 
$(X^{\op} \times X')\gq \Delta(G)$.
We also have
\[
 \mu^{-1}(0) \simeq X \times_{\mu_X,\frkg^*,\mu_{X'}} X' 
 \simeq G \times (S \times_{\mu_S,\frkg^*,\mu_{X'}}X'),
\]
where $\mu_S:=\rst{\mu_X}{S}: S \to \frkg^*$ is the restriction.
Thus we have 
\[
 (X^{\op} \times X')\gq \Delta(G) \simeq \mu^{-1}(0)/G \simeq 
 S \times_{\frkg^*}X'
\]
as schemes.
On the other hand, from Proposition \ref{prp:pr:glue} we have an isomorphism 
%the cohomology of each side of the above equivalence vanishes except degree $0$
\[
 \cBRST(\frkg,R^{\op} \otimes R',\mu) \simeq 
 (R^{\op} \otimes R') \gq_{\mu} \Sym(\frkg) 
 \simeq \bbk[(X^{\op} \times X')\gq \Delta(G)]
\] 
of graded Poisson algebras.
Here the symbol $\gq_{\mu}$ in the middle denotes 
the non-derived Hamiltonian reduction of Poisson algebra.
Combining these isomorphisms, we have 
\[
 H^\bl \cBRST(\frkg,R^{\op} \otimes R',\mu) \simeq 
 H^\bl \bbk[\mu^{-1}(0)/G] \simeq 
 \bbk[S \times_{\frkg^*}X'] \otimes H^{\bl}(G,\bbC).
\]
%$H^{\bl}(\frkg,\bbC)$ denotes the Lie algebra cohomology of 
%the trivial representation of $\frkg$.
Thus we recover the formula \cite[(13)]{A18}.
\end{rmk}

%%%%%%%%%%%%%%%%%%%%%%%%%%%%%%%%%%%%%%%%%%%%%%%%%%%%%%%%%%%%%%%%%%%%%%%%%%%%%%%%
%%%%%%%%%%%%%%%%%%%%%%%%%%%%%%%%%%%%%%%%%%%%%%%%%%%%%%%%%%%%%%%%%%%%%%%%%%%%%%%%
%%%%%%%%%%%%%%%%%%%%%%%%%%%%%%%%%%%%%%%%%%%%%%%%%%%%%%%%%%%%%%%%%%%%%%%%%%%%%%%%
\section{Jet and arc spaces in derived setting}
\label{s:jet}

%%%%%%%%%%%%%%%%%%%%%%%%%%%%%%%%%%%%%%%%%%%%%%%%%%%%%%%%%%%%%%%%%%%%%%%%%%%%%%%%
%%%%%%%%%%%%%%%%%%%%%%%%%%%%%%%%%%%%%%%%%%%%%%%%%%%%%%%%%%%%%%%%%%%%%%%%%%%%%%%%
\subsection{Jet and arc spaces for ordinary schemes}
\label{ss:jet:ord}

Let us recall the definition of jet spaces in the ordinary setting.
We refer \cite[\S 7.6]{BLR} and \cite{EM} for the detail.

Let $\bbk$ be a field, and $\Sch$ be the category of schemes over $\bbk$.

\begin{fct}[{\cite[\S 7.6, Theorem 4]{BLR}}]\label{fct:jet:Jn}
For $n \in \bbN$, 
there exists a functor $J_n: \Sch \to \Sch$ such that 
%scheme $J_n(X)$ of finite type over $\bbk$
% and a scheme $X$ of finite type over $\bbk$,
%there exists a scheme $J_n(X)$ of finite type over $\bbk$
such that for any $S \in \Sch$  %every $\bbk$-algebra $A$ 
we have a functorial bijection
\[
 \Hom_{\Sch}(S,J_n(X)) \simeq 
 \Hom_{\Sch}(S \times_{\Spec(\bbk)} \Spec(\bbk[t]/(t^{n+1})),X).
\]
We call $J_n(X)$ the \emph{$n$-th jet space of $X$}.
\end{fct}

\begin{cor}\label{cor:jet:Jn}
By the defining condition as a functor of points, we have
\begin{enumerate}[nosep] 
\item
$J_n(X)$ is unique up to a canonical isomorphism.
\item
$J_0(X) \simeq X$.
\item\label{i:cor:jet:Jn:pi}
For $m,n \in \bbN$ with $m \ge n$, 
there is a morphism $\pi_{m,n}: J_m \to J_n$ of functors 
induced by the truncation $\bbk[t]/(t^{m+1}) \surj \bbk[t]/(t^{n+1})$. 
These morphisms satisfy $\pi_{n,p} \circ \pi_{m,n} = \pi_{m,p}$
for $m \ge n \ge p$, so that we have an inverse system 
$\{\pi_{m,n} \mid m,n \in \bbN, m \ge n\}$ on the direct set $\{\bbN, \le\}$.
%\item
%If $f: X \to Y$ is a morphism of schemes of finite type over $\bbk$,
%then we have a morphism $J_n(f): J_n(X) \to J_n(Y)$.
%Thus we have a functor $J_n$ from $\Sch^{\ft}$ to itself.
%\item\label{i:cor:jet:Jn:nf}
%The morphism $\pi_{m,n}$ in \eqref{i:cor:jet:Jn:pi} gives a natural functor 
%between $J_m$ and $J_n$.
In particular, for $n \in \bbN$, 
we have the following commutative diagram in $\Sch$.
\[
 \xymatrix{J_n(X) \ar[r]^{J_n(f)} \ar[d]_{\pi^X_{n,0}} & 
           J_n(Y) \ar[d]^{\pi^Y_{n,0}} \\ X \ar[r]_{f} & Y}
\]
\end{enumerate}
\end{cor}

We denote by $\Sch^{\ft}$ the full subcategory of schemes of finite type over $\bbk$.
We then have
%By the construction (see the proof of \cite[Proposition 2.2]{EM}),
%we also have

\begin{fct}\label{fct:jet:EM2.2}
Let $n \in \bbN$ be arbitrary.
\begin{enumerate}[nosep]
\item
If a morphism $f$ in $\Sch$ is smooth, then so is $J_n(f)$
\cite[\S 7.6, Proposition 5]{BLR}.

\item 
$J_n$ gives a functor $J_n: \Sch^{\ft} \to \Sch^{\ft}$.
In other words, for a scheme $X$ of finite type over $\bbk$,
the $n$-th jet space $J_n(X)$ is also of finite type over $\bbk$
\cite[\S 7.6, Proposition 5]{BLR}.

\item
For any $X \in \Sch^{\ft}$, 
the morphism $\pi_{n,0}: J_n(X) \to J_{n,0}(X)=X$ is affine
\cite[\S 2]{EM}.
\end{enumerate}
\end{fct}

By \cite[Lemma 2.9]{EM}, 
if $f: X \to Y$ is an \'etale morphism in $\Sch^{\ft}$,
then for every $n \in \bbN$ the commutative diagram in 
Corollary \ref{cor:jet:Jn} \eqref{i:cor:jet:Jn:pi} is cartesian.
Using this fact, we can prove

\begin{fct*}[{\cite[Lemma 2.9, Remark 2.10]{EM}}]
Let $f: X \to Y$ be a morphism in $\Sch^{\ft}$, 
and let $n \in \bbN$ be arbitrary.
If $f$ is \'etale, then  $J_n(f)$ is also \'etale.
\end{fct*}

\begin{rmk}\label{rmk:jet:aff}
Following \cite[\S3]{EM} and \cite[\S3]{AM},
we give an explicit description of $J_n(X)$ for an affine $X$.
We assume the characteristics of $\bbk$ is $0$.
\begin{enumerate}[nosep]
\item 
\label{i:jet:aff:A^N}
First we assume $X=\bbA^N=\Spec(\bbk[x^1,\ldots,x^N])$.
Then we have 
\[
 J_n(\bbA^N) = \Spec(\bbk[x^i_{(-j-1)} \mid i=1,\ldots,N, \, j=0,\ldots,n]).
\]
Indeed, for an affine scheme $S=\Spec(A)$,
a morphism $a: \Spec(A[t]/(t^{n+1})) \to \bbA^N$ 
corresponds to a morphism $a^*: k[x^i] \to A[t]/(t^{n+1})$.
We set $a^*(x^i)=\sum_{j=0}^n a^i_{(-j-1)}t^j/j!$.
Then we have the morphism 
$\alpha^*: \bbk[x^i_{(-j-1)}] \to A$, $(x^i_{(-j-1)}) \mapsto (a^i_{(-j-1)})$
corresponding to an $A$-valued point $\alpha: \Spec(A) \to J_N(\bbA^N)$.
%$(a^i_{(-j-1)})$ of $\Spec(\bbk[x^i_{(-j-1)}])$.
The correspondence $a \mapsto \alpha$ gives the desired functorial bijection
\[
 \Hom_{\Sch}(\Spec(A) \times_{\Spec(\bbk)} \Spec(\bbk[t]/(t^{n+1})),\bbA^N)
 \simeq \Hom_{\Sch}(\Spec(A),J_n(\bbA^N)).
\]

\item
\label{i:jet:aff:aff}
Next we consider the case $X=\Spec(R)$, 
$R=\bbk[x^1,\ldots,x^N]/I$ with $I=(f_1,\ldots,f_M)$.
Define a $0$-derivation $T$ on $\bbk[x^i_{(-j-1)}]$ 
(Definition \ref{dfn:dga:der} \eqref{i:dga:der:3}) 
by $T x^i_{(-j)} = j x^i_{(-j-1)}$ for $j \in \bbZ_{>0}$.
Then we have 
\begin{align*}
&J_n(\Spec(R)) = \Spec(J_n(R)), \\
&J_n(R) := \bbk[x^i_{(-j-1)} \mid i=1,\ldots,N, \, j=0,\ldots,n]/
           (T^k f_l \mid k=0,\ldots,n, \, l=1,\ldots,M).
\end{align*}
\end{enumerate}
\end{rmk}

Now we turn to the arc space.
By Corollary \ref{cor:jet:Jn} \eqref{i:cor:jet:Jn:pi} and 
Fact \ref{fct:jet:EM2.2}, we have an inverse system 
$\{\pi_{m,n}: J_m(X) \to J_n(X) \mid m,n \in \bbN, m \ge n\}$
of affine morphisms in $\Sch^{\ft}$.
Thus the limit exists in $\Sch$.
In the affine case $X=\Spec(R)$, the limit is also affine.

\begin{dfn}\label{dfn:jet:arc}
\begin{enumerate}[nosep]
\item 
Let $X$ be a scheme over $\bbk$.
We denote the limit of $\{\pi_{m,n}: J_m(X) \to J_n(X)\}$ 
in $\Sch$ by $J_{\infty}(X)$,
and call it the \emph{arc space} or the \emph{$\infty$-jet space} of $X$.
We also denote by $\psi_n: J_{\infty}(X) \to J_n(X)$ the projection.

\item
For a commutative ring $R$, 
we denote by $J_{\infty}(R)$ the commutative ring whose spectrum 
gives the arc space: $\Spec(J_{\infty}(R)) = J_{\infty}(\Spec(R))$.
\end{enumerate}
\end{dfn}

The properties of the $n$-jet space $J_n(X)$ are inherited by $J_\infty(X)$.
For example, we have: 

\begin{lem*}
\begin{enumerate}[nosep]
\item 
For $X \in \Sch^{\ft}$ and a commutative algebra $A$ over $\bbk$, 
there is a bijection 
\[
  \Hom_{\Sch}(\Spec(A),J_\infty(X)) \simeq \Hom_{\Sch}(\Spec(A[[t]]),X).
\]
\item
The correspondence $X \mapsto J_\infty(X)$ gives a functor $\Sch^{\ft} \to \Sch$.
\item
If $f: X \to Y$ is an \'etale morphism, then there is a cartesian diagram
\[
 \xymatrix{J_\infty(X) \ar[r]^{J_\infty(f)} \ar[d]_{\psi^X_{0}} & 
           J_\infty(Y) \ar[d]^{\psi^Y_{0}} \\ X \ar[r]_{f} & Y}
\]
\end{enumerate}
\end{lem*}

Let us consider the case when 
$X$ is an affine scheme of finite type over $\bbk$ of characteristics $0$.
Expressing $X=\Spec(R)$ with $R=\bbk[x^1,\ldots,x^N]/(f_1,\ldots,f_M)$, 
we have $J_\infty(\Spec(R)) = \Spec(J_\infty(R))$ with 
\begin{align*}
J_\infty(R) := \bbk[x^i_{(-j-1)}\mid i=1,\ldots,N, \, j \in \bbN]/
                (T^k f_l \mid k \in \bbN, \, l=1,\ldots,M).
\end{align*}
Note that the $\bbk$-algebra $J_\infty(R)$ inherits the $0$-derivation $T$ on 
$J_n(\bbk[x^1,\ldots,x^N])$.
%=\bbk[x^i_{(-j-1)}\mid i=1,\ldots,N, \, j =0,\ldots,n]$.
We denote the induced $0$-derivation on $J_\infty(R)$ by the same symbol $T$.
Then the description above yields:

\begin{fct}[{\cite[Remark 3.1]{EM}}]\label{fct:jet:da}
For a commutative algebra $R$ of finite type, we have a morphism 
$j: R \to J_\infty(R)$ of algebras such that given a commutative algebra $R'$ 
with a $0$-derivation $T'$ and a morphism $j': R \to R'$ of algebras, 
there is a unique morphism 
$h: J_\infty(R) \to R'$ of algebras making the diagram 
\[
 \xymatrix{R \ar[rr]^j \ar[rd]_{j'} & & J_\infty(R) \ar@{.>}[ld]^h \\ & R'}
\] 
commute and satisfying $T' h = h T$, i.e., giving a morphism 
$h: (J_\infty(R),T) \to (R',T')$ of differential algebras.
\end{fct}

Here we used:

\begin{dfn}\label{dfn:jet:da}
A commutative algebra equipped with a $0$-derivation 
(Definition \ref{dfn:dga:der} \eqref{i:dga:der:3})
is called a \emph{differential algebra}.
\end{dfn}

Using this Fact, we can show:

\begin{lem}\label{lem:jet:G}
For a linear algebraic group $G$ over $\bbk$, 
the arc space is given by the proalgebraic group $J_{\infty}(G)=G[[t]]$.
\end{lem}

%%%%%%%%%%%%%%%%%%%%%%%%%%%%%%%%%%%%%%%%%%%%%%%%%%%%%%%%%%%%%%%%%%%%%%%%%%%%%%%%
%%%%%%%%%%%%%%%%%%%%%%%%%%%%%%%%%%%%%%%%%%%%%%%%%%%%%%%%%%%%%%%%%%%%%%%%%%%%%%%%
\subsection{Jet and arc spaces for derived schemes}

Let us give a derived analogue of the previous \S \ref{ss:jet:ord}.
%In the context of derived algebraic geometry,
%we have the \emph{mapping stack} 
In this subsection we work over a field $\bbk$ containing $\bbQ$.

%%%%%%%%%%%%%%%%%%%%%%%%%%%%%%%%%%%%%%%%%%%%%%%%%%%%%%%%%%%%%%%%%%%%%%%%%%%%%%%%
\subsubsection{Recollection on derived algebraic geometry}
\label{sss:jet:dag}

In this part we recall the terminology on derived schemes.

We start with the terminology on affine derived schemes, some of which 
are already recalled in Notations \ref{sss:sp:daff} and \ref{sss:sp:dSt}.
\begin{enumerate}[nosep]
\item
We denote by $\idAff$ the $\infty$-category of affine derived schemes over $\bbk$.
It is the opposite of the $\infty$-category %of commutative simplicial algebras,
%the latter of which is equivalent to the $\infty$-category 
$\icdga^{\le 0}$ of cdgas over $\bbk$ concentrated in the non-positive
cohomological degrees .
For a cdga $R \in \icdga^{\le 0}$, 
we denote by $\Spec(R) \in \idAff$ the corresponding affine derived scheme.

\item 
For $R =(R,d_R) \in \icdga^{\le 0}$ and $n \in \bbN$, we denote by
\[
 \pi_n(R):=H^{-n}(R,d_R)
\] 
its $(-n)$-th cohomology (Definition \ref{dfn:dg:coh}).
Actually it coincides with the $n$-th homotopy group of the 
differential graded nerve of $(R,d_R)$. See \cite[\S 1.3.1]{Lu2} for the detail.
In particular, we have the functor
\[
 \pi_0: \icdga^{\le 0} \longto \ca
\]
of $\infty$-categories.
where $\ca$ denotes the ($\infty$-)category of commutative $\bbk$-algebras.
Recall here that a functor of $\infty$-category means a morphism 
of simplicial sets \cite[\S 1.2.7]{Lu1},
The functor $\pi_0$ is called the \emph{truncation} (\emph{functor}).

\item
We also have the inclusion functor $\iota: \ca \to \icdga^{\le 0}$
whose definition is an obvious one.
These two functors form an \emph{adjunction} 
\[
 \pi_0: \icdga^{\le 0} \rightleftarrows \ca :\iota
\]
of functors of $\infty$-categories.
See \cite[\S 5.2]{Lu1} for the detail of adjunctions 
of functors between $\infty$-categories .
%We denote by $f: \iC \rightleftarrows \iD: g$
%an adjunction $(f,g)$ of functors $f: \iC \to \iD$ and $g:\iD \to \iC$
%of $\infty$-categories.
%Given such an adjunction, 
In particular, for $A \in \icdga^{\le 0}$ and $B \in \ca$, 
we have an isomorphism
\[
 \Map_{\ca}(\pi_0(A),B) \longsimto \Map_{\icdga^{\le 0}}(A,\iota(B))
\]
in the homotopy category $\iH$ of spaces (\S \ref{ss:0:ntn}).
%This isomorphism is a higher category analogue 
Taking $\pi_0$ of the mapping spaces, 
we recover the ordinary adjunction property
$\Hom_{\ca}(\pi_0(A),B) \simto \Hom_{\icdga^{\le 0}}(A,\iota(B))$.

\item
We will use \emph{Zariski open immersions},
\emph{\'etale}, \emph{smooth}, \emph{flat} and 
\emph{locally finitely presented} morphisms in $\icdga^{\le 0}$ or $\idAff$.
See \cite[\S 1.2.6]{TVe} for the precise definitions.

\item
\label{i:jet:dag:fp}
We denote by $\icdga^{\le 0, \fp} \subset \icdga^{\le 0}$ the sub-$\infty$-category
spanned by those objects of finite presentation over $\bbk$
in the sense of \cite[Definition 1.2.3.1]{TVe}.
As noted in \cite[\S 2.2.1]{TVe}, 
the truncation functor $\pi_0$ gives 
\[
 \icdga^{\le 0,\fp} \longto \ca^{\ft},
\] 
where $\ca^{\ft} \subset \ca$ denotes the subcategory of finite type objects.
We also denote by $\idAff^{\fp}:= (\icdga^{\le 0, \fp})^{\op}$
and call its object an affine derived scheme of finite presentation over $\bbk$.

\item
The fiber product in $\idAff$ is denoted by $X \times^{\bbL}_Y Z$.
It corresponds to the derived tensor product 
$A \otimes^{\bbL}_B C$ in $\icdga^{\le 0}$,
and is represented by the two-sided bar complex (Definition \ref{dfn:pr:bcpx}).
\end{enumerate}

%We denote by $\sSet$ the category of simplicial sets.
Next we turn to the terminology on derived stacks and derived schemes.
Recall that we denote by $\iS$ the $\infty$-category of spaces 
(\S \ref{ss:0:ntn}). 
\begin{enumerate}[nosep]
\setcounter{enumi}{5}
\item
A $D^{-}$-stack over $\bbk$ in the sense of \cite[Chap.\ 2.2]{TVe}
is called a \emph{derived stack}.
Thus, it is a functor
\[
 F: \idAff^{\op} \longto \iS 
\]
of $\infty$-categories satisfying the sheaf condition for 
the \'etale $\infty$-topos $\et$ on $\idAff$.
We can construct an $\infty$-category of derived stacks by \cite[\S 1.3.2]{TVe},
and we denote it by $\idSt$.

\item
An affine derived scheme $\Spec(A)$ with $A \in \idAff$ defines a derived stack 
by the Yoneda embedding 
\[
 \uMap_{\idAff}(-,\Spec(A)): \idAff^{\op} \longto \iS.
\]
Here, for an $\infty$-category $\iC$, we denoted by $\uMap_{\iC}(-,-)$
the Kan simplicial set such that its homotopy type is $\Map_{\iC}(-,-)$.
The existence of such a simplicial set is shown in \cite[\S 1.2.2]{Lu1},
where it is called the space of right morphisms.
A derived stack which is equivalent to the one of the form 
$\uMap_{\idAff}(-,\Spec(A))$ is called a \emph{representable derived stack}
\cite[\S 1.3.2]{TVe}.
We identify an affine derived scheme and a representable derived stack.

\item
A geometric $D^-$-stack over $\bbk$ in the sense of \cite[\S 2.2.3]{TVe}
will be called a \emph{geometric derived stack}.
For $m \in \bbZ_{\ge -1}$, 
one defines an $n$-geometric derived stack inductively on $m$.
An ($-1$)-geometric derived stack is defined to be a representable derived stack,
and the inductive step defines an $m$-geometric derived stack to be 
a derived stack having an atlas of $(m-1)$-geometric derived stacks 
with respect to the smooth morphism in $\idAff$.

\item
\label{i:jet:dag:dsch}
For $n \in \bbN$, a derived stack $F$ is \emph{$n$-truncated} if $\pi_i(F(T),s)=0$
for any $i \in \bbZ_{>n}$, any $T \in \idAff$ and any $s \in \pi_0(F(T))$.
A \emph{derived scheme} (over $\bbk$) is defined to be 
an $n$-truncated $m$-geometric derived stack $X$ with some $m$ and $n$ 
such that there is an $n$-atlas $\sqcup_i U_i \to X$ 
of affine derived schemes $U_i$'s and each $U_i \to X$ 
is a monomorphism of stacks \cite[Definition 2.1.1.4]{TVe}.
By \cite[Remark 2.1.1.5 (1)]{TVe}, a derived scheme is automatically $1$-geometric.
Equivalently, a derived stack is a pair $(X,\shO_X)$
of a topological space $X$ and a sheaf $\shO_X$ valued in $\icdga^{\le 0}$
such that the truncation $(X,\pi_0(\shO_X))$ is an ordinary scheme and 
the sheaf $\pi_i(\shO_X)$ is a quasi-coherent sheaf of $\pi_0(\shO_X)$-modules.
We denote by $\idSch \subset \idSt$ the sub-$\infty$-category 
spanned by derived schemes.

\item
We denote by $\idSt^{\fp} \subset \idSt$ and $\idSch^{\fp} \subset \idSch$
the sub-$\infty$-categories spanned by those objects of categorically locally 
finite presentation over $\bbk$ in the sense of \cite[Definition 1.3.6.4]{TVe}.
\end{enumerate}

%%%%%%%%%%%%%%%%%%%%%%%%%%%%%%%%%%%%%%%%%%%%%%%%%%%%%%%%%%%%%%%%%%%%%%%%%%%%%%%%
\subsubsection{Jet and arc spaces for derived schemes}

We continue to use the notations given in the lase \S \ref{sss:jet:dag}.
In particular, 
we denote by $\idSch$ the $\infty$-category of derived schemes over $\bbk$.

We have the following derived analogue of the functor $J_n$ 
in Fact \ref{fct:jet:Jn}.

\begin{prp}\label{prp:jet:Jn}
For any $n \in \bbN$, there is a functor
$J_n: \idSch \to \idSch$ of $\infty$-categories such that 
we have an functorial isomorphism 
\begin{align*}
 \Map_{\idSch}(S,J_n(X)) \simeq 
 \Map_{\idSch}\bigl(S \times^{\bbL}_{\Spec(\bbk)}\Spec(\bbk[t]/(t^{n+1})),X \bigr)
\end{align*}
in the homotopy category $\iH$ of spaces.
We call $J_n(X)$ the \emph{$n$-th jet space of $X$}.
\end{prp}

Our proof basically follows the non-derived case of \cite[\S 7.6, Theorem 4]{BLR},
but we add some modification to work correctly in derived algebraic geometry.

For $X \in \idSch$, we define a functor 
$\ul{J_n}(X): \idSch^{\op} \to \iS$ of $\infty$-categories by
\[
 \ul{J_n}(X)(S) :=
 \uMap_{\idSch}\bigl(S \times^{\bbL}_{\Spec(\bbk)}\Spec(\bbk[t]/(t^{n+1})),X \bigr).
\]
We first show:

\begin{lem}\label{lem:jet:dst}
The functor $\ul{J_n}(X)$ is a derived stack.
\end{lem}

\begin{proof}
For any derived stacks $G, H \in \idSt$,
% and an 
we have the internal derived stack $\bbR_{\et}\ul{\Hom}_{\idSt}(G,H) \in \idSt$
satisfying the adjunction isomorphisms
$\Map_{\idSt}(F,\bbR_{\et}\ul{\Hom}(G,H)) \simeq 
 \Map_{\idSt}(F \times^{\bbL}_{\Spec(\bbk)} G, H)$ for any $F \in \idSt$.
%where $\times$ denotes the (homotopy) fiber product in $\idSt$.
See \cite[\S 1.4.1]{TVe} for the detail.
Taking $G := \Spec(\bbk[t]/(t^{n+1}))$ and $H=X$,
we have the consequence.
\end{proof}

Thus we want to show that this derived stack is (represented by) 
a derived scheme $J_n(X)$.
The case $n=0$ is trivial: $J_0(X) = X$.

Note that for a morphism $u: X \to Y$ in $\idSch$,
we have the induced morphism of functors:
\[
 \ul{J_n}(u): \ul{J_n}(X) \longto \ul{J_n}(Y).
\]
In particular, for $m,n \in \bbN$ with $m \ge n$,
the truncation morphism $\bbk[t]/(t^{m+1}) \surj \bbk[t]/(t^{n+1})$
induces a morphism $\ul{\pi_{m,n}}: \ul{J_m}(X) \to \ul{J_n}(X)$ of functors.
Then we can check the following statement by a set-theoretic argument
(see \cite[Lemma 2.3]{EM} for the detail).

\begin{lem}\label{lem:jet:opim}
Let $u: U \inj X$ be a monomorphism in $\idSch$ from
an affine derived scheme $U$ to a derived scheme $X$.
If a derived scheme $J_n(X)$ representing the functor $\ul{J_n}(X)$ exists,
then the representing derived scheme $J_n(U)$ exists and we have 
$J_n(U) \simeq \pi_{n,0}^{-1}(U)$, where $\pi_{n,0}: J_n(X) \to J_0(X) = X$
is the induced morphism from the truncation morphism 
$\bbk[t]/(t^{n+1}) \surj \bbk$.
\end{lem}

Next we show the existence of the representing object in the affine case.
Recall that we denote by $\idAff^{\fp}$ the $\infty$-category of 
affine derived schemes of finite presentations over $\bbk$
(\S \ref{sss:jet:dag} \eqref{i:jet:dag:fp}).

\begin{lem}\label{lem:jet:daff}
Let $n \in \bbN$.
For any $R \in \icdga^{\le 0}$, the functor 
\begin{align*}
\ul{J_n}(R): \icdga^{\le 0} \longto \iS, \quad 
\ul{J_n}(R)(A) :=
 \uMap_{\icdga^{\le 0}}\bigl(R,A \otimes^{\bbL}_\bbk \bbk[t]/(t^{n+1})\bigr)
\end{align*}
of $\infty$-categories
is represented by $J_n(R) \in \icdga^{\le 0}$.
Moreover the correspondence $R \mapsto J_n(R)$ determines a functor
$J_n: \icdga^{\le 0} \to \icdga^{\le 0}$ of $\infty$-categories.
\end{lem}

\begin{proof}
%For $X=\Spec(R)$ with $R \in \icdga^{\le 0}$, we have 
%\[
% \uMap_{\idAff}\bigl(T \times^{\bbL}_{\Spec(\bbk)}\Spec(\bbk[t]/(t^{n+1})),X \bigr)
% \simeq
% \Map_{\icdga}\bigl(R,A \otimes^{\bbL}_{\bbk}(\bbk[t]/(t^{n+1}))\bigr)
%\]
%represented by an object $J_n(R) \in \icdga^{\le 0}$.
%For a general cdga $R$ of finite presentation over $\bbk$,
We may replace an $R \in \icdga^{\le 0}$ by a free resolution, 
i.e., a cdga $\wt{R} \in \icdga^{\le 0}$ such that 
\begin{itemize}[nosep]
\item
the underlying graded algebra is a free algebra over $k$, and 
\item 
it is a K-flat complex quasi-isomorphic to $R$.
\end{itemize} 
We cite from \cite[0BZ6, 0BZ67 Lemma]{SP} 
an explicit construction of such a resolution:
\begin{enumerate}[nosep,label=(\roman*)]
\item 
Take a set of homogeneous elements $r_s \in \Ker(d_R: R \to R)$ ($s \in S_0$) 
such that the classes $\ol{r_s} \in H(R,d_R)$ generate the linear space $H(R,d_R)$.
We define a cdga 
\[
 R_0 = \bbk[x^{0,s} \mid s \in S_0]
\]
{}to be the free polynomial graded algebra generated by the letters 
$\{x^{0,s} \mid s \in S_0\}$ equipped with the grading 
$\abs{x^{0,s}}:=\abs{r_{0,s}}$ and trivial differential.
We have a morphism $f_0: R_0 \to R$ of cdgas given by $x^{0,s} \mapsto r_{0,s}$.
It is obvious that $R_{0}$ is a free commutative algebra over $\bbk$ 
and a K-flat complex, 
and that the induced morphism $H(f_0):H(R_0) \to H(R)$ is a surjection.

\item
Assume that we have constructed a sequence 
\[
 R_0 \to R_1 \to \cdots \to R_{m-1} \xr{i_{m-1}} R_m \xr{f_m} R
\]
of cdgas.
Take a set of homogeneous elements $r_{m,s} \in R_m$ ($s \in S_m$)
such that the classes $\ol{r_{m,s}} \in H(R_m,d_{R_m})$ 
span the linear space $\Ker(H(f_m): H(R_m,d_{R_m}) \to H(R,d_R))$.
We define a cdga 
\[
 R_{m+1} = R_m[x^{m,s} \mid s \in S_m]
\]
{}to be the free polynomial graded algebra generated by the letters 
$\{x^{m,s} \mid s \in S_m\}$ over $R_m$ equipped with 
the grading $\abs{x^{m,s}}:=\abs{r_{m,s}}-1$ and the differential
$d_{R_{m+1}}(x^{m,s}):=r_{m,s}$ and $\rst{d_{R_{m+1}}}{R_m}:=d_{R_m}$.
We have a natural embedding $i_m: R_m \inj R_{m+1}$ of cdgas,
and also have a morphism $f_{m+1}: R_{m+1} \to R$ of cdgas 
given by $f_{m+1}(x^{m,s}) := f_m(r_{m,s})$ and $\rst{f_{m+1}}{R_m}:=f_m$.
By induction, we obviously have that 
$R_{m+1}$ is a free commutative algebra over $\bbk$
and the induced morphism $H(f_{m+1}):H(R_{m+1}) \to H(R)$ is a surjection.
We can also check that $R_{m+1}$ is K-flat.

\item
Finally we define 
\[
 \wt{R} := \colim_m R_m,
\]
where the colimit is taken in the $\infty$-category $\icdga^{\le 0}$.
We also have a morphism $f: \wt{R} \to R$ of cdgas.
We can check that $\wt{R}$ is a free graded algebra over $\bbk$ and 
a K-flat complex, and that the induced morphism $H(f)$ is an isomorphism.
\end{enumerate}
Note that this construction respects quasi-isomorphisms:
If we are given a quasi-isomorphism $A \to B$ of cdgas,
then the construction induces a quasi-isomorphism $\wt{A} \to \wt{B}$.

Now we will define $J_n(R)$ by constructing $J_n(R_m)$'s inductively on $m$
and setting 
\[
 J_n(R) = J_n(\wt{R}) := \colim_m J_n(R_m).
\]
In the case $m=0$, we have $R_0=\bbk[x^{0,s} \mid s \in S_0]$. 
%where $\abs{x^{0,i}} \le 0$ and $I_0$ is a finite set.
%Let $f_{0,s} \in R$ ($s \in S_0$) be homogeneous elements such that  
Then, similarly as in Remark \ref{rmk:jet:aff} \eqref{i:jet:aff:A^N},
the desired cdga is given by the polynomial graded algebra 
\[
 J_n(R_0) = \bbk[x^{0,s}_{(-j-1)} \mid s \in S_0, j=0,\ldots,n]
\]
with $\bigl|x^{0,s}_{(-j-1)}\bigr|:=\abs{x^{0,s}}$ and trivial differential.
Recalling Fact \ref{fct:jet:da},
we equip with $J_n(R_0)$ with the $0$-derivation $T_{0}$ given by 
$T_0(x^{0,s}_{(-j)})=j x^{0,s}_{(-j-1)}$.
The pair $(J_n(R_0),T_0)$ represents the functor $\ul{J_0}(R_0)$.

Assume that we have constructed a pair $(J_n(R_d),T_d)$ 
representing the functor $\ul{J_d}(R_d)$ for $d \le m$.
Note that as a commutative graded algebra $R_{m+1}$ is isomorphic to 
$R_m[x^{m,s},y^{m,s} \mid s \in S_{m}]/(y^{m,s}-r_{m,s} \mid s \in S_m)$,
%$f_{m,s}=f_{m,s}(x^{m,t},y^{m,s}):= y^{m,s}-r_{m,s}
% \in R_m[x^{m,t},y^{m,s}\mid t \in S]$
with the grading $\abs{y^{m,s}}=\abs{r_{m,s}}$.
Then, similarly as in Remark \ref{rmk:jet:aff} \eqref{i:jet:aff:aff},
we define the cdga $J_n(R_{m})$ by 
\[
 J_n(R_{m+1}) := J_n(R_m)\bigl[ x^{m,s}_{(-j-1)}, y^{m,s}_{(-j-1)} 
   \mid s \in S_{m}, \, j=0,\ldots,n\bigr]/
   (T_{m+1}^k g_{m,s} \mid k=0,\ldots,n, \, s \in S_m),
\]
where 
$g_{m,s}:=y^{m,s}_{(-1)}-r_{m,s}(x^{m,t}_{(-1)})
 \in R_m[x^{m,t}_{(-1)},y^{m,s}_{(-1)}\mid t \in S]$
with $r_{m,s}=r_{m,s}(x^{m,t})$ regarded as a polynomial of $x^{m,t}$'s, 
and $T_{m+1}$ is a $0$-derivation defined by 
\[
 T_{m+1} x^{m,s}_{(-j)}  := j x^{m,s}_{(-j-1)}, \quad 
 T_{m+1} y^{m,s}_{(-j)}  := j y^{m,s}_{(-j-1)}, \quad 
 \rst{T_{m+1}}{J_n(R_m)} := T_m.
\]
The cohomological grading is given by 
$\sabs{x^{m,s}_{(-j)}}:=\abs{x^{m,s}}=\abs{r^{m,s}}-1$
and $\sabs{y^{m,s}_{(-j)}}:=\abs{y^{m,s}}=\abs{r^{m,s}}$,
and the differential $d_{m+1}$ is given by 
$d_{m+1} T_{m+1} = T_{m+1} d_{m+1}$, 
$d_{m+1}(x^{m,s}_{(-1)}):=y^{m,s}_{(-1)}$,
$d_{m+1}(y^{m,s}_{(-1)}):= r_{m,s}(x^{m,t}_{(-1)})$ .
Then the pair $(J_n(R_{m+1}),T_{m+1})$ represents the functor $\ul{J_n}(R_{m+1})$.

The injection $i_m: R_m \inj R_{m+1}$ induces a cofibration 
$J_n(i_m): J_n(R_m) \to J_n(R_{m+1})$ in $\icdga^{\le 0}$ which 
(homotopically) commutes with the $0$-derivations $T_m$ and $T_{m+1}$,
and we can take the colimit $(J_n(\wt{R}),T) := \colim_m (J_n(R_m),T_m)$.
The pair $(J_n(\wt{R}),T)$ represents the functor $\ul{J_n}(R)$ by construction.
\end{proof}

By the construction, we immediately have

\begin{cor*}
In Lemma \ref{lem:jet:daff}, if $R \in \icdga^{\le 0,\fp}$, 
then we have $J_n(R) \in \icdga^{\le 0,\fp}$.
We also have a functor 
$J_n: \icdga^{\le 0,\fp} \to \icdga^{\le 0,\fp}$ of $\infty$-categories.
\end{cor*}

Let us give the proof of Proposition \ref{prp:jet:Jn}.

\begin{proof}[{Proof of Propoisition \ref{prp:jet:Jn}}]
By the definition of a derived scheme $X$ 
(\S \ref{sss:jet:dag}, \eqref{i:jet:dag:dsch}),
we have an atlas $\sqcup_i U_i \to X$ of $X$ consisting of 
affine derived schemes $U_i$.
By Lemma \ref{lem:jet:daff}, we have 
the representing derived scheme $J_n(U_i)$.
These satisfy the gluing condition by Lemma \ref{lem:jet:opim}.
Thus we have a derived stack representing the functor $\ul{J_n}(X)$.
the consequence.
\end{proof}

\begin{rmk*}
Let us comment a more natural point of view 
in the context of derived algebraic geometry.
In Lemma \ref{lem:jet:dst} we recalled the internal derived stack 
$\bbR_{\et}\ul{\Hom}_{\idSt}(G,H) \in \idSt$ for derived stacks $G$ and $H$.
If $G$ is a scheme and $H$ is an $m$-geometric derived stack for some 
$m \in \bbZ_{\ge -1}$,
then it is denoted by $\Map(G,H)$ and called the \emph{mapping derived stack}.
By \cite[Theorem 2.2.6.11]{TVe},
the mapping derived stack $\Map(G,H)$ is also $m$-geometric
under some conditions,
which can be checked for our case 
$G=D_n:=\Spec(k[t]/(t^{n+1}))$ and $H=X \in \idSch$ with $m=1$.
Thus the $n$-th jet scheme is nothing but 
the mapping derived stack from the $n$-fattened infinitesimal disk $D_n$:
\[
 J_n(X) \simeq \Map(D_n,X).
\]
%Then by \cite[Lemma 2.1.1.2]{TVe},
%$\Map(D_n,X)$ is $2$-truncated.
%Thus in order to show that $\Map(D_n,X)$ is a derived scheme,
%it is enough to check that there is an atlas 
%$\sqcup_i U_i \to \Map(D_n,X)$ 
\end{rmk*}

By the construction, 
%the cdga $J_n(R)$ for $R \in \icdga^{\le 0}$ enjoys 
%similar properties of the $n$-th jet schemes for non-derived schemes.
the $n$-th jet spaces of derived schemes enjoy similar properties as
those of non-derived schemes in Fact \ref{fct:jet:EM2.2}.

\begin{lem*}
Let $X$ be a derived scheme over $\bbk$ and $n \in \bbN$.
\begin{enumerate}[nosep]
\item
We have $J_0(X) \simeq X$.

\item
For $m \in \bbN$ with $m \ge n$,
the truncation morphism $\bbk[t]/(t^{m+1}) \surj \bbk[t]/(t^{n+1})$
induces a morphism $\pi_{m,n}: J_m(X) \to J_n(X)$ in $\idSch$.
The morphisms 
$\{\pi_{m,n}: J_m(X) \to J_n(X) \mid m,n \in \bbN, m \ge n\}$
form an inverse system over the direct set $(\bbN,\ge)$.

\item
For $R \in \icdga^{\le 0, \fp}$, we have $J_n(R) \in \icdga^{\le 0, \fp}$.

\item
The (homotopy) fiber of the morphism $\pi_{n,0}: J_n(X) \to J_0(X)=X$ is 
equivalent to an affine derived scheme.
\end{enumerate}
\end{lem*}

The last property guarantee the existence of the arc space of derived schemes.

\begin{prp*}
For $X \in \idSch$, the limit of the inverse system 
$\{\pi_{m,n}: J_m(X) \to J_n(X) \mid m,n \in \bbN, m \ge n\}$ exists in $\idSch$.
We denote it by $J_\infty(X)$, and call it the \emph{arc space} or 
the \emph{$\infty$-jet space} of $X$.
\end{prp*}

In the sections below, we only deal with cdgas $R$ 
(but not necessarily concentrated in non-positive degrees).
For later reference, we set:

\begin{dfn}\label{dfn:jet:dgarc}
For a cdga $R$, we call the cdga $J_\infty(R)$ with $0$-derivation 
the \emph{arc space of $R$}.
\end{dfn}

%%%%%%%%%%%%%%%%%%%%%%%%%%%%%%%%%%%%%%%%%%%%%%%%%%%%%%%%%%%%%%%%%%%%%%%%%%%%%%%%
%%%%%%%%%%%%%%%%%%%%%%%%%%%%%%%%%%%%%%%%%%%%%%%%%%%%%%%%%%%%%%%%%%%%%%%%%%%%%%%%
%%%%%%%%%%%%%%%%%%%%%%%%%%%%%%%%%%%%%%%%%%%%%%%%%%%%%%%%%%%%%%%%%%%%%%%%%%%%%%%%
\section{Li filtration and vertex Poisson algebras in derived setting}
\label{s:li}

Recall that given an affine Poisson scheme,
the coordinate ring of its jet scheme has a structure of 
vertex Poisson algebra \cite{A15}.
In this section we give a dg analogue of this statement.
We will work over a field $\bbk$ of characteristics $0$.

%%%%%%%%%%%%%%%%%%%%%%%%%%%%%%%%%%%%%%%%%%%%%%%%%%%%%%%%%%%%%%%%%%%%%%%%%%%%%%%%
%%%%%%%%%%%%%%%%%%%%%%%%%%%%%%%%%%%%%%%%%%%%%%%%%%%%%%%%%%%%%%%%%%%%%%%%%%%%%%%%
%%%%%%%%%%%%%%%%%%%%%%%%%%%%%%%%%%%%%%%%%%%%%%%%%%%%%%%%%%%%%%%%%%%%%%%%%%%%%%%%
%%%%%%%%%%%%%%%%%%%%%%%%%%%%%%%%%%%%%%%%%%%%%%%%%%%%%%%%%%%%%%%%%%%%%%%%%%%%%%%%
\subsection{Dg vertex algebras}
\label{ss:li:vsa}

In this subsection we recall basics of vertex algebras
and define dg vertex algebras.
In order to make the text consistent with the literature of vertex algebras,
we begin with the preliminary on super objects.

%%%%%%%%%%%%%%%%%%%%%%%%%%%%%%%%%%%%%%%%%%%%%%%%%%%%%%%%%%%%%%%%%%%%%%%%%%%%%%%%
\subsubsection{Super convention}

In this part we collect the convention on super objects.
We express $\bbZ/2\bbZ=\{\ol{0},\ol{1}\}$.

\begin{dfn}\label{dfn:dgs:ls}
%Let $\bbk$ be a field, and consider the $\bbk$-linear monoidal category $\iVect$.
\begin{enumerate}[nosep]
\item
A \emph{linear superspace} is a $\bbZ/2\bbZ$-graded linear space.

\item
The $\bbZ/2\bbZ$-grading of a linear superspace $V$ is called the \emph{parity},
and the grade decomposition of $V$ is denoted by $V=V_{\ol{0}} \bigoplus V_{\ol{1}}$.
A homogeneous element $v$ of $V$ is also called an \emph{element of pure parity}, 
and we denote by $p(v)$ its parity.
An element of $V_{\ol{0}}$ is called \emph{even},
and an element of $V_{\ol{1}}$ is called \emph{odd}.

\item
\label{i:dgs:ls3}
Let $V$ and $W$ be linear superspaces.
A linear map $f: V \to W$ is called \emph{even} if 
$f(V_{\ol{i}})\subset W_{\ol{i}}$ for $i=0,1$,
and called \emph{odd} if 
$f(V_{\ol{i}})\subset W_{\ol{i+1}}$ for $i=0,1$.
We denote by 
\[
 \Hom_{\bbk}(V,W) = \Hom_{\bbk}(V,W)_{\ol{0}} \oplus \Hom_{\bbk}(V,W)_{\ol{1}}
\]
the corresponding linear superspace structure on $\Hom_{\bbk}(V,W)$.

\item
The category of linear superspaces and even linear maps will be denoted by $\sVec$.
%The subcategory of $\sVec$ with only even linear maps will be 
%denoted by $\ul{\sVec}$.
Thus we have 
$\Hom_{\sVec}(V,W) = \Hom_{\bbk}(V,W)_{\ol{0}}$.
\end{enumerate}
\end{dfn}

\begin{dfn}\label{dfn:dgs:lst}
Let $V$ and $W$ be linear superspaces.
\begin{enumerate}[nosep]
\item
\label{i:dgs:lst:1}
The \emph{tensor product} $V \otimes W \in \sVec$ of $V$ and $W$ is defined to be 
the linear superspace whose underlying linear space is 
$V \otimes_{\bbk} W \in \cVec$
and the parity decomposition is given by 
\[
 (V \otimes W)_{\ol{0}} := 
 V_{\ol{0}} \otimes_{\bbk} W_{\ol{0}} \oplus V_{\ol{1}} \otimes_{\bbk} W_{\ol{1}}, 
 \quad
 (V \otimes W)_{\ol{1}} := 
 V_{\ol{0}} \otimes_{\bbk} W_{\ol{1}} \oplus V_{\ol{1}} \otimes_{\bbk} W_{\ol{0}}.
\]
\item
\label{i:dgs:lst:2}
The braiding (or the commutativity) isomorphism on the tensor product $V \otimes W$ 
is an isomorphism in $\sVec$ given by 
\[
 V \otimes W \longsimto W \otimes V, \quad 
 v \otimes w \longmapsto (-1)^{p(v)p(w)} w \otimes v
\]
for elements $v \in V$, $w \in W$ of pure parity.
\end{enumerate}
These define a symmetric monoidal structure on $\sVec$,
which is denoted by $\sVec^{\otimes}$.
\end{dfn}

The tensor product and the braiding isomorphism gives $\sVec$ 
a structure of symmetric monoidal category.
Thus we have notions of \emph{associative algebras}, 
\emph{commutative algebras} and \emph{Lie superalgebras}.
In particular, for a linear superspace $V=V_{\ol{0}}\oplus V_{\ol{1}}$, we have 
the endomorphism superalgebra $\End_{\bbk}(V)$ and the commutator 
\[
  [f,g] := f g - (-1)^{p(f) p(g)} g f
\]
for $f,g \in \End_{\bbk}(V)$ of pure parity $p(f)$ and $p(g)$ respectively.
We will repeatedly use this commutator.

Recall the notion of a differential algebra (Definition \ref{dfn:jet:da}).
For later use, let us give a super analogue.

\begin{dfn}\label{dfn:li:der}
Let $R$ be a superalgebra.
\begin{enumerate}[nosep]
\item 
For $\ve \in \bbZ/2\bbZ$, a \emph{derivation $d$ of parity $\ve$} on $R$
is a linear endomorphism on $R$ of parity $\ve$ such that for any $r \in R$
we have $[d,r]= d r$  in the endomorphism superalgebra $\End(R)$, where
in the left hand side we regard $r \in \End(R)$ as a multiplication operator.

\item
We denote by $\Der(R)_{\ve}$ 
the linear space of derivations of parity $\ve$ on $R$.

\item
A superalgebra equipped with a derivation is called 
a \emph{differential superalgebra}.
\end{enumerate}
\end{dfn}

Finally we remark a construction of super objects from graded objects.
See \cite[1.1.16]{BD} for a more systematic ``dg super" conventions.

\begin{ntn}\label{ntn:li:s-g}
\begin{enumerate}[nosep]
\item 
For a complex $V=(V^{\bl},d)$, we have a linear superspace 
whose even part is $V^{\even}:=\bigoplus_{n \in \bbZ}V^{2n}$
and whose odd part is $V^{\odd}:=\bigoplus_{n \in \bbZ}V^{2n+1}$.
We denote it by $V^{\even} \oplus V^{\odd}$.

\item
The correspondence $V \mapsto V^{\even} \oplus V^{\odd}$ defines 
a monoidal functor $\dgVec^{\otimes} \to \sVec^{\otimes}$,
and it induces similar functors for algebraic structures.
We call those objects lying in the essential image 
the \emph{associated super objects} of given dg objects.
\end{enumerate}
\end{ntn}

For example, for a graded Lie algebra $L=L^{\bl}$,
we have the associated Lie superalgebra $L^{\even}\oplus L^{\odd}$.

%%%%%%%%%%%%%%%%%%%%%%%%%%%%%%%%%%%%%%%%%%%%%%%%%%%%%%%%%%%%%%%%%%%%%%%%%%%%%%%%
\subsubsection{Vertex superalgebras}
\label{sss:li:vsa}

For the definiteness, we begin with the recollection of vertex superalgebras.
See \cite{Ka} and \cite{FBZ} for the detail.
%Recall also the super language given in \S \ref{ss:dg}.

\begin{ntn*}
\begin{enumerate}[nosep]
\item 
We denote by $\bbk[[z]]$ the linear space of formal series,
and by $\bbk((z))$ the linear space of formal Laurent series.
Thus we have a decomposition 
$\bbk((z))=\bbk[[z]] \oplus z^{-1}\bbk[z^{-1}]$.

\item
For a linear space $V$, we denote 
$V[[z]]:= V \otimes_{\bbk} \bbk[[z]]$.
$V((z)):= V \otimes_{\bbk} \bbk((z))$,
$z V[z] := V \otimes_{\bbk} z \bbk[z]$ and so on.
\end{enumerate}
\end{ntn*}

\begin{dfn}\label{dfn:li:vsa}
A \emph{vertex superalgebra} (\emph{vsa} for short) is a data 
$(V,\vac,T,Y)$ consisting of 
\begin{itemize}[nosep]
\item
a linear superspace $V$, 
\item
an even element $\vac \in V_{\ol{0}}$ called the \emph{vacuum}, 
\item
an even endomorphism $T \in \End(V)_{\ol{0}}$ 
called the \emph{translation}, and 
\item
a linear map 
\[
 V \otimes V \longto V((z)), \quad 
 a \otimes b \longmapsto Y(a,z)b = \sum_n (a_{(n)}b) z^{-n-1}
\]
such that we have $a_{(n)} \in \End(V)_{\ve}$ for $a \in V_{\ve}$ of pure parity 
and for any $n \in \bbZ$.
The operation $Y$ is called the \emph{state-field correspondence}.
\end{itemize}
These should satisfy the following standard axioms.
\begin{enumerate}[nosep, label=(\roman*)]
\item (Vacuum axiom)
$Y(\vac,z)=\id_V$ and 
$Y(a,z)\vac \in a+ z V[[z]]$ for any $a \in V$.
\item (Translation axiom)
$T \vac =0$ and 
$[T,Y(a,z)]=\partial_z Y(a,z)$ for any $a \in V$.
%, i.e., $[T,a_{(n)}]=-n a_{(n-1)}$ for any $a \in V$.
\item (Locality axiom)
$Y(a,z)$'s are \emph{mutually local}, i.e., for any homogeneous $a,b \in V$,
there exists $N \in \bbN$ such that 
the following equation holds in $(\End(V))[[z^{\pm1},w^{\pm1}]]$.
\[
 (z-w)^N Y(a,z)Y(b,w)=(-1)^{p(a) p(b)}(z-w)^N Y(b,w)Y(a,z).
\]
\end{enumerate}
We often denote it simply by $V$.

An even vertex superalgebra, i.e., $V=V_{\ol{0}}$, 
is called a \emph{vertex algebra} (\emph{va} for short).
\end{dfn}

Several remarks are in order.

\begin{rmk}\label{rmk:li:vsa}
\begin{enumerate}[nosep]
\item
\label{i:rmk:li:vsa:1}
The locality axiom implies the following relations.
\begin{align*}
&[a_{(m)},b_{(n)}] = \sum_{l \in \bbN}\binom{m}{l}(a_{(l)}b)_{(m+n-l)},
\\
&(a_{(m)}b)_{(n)}c = \sum_{l \in \bbN}(-1)^l \binom{m}{l} \bigl(
  a_{(m-l)}(b_{(n+l)}c) - (-1)^{m+ p(a)p(b)} b_{(m+n-l)}(a_{(l)}c) \bigr),
\end{align*}
where $a,b,c \in V$ and $a,b$ are of pure parity.
We also have the skew-symmetry
\[
 Y(a,z)b = (-1)^{p(a)p(b)} e^{z T}Y(b,-z)a
\]
for homogeneous $a,b \in V$.
See \cite[Chap.\ 3]{FBZ} and \cite[\S\S 4.2, 4.6, 4.8]{Ka} for the detail.

\item
\label{i:rmk:li:vsa:2}
The translation $T$ is completely determined by the operation $Y$
as $T a = a_{(-2)}\vac$.

\item
\label{i:rmk:li:vsa:3}
The correspondence $a \mapsto a_{(-1)}$ is injective \cite[1.3.2, Remarks 4]{FBZ}.
\end{enumerate}
\end{rmk}

Let us also recall:

\begin{dfn*}%\label{dfn:vsa:mor}
Let $V$ and $W$ be vertex superalgebras.
\begin{enumerate}[nosep]
\item 
A \emph{morphism $V \to W$ of vsas} is an even linear map $\varphi: V \to W$
such that $\varphi(a_{(n)}b)=\varphi(a)_{(n)}\varphi(b)$ 
for any $a,b \in V$ and $n \in \bbZ$.

\item
The \emph{tensor product $V \otimes W$ of vsas} is given by
$(V \otimes W, \vac_V \otimes \vac_W, T_V \otimes \id_W + \id_V \otimes T_W,
  Y_{V \otimes W})$,
where the first item denotes the tensor product as linear superspaces
(Definition \ref{dfn:dgs:lst} \eqref{i:dgs:lst:1})
and $Y_{V \otimes W}(a \otimes b,z):=Y_V(a,z) \otimes Y_W(b,z)$.
Thus we have $(a \otimes b)_{(n)}=\sum_{m \in \bbZ} a_{(m)} \otimes b_{(n-m-1)}$
using the tensor product 
$(f \otimes g)(c \otimes d)=(-1)^{p(g)p(c)}f(c) \otimes g(d)$
of linear operators $f$ and $g$ coming from the braiding isomorphism 
(Definition \ref{dfn:dgs:lst} \eqref{i:dgs:lst:2}).
\end{enumerate}
\end{dfn*}

Note that the tensor product $\otimes$ gives a unital symmetric monoidal structure
on the category of vsas, where the unit is the trivial vertex algebra $\bbk$.

We also have the standard notion of \emph{vertex super subalgebras} and 
those of \emph{ideals} and \emph{quotients} of vsas.
See \cite[\S 4.3]{Ka} for the detail.

Finally we give some terminology for modules over vsas.
There are several different definitions,
and we cite a version from \cite[\S 2.1]{A12}.

\begin{dfn*}%\label{dfn:vsa:Mod}
Let $V$ be a vertex superalgebra.
A \emph{$V$-module} is a linear superspace $M$ equipped with a linear map
\[
 V \otimes M \longto M((z)), \quad 
 a \otimes m \longmapsto Y^M(a,z)m = \sum_{k \in \bbZ}a^M_{(k)} m z^{-k-1}
\] 
where for $a \in V_{\ve}$ we have $a^M_{(k)} \in \End(M)_{\ve}$,
which should satisfy the following conditions.
\begin{enumerate}[nosep,label=(\roman*)]
\item 
$Y^M(\vac,z)=\id_M$,
\item
For elements $a,b \in V$ of pure parity and any $j,k,l \in \bbZ$, we have 
\[
 \sum_{n \in \bbN}\binom{k}{n}(a_{(l+n)}b)^M_{k+j-n}=
 \sum_{n \in \bbN}(-1)^n \binom{l}{n} \bigl( a^M_{(k+l-n)}b^M_{(j+n)}
 - (-1)^{l+p(a)p(b)} b^M_{(l+j-n)}a^M_{(k+n)} \bigr).
\]
\end{enumerate}
A \emph{morphism} of $V$-modules is naturally defined.
We denote by $\VMod{V}$ the category of $V$-modules.
\end{dfn*}

\begin{eg}\label{eg:li:h-m}
Let $\varphi:V \to W$ be a morphism of vsas.
Then $W$ is naturally a $V$-module.
In fact, denoting the state-field correspondence of $W$ as a vsa by 
$\alpha \otimes \beta \mapsto Y_W(\alpha,z)\beta$, 
we have the $V$-module structure $V \otimes W \to W((z))$, 
$Y^W(a,z)\beta := Y_W(\varphi(a),z)\beta$
for $a \otimes \beta \in V \otimes W$.
\end{eg}

Let us recall an equivalent description of modules over vertex superalgebras.:

\begin{fct}[{\cite[5.1.6.\ Theorem]{FBZ}}]\label{fct:co:wtUV}
Let $\wt{U}(V)$ be the associative algebra 
attached to a vertex algebra $V$ \cite[4.3.1.\ Definition]{FBZ}.
We have an equivalence of the category of $V$-modules
and the category of \emph{smooth} modules over $\wt{U}(V)$.
\end{fct}

See \cite[\S\S 4.1--4.3, 5.1]{FBZ} for the detail.
We will give a vertex Poisson analogue of this fact in \S \ref{sss:li:vpm}.

%%%%%%%%%%%%%%%%%%%%%%%%%%%%%%%%%%%%%%%%%%%%%%%%%%%%%%%%%%%%%%%%%%%%%%%%%%%%%%%%
\subsubsection{Graded vertex superalgebras}

Let us cite from \cite[1.3.1]{FBZ} 
the notion of \emph{graded vertex superalgebras}.

%\begin{rmk}%[{\cite[\S 4.9]{Ka}}]
%\label{rmk:vsa:Ham}
%\begin{enumerate}[nosep]
%\item 
%The condition \ref{i:vsa:Ham:2} of Hamiltonian $H$ is equivalent to 
%$[H,a_n]=-n a_n$ for $a \in V_{-\Delta}$, 
%where we expand $Y(a,z)=\sum_{n \in -\Delta+\bbZ}a_n z^{-n-\Delta}$.
%\item
%We automatically have $\Delta_{\vac}=0$ and 
%$\Delta_{a_{(n)}b}=\Delta_a+\Delta_b-n-1$.
%\item
%\label{i:rmk:vsa:Ham:inv}
%\end{enumerate}
%\end{rmk}

\begin{dfn}\label{dfn:li:gvsa}
%Let $I \subset \bbk$ be an additive subgroup containing $\bbZ$.
A vertex superalgebra $V$ is \emph{graded} if there is a decomposition
$V= \bigoplus_{\Delta \in \bbk}V_\Delta$ as a linear superspace such that 
the following conditions hold.
\begin{enumerate}[nosep,label=(\roman*)]
\item
$\vac \in V_{0}$.

\item
$T(V_{\Delta}) \subset V_{\Delta+1}$ for any $\Delta \in \bbk$.

\item 
$a_{(n)}b \in V_{-(\Delta+\Delta'-n-1)}$
holds for any $a \in V_{-\Delta}$, $b \in V_{-\Delta'}$ and $n \in \bbZ$.
\end{enumerate}
%For a homogeneous element $a \in V_{-\Delta}$, we denote $\Delta_a := \Delta$.
Setting $I:=\{\Delta \in \bbk \mid V_{-\Delta} \neq 0\}$,
we also call $V$ an \emph{$I$-graded vertex superalgebra}.
\end{dfn}

\begin{rmk}\label{rmk:li:Ham}
There is an equivalent notion of 
a \emph{vertex superalgebra $V$ with Hamiltonian $H$} \cite[\S 4.9]{Ka}.
Such $V$ is equipped with an even linear endomorphism $H$
such that the following conditions hold.
\begin{enumerate}[nosep]
\item 
$H$ acts semisimply on $V$.
\item
For any $a,b \in V$ the equation
$[H,Y(a,z)] b = \bigl(Y(H a,z) + z \partial_z Y(a,z)\bigr)b$ in $V((z))$ holds.
\end{enumerate}

Given such $(V,H)$, we have the $H$-eigen decomposition 
$V=\tboplus_{\Delta \in \bbk} V_\Delta$, 
$V_{\Delta} := \{a \in V \mid H a = -\Delta a\}$.
Setting $I:= \{\Delta \in \bbk \mid V_{\Delta} \neq 0\}$,
we recover Definition \ref{dfn:li:gvsa}.

Conversely, given an $I$-graded vertex superalgebra $V$,
defining $H a := \Delta a$ for $a \in V_{-\Delta}$
we have a Hamiltonian $H$ on $V$.
\end{rmk}

We also have the notion of \emph{graded modules} of graded vsas.

\begin{dfn}\label{dfn:li:gvmod}
%Let $V$ be a vsa with a Hamiltonian $H$.
%Then a $V$-module $M$ is \emph{graded} if there is a semisimple action of 
%$H$ on $M$ such that for $a \in V_{-\Delta}$ we have $[H,a^M_n]=-n a^M_n$ 
%where $a^M_n$ is defined by the expansion
%$Y^M(a,z)=\sum_{n \in -\Delta+\bbZ}a^M_n z^{-n-\Delta}$.
Let $V$ be a graded vertex superalgebra.
A $V$-module $M$ is \emph{graded} if there is a linear decomposition
$M = \oplus_{\Delta \in \bbk}M_{\Delta}$ such that 
$a^M_{(n)}b \in M_{-(\Delta+\Delta'-n-1)}$
holds for any $a \in V_{-\Delta}$, $b \in V_{-\Delta'}$ and $n \in \bbZ$. 
\end{dfn}

%%%%%%%%%%%%%%%%%%%%%%%%%%%%%%%%%%%%%%%%%%%%%%%%%%%%%%%%%%%%%%%%%%%%%%%%%%%%%%%%
\subsubsection{Universal affine vertex algebras}
\label{sss:li:uava}

In this part we recall the universal affine vertex algebra.
For simplicity, we work over $\bbC$.
See \cite[\S 2.4]{FBZ} for the detail.
%We work over $\bbC$ for simplicity.
We first fix the notation on Lie algebras.

\begin{ntn}\label{ntn:li:g}
\begin{enumerate}[nosep]
\item 
Let $G$ be a simply connected semisimple algebraic group over $\bbC$,
and $\frkg := \Lie G$ be its Lie algebra.

\item
We denote by $\kappa_{\frkg}$ the Killing form 
and by $h^\vee$ the dual Coxeter number of $\frkg$.

\item
We denote by $\frkg((t)):=\frkg \otimes_{\bbC} \bbC((t))$ 
the tensor product with the linear space $\bbC((t))$ of formal Laurent series,
and denote its element by $x f := x \otimes f \in \frkg((t))$ 
with $x \in \frkg$ and $f=f(t) \in \bbC((t))$.
Similarly we will denote $\frkg[t] := \frkg \otimes_{\bbC} \bbC[t]$, 
$\frkg[[t]] := \frkg \otimes_{\bbC} \bbC[[t]]$ and so on.

\item
We denote by $\wt{\frkg}$ the non-twisted affine Kac-Moody Lie algebra 
associated with $\frkg$.
It is a Lie algebra of which 
\begin{itemize}
\item 
the underlying linear space is $\frkg((t)) \oplus \bbC K \oplus \bbC D$, and
\item
the commutation relations are 
\[
 [x t^m , y t^n] = [x,y] t^{m+n} + 
 m \delta_{m+n,0} \tfrac{1}{2 h^\vee}\kappa_{\frkg}(x,y) K,
 \quad
 [D,x t^m]=m x t^m, \quad [K,\wh{\frkg}]=0
\]
for $x,y \in \frkg$. 
%and $f,g \in \bbC((t))$. 
%Here we denoted $g'(t)=\partial_t g(t)$ the differential of $g$.
\end{itemize}
%The standard derivation $D$ on $\wh{\frkg}$ is 
%given by $D(x f + \alpha \one)=x(t \partial_t f)$ for 
%$x \in \frkg$, $f \in \bbC((t))$ and $\alpha \in \bbC$.

\item
We denote by $\wh{\frkg}$ the derived algebra of $\wt{\frkg}$.
Thus we have $\wh{\frkg}=\frkg((t)) \oplus \bbC K$ as a linear space.
\end{enumerate}
\end{ntn}

Here is the definition of the universal affine vertex algebra.
%We follow \cite[\S 2]{A18} for the notation.

\begin{dfn}\label{dfn:li:uava}
%Let $\kappa$ be an invariant symmetric bilinear form on $\frkg$.
For $k \in \bbC$, we define a $\wt{\frkg}$-module $V_{k}(\frkg)$ by
\[
 V_{k}(\frkg) := 
 U(\wt{\frkg}) \otimes_{U(\frkg[t] \oplus \bbC K \oplus \bbC D)}\bbC_k,
\]
where $\bbC_k$ is the one-dimensional representation of the Lie subalgebra
$\frkg[t] \oplus \bbC K \oplus \bbC D \subset \wt{\frkg}$
on which $\frkg[t] \oplus \bbC D$ acts trivially 
and $K$ acts as multiplication by $k$, 
Then $V_{k}(\frkg)$ is a vertex algebra of which 
\begin{itemize}[nosep]
\item
the vacuum vector is $\vac = 1 \otimes 1$, 
\item
the translation is $T = -\partial_t$, and 
\item
$Y(x t^{-1}\vac,z)=x(z):=\sum_{n \in \bbZ} (x t^n) z^{-n-1}$
for $x \in \frkg$.
\end{itemize}
These conditions determines the vertex algebra structure uniquely.
It is $\bbN$-graded by the Hamiltonian $-D$ (Remark \ref{rmk:li:Ham}).
We call the obtained vertex algebra $V_{k}(\frkg)$
the \emph{universal affine vertex algebra associated with $\frkg$ at level $k$}.
\end{dfn}

\begin{rmk*}
We used the normalization of invariant bilinear form of \cite{FBZ}.
Comparing it with the notation $V^{\kappa}(\frkg)$ in \cite{A18}
where $\kappa$ denotes a bilinear invariant form on $\frkg$, 
we have $V_{1}(\frkg)=V^{\kappa_{\frkg}}(\frkg)$
and $V_{-h^\vee}(\frkg)=V^{\kappa_{c}}(\frkg)$
with $\kappa_{c}:=-\frac{1}{2}\kappa_{\frkg}$.
%Thus the critical level $k=-h^\vee$ in the sense of \cite{FBZ} 
By the latter relation, 
$\kappa_c$ is called the \emph{critical level} in \cite{A18}.
\end{rmk*}

For later use, we recall:

\begin{dfn}\label{dfn:li:nord}
Let $R$ be a associative superalgebra.
Let also $a(z)=\sum_{n \in \bbZ} a_n z^n$ and 
$b(z)=\sum_{n \in \bbZ} b_n z^n$ be formal power series
with coefficients $a_n$ and $b_n$ in $R$
such that the parity of the coefficients are constant for each series.
We denote $p(a):=p(a_n)$ and $p(b):=p(b_n)$.
Then the \emph{normal ordering} $:a(z)b(z):$ of $a(z)$ and $b(z)$ 
is defined to be the formal power series
\[
 :a(z)b(z): \, :=  \, a(z)_{+} b(z) + (-1)^{p(a)p(b)} b(z)a(z)_{-},
\]
where we set $a(z)_{+}:=\sum_{n \ge 0}a_{(n)}z^n$ and 
$a(z)_{-}:=\sum_{n \le -1}a_{(n)}z^n$.
\end{dfn}

\begin{fct}[{\cite[\S 2.4]{FBZ}}]\label{fct:uava:pbw}
Let $\{x_i \mid i=1,\ldots,\dim \frkg\}$ be a linear basis of $\frkg$.
Then 
\begin{enumerate}[nosep]
\item 
$V_{k}(\frkg)$ has a linear basis of monomials of the form
\[
 v = (x_{i_1}t^{n_1})\cdots (x_{i_l}t^{n_l})\vac,
\]
where $n_i \in \bbZ$, 
$n_1 \le \cdots \le n_l <0$ and if $n_j=n_{j+1}$ then $i_j \le i_{j+1}$.

\item
For the element $v \in V_k(\frkg)$ in the previous item, we have 
\[
 Y(v,z) = \tfrac{1}{(-n_1+1)!} \cdots \tfrac{1}{(-n_l+1)!}
 :\partial_z^{-n_1+1}x_{i_1}(z) \cdots \partial_z^{-n_l+1}x_{i_l}(z):.
\]
\end{enumerate}
\end{fct}

Next we consider the category $\VMod{V_{k}(\frkg)}$ of modules 
over the universal affine vertex algebra $V_{k}(\frkg)$.
By Fact \ref{fct:co:wtUV}, it is equivalent to the smooth modules over
the associative algebra $\wt{U}(V_{k}(\frkg))$ attached to $V_{k}(\frkg)$.
Denoting by $\wt{U}_k(\wh{\frkg})$ the $t$-adic completion of 
$U_k(\wh{\frkg}):= U(\wh{\frkg})/(K-k)$, we have 
$\wt{U}(V_{k}(\frkg)) \simeq \wt{U}_k(\wh{\frkg})$ by \cite[4.3.2.\ Lemma]{FBZ}.
Then a smooth $\wt{U}(V_k(\frkg))$-module is nothing but a smooth $\frkg_k$-module.
In other words, we have:

\begin{fct}[{\cite[5.1.8]{FBZ}}]\label{fct:li:sm}
A $V_{k}(\frkg)$-module is equivalent to a \emph{smooth} $\wh{\frkg}$-module 
of level $k$, i.e., a representation $M$ of the Lie algebra $\wh{\frkg}$ 
such that $K$ acts by multiplication $k$ and 
$(x t^n).m=0$ for any $x \in \frkg$, $m \in M$ and $n \gg 0$.
\end{fct}

%Let $k$ and $k'$ be invariant symmetric bilinear forms on $\frkg$.
%For $M \in \VMod{V_{k}(\frkg)}$ and $N \in \VMod{V_{k'}(\frkg)}$,
%the tensor product $M \otimes N$ is an object of $\VMod{V_{k+\kappa'}(\frkg)}$
%By Remark \ref{rmk:KL:sm}, we can also check:

Recall that there is a monoidal structure $\otimes$ on the category of 
representations of a Lie algebra $L$,
which is induced by the comultiplication $\Delta(x)=x \otimes 1 + 1 \otimes x$
for each element $x \in L$.
We call this action on the tensor the \emph{diagonal action} of $L$.
Together with Fact \ref{fct:li:sm}, we have the following tensor structure on
$V_k(\frkg)$-modules with different $k$'s:

\begin{lem}\label{lem:li:tens}
Let $k,k' \in \bbC$, $M \in \VMod{V_{k}(\frkg)}$ and $N \in \VMod{V_{k'}(\frkg)}$.
Then the tensor product $M \otimes N$ in $\cVec$ is an object 
of $\VMod{V_{k+k'}(\frkg)}$ by the diagonal action of $\wh{\frkg}$. 
\end{lem}

%%%%%%%%%%%%%%%%%%%%%%%%%%%%%%%%%%%%%%%%%%%%%%%%%%%%%%%%%%%%%%%%%%%%%%%%%%%%%%%%
\subsubsection{Dg vertex algebras}

Now we introduce a dg analogue of a vertex algebra.
There are in fact several choices of definition, 
and we give one suitable for our purpose.
Let us first recall 

\begin{dfn}[{\cite[\S 4.3]{Ka}}]\label{dfn:li:dv}
Let $V$ be a vertex superalgebra, and let $\ve \in \bbZ/2\bbZ$.
A \emph{derivation} $d$ of parity $\ve$ on $V$
is a linear endomorphism of parity $\ve$, i.e., $d \in \End_{\bbk}(V)_{\ve}$, 
such that for any $a \in V$ we have
\[
 [d,Y(a,z)] = Y(d a, z).
\] 
\end{dfn}

\begin{rmk}\label{rmk:li:dv}
Let $d$ be a derivation of a vertex superalgebra $(V,\vac,T,Y)$.
\begin{enumerate}[nosep]
\item
\label{i:rmk:li:dv:dT}
We have $[d,T]=0$ by Remark \ref{rmk:li:vsa} (\ref{i:rmk:li:vsa:2}).
\item 
We have $d \vac = 0$ by Remark \ref{rmk:li:vsa} (\ref{i:rmk:li:vsa:3}).
\item
\label{i:rmk:li:dv:dab}
The condition on $d$ and $Y$ is equivalent to 
$d(a_{(n)}b)=(d a)_{(n)}b+(-1)^{p(a) \ve}a_{(n)}(d b)$
for any $a, b \in V$ with $a$ of pure parity.
\end{enumerate}
\end{rmk}

Here is our definition of a dg vertex algebra:

\begin{dfn}\label{dfn:li:dgva}
\begin{enumerate}[nosep]
\item 
A \emph{dg vertex algebra} (\emph{dgva} for short) is a complex $(V^{\bl},d)$
equipped with a vertex superalgebra structure $(\vac,T,Y)$ on the associated
linear superspace $V^{\even} \oplus V^{\odd}$ (Notation \ref{ntn:li:s-g})
such that the following conditions hold.
\begin{enumerate}[nosep, label=(\roman*)]
\item
$\vac \in V^0$ and $T \in \ul{\End}(V)^0=\Hom_{\dgVec}(V,V)$.
\item
$d$ is an odd derivation (Definition \ref{dfn:li:dv}) of the vertex superalgebra
$(V^{\even} \oplus V^{\odd},\vac,T,Y)$.
\item
The state-field correspondence $Y$ is homogeneous.
In other words, we have 
$a_{(n)} V^{j} \subset V^{i+j}$ for any $a \in V^i$ and $n \in \bbZ$.
\end{enumerate}
%The $\bbZ$-grading will be called the \emph{cohomological grading}.
We denote a dg vertex superalgebra as $V  = (V^{\bl},d,\vac,T,Y)$.

\item
We also have the notions of \emph{morphisms} and \emph{tensor products} of 
dg vertex algebras.
We denote by $\dgVA$ the category of dg vertex algebras.
\end{enumerate}
\end{dfn}

%The details are omitted.

For the later use, let us also introduce the corresponding notion of dg modules.

\begin{dfn}\label{dfn:li:dgMod}
Let $V$ be a dg vertex algebra.  %with Hamiltonian $H$.
\begin{enumerate}[nosep]
\item
A \emph{dg $V$-module} $M = (M^{\bl},d_M,Y^M)$ consists of 
\begin{itemize}[nosep]
\item 
a complex $(M^{\bl},d_M)$ and 
\item
a $(V^{\even} \oplus V^{\odd})$-module structure on $M^{\even} \oplus M^{\odd}$,
$Y^M(a,z)m=\sum_{n \in \bbZ} a^M_{(n)}m z^{-n-1}$, 
in the sense of Definition \ref{dfn:li:gvmod},
where $V^{\even} \oplus V^{\odd}$ is regarded as a vertex superalgebra,
\end{itemize}
such that %the following conditions hold.
%\begin{enumerate}[nosep,label=(\roman*)]
%\item 
%$[H,d_M]=0$.
%\item
\begin{align*}
 a_{(n)}^M M^j \subset M^{\abs{a}+j}, \quad 
 d_M(a^M_{(n)}m)=(d_V a)^M_{(n)}m+(-1)^{\abs{a}}a^M_{(n)}(d_M m)
\end{align*}
for any homogeneous $a \in V$, any $m \in M$ and any $j,n \in \bbZ$.
%\end{enumerate}

\item
We denote by $\dgVMod{V}$ the category of dg $V$-modules.
For a dg $V$-module $M$,
the $\bbZ$-grading of the underlying complex $(M^{\bl},d_M)$
is called the \emph{cohomological degree}.
\end{enumerate}
\end{dfn}

As for the cohomology of a dg vertex algebra, we have:

\begin{lem}\label{lem:dgv:coh}
For a dg vertex algebra $V=(V^{\bl},d,\vac,T,Y)$, consider the cohomology
$H^{\bl}=H(V^{\bl},d)$ of the underlying complex (Definition \ref{dfn:dg:coh}).
Then the associated linear superspace $H^{\even}\oplus H^{\odd}$
(Notation \ref{ntn:li:s-g}) has a structure of vertex superalgebra.
\end{lem}

\begin{proof}
The argument in \cite[\S 5.7.3]{FBZ} works. 
For the completeness, let us write down it.
By Remark \ref{rmk:li:dv} (\ref{i:rmk:li:dv:dT}), 
the translation $T$ preserves $\Ker d$ and $\Img d$.
The derivation property $[d,Y(a,z)]=Y(d a,z)$ yields that $\Ker d$ is 
a vertex sub-superalgebra and $\Img d$ is a vsa ideal of $V$.
Thus we have a quotient vsa $H(V^{\bl},d)$.
\end{proof}

The cohomology has a $\bbZ$-grading $H^{\bl}(V^{\bl},d)$ induced by 
the cohomological degree $V^{\bl}$.
But it is not necessarily a graded vertex superalgebra in the sense of 
Definition \ref{dfn:li:gvsa}.
In other words, the $\bbZ$-grading may not be given by the eigenvalue 
of any Hamiltonian.
In order to avoid this conflict, we introduce

\begin{dfn}\label{dfn:li:dgvh}
A \emph{graded dg vertex algebra} is a dg vertex algebra $V$
equipped with an additional linear decomposition
$V=\oplus_{\Delta \in \bbk} V_{\Delta}$ such that the following conditions hold.
\begin{enumerate}[nosep,label=(\roman*)]
\item
The underlying vsa $(V^{\even} \oplus V^{\odd},\vac,T,Y)$ 
is graded in the sense of Definition \ref{dfn:li:gvsa}.
%Hamiltonian $H \in \End_{\bbk}(V^{\even} \oplus V^{\odd})_{\ol{0}}$ 
%for the underlying vsa  (Definition \ref{dfn:vsa:Ham}) 
\item
$d$ preserves the additional gradings.
%which is compatible with $d$, i.e., $[d,H]=d H + H d = 0$.
\end{enumerate} 
%We denote by $V_{\bl}$ the additional 
\end{dfn}

A motivational example is \emph{weak BRST complex of vertex algebras} 
in \cite[3.15]{A07}.
We will give explicit examples in  \S \ref{sss:li:ff}.

Similarly as Definition \ref{dfn:li:gvmod}, 
we also have the notion of \emph{graded dg $V$-modules}
of a graded dg vertex algebra $V$.
We then have an obvious analogue of Example \ref{eg:li:h-m}:

\begin{lem}\label{lem:li:dgh-m}
%Let $(V,H_V)$ and $(W, H_W)$ be dg vertex algebras with Hamiltonian.
%If $\varphi:V \to W$ is a morphism of dg vertex algebras  
%which is compatible with Hamiltonians, i.e., $\varphi \circ H_V = H_W$,
Let $V$ and $W$ be graded dg vertex algebras
and $\varphi: V \to W$ be a morphism of dg vertex algebras
which preserves the additional gradings.
Then $W$ is naturally a graded dg $V$-module.
\end{lem}

%%%%%%%%%%%%%%%%%%%%%%%%%%%%%%%%%%%%%%%%%%%%%%%%%%%%%%%%%%%%%%%%%%%%%%%%%%%%%%%%
\subsubsection{The free fermionic vertex algebra}
%free fermionic vertex superalgebra
\label{sss:li:ff}

Let us recall the \emph{dg vertex algebra of free fermions}, 
a standard example of dg vertex algebra.
See \cite[4.3.1]{FBZ} for more information.
See also \cite[3.8.6]{BD}, where the coordinate free version is introduced 
under the name \emph{chiral Clifford algebra}.

Let $U$ be a bounded complex, and $U^* := \uHom(U,\bbk)$ be its dual.
We have the canonical pairing $\pair{\cdot,\cdot}: U^* \otimes U \to \bbk$.
We denote by $\bbk((t)) d t$ the linear space of one-forms on formal Laurent series.
We have the residue pairing $(f(t),g(t) d t) \mapsto \Res_t (f(t) g(t) d t)$
on $\bbk((t)) \otimes \bbk((t)) d t$.
Thus, denoting $U ((t)) := U \otimes \bbk((t))$ and 
$U^*((t)) d t := U \otimes \bbk((t)) d t $, we have a skew-symmetric pairing
$(\cdot,\cdot)$ on $U ((t))[1] \oplus U^*((t)) d t [-1]$
and itself induced by $\pair{\cdot,\cdot}$ and $\Res_t$.

For the pair $U ((t))[1] \oplus U^*((t)) d t [-1]$ and $(\cdot,\cdot)$,
we denote the associated Clifford algebra by 
\[
 \Cl = \Cl\bigl(U((t))[1] \oplus U^*((t))[-1],(\cdot,\cdot)\bigr).
\] 
It is a complete topological unital dg algebra.
Explicitly, taking a homogeneous linear basis $\{u_i \mid i \in I\}$ of $U$ 
and the dual basis $\{u^*_i \mid i \in I\}$ of $U^*$, 
we have that $\Cl$ is generated by 
\[
 \psi_{i,n}   := (u_i   \otimes t^n)[1] \in U((t))[1], \quad 
 \psi^*_{i,n} := (u^*_i \otimes t^{n-1}d t)[-1] \in U^*((t)) d t [-1] \quad
 (i \in I, \, n \in \bbZ)
\] 
and they satisfy
\begin{equation}\label{eq:li:ff-com}
 [\psi_{i,m},\psi_j t^n] = [\psi^*_i t^{m-1} d t,\psi^*_j t^{n-1} d t] = 0, \quad
 [\psi_i t^{m},\psi^*_j t^{n-1} d t] = \delta_{i,j}\delta_{m,-n}.
\end{equation}

Let us denote by $\Wedge^{\sinf}(U)$
the \emph{fermionic Fock module} of $\Cl$.
It is a left dg $\Cl$-module generated by an element $\vac$ such that
\[
 \psi_i   t^m \vac=0 \ (m \ge 0), \quad 
 \psi_i^* t^{n-1} d t \vac=0 \quad (n \ge 1).
\] 
It has a homogeneous basis consisting of the elements of the form 
\begin{align}\label{eq:li:ff}
 v = (\psi_{i_1} t^{m_1}) \cdots (\psi_{i_k} t^{m_k}) 
     (\psi_{j_1}^* t^{n_1-1}d t) \cdots (\psi_{j_l}^* t^{n_l-1}d t) \vac \quad
 (m_1< \cdots < m_k <0, \, n_1< \cdots < n_l \le 0).
\end{align}
We have an additional $\bbZ$-grading to the dg structure, 
called the \emph{charge grading}, 
under which the vector \eqref{eq:li:ff} is attached with $\chg =-k+l$.
We denote $\Wedge^{\sinf+\bl}(U)$ to emphasize the charge grading.

If $U$ is concentrated in degree $0$, then the charge is equal to the minus of 
the cohomological degree, and the parity is given by the charge modulo $2$.
In this case $\Wedge^{\sinf}(U)$ is also called 
the \emph{space of semi-infinite wedges}.

We now recall:

\begin{dfn}\label{dfn:li:ff}
The fermionic Fock module $\Wedge^{\sinf}(U)$ has a structure of 
dg vertex algebra with Hamiltonian, where
\begin{itemize}[nosep]
\item
the vacuum is given by $\vac$,

\item
the translation is given by $T=-\partial_t$, and 

\item 
the state-field correspondence is given by
\begin{align*}
&Y(\psi_i t^{-1} \vac,z) = \psi_i(z) := \tsum_{n \in \bbZ}(\psi_i t^n)z^{-n-1},
\\
&Y(\psi^*_i t^{-1} d t \vac,z)  = \psi^*_i(z)
 := \tsum_{n \in \bbZ}(\psi^*_i t^{n-1} d t)z^{-n},
\end{align*}
and for the element 
$v=(\psi_{i_1} t^{m_1}) \cdots (\psi_{i_k} t^{m_k}) 
   (\psi_{j_1}^* t^{n_1-1}d t) \cdots (\psi_{j_l}^* t^{n_l-1}d t) \vac$
in \eqref{eq:li:ff}, it is 
\begin{align*}
 Y(v,z) = \tprd_{a=1}^k \tfrac{1}{(-m_a-1)!} \tprd_{b=1}^l \tfrac{1}{(-n_b-1)!}
  :\partial_z^{-m_1-1}\psi_{  i_1}(z) \cdots \partial_z^{-m_k-1}\psi_{  i_k}(z)
   \partial_z^{-n_1  }\psi^*_{j_1}(z) \cdots \partial_z^{-n_l  }\psi^*_{j_l}(z):,
\end{align*}
where we used the normal ordering (Definition \ref{dfn:li:nord}).
\item
The additional $\bbZ$-grading is given by the charge grading.
%(The Hamiltonian is given by Remark \ref{rmk:vsa:Ham} 
% \eqref{i:rmk:vsa:Ham:inv}.)
\end{itemize}
We call it the \emph{free fermionic vertex algebra},
or the \emph{chiral Clifford algebra} and denote by 
\[
 \Wedge^{\sinf}(U) = 
 \chCl\bigl(U((t)),U^*((t))d t,(\cdot,\cdot)\bigr).
\]
For the second notation, see Remark \ref{rmk:ch:chCl}
\end{dfn}

%For $\alpha \in \bbk$, define the vector 
%\[
% \omega_\alpha := 
% \Bigl(\alpha \sum_{i=1}^{\dim U} (\psi^*_i t^{-2}d t) (\psi_i t^{-1})
%  + (1-\alpha)\sum_{i=1}^{\dim U} (\psi_i t^{-2}) (\psi^*_i t^{-1} d t)\Bigr)\vac.
%\]
%Then $Y(\omega_\alpha,z)=\sum_{n \in \bbZ}L_n z^{n-2}$
%satisfies the Virasoro condition
%%$[L_m,L_n]=(m-n) L_{m+n}+\frac{c_\alpha}{12}(m^3-m) \delta_{m,-n}\id$
%\eqref{eq:vsa:Vir} with $c = -2 (6\alpha^2-6\alpha+1) \dim U$.
%Now, setting $\alpha=0$, we have:
%
%\begin{fct*}[{\cite[\S 5.3.1]{FBZ}}]
%The free fermionic vsa $\Wedge^{\sinf}(U)$ is conformal
%with the conformal vector 
%\[
% \omega_0=(\psi_i t^{-2})(\psi^*_i t^{-1} d t)\vac
%\]
%of central charge $-2 \dim U$.
%\end{fct*}
%
%Note that $L_0$-eigenvalues, or conformal weights, 
%$\Delta$ are non-positive integers, and for the vector \eqref{eq:li:ff}
%we have $\Delta=-\sum_{a=1}^k m_k-\sum_{b=1}^l n_l$.
%In particular, we have
%\[
% \Wedge^{\sinf}(U) = 
% \tboplus_{\Delta \in \bbN} \bigl(\Wedge^{\sinf}(U)\bigr)_{\Delta}, \quad
% \bigl(\Wedge^{\sinf}(U)\bigr)_0 \simeq \Wedge(U^*),
%\]
%where the last term is the space of exterior products on $U^*$. 

%%%%%%%%%%%%%%%%%%%%%%%%%%%%%%%%%%%%%%%%%%%%%%%%%%%%%%%%%%%%%%%%%%%%%%%%%%%%%%%%
%%%%%%%%%%%%%%%%%%%%%%%%%%%%%%%%%%%%%%%%%%%%%%%%%%%%%%%%%%%%%%%%%%%%%%%%%%%%%%%%
\subsection{Dg vertex Li algebras}
\label{ss:li:dgvl}

In this subsection we recall the notion \emph{vertex Lie algebras}.
It is designed to encode the ``polar part" of vertex algebra.
We work over a field $\bbk$ of characteristic $0$.

%%%%%%%%%%%%%%%%%%%%%%%%%%%%%%%%%%%%%%%%%%%%%%%%%%%%%%%%%%%%%%%%%%%%%%%%%%%%%%%%
\subsubsection{Definition}

\begin{dfn}\label{dfn:li:vla}
\begin{enumerate}[nosep]
\item 
A \emph{dg vertex Lie algebra} is a data $(L,d,T,Y_-)$ consists of 
\begin{itemize}[nosep]
\item 
a complex $(L,d)$,
\item
an endomorphism $T \in \uEnd(L)^0 = \Hom_{\dgVec}(L,L)$, and
\item
a morphism $L \otimes L \to z^{-1} L[z^{-1}]$ of complexes, 
$a \otimes b \mapsto Y_-(a,z)b = \sum_{n \in \bbN} (a_{(n)}b)z^{-n-1}$
\end{itemize}
satisfying the following conditions.
\begin{enumerate}[nosep, label=(\roman*)]
\item\label{i:li:vla:1} 
$Y_-(T a,z)=\partial_z Y_-(a,z)$.

\item\label{i:li:vla:2}
For any homogeneous $a,b \in L$, we have 
\[
 Y_-(a,z)b = \bigl((-1)^{\abs{a}\abs{b}} e^{z T}Y_-(b,-z)a\bigr)_{-}.
\]
Here for a formal Laurent series $f(z) = \sum_i f_i z^i \in L((z))$
we set $f(z)_-:=\sum_{i<0} f_i z^i \in z^{-1} L[z^{-1}]$.

\item\label{i:li:vla:3}
For $a,b \in L$ we have 
$[a_{(m)},b_{(n)}] = \sum_{i \in \bbN} \binom{m}{i}(a_{(i)}b)_{(m+n-i)}$
in $\uEnd(L)=\uHom_{\dgVec}(L,L)$.

\item\label{i:li:vla:4}
For any $a,b \in L$, we have $[d,Y_-(a,z)]b = Y_-(d a,z)b$.
\end{enumerate}
%[a_{(m)},Y_-(b,w)]=\sum_{l \in \bbN} \binom{m}{l} (w^{m-l}Y_-(a_{(l)}b,w)_-
%=\sum_{l \in \bbN} \binom{m}{l} (w^{m-l} \sum_k (a_{(l)}b)_{(k)} w^{-k-1})_-
%=\sum_{l \in \bbN} \binom{m}{l} (\sum_k (a_{(l)}b)_{(k)} w^{m-l-k-1})_-
%-n = m-l-k \le 0
%[a_{(m)},b_{(n)}]
%=\sum_{l \in \bbN} \binom{m}{l} (a_{(l)}b)_{(m+n-l)} 

\item
We have obvious notions of \emph{morphism} and \emph{tensor product} of 
dg vertex Lie algebras.
We denote by $\dgVL$ the category of dg vertex Lie algebras.
\end{enumerate}
\end{dfn}

%First we should give a dg version of vertex Lie superalgebras.
%Recall Definition \ref{dfn:li:dgva} of dg vertex algebras.
%
%\begin{dfn*}
%\begin{enumerate}
%\item
%A \emph{derivation} of parity $\ve \in \bbZ/2\bbZ$ on a vertex Lie superalgebra 
%$(V,T,Y_-)$ is a linear endomorphism of parity $\ve$ on the underlying linear 
%superspace $V$ such that for any $a \in V$ we have $[d,Y_-(a,z)] = Y_-(d a,z)$.
%
%\item 
%A \emph{dg vertex Lie algebra} is a complex $(V^{\bl},d)$ equipped with 
%a vertex Lie superalgebra structure $(T,Y_-)$ on the associated linear superspace
%$V^{\even} \oplus V^{\odd}$ (Notation \ref{ntn:li:s-g})
%such that $d$ is an odd derivation and $Y_-$ is homogeneous.
%\end{enumerate}
%\end{dfn*}

Similarly we have the notions of \emph{vertex Lie superalgebra}
and \emph{vertex Lie algebra}.
We refer \cite[\S 16.1]{FBZ} for the detail of the latter notion.

\begin{rmk}\label{rmk:li:vla}
\begin{enumerate}[nosep]
\item 
\label{i:rmk:li:vla:1}
The condition \ref{i:li:vla:2} is equivalent to
\[
 a_{(n)}b = 
 \tsum_{i \in \bbN}(-1)^{\abs{a}\abs{b}+n+i+1}\frac{1}{i!}T^i(b_{(n+i)}a)
 \quad (n \in \bbN).
\]

\item
\label{i:rmk:li:vla:2}
We can check that the quotient space $L/\Img(T)$ is a dg Lie algebra
with the Lie bracket 
\[
 [\ol{a},\ol{b}] := \ol{a_{(0)}b}
\]
for $a,b \in L$.
Indeed, by the condition \ref{i:li:vla:1} we have $(T a)_{(0)}b=0$ 
so that the Lie bracket is well-defined,
by the condition \ref{i:li:vla:2} and the previous item \eqref{i:rmk:li:vla:1}
we have $a_{(0)}b = -(-1)^{\abs{a}\abs{b}}b_{(0)}a \mod \Img(T)$ 
so that the anti-commutativity holds,
and by \ref{i:li:vla:3} we have $[a_{(0)},b_{(0)}]=(a_{(0)}b)_{(0)}$
and can check the Jacobi relation by direct computation.

\item
By the theory of coisson algebras \cite[\S 2.6]{BD},
a vertex Lie algebra satisfying some conditions can be regarded 
as a Lie algebra object in a certain symmetric monoidal category.
See Remark \ref{rmk:co:star} for more information.
\end{enumerate}
\end{rmk}

By the definition we immediately have
the following construction of vertex Lie algebras from vertex algebras.
See \cite[16.1.2--3]{FBZ} for the detail.

\begin{lem}\label{lem:li:VL}
For a dg vertex algebra $V=(V,d,\vac,T,Y)$, 
we have a dg vertex Lie algebra 
\[
 V_{\tLie} := (V,d,T,Y_-)
\]
by setting $Y_-(a,z) := Y(a,z)_ = \sum_{n \in \bbN} a_{(n)} z^{-n-1}$.
This construction gives rise to a functor 
\[
 \dgVA \longto \dgVL, \quad V \longmapsto V_{\tLie},
\]
which we call the \emph{polar part construction}.
\end{lem}

See \S \ref{sss:li:UL} for the adjoint of the polar part construction.
For later citation, we give an example.

\begin{eg}\label{eg:li:vk}
We use the notations in \S \ref{sss:li:uava} here.
Consider the universal affine vertex algebra $V_k(\frkg)$ over $\bbC$.
Following \cite[16.1.5]{FBZ}, we denote by 
\[
 v_k(\frkg) \subset V_k(\frkg)
\]
the linear subspace spanned by the vacuum vector $\vac$ and 
the vectors $x t^m \vac$ ($x \in \frkg$, $m \in \bbZ_{<0}$).
Thus we have $v_k(\frkg) \simeq \frkg[t^{-1}] \oplus \bbC$ as linear spaces.
The subspace $v_k(\frkg)$ is invariant under the action of 
$x_-(z) := Y(x t^{-1}\vac,z)_- = \sum_{n \in \bbN}(x t^n) z^{-n-1}$,
and we can see that $(v_k(\frkg),T,Y_-)$ with $Y_-(x t^{-1}\vac,z):= x_-(z)$
is a vertex Lie algebra.
\end{eg}

Here is another construction of dg vertex Lie algebras:

\begin{lem}\label{lem:li:0vla}
Let $\frkl$ be dg Lie algebra, 
and consider $\frkl[[t]] = \frkl \otimes \bbk[[t]]$.
It is a dg Lie algebra with the Lie bracket 
$[x \otimes f, y \otimes g] := [x,y]_{\frkl} \otimes (f g)$
for $x,y \in \frkl$ and $f,g \in \bbk[[t]]$.
Then $\frkl[[t]]$ has the following dg vertex Lie algebra structure:
%For $x \in \frkl$ and $i \in \bbZ_{>0}$, 
%we set $x[-i]:=\frac{1}{(i-1)!}x \otimes t^{i-1} \in \frkl[[t]]$.
%Then the structure is given by:
\begin{itemize}[nosep]
\item 
The translation $T$ is given by the multiplication by $t$.
%Thus $T(x[-i]):=i x[-i-1]$.
Thus $T(x \otimes t^i)=x \otimes t^{i+1}$.

\item
The operation $Y_-$ is determined by 
\begin{align}
\nonumber
&Y_-(x \otimes 1,z) = x_-(z) := \tsum_{n \in \bbN} z^{-n-1} \partial_t^n([x,-]),
\\
\label{eq:li:vla:ij}
&Y_-(x \otimes t^i,z)(y \otimes t^j) := 
 \bigl((-1)^{\abs{x} \abs{y}} e^{z t} (-\partial_z)^j y_-(-z)(x \otimes t^i)\bigr)_-.
\end{align}
Explicitly, we have 
\begin{equation}\label{eq:li:vla:ex}
 (x \otimes t^i)_{(n)}(y \otimes t^j) = 
 (-1)^{i}\frac{n!}{(n-i)!}\frac{j!}{(j-n+i)!} [x,y]_{\frkl} \otimes t^{i+j-n}
\end{equation}
for $i \le n \le i+j$ and $0$ otherwise.
\end{itemize}
We denote the obtained dg vertex Lie algebra by
\[
 J_\infty(\frkl) := (\frkl[[t]],d_{\frkl},T,Y_-),
\]
and call it the \emph{level $0$ dg vertex Lie algebra}.
\end{lem}

\begin{proof}
It is enough to check the conditions \ref{i:li:vla:1}--\ref{i:li:vla:4} 
in Definition \ref{dfn:li:vla}.
The condition \ref{i:li:vla:1} follows from \eqref{eq:li:vla:ex}.
The condition \ref{i:li:vla:2} follows from the definition \eqref{eq:li:vla:ij}.
The condition \ref{i:li:vla:3} can be checked with the help of Jacobi relation 
of $\frkl$.
The condition \ref{i:li:vla:4} follows from \eqref{eq:li:vla:ex} 
and the dg Lie algebra structure on $\frkl$.
\end{proof}

For the naming ``level $0$", see Example \ref{eg:li:arc-sg}.

%%%%%%%%%%%%%%%%%%%%%%%%%%%%%%%%%%%%%%%%%%%%%%%%%%%%%%%%%%%%%%%%%%%%%%%%%%%%%%%%
\subsubsection{Enveloping vertex algebra}
\label{sss:li:UL}

We explain the universal construction of dg vertex algebra from 
a dg vertex Lie algebra, following the non-dg version in \cite[\S 16.1]{FBZ}.
See also \cite[\S 3.7]{BD} for the chiral algebra version.

Let $L=(L,d,T,Y_-)$ be a dg vertex Lie algebra, and consider the operator
$\partial := T \otimes \id + \id \otimes \partial_t$
on  $L((s)) = L \otimes \bbk((s))$.
We define 
\[
 \Lie(L) := L((s)) / \Img(\partial),
\]
which inherits the dg structure of $L$.
We also denote 
$x_{[n]} := \ol{x \otimes s^n} \in \Lie(L)$ for $x \in L$ and $n \in \bbZ$,
and define 
\[
 \Lie(L)_+ \subset \Lie(L)
\]
{}to be the subspace which is the completion 
of the linear span of $x_{[n]}$ with $n \in \bbN$.
It is in fact a subcomplex of $\Lie(L)$.

\begin{lem*}[{\cite[16.1.7.\ Lemma]{FBZ}}]
\begin{enumerate}[nosep]
\item 
$\Lie(L)$ is a dg Lie algebra with the Lie bracket 
\[
 [x_{[m]},y_{[n]}] := 
 \sum_{l \in \bbN} \binom{m}{l} (x_{(l)}y)_{[m+n-l]},
\]
where we denoted by $Y_-(x,z)=\sum_{n \in \bbN} x_{(n)}z^{-n-1}$
the vertex Lie structure of $L$.

\item
$\Lie(L)_+$ is a dg Lie subalgebra of $\Lie(L)$.

\item
The correspondence $\Lie(L)_+ \to \uEnd(L)$, $x_{[n]} \to x_{(n)}$ 
gives a morphism of dg Lie algebras.
\end{enumerate}
\end{lem*}

Consider the universal enveloping algebra $U(\Lie(L))$ and $U(\Lie(L)_+)$
of the dg Lie algebras in the above lemma.
We define a left dg $U(\Lie(L))$-module
\[
 U(L) := U(\Lie(L)) \otimes_{U(\Lie(L)_+)}\bbk,
\]
where $\bbk$ denotes the trivial one-dimensional representation
of $U(\Lie(L)_+)$.

\begin{prp*}[{\cite[16.1.12.\ Propsition]{FBZ}}]
There is a unique dg vertex algebra structure on $U(L)$ such that 
\begin{itemize}[nosep]
\item
the vacuum is $\vac := 1 \otimes 1$,
\item 
the translation $T$ is defined by $T \vac=0$ and $[T,x_{[n]}]=-n x_{[n-1]}$,
and 
\item
$Y(x_{[-1]}\vac,z)=\sum_{n \in \bbZ}x_{[n]} z^{-n-1}$ for $x \in L$.
\end{itemize}
Moreover, the correspondence $L \mapsto U(L)$ gives a functor 
that is a left adjoint of the polar part construction $V \mapsto V_{\tLie}$
in Lemma \ref{lem:li:VL}:
\[
% \Hom(U(L),V) \simeq \Hom(L,V_{\tLie}).
 (-)_{\tLie}: \dgVA \rightleftarrows \dgVL :U(-).
\]
We call $U(L)$ the \emph{enveloping vertex algebra} of $L$.
\end{prp*}

Finally let us explain a twisted version of $U(L)$.
See \cite[3.7.20]{BD} for the original definition.
We note that the category $\dgVL$ of dg vertex algebras is an abelian category.

\begin{dfn}\label{dfn:li:ULf}
Let $L$ be a dg vertex Lie algebra.
\begin{enumerate}[nosep]
\item
A \emph{one-dimensional central extension} of $L$
is a dg vertex Lie algebra $L^{\flat}$ which sits in an exact sequence
\[
 0 \longto \bbk_{\tLie} \longto L^{\flat} \longto L \longto 0
\]
in the abelian category $\dgVL$.
Here $\bbk_{\tLie}$ denotes the one-dimensional linear space 
with trivial dg vertex Lie algebra structure.
We denote by $1^{\flat} \in L^{\flat}$ the image of $1_{\bbk} \in \bbk_{\tLie}$.

\item
Assume that we are given a one-dimensional central extension $L^\flat$ of $L$.
We define 
\[
 U(L)^\flat := U(L^{\flat})/(1_{U(L)} - 1^{\flat}),
\]
where $1^{\flat}$ denotes the image of $1^{\flat} \in L^{\flat}$
in the enveloping vertex algebra $U(L^{\flat})$.
This quotient inherits the dg vertex algebra structure of $U(L^{\flat})$.
We call the resulting dg vertex algebra $U(L^{\flat})$ 
the \emph{twisted enveloping vertex algebra of $L$}.

\item\label{i:li:ULf:PBW}
$U(L)$ and $U(L)^\flat$ inherit the PBW filtration of
$U(\Lie(L))$ and  $U(\Lie(L^\flat))$ respectively.
We call the resulting filtration the \emph{PBW filtration} of 
the (twisted) enveloping vertex algebra,
and denote it by $U(L)_{\bl}$, $U(L)^{\flat}_{\bl}$.
\end{enumerate}
\end{dfn}

\begin{eg}[{\cite[16.1.9]{FBZ}}]\label{eg:li:UL}
Consider the vertex Lie algebra $v_k(\frkg)$ over $\bbC$ 
in Example \ref{eg:li:vk}.
Recalling that $v_k(\frkg)=\frkg[t^{-1}]\vac \oplus \bbC \vac$
as linear spaces, we can check that $\Lie(v_k(\frkg))$ is spanned by 
$(x t^{-1}\vac)_{[n]}$ ($n \in \bbZ$) and $\vac_{[-1]}$,
and the commutation relation is 
$[(xt^{-1}\vac)_{[m]},(yt^{-1}\vac)_{[n]}] = ([x,y]_{\frkg}t^{-1}\vac)_{[m+n]}+
 \frac{k}{2h^{\vee}} m \delta_{m+n,0}\kappa_{\frkg}(x,y)\vac_{[-1]}$.
Thus we have 
\[
 \Lie(v_k(\frkg)) \simeq \wh{\frkg}_k, \quad
 \Lie(v_k(\frkg))_+ \simeq \frkg[[t]],
\]
where $\wh{\frkg}_k = \frkg((t)) \oplus \bbC \one$ denotes the affine Lie algebra
obtained by replacing $K$ in $\wh{\frkg}$ with $k \one$.
Thus $U(v_k(\frkg)) \simeq U(\wh{\frkg}) \otimes_{U(\frkg[[t]])} \bbC$.

The linear space $L := \frkg[t^{-1}]$ has a structure of vertex Lie algebra
induced by the Lie algebra $\frkg((t))$.
Explicitly, $L$ is regarded as a subspace of the $\frkg((t))$-module 
$U(\frkg((t)))\otimes_{U(\frkg[[t]])} \bbC$, 
and the vertex Lie structure is given by $T=-\partial_t$ and 
$Y_-(x t^{-1},z)=\sum_{n \in \bbN}(x t^n)z^{-n-1}$.
We have $\Lie(L)\simeq \frkg((t))$ and $\Lie(L)_+ \simeq \frkg[[t]]$.

Then $v_k(\frkg)$ is a one-dimensional central extension of $L$:
\[
 0 \longto \bbC_{\tLie} \longto L^\flat=v_k(\frkg) \longto 
 L = \frkg[t^{-1}] \longto 0.
\]
We can identify $1^\flat=\one$, so that we have
$U(L)^\flat = U(v_k(\frkg))/ (1-1^\flat) \simeq V_k(\wh{\frkg})$.
\end{eg}

%%%%%%%%%%%%%%%%%%%%%%%%%%%%%%%%%%%%%%%%%%%%%%%%%%%%%%%%%%%%%%%%%%%%%%%%%%%%%%%%
%%%%%%%%%%%%%%%%%%%%%%%%%%%%%%%%%%%%%%%%%%%%%%%%%%%%%%%%%%%%%%%%%%%%%%%%%%%%%%%%
\subsection{Dg vertex Poisson algebras}
\label{ss:li:dgvp}

Recall that a \emph{vertex Poisson algebra} is an analogue of Poisson algebra
in the category of vertex algebras.
As Poisson algebra is defined to be a combined structure commutative algebra
and Lie algebra, the definition of vertex Poisson algebra is a combination of 
commutative vertex algebra and vertex Lie algebra.
A vertex Poisson algebra is also regarded as a ``classical limit" of 
a vertex algebra, as from a Poisson algebra can be obtained from 
the PBW filtration of the universal enveloping algebra of a Lie algebra.
We refer \cite[Chap.\ 16]{FBZ} for the detail of vertex Poisson algebras.

In this subsection we introduce a dg analogue of vertex Poisson algebras.
It is in fact a special case of \emph{coisson algebras}
discussed in \cite[\S 2.6]{BD}.
See also Remark \ref{rmk:co:star} for more information on this point.

%%%%%%%%%%%%%%%%%%%%%%%%%%%%%%%%%%%%%%%%%%%%%%%%%%%%%%%%%%%%%%%%%%%%%%%%%%%%%%%%
\subsubsection{Commutative dg vertex algebras}

First we introduce a dg version of commutative vertex algebras.

\begin{dfn*}
A dg vertex algebra $V$ is \emph{commutative} if the integer $N$ 
in the locality axiom (Definition \ref{dfn:li:vsa}) can be taken to be $0$ 
for any homogeneous $a,b \in V$.
In other words, we have $[a_{(m)},b_{(n)}]=0$ for any $m,n \in \bbZ$.
\end{dfn*}

We can see that for a commutative dg vertex algebra $V$, the state-field 
correspondence satisfies $a_{(n)}=0$ for any $n \in \bbN$ and $a \in V$.
Hereafter we denote a commutative dg vertex algebra by 
\begin{equation}\label{eq:li:cvsa}
 (V,d,\vac,T,Y_+)
\end{equation}
in order to emphasize 
$Y_+(a,z)=\sum_{n \in \bbZ_{<0}}a_{(n)} z^{-n-1} \in \uEnd(V)[[z]]$.

Now recall the notion of derivations on a dg algebra 
(Definition \ref{dfn:dga:der}).
The non-dg version of the next statement is due to Borcherds:

\begin{lem}\label{lem:vpa:cvs=cds}
A commutative dg vertex algebra $ (V,d,\vac,T,Y_+)$ is equivalent to 
a unital commutative dg algebra $(V,d,\cdot,\vac)$
equipped with an extra $0$-derivation $T$ commuting with the differential $d$.
\end{lem}

\begin{proof}
As in the non-dg case (see \cite[\S 1.4]{FBZ}), 
we attach to a commutative vertex algebra structure $Y_+$ 
the multiplication $a \cdot b := a_{(-1)}b$,
and attach to the differential algebra structure $(\cdot,T)$
the vertex algebra structure $Y_+(a,z)b := (e^{z T}a)$.
%The non-dg case is well-known and written down in \cite[\S 1.4]{FBZ}), 
%so we only present an outline.
%If $(V,d,\vac,T,Y_+)$ is a commutative dgva,
%then $a \cdot b := a_{(-1)}b$ gives the multiplication,
%$\vac$ gives the unit, and the translation $T$ gives the $0$-derivation.
%Conversely, given a unital cdga $(V,d,\cdot,1_V,T)$ with $0$-derivation, 
%we can equip $V$ with a structure of commutative dgva by  
%$Y_+(a,z)b := (e^{z d}a) \cdot b$, $T:=d$ and $\vac:=1_V$.
\end{proof}

By this lemma, we can denote the commutative dgva by \eqref{eq:li:cvsa} or by 
\[
 (V,d,\vac,T,\cdot).
\]

%%%%%%%%%%%%%%%%%%%%%%%%%%%%%%%%%%%%%%%%%%%%%%%%%%%%%%%%%%%%%%%%%%%%%%%%%%%%%%%%
\subsubsection{Dg vertex Poisson algebras}

We turn to the definition of a dg vertex Poisson algebra.
Recall that we denote by $\Der(R)^n$ the linear space of $n$-derivations
for a dg algebra $R$ (Definition \ref{dfn:dga:der}).
%of parity $\ve$ on a superalgebra $R$ (Definition \ref{dfn:li:der}).
For a commutative dgva $(P,d,\vac,T,\cdot)$,
we denote by  $\Der(P,\cdot)^n$ the corresponding space
for the cdga $(P,\cdot)$.
Using this notation, we have:

\begin{dfn}\label{dfn:li:dgvp}
\begin{enumerate}[nosep]
\item 
A \emph{dg vertex Poisson algebra} (\emph{dg vpa} for short) 
is a complex $(P,d)$ equipped with 
\begin{itemize}[nosep]
\item 
a structure $(P,d,\vac,T,Y_+)$ of commutative dgva, or equivalently, 
a structure $(P,d,\vac,T,\cdot)$ of unital cdga,  and 
\item
a structure $(P,d,T,Y_-)$ of dg vertex Lie algebra 
\end{itemize}
such that for any homogeneous $a \in P$ we have
\begin{equation}\label{eq:li:dgvp}
Y_-(a,z) \in (\Der(P,\cdot)^{\abs{a}})[[z^{-1}]]z^{-1}.
\end{equation} 
We denote it by $(P,d,\vac,T,\cdot,Y_-)$ or by $(P,d,\vac,T,Y_+,Y_-)$, or just by $P$.

\item
\label{i:li:dgvp:an}
Let $V$ be a dg vpa and $a \in V$.
We define $a_{(n)}$ for $n \in \bbZ$ in the following way.
\begin{itemize}[nosep]
\item 
For $n \ge 0$, we denote by $Y_-(a,z)=\sum_{n \ge 0}a_{(n)}z^{-n-1}$
the vertex Lie algebra structure.
\item
For $n <   0$, we denote by $Y_+(a,z)=\sum_{n <   0}a_{(n)}z^{-n-1}$
the commutative vertex algebra structure.
Explicitly we have $a_{(n)}=\frac{1}{(-n-1)!}T^{-n-1}a$.
\end{itemize}

\item
We have an obvious notion of \emph{morphisms} of dg vertex Poisson algebras.
We denote by $\dgVP$ the category of dg vertex Poisson algebras.
\end{enumerate}
\end{dfn}

Similarly we have the notions of \emph{vertex Poisson superalgebra} and
\emph{vertex Poisson algebra} (\emph{vpa} for short).

Note that the translation $T$ of the commutative vertex algebra structure 
and that of the vertex Lie algebra structure are common.

We have notions of dg subalgebras, dg ideals, and tensor products of dg vpas.
Details are omitted.

Since the structure maps $Y$ and $Y_-$ of a dg vertex Poisson algebras
are homogeneous with respect to the cohomological grading
and the differential $d$ has the derivation property both for them,
the proof of Lemma \ref{lem:dgv:coh} works in the Poisson case, and we have:

\begin{lem*}
For a dg vpa $V$, the cohomology $H(V^{\bl},d)$ of the underlying complex
has a structure of dg vertex Poisson algebra with trivial differential,
called the \emph{cohomology vertex Poisson algebra}.
\end{lem*}

We continue the observation in Remark \ref{rmk:li:vla} \eqref{i:rmk:li:vla:2}.
For a dg vertex Poisson algebra $P=(P,d,\vac,T,\cdot,Y_-)$, 
the quotient space $P/\Img(T)$ has a structure of dg Poisson algebra,
of which the commutative multiplication and the Poisson bracket are given by 
\[
 \ol{a} \cdot \ol{b} := \ol{a \cdot b}, \quad 
 \{\ol{a},\ol{b}\}:=\ol{a_{(0)}b}
\]
for $a, b \in P$.
The Leibniz rule follows from \eqref{eq:li:dgvp}.
Since we have $[d,T]=0$, 
the quotient $P/\Img(T)$ inherits the dg structure $(P,d)$.

\begin{dfn}\label{dfn:li:dgapa}
\begin{enumerate}
\item 
We call the obtained dg Poisson algebra $P/\Img(T)$ 
the \emph{associated dg Poisson algebra} of the dg vertex Poisson algebra $P$.
We denote it by 
\[
 R^{\cois}_P := P/\Img(T).
\]
\item
We denote the functor induced by the construction $P \mapsto R^{\cois}_P$ by 
\[
 R^{\cois}_{(-)}: \dgVP \longto \dgpa, 
\]
where $\dgpa$ denotes the category of dg Poisson algebras.
\end{enumerate}
\end{dfn}

We can immediately check that the functor $R^{\cois}_{(-)}$ is 
a monoidal functor between the corresponding symmetric monoidal categories.

\begin{rmk*}
We refer Remark \ref{rmk:li:C2} \eqref{i:li:RV=RP} for the relation to 
\emph{Zhu's $C_2$-algebra} $R_V$ of a vertex algebra $V$. 
\end{rmk*}

%%%%%%%%%%%%%%%%%%%%%%%%%%%%%%%%%%%%%%%%%%%%%%%%%%%%%%%%%%%%%%%%%%%%%%%%%%%%%%%%
\subsubsection{Vertex Poisson modules}
\label{sss:li:vpm}

We introduce the notion of modules over dg vertex Poisson algebras,
following the non-dg version given in \cite[\S 2.2]{A12} and \cite[\S 2.2]{A15}.

\begin{dfn}\label{dfn:li:vpm}
Let $P=(P,d,\vac,T,\cdot,Y_-)$ be a dg vertex Poisson algebra.
\begin{enumerate}
\item 
A \emph{dg vertex Poisson $P$-module} is a complex $M=(M^{\bl},d_M)$ equipped with
\begin{itemize}[nosep]
\item 
a structure $P \otimes M \to M$, $a \otimes m \mapsto a.m$ 
of dg module over the cdga $(P,\cdot,\vac)$, and
%$a \cdot b := a_{(-1)} b$ 
%associated to the commutative vertex superalgebra $(P,\vac,T,Y)$, and 
\item
a morphism of complexes $P \otimes M \to M[z^{-1}] z^{-1}$,
$a \otimes m \mapsto Y^M_-(a,z)m = \sum_{n \in \bbN} (a^M_{(n)}m) z^{-n-1}$
%of module over the vertex Lie algebra $(P,T,Y_-)$
\end{itemize}
satisfying the following conditions for homogeneous $a,b \in P$,
$m \in M$ and $l,n \in \bbN$.
\begin{enumerate}[nosep,label=(\roman*)]
\item
$Y^M_{-}(\vac,z)=\id_M$.
\item 
$(T a)^M_{(n)}=-n a^M_{(n-1)}$, 
where for $n=0$ we read the right hand side as $0$.
\item 
$[a^M_{(l)},b^M_{(n)}]=\sum_{i \in \bbN}\binom{l}{i}(a_{(i)}b)^M_{(l+n-i)}$.
\item
$a^M_{(n)} (b.m) = (a_{(n)}b).m+(-1)^{\abs{a}\abs{b}} b.(a^M_{(n)}m)$.
%for $a,b \in V$ of pure parity, $n \in \bbN$ and $v \in M$.
\item
$(a \cdot b)^M_{(n)} = \sum_{l \in \bbN}(a_{(-l-1)}b^M_{(n+l)}
 +(-1)^{\abs{a}\abs{b}}b_{(-l-1)}a^M_{(n+l)})$.
\item
$[d_M,Y^M_-(a,z)]m=Y^M_-(d a,z)m$.
\end{enumerate}

\item
A morphism of dg vertex Poisson $P$-modules is naturally defined.
We denote by $\dgVPMod{P}$ the category of dg vertex Poisson $P$-modules,
and by $\ddgVPMod{P}$ the corresponding $\infty$-category.

\item
A \emph{vertex Poisson $P$-module} is 
a dg vertex Poisson $P$-module concentrating on the cohomological degree $0$.
We denote by $\VPMod{P}$ the category of vertex Poisson $P$-modules.
\end{enumerate}
\end{dfn}

Now we give a vertex Poisson analogue of Fact \ref{fct:co:wtUV}
on the equivalence of module categories.
Let $P=(P,d,\vac,T,\cdot,Y_{-})$ be a dg vertex Poisson algebra.
Recall the notation $a_{(n)}$ for an element $a \in P$ 
(Definition \ref{dfn:li:dgvp} \eqref{i:li:dgvp:an}).
We have the following line of constructions.
\begin{enumerate}[nosep]
\item 
We define the complex $U'(P)$ to be the quotient of 
$P[t^{\pm1}] = P \otimes \bbk[t^{\pm 1}]$ by the relation 
$(T a) \otimes t^n + n a \otimes t^n=0$.
We denote the image of $a \otimes t^n$ by  $a_{[n]}$.
We have $a_{[-n-1]}=\frac{1}{n!}(T^n a)_{[-1]}$ for $n \in \bbN$.

\item
We can define the binary operation $\cdot: U'(P)^{\otimes 2} \to U'(P)$ by
\[
 a_{[-m-1]} \cdot b_{[-n-1]} := 
 \begin{cases} 
  \bigl((\tfrac{1}{m!}T^m a)\cdot(\tfrac{1}{n!}T^n b)\bigr)_{[-1]} 
    & (m,n \in \bbN), \\
  0 & (\text{otherwise}).
 \end{cases}
\]
Then $(U'(P),\cdot)$ is a cdga with unit $\vac_{[-1]}$.

\item
We can define the binary operation $\{-,-\}: U'(P)^{\otimes 2} \to U'(P)$ by
\[
 \{a_{[m]},b_{[n]}\} := 
 \begin{cases} 
  \tsum_{i \in \bbN} \binom{m}{i}(a_{(i)}b)_{[m+n-i]}
    & (m,n \in \bbN), \\
  0 & (\text{otherwise}).
 \end{cases}
\]
Then $(U'(P),\{-,-\})$ is a dg Lie algebra 
by the proof of \cite[4.1.2.\ Theorem]{FBZ}.

\item
We can  check that 
$(U'(P),\cdot,\{-,-\})$ is a dg Poisson algebra with unit $\vac_{[-1]}$.

\item
We define $U''(P)$ to be the completion of $U'(P)$ 
with respect to the $t$-adic topology.
Thus we can regard it as a quotient of $P((t)) = P \otimes \bbk((t))$.
The completion $U''(P)$ inherits the dg Poisson algebra structure of $U'(P)$.
We further define $U(P):=\lim_N U''(P)/I_N$,
where $I_N$ denotes the left Poisson ideal generated by 
$a_{[n]}$ ($a \in P$, $n \in \bbZ$, $n>N$).
It also inherits the dg Poisson algebra structure.

\item
For $a \in P$, we denote $Y[a,z]:=\sum_{n \in \bbZ}a_{[n]}z^{-n-1}$.
We define $\wt{U}(P)$ to be the quotient of $U(P)$
with respect to the relation $Y[a_{(-1)}b,z]=:Y[a,z]Y[b,z]:$, where 
the left hand side denotes the normal ordering (Definition \ref{dfn:li:nord}).

\item
\label{i:li:vpm:sm}
A dg Poisson $\wt{U}(P)$-module $M$ is called \emph{smooth} if for any $m \in M$
and $a \in P$ we have $a_{[n]}.m=0$ for $n \gg 0$.
\end{enumerate}

Now we can apply the same argument in the proof of \cite[5.1.6.\ Theorem]{FBZ},
and we have;

\begin{prp}\label{prp:li:vpmod}
For a dg vertex Poisson algebra $P$, 
there is an equivalence of the category $\dgPMod{P}$ of dg vertex Poisson 
$P$-modules and the category of smooth dg Poisson $\wt{U}(P)$-modules.
\end{prp}

In the following, 
we will explain three constructions of dg vertex Poisson algebras.
\begin{enumerate}[nosep]
\item 
Symmetric algebra of a dg vertex Lie algebra (\S \ref{sss:li:SymL}).
\item
Quasi-classical limit of a dg vertex algebra (\S \ref{sss:li:qc}).
\item
Level $0$ vertex Poisson structure on the arc space of a dg Poisson algebra
(\S \ref{sss:li:arc}).
\item
Associated graded space of Li filtration of a vertex algebra (\S \ref{ss:li:li}).
\end{enumerate}

%%%%%%%%%%%%%%%%%%%%%%%%%%%%%%%%%%%%%%%%%%%%%%%%%%%%%%%%%%%%%%%%%%%%%%%%%%%%%%%%
\subsubsection{Symmetric vertex Poisson algebra of vertex Li algebra}
\label{sss:li:SymL}

We first explain the symmetric algebra construction.
See also \cite[16.2.2]{FBZ} and \cite[1.4.18]{BD}.

Let $L=(L^{\bl},d,T_L,Y_{-}^L)$ be a dg vertex Lie algebra.
Regarding $L$ as a complex, we denote by 
\[ 
 P := \Sym(L)
\]
the symmetric algebra, which is a commutative dg algebra.
Then $P$ has a structure $(\vac,T,Y_-)$ of dg vertex Poisson algebra where
\begin{itemize}[nosep]
\item 
the vacuum $\vac$ is the unit element of $P$,
\item
the translation $T$ is the extension of $T_L$ to $P$ 
by the Leibniz rule and the condition $T \vac=0$. and
\item
the operation $Y_-$ is uniquely determined by the condition
that the injection $L \inj P$ is a morphism of dg vertex Lie algebras
and by the derivation condition.
\end{itemize}

The following statement can be regarded as a vertex analogue of 
Remark \ref{rmk:sp:mu} \eqref{i:sp:mu:3} on the momentum map.
We omit the proof.

\begin{lem*}%\label{lem:li:mu}
Let $L$ be a dg vertex Lie algebra, $P$ be a dg vertex Poisson algebra,
and $\mu: L \to P$ be a morphism of dg vertex Lie algebras.
Then we have a unique morphism $\Sym(L) \to P$ of dg vertex Poisson algebras
such that the restriction to $L$ coincides with $\mu$.
\end{lem*}

We have a twisted version 
which is similar as \S \ref{ss:pr:BRST} \eqref{i:pr:cCl}.
Let $L$ be a dg vertex Lie algebra, and assume that 
we are given a one-dimensional central extension $L^\flat$.
In other words, we have a short exact sequence 
$0 \to \bbk \to L^\flat \to L \to 0$ of dg vertex Lie algebras,
where $\bbk$ is the one-dimensional trivial vertex Lie algebra.
Regarding $L^\flat$ as a complex, we define
\[ 
 \Sym^\flat(L) := \Sym(L^\flat)/I
\]
{}to be the quotient of the symmetric algebra $\Sym(L^\flat)$
by the ideal $I$ generated by the difference of embeddings  
$\bbk = \Sym(L^\flat)^0 \inj \Sym(L^\flat)$ and
$\bbk \inj \frkl^\flat=\Sym(L^\flat)^1 \inj \Sym(L^\flat)$.
It inherits the dg vertex Poisson structure from $\Sym(L^\flat)$.

%We have the same three constructions of dg vertex Poisson algebras 
%as those of vertex Poisson superalgebras in \S \ref{sss:li:3vpsa}.
For later reference, we name the obtained dg vertex Lie algebras as:

\begin{dfn}\label{dfn:li:sym-vp}
Let $L$ be a dg vertex Lie algebra.
\begin{enumerate}[nosep]
\item 
W call the dg vertex Poisson algebra $\Sym(L)$ the \emph{symmetric dg vpa}.

\item
Assume that we are given a one-dimensional central extension $L^\flat$ of $L$.
We call the dg vertex Poisson algebra $\Sym^\flat(L)$ 
the \emph{twisted symmetric dg vpa}.
\end{enumerate}
\end{dfn}

%%%%%%%%%%%%%%%%%%%%%%%%%%%%%%%%%%%%%%%%%%%%%%%%%%%%%%%%%%%%%%%%%%%%%%%%%%%%%%%%
\subsubsection{Quasi-classical limit construction}
\label{sss:li:qc}

Next we turn to the second construction by
a quasi-classical limit of a vertex superalgebra.
See \cite[16.2.3--7]{FBZ} for the detail.
Here we only explain a particular case.

Let $U$ be a complex, and consider the free fermionic vertex algebra
$\Wedge^{\sinf}(U)$ in Definition \ref{dfn:li:ff}.
We use the linear basis $\{\psi_i \mid i \in I\}$ of $U$ and the dual basis
$\{\psi^*_j \mid j \in I\}$ of $U^*$ therein. 
In particular, $\Wedge^{\sinf}(U)$ has a basis consisting of 
PBW monomials \eqref{eq:li:ff}.

We define $\Wedge_{\hbar}$ to be the $\bbk[\hbar^{-1}]$-lattice
spanned by the rescaled monomials
\[
 \hbar^{-k} (\psi_{i_1} t^{m_1}) \cdots (\psi_{i_k} t^{m_k}) 
 (\psi_{j_1}^* t^{n_1-1}d t) \cdots (\psi_{j_l}^* t^{n_l-1}d t) \vac \quad
 (m_1< \cdots < m_k <0, \, n_1< \cdots < n_l \le 0).
\]
In other words, 
we shift $\psi_i \mapsto \hbar^{-1}\psi_i$ and preserve $\psi_j^*$.
Then we set 
\[
 \ol{\Wedge}^{\sinf}(U) :=  \Wedge_{\hbar}/\hbar^{-1}\Wedge_{\hbar}.
\]
The image of the generators are denoted by 
$\ol{\psi}_i t^m$ and $\ol{\psi}^*_j t^{n-1}d t$.
It inherits the left dg $\Cl$-module structure of $\Wedge^{\sinf}(U)$,
and has an action of the dg subalgebra of $\Cl[\hbar^{-1}]$ generated by  
\[
 \ol{\psi_i} t^m := \hbar \psi_i t^m, \quad 
 \ol{\psi}^*_j t^{n-1} d t := \psi^*_j t^{n-1} d t.
\] 
The commutation relations of these generators are 
\[
 [\ol{\psi}_i t^m ,\ol{\psi}^*_j t^{n-1} d t]=\delta_{i,j}\delta_{m+n,0}, \quad 
 [\ol{\psi}_i t^m, \ol{\psi}_j t^n] = 0 = 
 [\ol{\psi}^*_i t^{m-1} d t, \ol{\psi}^*_j t^{n-1} d t].
\]
%$\ol{\Wedge}^{\sinf}(U)$ also inherits the dg structure of 
%$\Wedge^{\sinf+\bl}(U)$, 
%We denote $\ol{\Wedge}^{\sinf+\bl}(U)$ to emphasize the charge grading.
Thus the complex $\ol{\Wedge}^{\sinf}(U)$ has a linear basis consisting 
of the elements of the form 
\[
 (\ol{\psi}_{i_1} t^{m_1})        \cdots (\ol{\psi}_{i_k} t^{m_k}) 
 (\ol{\psi}_{j_1}^* t^{n_1-1}d t) \cdots (\ol{\psi}_{j_l}^* t^{n_l-1}d t) \vac
 \quad (m_1< \cdots < m_k <0, \, n_1< \cdots < n_l \le 0).
\]

The complex $\ol{\Wedge}^{\sinf}(U)$ has a structure of dg vertex Poisson algebra
induced by the dg vertex algebra structure on $\Wedge^{\sinf}(U)$.
As a commutative dg algebra, it is isomorphic to 
$\Sym\bigl(U((t))[1] \oplus U^*((t))[-1] d t \bigr)$,
and the differential is given by $\partial_t$.
The vertex Lie algebra structure $Y_-$ is determined by
\begin{align*}
&Y_-(\ol{\psi}_i t^{-1} \vac,z) = 
 \ol{\psi}_i(z) := \sum_{l \in \bbN}(\ol{\psi}_i t^l)z^{-l-1}, \quad
 Y_-(\ol{\psi}^*_j t^{-1} d t \vac,z) = 
 \ol{\psi}_j^*(z) := \sum_{l \in \bbN}(\ol{\psi}^*_j t^{l})z^{-l-1}, \\
&Y_-(\ol{\psi}_i t^{m} \vac,z) = 
 \tfrac{1}{(-m-1)!}\partial_z^{-m-1} \ol{\psi}_i(z), \quad
 Y_-(\ol{\psi}^*_j t^{n-1}d t \vac,z) = 
 \tfrac{1}{(-n)!}\partial_z^{-n} \ol{\psi}^*_i(z).
\end{align*}
It also inherits the charge grading $\Wedge^{\sinf+\bl}(U)$
given by $\chg(\ol{\psi_i} t^{m})=-1$, $\chg(\ol{\psi_i} t^{n-1} d t)=1$.
We denote $\ol{\Wedge}^{\sinf}(U)$ to emphasize this charge grading.

\begin{dfn}\label{dfn:li:ffp}
We call the obtained dg vertex Poisson algebra $\ol{\Wedge}^{\sinf}(U)$ 
the \emph{free fermionic vertex Poisson algebra}.
\end{dfn}

We can check that the associated dg Poisson algebra 
(Definition \ref{dfn:li:dgapa}) is the classical Clifford algebra $\cCl(U)$
(Definition \ref{dfn:pr:cCl}).

%%%%%%%%%%%%%%%%%%%%%%%%%%%%%%%%%%%%%%%%%%%%%%%%%%%%%%%%%%%%%%%%%%%%%%%%%%%%%%%%
\subsubsection{Vertex Poisson structure of arc space}
\label{sss:li:arc}

We cite from \cite[\S 2.3]{A12} the third construction of 
a dg vertex Poisson algebra via the arc space of a dg Poisson algebra.
Recall the arc space $J_\infty(R)$ of a commutative dg algebra $R$ 
(Definition \ref{dfn:jet:dgarc}).
%Assuming $R$ is of finite type, 
We denote by $T$ the $0$-derivation on $J_\infty(R)$.
%We use the same notation for a commutative superalgebra $R$,
%omitting the precise definitions in super setting.
%whose spectrum gives the arc space:
%$\Spec(J_\infty(R)) = J_\infty(\Spec(R))$ .

\begin{prp}\label{fct:li:0vpa}
For a dg Poisson algebra $R$ which is of finite type as a commutative algebra,
there is a unique dg vertex Poisson algebra structure on $J_\infty(R)$ such that
\begin{align}\label{eq:li:0vpa}
 u_{(n)}(T^l v) = 
 \begin{cases} 
  \frac{l!}{(l-n)!}T^{l-n}\{u,v\}_R & (l \ge n), \\ 
  0 & (l<n).
 \end{cases}
\end{align}
for $u,v \in R \subset J_\infty(R)$ and $l \in \bbN$.
It is called the \emph{level $0$ dg vertex Poisson algebra}.
\end{prp}

\begin{proof}
We follow the proof of the non-dg case given in \cite[Proposition 2.3.1]{A12}.
By the formula \eqref{eq:li:0vpa}, we have a well-defined morphism of complexes
\[
 R \longto \Der(J_\infty(R))[[z^{-1}]]z^{-1}, \quad 
 u \longmapsto u_-(z) := \tsum_{n \in \bbN} u_{(n)} z^{-n-1}.
\]
We obviously have $u_-(z) \in \Der(J_\infty(R))^{\abs{u}}[[z^{-1}]]z^{-1}$ 
for homogeneous $u \in R$, and can check the conditions 
\ref{i:li:vla:1}--\ref{i:li:vla:4} in Definition \ref{dfn:li:vla} of 
dg vertex Lie algebra structure (see Proof of Lemma \ref{lem:li:0vla}).
This morphism is extended to 
$Y_-(-,z): J_\infty(R) \to \Der(J_\infty(R))[[z^{-1}]]z^{-1}$ by 
\[
 Y_-(a,z)(T^l u) := \bigl(e^{z T}(-\partial_z)^l u_-(-z)a\bigr)_-
\]
for $a \in J_\infty(R)$, $u \in R$ and $l \in \bbN$.
By \cite[Proposition 2.3.1]{A12}, we know that it defines a vertex Poisson
superalgebra structure on the associated commutative superalgebra 
$J_\infty(R)^{\even} \oplus J_\infty(R)^{\odd}$.
Thus it remains to check the condition \ref{i:li:vla:4} 
in Definition \ref{dfn:li:vla} and the condition 
$Y_-(a,z) \in \Der(J_\infty(R))^{\abs{a}}[[z^{-1}]]z^{-1}$ 
for homogeneous $a \in J_\infty(R)$, which are obvious.
\end{proof}

Hereafter we regard $J_\infty(R)$ as the level $0$ vertex Poisson algebra
unless otherwise stated.
Recall the associated dg Poisson algebra $R^{\cois}_P$ 
for a dg vertex Poisson algebra $P$ (Definition \ref{dfn:li:dgapa}).
By the construction, we have:

\begin{lem}\label{lem:li:RJ=R}
For a commutative dg algebra $R$ of finite type, we have 
\[
 R^{\cois}_{J_\infty(R)} \simeq R.
\]
\end{lem}

\begin{eg}\label{eg:li:arc-sg}
Let $\frkl$ be a dg Lie algebra and consider the Kostant-Kirillov dg Poisson
algebra $R=\Sym(\frkl)$ (Notation \ref{ntn:sp:KK}).
Then we have the level $0$ vertex Poisson superalgebra $J_\infty(\Sym(\frkl))$.

By Remark \ref{rmk:jet:aff} of the description of $J_\infty(\Sym(\frkl))$, 
we have
\[
 J_{\infty}(\Sym(\frkl)) \simeq \Sym(\frkl[[t]])
\]
as commutative dg algebras,
where $\frkl[[t]]=\frkl \otimes \bbk[[t]]$ is the tensor product in $\dgVec$.
Let us denote an element of $\frkl[[t]]$ by $x f := x \otimes f$
with $x \in \frkl$ and $f=f(t) \in \bbk[[t]]$.
Then the level $0$ dg vertex Poisson structure is given by 
\[
 (x f)_{(n)}(y g)= \delta_{n,0} [x,y]_{\frkl} \otimes (f g).
\]
This is nothing but the standard Lie algebra structure on $\frkl[[t]]$.
We can also check that the vertex Lie algebra structure on the restriction
$\frkl[[t]] = \Sym(\frkl[[t]])^1 \inj J_\infty(\Sym(\frkl))$ 
coincides with the level $0$ vertex Lie algebra $J_\infty(\frkl)$ 
in Lemma \ref{lem:li:0vla}
(compare \eqref{eq:li:vla:ex} and \eqref{eq:li:0vpa}).
Thus we have an isomorphism of dg vertex Poisson algebras  
\[
 J_\infty(\Sym(\frkl)) \simeq \Sym(J_\infty(\frkl)),
\]
where the right hand side denotes the symmetric vertex Poisson algebra
(Definition \ref{dfn:li:sym-vp}).
%In total, we identify the level $0$ dg vpa $J_{\infty}(\Sym(\frkl))$ with
%the Kirillov-Kostant Poisson structure for the dg Lie algebra $\frkl[[t]]$.
\end{eg}

%We have the following twisted analogue of $J_{\infty}(\Sym(\frkl))$:
%\begin{eg}\label{eg:li:Sb} 
%In the beginning of \S \ref{ss:pr:BRST}, 
%we introduced the twisted symmetric algebra $\Sym^\flat(\frkl)$ for 
%a dg Lie algebra $\frkl$ and a one-dimensional central extension $\frkl^\flat$.
%It is a dg Poisson algebra inheriting the Poisson bracket from 
%the Kirillov-Kostant bracket on $\Sym(\frkl^\flat)$.
%Thus $\Sym(\frkl^\flat)$ is a dg Poisson algebra,
%and we have the level $0$ dg vertex Poisson algebra 
%\[ 
% J_{\infty}(\Sym^\flat(\frkl)).
%\]
%By the same argument as in Example \ref{eg:li:arc-sg}, we have 
%$J_{\infty}(\Sym^\flat(\frkl)) \simeq \Sym^{\flat}(J_\infty(\frkl))$.
%\end{eg}

Recall the twisted symmetric vertex Poisson algebra $\Sym^\flat(L)$
for a dg vertex Poisson algebra (Definition \ref{dfn:li:sym-vp}).
Using this construction, we can recover the free fermion vertex Poisson algebra
$\ol{\Wedge}^{\sinf}(U)$ in Definition \ref{dfn:li:ffp}.
The present construction is in fact a coordinate-dependent version of 
the \emph{Clifford coisson algebra} in \cite[1.4.21]{BD}.
%introduce a vertex Poisson analogue of
%the classical Clifford algebra (Definition \ref{dfn:pr:cCl}).

Let $U$ be a complex.
We denote by $U^*=\uHom(U,\bbk)$ the dual complex, and by 
$\pair{\cdot,\cdot}: U^* \otimes U \to \bbk$ the natural pairing.
We also denote $J_\infty(U) := U[[t]] $ and consider it as a 
dg vertex Lie algebra attached to the trivial dg Lie algebra $U$.
(Lemma \ref{lem:li:0vla}).
Similarly we have a dg vertex Lie algebra 
$J_\infty(U)^* = U^*[[t^{-1}]]$,
and the direct sum $J_\infty(U)[1] \oplus J_\infty(U)^*[-1]$.
By the pairing $\pair{\cdot,\cdot}$, 
we have a one-dimensional central extension of the direct sum, 
and thus we have the twisted symmetric dg vertex Poisson algebra
$\Sym^\flat\bigl( J_\infty(U)[1] \oplus J_\infty(U)^*[-1] \bigr)$.

\begin{dfn}\label{dfn:li:coCl}
For a complex $U$, we denote the above dg vertex Poisson algebra by
\[
 \coCl(J_\infty(U)) := 
 \Sym^\flat\bigl(J_\infty(U)[1] \oplus J_\infty(U)^*[-1] \bigr)
\]
and call it the \emph{Clifford vertex Poisson algebra}.
\end{dfn}

By comparing the description in \S \ref{sss:li:qc}, we have 
\[
 \coCl(J_\infty(U)) \simeq \ol{\Wedge}^{\sinf}(U)
\]
as dg vertex Poisson algebras.
In particular, we have:

\begin{cor}\label{cor:li:coCl}
The associated dg Poisson algebra (Definition \ref{dfn:li:dgapa})
of $\coCl(J_\infty(U))$ 
coincides with the classical Clifford algebra $\cCl(U)$:
\[
 R^{\cois}_{\coCl(J_\infty(U))} \simeq \cCl(U)
\]
\end{cor}

Using $\coCl(J_\infty(U))$ we will introduce \emph{coisson BRST reduction} 
in the following \S \ref{s:co}.

\begin{rmk}\label{rmk:li:coCl}
In some literature such as \cite[\S 16.7]{FBZ}, what we call coisson BRST 
reduction is called the classical BRST reduction.
In that terminology $\coCl(J_\infty(U))$ should  be called 
the \emph{classical} Clifford algebra,
conflicting our terminology (Definition \ref{dfn:pr:cCl}).
\end{rmk}

%%%%%%%%%%%%%%%%%%%%%%%%%%%%%%%%%%%%%%%%%%%%%%%%%%%%%%%%%%%%%%%%%%%%%%%%%%%%%%%%
%%%%%%%%%%%%%%%%%%%%%%%%%%%%%%%%%%%%%%%%%%%%%%%%%%%%%%%%%%%%%%%%%%%%%%%%%%%%%%%%
\subsection{Li filtration}
\label{ss:li:li}

In this subsection we introduce the \emph{Li filtration} for a dg vertex algebra.
Main references are \cite{Li} and \cite{A15}.
%We give super versions in order to introduce dg analogues later in \S \ref{ss:li:dg}.

%%%%%%%%%%%%%%%%%%%%%%%%%%%%%%%%%%%%%%%%%%%%%%%%%%%%%%%%%%%%%%%%%%%%%%%%%%%%%%%%
\subsubsection{Definition}
\label{sss:li:li}

%Let us now recall the \emph{Li filtration of a vertex algebra} \cite{Li}.
In the following we use:

\begin{ntn*}
For a linear space $V$ and a subset $S \subset V$,
we denoted by $\lsp{S}$ the linear subspace spanned by the elements in $S$.
In the case $S=\{s_i \mid i \in I\}$, 
we also denote it by $\lsp{s_i \mid i \in I}$.
\end{ntn*}

For a dg vertex algebra $V$, $a \in V$, and a dg $V$-module $M$,
we denote $Y_M(a,z)=\sum_{n \in \bbZ} a_{(n)}^M z^{-n-1}$ as before.
Then we define a linear subspace $F^p M \subset M$ by 
\[
 F^p M := \lsp{(a^M_1)_{(-n_1-1)} \cdots (a^M_r)_{(-n_r-1)} m \mid 
                m \in M, \ r \in \bbZ_{>0}, \ a_i \in V, \ n_i \in \bbN, \ 
                \tsum_{i=1}^r n_i \ge p}.
\]
Then we have a decreasing filtration 
$M = F^0 M \supset F^1 M \supset  F^2 M \supset \cdots$
of linear spaces.
By induction using $a_{(n)}M^i \subset M^{i+j}$ and 
$d_M(a_{(n)}^M m)=(d_V a)_{(n)}m+(-1)^{\abs{a}} a_{(n)}(d_M m)$ 
in Definition \ref{dfn:li:dgMod} of dg $V$-module, 
we can check that the differential $d_M$ preserves the Li filtration: 
$d_M(F^p M) \subset F^p M$ for any $p \in \bbZ$.
Thus $F^{\bl} M$ is in fact a filtration of complexes.

\begin{dfn}[{\cite[Definition 2.7, Lemma 2.8]{Li}}]\label{dfn:li:li}
For a dg vertex algebra $V$ and a dg $V$-module $M$,
we call the decreasing filtration of complexes 
\[
 M = F^0 M \supset F^1 M \supset  F^2 M \supset \cdots
\]
the \emph{Li filtration} of $M$.
In particular, we can take $M:=V$ and have the Li filtration of $V$:
\[
 V = F^0 V\supset F^1 V \supset  F^2 V \supset \cdots.
\]
\end{dfn}

Below we set $F^{p}M := M$ for $p \in \bbZ_{<0}$.

\begin{lem}[{\cite[Lemma 2.9, Proposition 2.11]{Li}}]\label{lem:li:li}
Let $V$ and $M$ be as in Definition \ref{dfn:li:li}.
\begin{enumerate}[nosep]
\item
\label{i:lem:li:li:1}
For any $p \in \bbZ_{\ge 1}$, we have 
\[
 F^p M = \lsp{a_{(-i-1)}m \mid 
 a \in V, i \in \bbZ, 1 \le i \le p, m \in F^{p-i}M}.
\]
In particular, taking $M:=V$, we have 
\[
 F^1 V = C_2(V):=\lsp{a_{(-2)}b \mid a,b \in V}.
\]

\item 
$a_{(n)}F^q M \subset F^{p+q-n-1}M$ 
for any $p,q \in \bbZ$, $a \in F^p V$ and $n \in \bbZ$.

\item
$a_{(n)}F^q M \subset F^{p+q-n}M$ 
for any $p,q \in \bbZ$, $a \in F^p V$ and $n \in \bbN$.

\item
$T(F^p V) \subset F^{p+1} V$ for any $p \in \bbZ$.
\end{enumerate}
\end{lem}

By these properties of the Li filtration, we have a canonical construction of
vertex Poisson algebra from any vertex algebra.
We follow \cite{A15} for the notation.

\begin{fct}[{\cite[Theorem 2.12]{Li}}]\label{fct:li:li}
Let $V=(V,d,\vac,T,Y)$ be a dg vertex algebra.
\begin{enumerate}[nosep]
\item \label{i:li:cdga}
The associated graded space
\[
 \gr^F V := \tboplus_{p \in \bbN} F^p V/F^{p+1}V
\]
is a commutative dg algebra equipped with an additional $0$-derivation $\delta$
given by
\[
 \sigma_p(a) \cdot \sigma_q(b) := \sigma_{p+q}(a_{(-1)}b), \quad 
 \delta \sigma_p(a) := \sigma_{p+1}(a_{(-2)}\vac) = \sigma_{p+1}(T(a))
\]
Here $\sigma_p: F^p V \surj F^p V/F^{p+1}V$ denotes the projection.

\item
$\gr^F V$ has a structure of dg vertex Poisson algebra 
whose commutative dg vertex algebra  structure is given by \eqref{i:li:cdga}
and Lemma \ref{lem:vpa:cvs=cds},
and whose dg vertex Lie algebra structure is given by
\[
 \sigma_p(a)_{(n)}\sigma_q(b) := \sigma_{p+q-n}(a_{(n)}b) \quad (n \in \bbN).
\]
Here we understand $\sigma_r(a)=0$ for $r<0$.
\end{enumerate}
\end{fct}

Precisely speaking, the statement in \cite[Theorem 2.12]{Li}
is for the non-dg case, but the same proof works in the dg case.

For a morphism $\phi: V \to W$ of dg vertex algebras, 
we have $\phi(F^p V) \subset F^p W$,
which induces a morphism  $\gr^F V \to \gr^F W$ of dg vertex Poisson algebras.
Thus we obtain:

\begin{lem}\label{lem:li:grdgv}
Taking the associated graded space of the Li filtration, 
we have a functor 
\[
 \gr^F: \dgVA \longto \dgVP.
\]
%from the category of dg vertex algebras to 
%the category of dg vertex Poisson algebras.
It is a monoidal functor of the symmetric monoidal structures.
\end{lem}

As a corollary of Fact \ref{fct:li:li}, 
we have a dg Poisson algebra on the quotient $F^0 V/F^1 V$.
Following the terminology in \cite{A12, A18}, we give:

\begin{dfn}\label{dfn:li:C2}
Let $V$ be a dg vertex algebra.
\begin{enumerate}[nosep]
\item
The complex 
\[
 R_V := F^0 V /F^1 V = V/C_2(V) \subset \gr^F V
\]
has a structure of dg Poisson algebra 
whose multiplication $\cdot$ and Poisson bracket $\{-,-\}$ are given by
\[
 \ol{a} \cdot \ol{b} := \ol{a_{(-1)}b}, \quad
 \{\ol{a},\ol{b}\} := \ol{a_{(0)}b}.
\]
Here $\ol{a}:= \sigma_0(a)$ denotes the image of $a \in V$ in $R_V$.
The resulting dg Poisson algebra $R_V$ is called 
\emph{Zhu's $C_2$-algebra} of $V$.

\item
We denote the functor induced by the construction $V \mapsto R_V$ by 
\[
 R_{(-)}: \dgVP \longto \dgpa,
\]
where $\dgpa$ denotes the category of dg Poisson algebras.

\item
If $R_V$ is concentrated in negative cohomological degrees,
then we denote the corresponding affine derived Poisson scheme by
\[
 X_V := \Spec(R_V)
\]
and call it the \emph{associated derived scheme} of $V$.
\end{enumerate}
\end{dfn}

If $V$ is a plain vertex algebra, i.e., concentrated in cohomological degree $0$, 
then $R_V$ is a Poisson algebra, 
and we call the corresponding affine Poisson scheme $X_V$
the \emph{associated scheme}, which recovers the terminology in \cite{A15, A18}.

\begin{rmk}\label{rmk:li:C2}
\begin{enumerate}[nosep]
\item 
In the non-dg case, the Poisson algebra $R_V$ was introduced by Y.~Zhu to
give a nice finiteness condition on vertex operator algebras \cite[\S 4.4]{Z}.

\item
We can confirm the Poisson structure directly by using 
Remark \ref{rmk:li:vsa} \eqref{i:rmk:li:vsa:1}.

\item\label{i:li:RV=RP}
We can easily find that Zhu's $C_2$-algebra $R_V$ coincides with 
the associated dg Poisson algebra of $\gr^F V$ (Definition \ref{dfn:li:dgapa}):
\[
 R_V \simeq R^{\cois}_{\gr^F V}.
\]
\end{enumerate}
\end{rmk}

\begin{eg}[{\cite[Lemma 4.5]{A15}}]
Let us consider the free fermionic vertex algebra $\Wedge^{\sinf}(U)$ 
(Definition \ref{dfn:li:ff}) for a complex $U$.
Using the commutation relation \eqref{eq:li:ff-com} and 
the linear basis \eqref{eq:li:ff}, we can check
\[
 \gr^F \Wedge^{\sinf}(U) \simeq \ol{\Wedge}^{\sinf}(U)
\]
as dg vertex Poisson algebras, where the right hand side denotes 
the free fermionic vertex Poisson algebra (Definition \ref{dfn:li:ffp}).
In particular, on Zhu's $C_2$-algebras we have 
\[
 R_{\wedge^{\infty/2}(U)} \simeq \ol{\Cl}(U)
\]
as dg Poisson algebras, where $\ol{\Cl}(U)$ denotes 
the classical Clifford algebra (Definition \ref{dfn:pr:cCl}).
\end{eg}

%We also have an analogous fact for vertex algebra modules.
%
%\begin{fct*}[{\cite[Proposition 2.13]{Li}}]
%Let $V$ be a dg vertex algebra and $M$ be a dg $V$-module.
%Then the associated graded space 
%$\gr^F M = \bigoplus_{p \in \bbN} F^p M/F^{p+1}M$ to the Li filtration 
%$F^{\bl} M$ is a dg module over the dg vertex Poisson algebra $\gr^F V$.
%\end{fct*}

For later use, let us recall:

\begin{dfn}\label{dfn:li:sep}
A vsa $V$ is \emph{separated} if $\bigcap_{n \in \bbN} F^n V = 0$.
\end{dfn}

By \cite[Proposition 3.9]{Li}, 
if a vertex algebra  $V$ has a lower truncated $\bbZ$-gradation
(i.e., there is some $n \in \bbZ$ such that $V$ is $\bbZ_{\ge n}$-graded),
then it is separated.

\subsubsection{The case of universal affine vertex algebra}
%\label{sss:li:uavp}

Let us explain the notions introduced so far in the case of 
the universal affine vertex algebra (\S \ref{sss:li:uava}).

Recall Notation \ref{ntn:li:g} and Definition \ref{dfn:li:uava}.
In particular 
\begin{itemize}[nosep]
\item 
$\frkg$ is the Lie algebra 
of the semi-simple algebraic group over $\bbC$, and 

\item
$\wh{\frkg}=\frkg((t)) \oplus \bbC K$ is the the derived algebra of 
the non-twisted affine Lie algebra associated to $\frkg$.

\item
$V_k(\frkg)$ is the universal affine vertex algebra at level $k \in \bbC$.
\end{itemize}

Let $\{x_i \mid i=1,\ldots,\dim \frkg\}$ be a linear basis of $\frkg$.
Then we have a PBW basis of $V_{k}(\frkg)$ (Fact \ref{fct:uava:pbw}):
It consists of monomials of the form
$(x_{i_1}t^{n_1})\cdots (x_{i_l}t^{n_l})\vac$,
where $n_1 \le \cdots \le n_l <0$ and if $n_j=n_{j+1}$ then $i_j \le i_{j+1}$.

Recall Definition \ref{dfn:li:li} of the Li filtration $F^{\bl} V_k(\frkg)$.
By the remark after Definition \ref{dfn:li:sep},
we find that $F^{\bl} V_k(\frkg)$ is separated.
Also recall Zhu's $C_2$-algebra $R_{V_k(\frkg)} := V_k(\frkg)/F^1 V_k(\frkg)$
(Definition \ref{dfn:li:C2}).
By the PBW basis above we find that the set 
$\{x_i t^{n} \mid i=1,\ldots,\dim \frkg, n \in \bbZ_{<0}\}$
generates $R_{V_k(\frkg)}$ as a commutative algebra,
and that we have $F^1 V_k(\frkg)=\frkg[t^{-1}]t^{-2} V_k(\frkg)$.
Then we further find that there is an isomorphism
\begin{equation}\label{eq:li:SV=RV}
 \Sym(\frkg) \longsimto R_{V_k(\frkg)} = V_k(\frkg)/F^1 V_k(\frkg), \quad
 x \longmapsto \ol{(x t^{-1})\vac}
\end{equation}
of commutative algebras.
By checking the Poisson brackets on both sides, 
we have \eqref{i:li:SV=RV:1} in the following fact.

\begin{fct}[{\cite[Proposition 2.7.1]{A12}}]\label{fct:li:SV=RV}
\begin{enumerate}[nosep]
\item 
\label{i:li:SV=RV:1}
The map \eqref{eq:li:SV=RV} gives the following isomorphism of Poisson algebras.
\[
 \Sym(\frkg) \longsimto R_{V_k(\frkg)}.
\]

\item
\label{i:li:SV=RV:2}
The same map induces an isomorphism 
\[
 J_\infty(\Sym(\frkg)) \longsimto \gr^F V_k(\frkg)
\]
of vertex Poisson algebras,
where the left hand side denotes the level $0$ vertex algebra 
(Fact \ref{fct:li:0vpa})
associated to the Kirillov-Kostant Poisson algebra $\Sym(\frkg)$. 
\end{enumerate}
\end{fct}

In particular, both $R_{V_k(\frkg)}$ and $\gr^F V_k(\frkg)$ 
are independent of the level $k$.
See \cite[Proposition 2.7.1]{A12} for a proof of this fact.

Finally we consider the module category of $J_\infty(\Sym(\frkg))$.
By Proposition \ref{prp:li:vpmod} a vertex Poisson module 
is equivalent to a smooth Poisson module of the associated Poisson algebra.
By the construction in \S \ref{sss:li:vpm}, we can check:

\begin{lem*}
As Poisson algebras we have 
$\wt{U}(J_\infty(\Sym(\frkg))) \simeq \Sym(\frkg[[t]])$,
where the latter denotes the Kirillov-Kostant Poisson algebra
for the Lie algebra $\frkg[[t]]$.
\end{lem*}

Then by the definition of smooth Poisson modules 
(\S \ref{sss:li:vpm} \eqref{i:li:vpm:sm}),
we can restate Proposition \ref{prp:li:vpmod} as:

\begin{prp}\label{prp:li:Jgm}
A vertex Poisson module over $J_\infty(\Sym(\frkg))$ is equivalent to 
a smooth representation of the Lie algebra $\frkg[[t]]$,
i.e., a representation $M$ such that $(x t^n).m=0$ for any $x \in \frkg$, 
$m \in M$ and $n \gg 0$.
\end{prp}

%Since $\frkg$ is finite-dimensional. we can further restate it as
%a locally finite representation of the Lie algebra $\frkg[[t]]$.

%%%%%%%%%%%%%%%%%%%%%%%%%%%%%%%%%%%%%%%%%%%%%%%%%%%%%%%%%%%%%%%%%%%%%%%%%%%%%%%%
%%%%%%%%%%%%%%%%%%%%%%%%%%%%%%%%%%%%%%%%%%%%%%%%%%%%%%%%%%%%%%%%%%%%%%%%%%%%%%%%
%%%%%%%%%%%%%%%%%%%%%%%%%%%%%%%%%%%%%%%%%%%%%%%%%%%%%%%%%%%%%%%%%%%%%%%%%%%%%%%%
\section{Coisson BRST reduction and gluing procedure for arc spaces}
\label{s:co}

In this section we introduce the BRST reduction for dg vertex Poisson algebras
and a dg vertex Poisson analogue of the category $\MT$.

%%%%%%%%%%%%%%%%%%%%%%%%%%%%%%%%%%%%%%%%%%%%%%%%%%%%%%%%%%%%%%%%%%%%%%%%%%%%%%%%
%%%%%%%%%%%%%%%%%%%%%%%%%%%%%%%%%%%%%%%%%%%%%%%%%%%%%%%%%%%%%%%%%%%%%%%%%%%%%%%%
\subsection{Coisson BRST complex}
\label{ss:co:coBRST}

In this subsection we introduce the BRST complex for dg vertex Poisson algebras.
The construction is a natural analogue of 
classical BRST complex in \S \ref{ss:pr:BRST}.
We work over a field $\bbk$ of characteristic $0$ in this subsection.

\begin{rmk}\label{rmk:co:star}
Our argument is in fact a coordinate-dependent version of the arguments 
in \cite[1.4.21--26]{BD}, 
where the BRST complex is introduced for \emph{coisson algebras}.
Let us give a brief comment on the theory of coisson algebras.
See \cite[0.15, 2.6, 3.8.6]{BD} and \cite{FG} for the detail,
and also \cite[Chap. 19]{FBZ} for a related exposition.

Let $X$ be a smooth algebraic variety.
A \emph{right D-module $\shM$ on the Ran space $\Ran(X)$ of $X$} is 
a family of right D-modules $\shM_{X^I} \in \DMod(X^I)$ for finite sets $I$
satisfying a compatibility condition for every surjection $I \surj J$.
We have a natural fully faithful embedding $\iota: \DMod(X) \to \DMod(\Ran(X))$
of the category of right D-modules on $X$ into that on $\Ran(X)$.
On the category $\DMod(\Ran(X))$, we can also construct two structures of
symmetric monoidal categories denoted by $\otimes^!$ and $\otimes^\star$.
The monoidal structure $\otimes^!$ is equivalent to the standard one 
on left D-modules on $X$, and the other structure $\otimes^\star$ is designed to
reflect operator product expansion nicely.
A Lie algebra object in the monoidal category $\DMod(\Ran(X))^{\otimes^\star}$
which lies in the essential image of $\iota$ 
is called a \emph{$\star$-Lie algebra on $X$}.
Similarly, we have the notion of $!$-commutative algebra on $X$.
Then a \emph{coisson algebra over $X$} is a ``compound Poisson algebra object" 
in $\DMod(\Ran(X))$, i.e., 
a combination of $!$-commutative ring structure and $\star$-Lie algebra structure.

A vertex Lie algebra is equivalent to a $\star$-Lie algebra on $X=\bbA^1$
which is equivariant under the action of affine transformations of $\bbA^1$.
Similarly, a vertex Poisson algebra is equivalent to an equivariant 
coisson algebra on $\bbA^1$.
Below we introduce the BRST complex for vertex Poisson algebras
by replacing ``Lie algebra" in the argument on classical BRST complex
with ``$\star$-Lie algebra", or ``vertex Lie algebra".
\end{rmk}

Recall that for a dg Lie algebra $\frkl$,
we have a contractible dg Lie algebra $\frkl_\dagger$ (\S \ref{sss:dga:CE}).
By the same way, for a dg vertex Lie algebra $L$, we have a dg vertex Lie
algebra structure on the contractible complex $\Cone(\id_L)=L\oplus L[1]$,
which will be denoted by $L_{\dagger}$.

Let $\frkl$ be a dg Lie algebra, and $L := J_\infty(\frkl)$ be 
the level $0$ vertex Lie algebra (Lemma \ref{lem:li:0vla}).
As a complex we have $L \simeq \frkl[[t]]$.
Recall also the Clifford vertex Poisson algebra
$\coCl(J_\infty(\frkl))= \Sym^\flat\bigl(L[1] \oplus L^*[-1]\bigr)$
in Definition \ref{dfn:li:coCl},

Using these notations, we have the following vertex Poisson analogue
of Lemma \ref{lem:pr:cBRST}.
We omit the proof.
%See \cite[1.4.23--25]{BD} for the detail.

\begin{lem*}
%$\Sym(L)$ be the symmetric 
%vertex Poisson algebra (Definition \ref{dfn:li:sym-vp}).
Let $\frkl$ be a dg Lie algebra, $L := J_\infty(\frkl)$,
$P$ be a dg vertex Poisson algebra,
and $\mu_{\cois}: L \to P$ be a morphism of dg vertex Lie algebras.
%(Remark \ref{rmk:sp:mu} \eqref{i:sp:mu:3}).
\begin{enumerate}[nosep]
\item
The adjoint action of $L:=J_\infty(\frkl)=\frkl[[t]]$ on itself as a dg Lie algebra
yields a morphism $\beta: L \to \coCl(L)$ of dg vertex Lie algebras 
as a composition 
\[
 \beta: L \longto L \otimes L^* \longsimto L[1] \otimes L^*[-1] 
 \longinj \coCl(L)^0.
\]

\item
Let
\[
 \ell: L_{\dagger} \longto \coCl(L) \otimes P
\]
be the morphism of complex given by
\[
 \ell^{0}:=1 \otimes \mu_{\cois} + \beta \otimes 1: 
 L \longto \coCl(L)^0 \otimes P, \quad
 \ell^{-1}: L[1] \longinj \coCl(L)^{-1} \longto \coCl(L) \otimes P.
\]
Then $\ell$ is a morphism of dg vertex Lie algebras.

\item
We define the following elements.
\begin{itemize}[nosep]
\item
$\wt{\mu}_{\cois} \in L^* \otimes P \subset (\coCl(L) \otimes P)^1$ 
is the element corresponding to $\mu_{\cois}$.
\item 
$\beta'  \in L^* \otimes \coCl(L)^0$  is the element corresponding to $\beta$.
\item
$\beta'' \in \coCl(L)^1$ is the image of $\beta'$
by the product map $L^*[-1] \otimes \coCl(L) \to \coCl(L)$.
\item
$\wt{\beta} \in (\coCl(L) \otimes P)^1$ is the image of $\beta''$
by $\coCl(L) \to \coCl(L) \otimes P$.
\end{itemize}
Then the \emph{coisson BRST charge} 
\[
 Q_{\cois} := \wt{\mu}_{\cois}+\tfrac{1}{2}\wt{\beta} \in (\coCl(L) \otimes P)^1
\]
satisfies ${(Q_{\cois})_{(0)}}^2=0$,
where $Y_-(Q_{\cois},z) = \sum_{n \in \bbN}(Q_{\cois})_{(n)} z^{-n-1}$ 
denotes the vertex Lie algebra structure of $\coCl(L) \otimes P$.
\end{enumerate}
\end{lem*}

Now we can introduce:

\begin{dfn}\label{dfn:co:coBRST}
Let $\frkl$ be a dg Lie algebra, and $P$ be a dg vertex Poisson algebra.
\begin{enumerate}[nosep]
\item 
We call a morphism $\mu_{\cois}: J_\infty(\frkl) \to P$ of dg vertex Lie algebras
a \emph{coisson momentum map}.

\item
Given a coisson momentum map $\mu_{\cois}: J_\infty(\frkl) \to P$,
%As in Lemma \ref{lem:li:mu}, we can extend $\mu_{\cois}$ to a morphism of dg vertex 
%Poisson algebras $\Sym(L) \to P$, which will be denoted by the same symbol $\mu$.
we define the \emph{coisson BRST complex} to be the dg vertex Poisson algebra
\[
 \coBRST(J_\infty(\frkl),P,\mu_{\cois}) := 
 (\coCl(J_\infty(\frkl)) \otimes P, d_{\cois}),
\]
consisting of the followings.
\begin{itemize}[nosep]
\item 
$\coCl(J_\infty(\frkl)) \otimes P$ denotes the tensor product
as graded vertex Poisson algebra (forgetting the differential).
\item
The differential is given by 
$d_{\cois}:= (Q_{\cois})_{(0)}+d_{\coCl(J_\infty(\frkl)) \otimes P}$,
where the second term is the differential of the tensor product as complex.
\end{itemize}

\item
Given a coisson momentum map $\mu_{\cois}: J_\infty(\frkl) \to P$,
the cohomology of the coisson BRST complex 
$\coBRST(J_\infty(\frkl),P,\mu_{\cois})$ is denoted by
\[
 H^{\sinf+\bl}_{\cois}(J_\infty(\frkl),P,\mu_{\cois}) := 
 H^{\bl}\coBRST(J_\infty(\frkl),P,\mu_{\cois}),
\]
which is a graded Poisson algebra.
\end{enumerate}
\end{dfn}

\begin{rmk}\label{rmk:co:cmm}
A coisson momentum map $\mu_{\cois}$ in our sense 
is called a chiral momentum map in \cite{A18}.
\end{rmk}

The coisson BRST complex satisfies similar properties 
as in Lemma \ref{lem:pr:BRST2}.

Let $\frkl$ be a dg Lie algebra, $R$ be a dg Poisson algebra,
and $\mu: \frkl \to R$ be a momentum map.
Then we have the symmetric vertex Poisson algebra 
$J_\infty(\Sym(\frkl))=\Sym(J_\infty(\frkl))$ 
and the level $0$ vertex Poisson algebra $J_\infty(R)$,
and $\mu$ induces a coisson momentum map
\[
 J_\infty(\mu): J_\infty(\Sym(\frkl)) \longto J_\infty(R).
\]
Then we have:

\begin{lem*}
For $\frkl,R,\mu$ as above, 
the associated dg Poisson algebra (Definition \ref{dfn:li:dgapa}) 
of the coisson BRST complex $\coBRST(\frkl,J_\infty(R),J_\infty(\mu))$
is isomorphic to the classical BRST complex $\cBRST(\frkl,R,\mu)$.
\end{lem*}

%\begin{rmk*}
%By the chiral momentum map $\mu: L \to P$,
%we can regard $P$ as a vertex Poisson $\Sym(L)$-module.
%Then, since we have $\Sym(L) = \Sym(\frkl[[t]]) \simeq J_\infty(\Sym(\frkl))$
%and  the vertex Poisson structure on $J_\infty(\Sym(\frkl))$ 
%is equivalent to the Kirillov-Kostant Poisson structure on $\Sym(\frkl[[t]])$
%by Example \ref{eg:li:arc-sg},
%we can regard $P$ as a dg module over the dg Lie algebra $\frkl[[t]]$.
%Thus the coisson BRST complex can be regarded as the classical BRST complex
%for the dg Lie algebra $\frkl[[t]]$.
%\end{rmk*}
%
%Let $\mu: J_\infty(\Sym(\frkg)) \to P$ be a morphism of dg vpas as above.
%Since $\mu$ is a morphism of cdgas, we have the Koszul complex 
%$\Kos(\frkg[[t]],P,\mu)$ in Definition \ref{dfn:pr:dBRST} \eqref{i:pr:dBRST:Kos}.
%Thus we can have the underlying graded linear space 
%$\CE(\frkg[[t]],\Kos(\frkg[[t]],P,\mu))$ of the BRST complex
%in Definition \ref{dfn:pr:dBRST} \eqref{i:pr:dBRST}.
%
%We will see an equivalence of the derived Hamiltonian reduction
%$R \gq_{\mu}^{\bbL} J_\infty(\Sym(\frkg))$ and
%the coisson BRST complex $\BRST_{\cois}(J_\infty(\frkg),P,\mu)$ 
%under a certain condition in the next \S \ref{sss:co:G}.

%%%%%%%%%%%%%%%%%%%%%%%%%%%%%%%%%%%%%%%%%%%%%%%%%%%%%%%%%%%%%%%%%%%%%%%%%%%%%%%%
\subsection{Coisson gluing procedure}
\label{sss:co:G}

In this part we give an analogue of the discussion in \S \ref{ss:pr:G} for arc spaces.
We basically follow the argument in \cite[\S3]{A18}, but with a slight modification.
We work over $\bbC$ here.

As in \S \ref{ss:pr:G}, let $G$ be a simply connected semisimple algebraic group.
We denote by $\frkg:=\Lie(G)$ the Lie algebra of $G$.
We have the affine Poisson scheme $\frkg^*$ whose coordinate ring is $\Sym(\frkg)$ 
with the Kirillov-Kostant Poisson structure.

Let us consider the arc space 
\[
 J_\infty(\frkg^*) = \Spec\bigl(J_\infty(\Sym(\frkg))\bigr).
\]
Recall that the arc space $J_\infty(G)$ of $G$ is isomorphic to 
the proalgebraic group $G[[t]]$ (Lemma \ref{lem:jet:G}).
We denote its Lie algebra by $J_\infty(\frkg):=\frkg[[t]]$.
The coadjoint action of $G$ on $\frkg$ extends to an action of 
$J_\infty(G) \simeq G[[t]]$ on $J_\infty(\frkg^*)$.
Thus, using Notation \ref{ntn:pr:Gsch}, 
we have the category $\QCoh^{J_\infty(G)}(J_\infty(\frkg^*))$.

On the other hand, by Fact \ref{fct:li:0vpa}, 
we can regard $J_\infty(\Sym(\frkg))$ as a vertex Poisson algebra.
Recall Proposition \ref{prp:li:Jgm}:
The category $\VPMod{J_\infty(\Sym(\frkg))}$ of vertex Poisson modules
is equivalent to the category of smooth representations 
of the Lie algebra $J_{\infty}(\frkg)=\frkg[[t]]$.
Then, by the same argument as in \S \ref{ss:pr:G}, 
we have the equivalence of categories
\[
 \QCoh^{J_\infty(G)}(J_\infty(\frkg^*)) \simeq 
 \VPMod{J_\infty(\Sym(\frkg))}^{\lf},
\]
where $\VPMod{P}^{\lf}$ denotes the full subcategory of $\VPMod{P}$ spanned by 
those objects on which the adjoint action of $J_\infty(\frkg)$ is locally finite.

Now we have an arc space analogue of Lemma \ref{lem:pr:pao}:

\begin{lem*}%\label{lem:co:pao}
Let $G$ and $\frkg=\Lie(G)$ be as above.
A Poisson algebra object in the symmetric monoidal category 
$\QCoh^{J_{\infty}(G)}(J_{\infty}(\frkg^*))$ 
is equivalent to a vertex Poisson algebra $P$ equipped with a morphism 
$\mu_{\cois}: \Sym(\frkg) \to P$ of vertex Poisson algebras
under which the adjoint action of $J_{\infty}(\frkg)=\frkg[[t]]$ is locally finite.
%The same statement holds for a Poisson algebra object 
%in the $\infty$-category $\eLQCoh{J_{\infty}(G)}(J_{\infty}(\frkg^*))$.
\end{lem*}

In view of this lemma, we set:

\begin{dfn}\label{dfn:co:pao}
A \emph{Poisson algebra object in $\eLQCoh{J_{\infty}(G)}(J_{\infty}(\frkg^*))$}
is a dg Poisson vertex algebra $P$ equipped with a coisson momentum map
$\mu_P: J_{\infty}(\Sym(\frkg)) \to P$
under which the adjoint action of $J_{\infty}(\frkg)=\frkg[[t]]$ is locally finite.
We denote such an object by $(P,\mu_P)$.
\end{dfn}

\begin{rmk*}
A genuine definition should be given in terms of 
``homotopy vertex Poisson algebra", which would be the combination of 
``homotopy vertex Lie algebra" structure and cdga structure.
We will come back to this point in a future work.
\end{rmk*}

We are now interested in a vertex Poisson analogue of 
the category $\MT$ (Definition \ref{dfn:pr:MT}).
It is enough to consider the composition of morphisms.
Following Proposition \ref{prp:pr:glue} of the relation between
the derived Hamiltonian reduction and the classical BRST complex, 
we will define the \emph{coisson gluing of vertex Poisson algebras}
by the coisson BRST complex.
For that, some preparations are in order:
\begin{itemize}[nosep]
\item 
Let $\frkg_1$ and $\frkg_2$ be the Lie algebras of 
semisimple algebraic groups $G_1$ and $G_2$ respectively.
For a Poisson algebra object $(P,\mu_P)$ in 
$\eLQCoh{J_{\infty}(G_1 \times G_2)}(J_{\infty}(\frkg_1^* \times \frkg_2^*))$
we define 
\[
 \mu_P^i: J_\infty(\Sym(\frkg_i)) \longto P \quad (i=1,2)
\]
by $\mu_P^1(x)=\mu_P(x \otimes 1)$ for $x \in \frkg_1$ and
$\mu_P^2(y)=\mu_P(1 \otimes y)$ for $y \in \frkg_2$. 
Here we used
$J_\infty(\Sym(\frkg_1 \oplus \frkg_2)) \simeq 
 J_\infty(\Sym(\frkg_1)) \otimes J_\infty(\Sym(\frkg_2))$.

\item
For a dg vertex Poisson algebra $P=(P,d,\vac,T,Y_+,Y_-)$, 
we denote by $P^{\op}$ the \emph{opposite} dg vertex Poisson algebra of $P$.
Explicitly, it is given by
\[
 P^{\op} := (P,d,\vac,-T,Y_+^{\op},Y_-^{\op})
\]
with  $Y_{\pm}^{\op}(a,z) := Y_{\pm}(a,-z)$.
\end{itemize}

\begin{dfn}\label{dfn:co:glue}
Let $\frkg_1$, $\frkg_2$ and $\frkg_3$ be the Lie algebras of 
semisimple algebraic groups $G_1$, $G_2$ and $G_3$ respectively.
Let $(P,\mu_P)$ and $(P',\mu_{P'})$ be Poisson algebra objects
in $\eLQCoh{G}(\frkg_1^* \times \frkg^*_2)$ and 
in $\eLQCoh{G}(\frkg_2^* \times \frkg^*_3)$ respectively.
We define the Poisson algebra object $P' \wtc P$
in $\eLQCoh{G}(\frkg_1^* \times \frkg^*_3)$ by 
\[
 P' \wtc P := 
 \BRST_{\cois}(J_\infty(\frkg_2), P^{\op} \otimes P',\mu_{\cois}).
 %(P^{\op} \otimes P') \gq_{\mu}^{\bbL} J_{\infty}(\Sym(\frkg_2))
\]
with $\mu_{\cois} :=-\mu_P^2 + \mu_{P'}^1$.
The coisson momentum map $J_\infty(\Sym(\frkg_1 \oplus \frkg_3)) \to P' \wtc P$
is naturally induced by the given $\mu_P^1$ and $\mu_{P'}^3$.
We call $P' \wtc P$ the \emph{coisson gluing} of $P$ and $P'$.
\end{dfn}

Here is our definition of vertex Poisson analogue of $\MT$:

\begin{dfn}\label{dfn:co:coMT}
We define the category $\coMT$ by the following description.
\begin{itemize}[nosep]
\item 
An object is a simply connected semi-simple algebraic group $G$ over $\bbC$.
We identify it with the associated Lie algebra $\frkg$.

\item
A morphisms from $\frkg_1$ to $\frkg_2$ is a Poisson algebra object 
$(P,\mu_P)$ in $\eLQCoh{G}(\frkg_1^* \times \frkg^*_2)$.

\item
The composition of $(P,\mu_P) \in \Hom_{\coMT}(\frkg_1,\frkg_2)$ and 
$(P',\mu_{P'}) \in \Hom_{\coMT}(\frkg_2,\frkg_3)$ is given by
the coisson gluing $P' \wtc P$.
\end{itemize}
\end{dfn}

Let us study the compatibility 
with the derived gluing (Definition \ref{dfn:pr:MT}).
%By Proposition \ref{prp:co:HBcp}, we have
Recall the monoidal functor $R^{\cois}_{(-)}: \dgVP \to \dgpa$ taking 
the associated dg Poisson algebra (Definition \ref{dfn:li:dgapa}).
Applying it to a coisson momentum map $\mu_{\cois}: J_\infty(\frkg) \to P$
and using Lemma \ref{lem:li:RJ=R}, we have a morphism
$\mu:=R^{\cois}_{\mu_{\cois}}: \Sym(\frkg) \to R^{\cois}_P$ of dg Poisson algebra,
which is equivalent to a momentum map $\mu: \frkg \to R^{\cois}_P$.
We also have $R^{\cois}_{P^{\op}} \simeq (R^{\cois}_P)^{\op}$,
where the second term denotes the opposite dg Poisson algebra 
(Definition \ref{dfn:sp:whP}), and 
$R^{\cois}_{\coCl(J_\infty(\frkg))}=\cCl(\frkg)$ by Corollary \ref{cor:li:coCl}.
Combining these facts, we obtain:

\begin{prp}\label{prp:co:glue}
Let $P$, $P'$ and $\mu_{\cois}: J_\infty(\Sym(\frkg_2)) \to P^{\op} \otimes P'$ 
be as in Definition \ref{dfn:co:glue}, and define 
$\mu=R^{\cois}_{\mu_{\cois}}: 
 \frkg_2 \to (R^{\cois}_P)^{\op} \otimes R^{\cois}_{P'}$
as above. 
Then we have 
\[
 R^{\cois}_{P' \wtc P} \simeq 
 \cBRST(\frkg_2, (R^{\cois}_{P})^{\op} \otimes R^{\cois}_{P'},\mu)
\]
as dg Poisson algebras.
\end{prp}

In particular, combining it with Proposition \ref{prp:pr:glue},
we have a quasi-isomorphism of homotopy Poisson algebras
\[
 R^{\cois}_{P' \wtc P} \dqiseq  R^{\cois}_{P'} \wtc R^{\cois}_P = 
 ((R^{\cois}_P)^{\op} \otimes R^{\cois}_{P'}) \gq^{\bbL}_{\mu} \Sym(\frkg).
\]
In other words, we have:

\begin{thm}\label{thm:co:glue}
The functor $R^{\cois}_{(-)}: \dgVP \to \dgpa$ taking 
the associated dg Poisson algebra induces a functor
\[
 R^{\cois}_{(-)}: \coMT \longto \MT.
\] 
\end{thm}

%In particular, for Poisson objects $R$ and $R'$ in $\QCoh{G}(\frkg^*)$,
%we have the Poisson objects $J_\infty(R)$ and $J_\infty(R')$ 
%$\QCoh^{J_\infty(G)}(J_\infty(\frkg^*))$. 
%Then, as a corollary of Proposition \ref{prp:co:glue}, we have 
%\[
% R_{J_\infty(R') \wtc J_\infty(R)} \dqiseq  R' \wtc R 
%\]

In the remaining part we consider coisson analogue of 
Hamiltonian reduction of Poisson algebra objects.
Recall Definition \ref{dfn:pr:dhr} of the derived Hamiltonian reduction:
For a dg Lie algebra $\frkl$, a dg Poisson algebra $R$ and 
a momentum map $\mu: \frkl \to R$ we have 
$R \gq^{\bbL}_{\mu} \Sym(\frkl) := 
 \CE(\frkl,\bbk) \otimes^{\bbL}_{\CE(\frkl,\Sym(\frkl))} \CE(\frkl,R)$.
It is natural to guess that in the coisson setting 
we replace the dg Lie algebra $\frkl$ by 
the level $0$ vertex Lie algebra $J_\infty(\frkl)$.
%Definition \ref{dfn:co:coBRST} of the coisson BRST complex.
We should be careful here that as a Lie algebra 
$J_\infty(\frkg)=\frkg[[t]]$ is infinite-dimensional.
Recalling Proposition \ref{prp:pr:HBinf}, 
we modify Definition \ref{dfn:dga:CE} of Chevalley-Eilenberg complex as 
\[
 \CE(\frkg[[t]],M) := \uHom_{U(\frkg[[t]])}^{\trst}(U(\frkg[[t]]_{\dagger}),M)
\]
for a dg $\frkg[[t]]$-module $M$,
where $\uHom^{\trst}$ is given by Definition \ref{prp:pr:rst}
with the abelian group $\Gamma=\bbZ$ 
and the decomposition $\frkg[[t]] = \bigoplus_{n \in \bbN} \frkg \otimes t^n$.
Using this modified Chevalley-Eilenberg complex, we consider:

\begin{dfn*}%\label{dfn:co:cdHr}
For a Poisson algebra object $(P,\mu_\cois)$ of 
$\eLQCoh{J_{\infty}(G)}(J_{\infty}(\frkg^*))$,
we define a commutative dg algebra
\[
 P \gq_{\mu_{\cois}}^{\bbL} J_\infty(\Sym(\frkg)) := 
 \CE(\frkg[[t]],\bbk) 
 \tbotimes^{\bbL}_{\CE(\frkg[[t]],J_\infty(\Sym(\frkg)))} \CE(\frkg[[t]],P),
\]
where $P$ is regarded as a $\frkg[[t]]$-module by $\mu_\cois$.
We call it the \emph{coisson Hamiltonian reduction} of $P$ 
with respect to the coisson momentum map $\mu_{\cois}$.
\end{dfn*}

Then we can apply the argument in Proposition \ref{prp:pr:HBinf} to 
the finite-dimensional decomposition 
$J_{\infty}(\frkg)=\frkg[[t]] = \bigoplus_{n \in \bbN}\frkg \otimes t^n$.
%since the action of $J_{\infty}(\frkg)$ on $P$ is locally finite.
Since we have 
$\cBRST(\frkg[[t]],P,\mu_{\cois}) \simeq \coBRST(J_\infty(\frkg),P,\mu_{\cois})$ 
as cdgas, we have:

\begin{prp}\label{prp:co:HBcp}
Let $G$ and $\frkg$ be as above.
For a Poisson algebra object $(P,\mu_\cois)$ in 
$\eLQCoh{J_{\infty}(G)}(J_{\infty}(\frkg^*))$,
we have a quasi-isomorphism of cdgas
\[
 P \gq^{\bbL}_{\mu_\cois} J_\infty(\Sym(\frkg)) \dqiseq 
 \BRST_{\cois}(J_{\infty}(\frkg),P,\mu_{\cois}).
\] 
\end{prp}

\begin{rmk}\label{rmk:co:glue}
Let us continue Remark \ref{rmk:pr:glue}, where we considered 
the reduction of a Poisson algebra with Hamiltonian $G$-action.
Here we consider the action of $J_\infty(G)=G[[t]]$ on the arc space instead.

We use the same notation in Remark \ref{rmk:pr:glue}.
Thus, $(R,\mu_R)$ is a Poisson object in $\QCoh^{G}(\frkg^*)$,
and identified with $(X:=\Spec(R), \mu_X: X \to \frkg^*)$.
We assume that there is a closed subscheme $S \subset X$ 
such that the action map gives an isomorphism $G \times S \simto X$,
and that the momentum map $\mu_X$ is flat.
%with the momentum map $\mu_{X'}: X' \to \frkg^*$ and $X^{\op}:=\Spec(R^{\op})$,
%we have the affine Poisson scheme $X^{\op} \times X$ 
%with the flat momentum map $\mu(x,x'):=-\mu_X(x)+\mu_{X'}(x')$.
%Thus we have 
Note that the spectrum $J_{\infty}(R)$ of the arc space $J_{\infty}(X)$ 
is a Poisson algebra object in $\QCoh^{J_{\infty}(G)}(J_\infty(\frkg^*))$, and
the corresponding coisson momentum map (Definition \ref{dfn:co:coBRST}) is given
by $J_{\infty}(\mu_X): J_{\infty}(X) \to J_{\infty}(\frkg^*)$.

Let $(P',\mu_{P'})$ be another Poisson algebra object in 
$\QCoh^{J_{\infty}(G)}(J_\infty(\frkg^*))$,
which will be identified with $(Y':=\Spec(P'),\mu_{Y'})$.
Then we can consider the tensor product $J_{\infty}(R)^{\op} \otimes P'$ 
with coisson momentum map 
$\mu_{\cois}: J_\infty(X)^{\op} \times Y' \to J_{\infty}(\frkg^*)$,
$\mu_{\cois}(x,y):=-J_{\infty}(\mu_X)(x)+\mu_{Y'}(y)$.
By the assumption we have $X \simeq G \times S$, so that we also have 
$J_\infty(X) \simeq J_{\infty}(G) \times J_{\infty}(S)$ as schemes.
Thus the fiber $\mu_{\cois}^{-1}(0)$ is given by
\[
 \mu_{\cois}^{-1}(0) \simeq J_{\infty}(X) \times_{J_\infty(\frkg^*)} Y' 
 \simeq J_{\infty}(G) \times (J_{\infty}(S) \times_{J_\infty(\frkg^*)}Y').
\]

Now let us further assume that the morphism
$J_{\infty}(\mu_X): J_{\infty}(X) \to J_{\infty}(\frkg^*)$ is flat.
We have the non-derived Hamiltonian reduction
$(J_\infty(X)^{\op} \times Y')\gq \Delta(J_{\infty}(G))$.
Then, similarly as in the argument in Remark \ref{rmk:pr:glue},
we can deduce from Proposition \ref{prp:co:HBcp} that 
there is an quasi-isomorphism 
\[
 \BRST_{\cois}(J_{\infty}(\frkg),J_\infty(R)^{\op} \otimes P',\mu) 
 \simeq (R^{\op} \otimes P') \gq_{\mu_{\cois}} J_{\infty}(\Sym(\frkg)) 
 \simeq \bbk[\mu_{\cois}^{-1}(0) / \Delta(J_\infty(G))]
\] 
of cdgas, and on the cohomology we have 
\[
 H^\bl \BRST_{\cois}(J_{\infty}(\frkg),J_\infty(R)^{\op} \otimes P',\mu) 
 \simeq \bbk[J_{\infty}(S) \times_{J_{\infty}(\frkg^*)}X'] 
        \otimes H^{\bl}(G,\bbC).
\]
%$H^{\bl}(\frkg,\bbC)$ denotes the Lie algebra cohomology of 
%the trivial representation of $\frkg$.
Thus we recover the formula \cite[(15)]{A18}.
\end{rmk}

%%%%%%%%%%%%%%%%%%%%%%%%%%%%%%%%%%%%%%%%%%%%%%%%%%%%%%%%%%%%%%%%%%%%%%%%%%%%%%%%
%%%%%%%%%%%%%%%%%%%%%%%%%%%%%%%%%%%%%%%%%%%%%%%%%%%%%%%%%%%%%%%%%%%%%%%%%%%%%%%%
%%%%%%%%%%%%%%%%%%%%%%%%%%%%%%%%%%%%%%%%%%%%%%%%%%%%%%%%%%%%%%%%%%%%%%%%%%%%%%%%
\section{Derived gluing of dg vertex algebras}
\label{s:ch}

Finally we present the main result.
We introduce a vertex algebra analogue $\chMT$  of the category $\MT$.
The composition of morphisms in $\chMT$ will be called 
the \emph{chiral gluing} of dg vertex algebras.

%%%%%%%%%%%%%%%%%%%%%%%%%%%%%%%%%%%%%%%%%%%%%%%%%%%%%%%%%%%%%%%%%%%%%%%%%%%%%%%%
%%%%%%%%%%%%%%%%%%%%%%%%%%%%%%%%%%%%%%%%%%%%%%%%%%%%%%%%%%%%%%%%%%%%%%%%%%%%%%%%
\subsection{Chiral BRST complex}
\label{ss:ch:BRST}

In this subsection we explain BRST reduction for vertex algebras.
Our exposition is a coordinate-dependent version of the general argument 
for chiral algebras in \cite[\S 3.8]{BD}.

\begin{rmk*}
We give a brief account on the theory of \emph{chiral algebras}
developed by Beilinson and Drinfeld \cite[Chap.\ 3]{BD}.
We explained in Remark \ref{rmk:co:star} that 
a vertex Poisson algebra is a special case of coisson algebra,
which is a Poisson algebra object in the compound 
($\otimes^!$ and $\otimes^\star$) monoidal structure 
on the category $\DMod(\Ran(X))$ of D-modules on the Ran space of $X=\bbA^1$.
On similar footing, a vertex algebra can be regarded as a Lie algebra object
in the \emph{chiral monoidal structure} $\otimes^{\ch}$ on $\DMod(\Ran(X))$.
See also \cite{FG} for the chiral monoidal structure.

Thus the BRST complex for vertex algebra should be defined by replacing 
``Lie algebra" in the classical BRST complex (\S \ref{ss:pr:BRST})
with ``Lie algebra object in $\DMod(\Ran(X))^{\otimes^{\ch}}$".
This is the very rough explanation of the construction 
of chiral BRST complex in \cite[\S 3.8]{BD}. 
\end{rmk*}

%%%%%%%%%%%%%%%%%%%%%%%%%%%%%%%%%%%%%%%%%%%%%%%%%%%%%%%%%%%%%%%%%%%%%%%%%%%%%%%%
\subsubsection{Clifford vertex algebra and Tate extension}

We work over a field $\bbk$ of characteristic $0$.
In this part we give a new explanation the free fermion vertex algebra 
$\Wedge^{\sinf}(U)$ in \S \ref{sss:li:ff},
and also introduce the Tate extension of dg vertex Lie algebra.
These materials are given in \cite[\S 3.8]{BD}.

Let $L$ be a dg vertex Lie algebra.
We denote by $L^*:=\uHom(L,\bbk)$ the dual complex.
It has a dg vertex Lie algebra structure.
The canonical pairing $L^* \otimes L \to \bbk$ induces a pairing 
$(\cdot,\cdot)$ on $M := L^*[-1] \oplus L[1]$, 
and it defines a one-dimensional central extension $M^\flat$ 
of the commutative dg vertex Lie algebra $M$.
Thus we have the twisted enveloping vertex algebra $U(M)^\flat$ of $M$
(Definition \ref{dfn:li:ULf}).
%Below we denote $1^\flat := 1_{\bbk} \in M^\flat$.

\begin{dfn*}[{\cite[3.8.6]{BD}}]
We denote the twisted enveloping vertex algebra by 
\[
 \chCl(L) = \chCl\bigl(L,L^*,(\cdot,\cdot)\bigr) := U(M)^\flat
\]
and call it the \emph{Clifford vertex algebra}.

$M^\flat$ has an extra $\bbZ$-grading by
$(M^\flat)^{(-1)}:= L^*[-1]$, $(M^\flat)^{(0)} := \bbk$ and 
$(M^\flat)^{(1)} := L  [1]$.
It induces an extra $\bbZ$-grading $\chCl(L)^{(\bl)}$.
\end{dfn*}

\begin{rmk}\label{rmk:ch:chCl}
\begin{enumerate}[nosep]
\item 
In \cite[3.8.6]{BD}, $\chCl(L)$ is called the chiral Clifford algebra,
which is the origin of the symbol $\ch$.

\item
The notation $\chCl(U((t)),U^*((t))d t,(\cdot,\cdot))$ 
in Definition \ref{dfn:li:ff} agrees with this definition.
In this case the $\bbZ$-grading $(\bl)$ is nothing but the charge grading.
%In the case that $L$ is concentrated in cohomological degree $0$
%(or generally in an even degree $2n$),
\end{enumerate}
\end{rmk}

Recall the PBW filtration $\chCl(L)_{\bl}$ 
in Definition \ref{dfn:li:ULf} \eqref{i:li:ULf:PBW}.
We can check that 
$\chCl(L)^{(0)}_2 = (L^*[-1]) \otimes (L[1]) \oplus \bbk$
is a dg vertex Lie subalgebra of $\chCl(L)_{\tLie}$,
and is a one-dimensional central extension of $L^* \otimes L$.
On the other hand, the adjoint action of $\Lie(L)$ on itself yields
a morphism $L \to L^* \otimes L$ of dg vertex Lie algebras.
We denote by $L^\flat$ the pullback of the extension 
$0 \to \bbk_{\tLie} \to \chCl(L)_{\tLie} \to L^* \otimes L \to 0$
by this morphism.

\begin{dfn*}{\cite[3.8.7]{BD}}
We call the dg vertex Lie algebra $L^\flat$ the Tate extension.
\end{dfn*}

We have a morphism $L^\flat \to \chCl(L)_{\tLie}$ of dg vertex Lie algebras
satisfying $1^\flat \mapsto 1_{\chCl(L)}$.

%%%%%%%%%%%%%%%%%%%%%%%%%%%%%%%%%%%%%%%%%%%%%%%%%%%%%%%%%%%%%%%%%%%%%%%%%%%%%%%%
\subsubsection{Chiral BRST complex}

We continue to work over a field $\bbk$ of characteristic $0$.
We fix a dg vertex Lie algebra $L$.

\begin{dfn*}[{\cite[3.8.8]{BD}}]
A \emph{BRST datum} is a pair $(V,\alpha)$ of a dg vertex algebra $V$ 
and a morphism $\alpha: L \to V_{\tLie}$ of dg Lie algebras 
such that $\alpha(1^\flat)=-\vac$,
where $\vac$ denotes the vacuum of $V$.
\end{dfn*}

Given a BRST datum $(V,\alpha)$, we have a morphism
$\ell^{(0)} = \alpha+\beta: L \to V \otimes \chCl(L)^{(0)}$
of dg vertex Lie algebras.
As in the argument of \S \ref{ss:co:coBRST}, the contractible complex 
$L_\dagger=(L[1] \to L)$ inherits the dg vertex Lie algebra structure of $L$.
Defining 
$\ell^{(1)}: L[1] \to V \otimes \chCl(L)^{(-1)}$ to be 
the composition $L[1] \inj \chCl(L)^{(-1)} \inj V \otimes \chCl(L)^{(-1)}$,
we have a morphism
\[
 l: L_{\dagger} \longto \chCl(L) \otimes V
\]
of graded vertex Lie algebras (forgetting the differential).

We now regard the symmetric dg algebra $\Sym(L^*[-1])$ 
as a Chevalley-Eilenberg complex of the trivial $L$-module, 
and denote by $\delta$ the differential.
By the embedding $\Sym(L^*[-1]) \subset \chCl(L) \subset V \otimes \chCl(L)$
we regard $\delta$ acting on $V \otimes \chCl(L)$.
%Then we have an operation
%$\chi: L \otimes L^* \to V \otimes \chCl(L)^{(1)}[1]$,
%$x \otimes y \mapsto Y(l^{(0)}(x),z-w)y-Y(l^{(-1)}(x),z-w)\delta(y)$.
%Identifying $L \otimes L^* \simeq \uEnd(L)$, we set
%\[
% Q := \chi(\id_L) \in V \otimes \chCl(L)^{(1)}[1],
%\]

\begin{fct}[{\cite[3.8.10]{BD}}]\label{fct:ch:BRST}
There is a unique element $Q \in V \otimes \chCl(L)^{(1)}$
of cohomological degree $1$ 
such that $[Q_{(0)},l^{(-1)}]=l^{(0)}$ and $Q_{(0)}^2=0$.
The morphism $d_{\ch}:= Q_{(0)}+d_{\chCl(L) \otimes V}$ defines 
a dg vertex algebra structure on $V \otimes \chCl(L)$.
We denote it by 
\[
 \BRST_{\ch}(V,\alpha) := (V \otimes \chCl(L),d_{\ch})
\]
and call it the (chiral) BRST complex for the BRST datum $(V,L,\alpha)$.
\end{fct}

The morphism $l$ gives a morphism 
$L_{\dagger} \to \BRST_{\ch}(V,\alpha)_{\tLie}$ of dg Lie algebras.

%%%%%%%%%%%%%%%%%%%%%%%%%%%%%%%%%%%%%%%%%%%%%%%%%%%%%%%%%%%%%%%%%%%%%%%%%%%%%%%%
\subsubsection{The case of universal affine vertex algebras}
\label{sss:ch:G}

In this part we describe in detail the chiral BRST complex 
for the universal affine vertex algebras,
which will be used to formulate the vertex algebra analogue 
of the category $\MT$ in  \S \ref{ss:ch:gl}.
%Historically, it goes back to the work of B.~Feigin on semi-infinite cohomology of 
%infinite-dimensional Lie algebras \cite{F}.
Hereafter we work over $\bbC$, 
and use notations in \S \ref{sss:li:uava} for Lie algebras.
In particular 
\begin{itemize}[nosep]
\item 
$\frkg$ is the Lie algebra of the semi-simple algebraic group over $\bbC$, and 

\item
$\wh{\frkg}=\frkg((t)) \oplus \bbC K$ is the the derived algebra of 
the non-twisted affine Lie algebra associated to $\frkg$.
\end{itemize}

Let $V_k(\frkg)$ be the universal affine vertex algebra at level $k \in \bbC$.
(Definition \ref{dfn:li:uava}).
Regarding it as a graded dg vertex algebra concentrated in 
cohomological degree $0$ with trivial differential (Definition \ref{dfn:li:dgvh}),
we have the category $\dgVMod{V_k(\frkg)}$ of dg $V_k(\frkg)$-modules
(Definition \ref{dfn:li:dgMod}).
Hereafter we regard $\frkg \subset V_k(\frkg)$ by the injective linear map 
\[
 \frkg \longinj V_k(\frkg), \quad x \longmapsto x t^{-1} \vac.
\]

Following \cite[\S 3, p.\ 10]{A18}, 
we introduce the notion of momentum maps in the category $\dgVMod{V_k(\frkg)}$:

\begin{dfn}\label{ntn:ch:chmm}
\begin{enumerate}[nosep]
\item 
A \emph{chiral momentum map} is a morphism $\mu: V_k(\frkg) \to V$
of dg vertex algebras.

\item
A \emph{dg vertex algebra object} (\emph{dgva object} for short) 
\emph{in $\dgVMod{V_k(\frkg)}$} is a pair $(V,\mu)$ of a dg vertex algebra $V$ 
and a chiral momentum map $\mu: V_k(\frkg) \to V$,
where we regard $V \in \dgVMod{V_k(\frkg)}$ by Lemma \ref{lem:li:dgh-m}.
%If $V$ is equipped with Hamiltonian, then we assume that 
%a chiral momentum map respects Hamiltonians on $V_k(\frkg)$ and $V$.
\end{enumerate}
\end{dfn}

\begin{rmk}\label{rmk:ch:chmm}
In \cite{A18} a chiral momentum map is called a chiral quantum momentum map.
See also Remark \ref{rmk:co:cmm}.
\end{rmk}

Given a chiral momentum map $\mu: V_k(\frkg) \to V$, we have a morphism
$\mu_{\tLie}: V_k(\frkg) \to V_{\tLie}$ of dg vertex Lie algebras
by the polar part construction (Lemma \ref{lem:li:VL}).
Restricting $\mu_{\tLie}$ to the subalgebra $v_k(\frkg) \subset V_k(\frkg)$
in Example \ref{eg:li:vk}, we have a BRST datum
$(V, \rst{\mu_{\tLie}}{v_k(\frkg)}: v_k(\frkg) \to V)$,
and thus we have the BRST complex (Fact \ref{fct:ch:BRST}).

\begin{dfn}\label{dfn:ch:BRST}
Let $k \in \bbC$ and $(V,\mu)$ be a dgva object in $\dgVMod{V_k(\frkg)}$.
\begin{enumerate}[nosep]
\item 
We denote the BRST complex for the BRST datum 
$(V,\rst{\mu_{\tLie}}{v_k(\frkg)})$ by
\[
 \BRST(\frkg_k,V,\mu) := \BRST_{\ch}(V,\rst{\mu_{\tLie}}{v_k(\frkg)}).
\]

\item
The cohomology vertex algebra (Lemma \ref{lem:dgv:coh}) of the BRST complex
is denoted by 
\[
 H^{\sinf+\bl}(\wh{\frkg}_{k},V,\mu) := 
 H^{\bl}(\BRST(\wh{\frkg}_{k},V,\mu), d_{\cl}),
\]
and called the \emph{BRST cohomology}.

\item
In the case $V=V_{k}(\frkg)$ and $\mu=\id$, we denote 
$\BRST(\wh{\frkg}_{k},V_{k}(\frkg)) := 
 \BRST(\wh{\frkg}_{k},V_{k}(\frkg),\id)$.
\end{enumerate}
\end{dfn}

Now recall the free fermionic dg vertex algebra $\Wedge^{\sinf}(\frkg)$
in Definition \ref{dfn:li:ff}.
Note that the $\bbZ$-grading $\Wedge^{\sinf}(\frkg)^{\bl}$ of the dg structure
is equal to the minus of the charge grading $\Wedge^{\sinf+\bl}(\frkg)$.
By the characterization of the BRST charge in Fact \ref{fct:ch:BRST},
we can write down this BRST complex in the following form:

\begin{lem}\label{lem:ch:BRST}
%Let $\kappa$ be an invariant symmetric bilinear form on $\frkg$. 
Let $k \in \bbC$ and $(V,\mu)$ be a dgva object in $\dgVMod{V_k(\frkg)}$.
Then the dg vertex algebra $\BRST(\wh{\frkg}_{k},V,\mu)$ is described as follows.
\begin{enumerate}[nosep]
\item
As a graded vertex algebra (forgetting the differential), we have 
\[
 \BRST(\wh{\frkg}_{k},V,\mu) \simeq V \otimes \Wedge^{\sinf}(\frkg).
\]

\item
The differential is given by 
$d_{\ch} = Q_{(0)}+d_{V \otimes \bigwedge^{\infty/2}(\frkg)}$, 
where $Q \in \BRST(\wh{\frkg}_{k},V,\mu)$ is the BRST charge
\[
 Q := \sum_{i=1}^{\dim \frkg}\mu(x_i) \otimes \psi_i^*
    - \frac{1}{2}\sum_{i,j,k=1}^{\dim \frkg} 
      1 \otimes c_{i j}^k \psi^*_i \psi^*_j \psi_k.
\]
Here we used the structure constant $c_{i j}^k$ of $\frkg$ 
as in Lemma \ref{lem:pr:olQ} and omit the vacuum $\vac$.
%Then we have $(Q_{(0)})^2=0$ and $[Q_{(0)},d_V \otimes 1]=0$.
\end{enumerate}
\end{lem}

%The BRST cohomology is a graded vertex algebra.
%If $V$ is conformal with central charge $c_V$,
%then it is also conformal with central charge $c_V-2\dim \frkg$.
%In the next \S \ref{ss:ch:KL}, we give an equivalence of 
%the chiral derived Hamiltonian reduction and the BRST complex.

Now recall the functor taking the associated graded space of the Li filtration
 (Lemma \ref{lem:li:grdgv}):
\[
 \gr^F: \dgVA \longto \dgVP.
\]
Given a dgva object$(V,\mu_V)$ in $\dgVMod{V_k(\frkg)}$.
the chiral momentum map $\mu_V: V_k(\frkg) \to V$ induces the coisson momentum map 
\[
 \mu_{\gr^F V}: \gr^F V_k(\frkg) \simeq J_\infty(\Sym(\frkg)) \longto \gr^F V.
\]
Here we used the isomorphism in Fact \ref{fct:li:SV=RV} \eqref{i:li:SV=RV:2}.
Then, comparing the description of BRST complex (Lemma \ref{lem:ch:BRST})
with the coisson BRST complex (Definition \ref{dfn:co:coBRST}),
we obtain:

\begin{lem}\label{lem:ch:grB=coB}
Let $(V,\mu_V)$ be a dgva object in $\dgVMod{V_k(\frkg)}$ with $k \in \bbC$.
Then we have an isomorphism of dg vertex Poisson algebras
\[
 \gr^F \BRST(\wh{\frkg}_k,V,\mu_V) \simeq 
 \BRST_{\cois}(J_\infty(\frkg),\gr^F V,\mu_{\gr^F V}).
\]
\end{lem}

\subsection{Chiral gluing procedure}
\label{ss:ch:gl}

In this subsection we give a vertex algebra analogue of 
composition of morphisms in the category $\MT$.
We work under the same setting as in \S \ref{sss:ch:G},
and use the same symbols $\frkg=\Lie(G)$, $V_k(\frkg)$ and so on.
%Recall the original gluing in \S \ref{ss:pr:MT}
%and the vertex Poisson version in \S \ref{sss:co:G}.

Some preparations are in order. 
\begin{itemize}[nosep]
\item 
Let $(V,\mu)$ and $(V',\mu')$ be dgva objects in 
$\dgVMod{V_k(\frkg)}$ and $\dgVMod{V_l(\frkg)}$ respectively.
Then we can regard $(V \otimes V',\mu \otimes \mu') \in \dgVMod{V_{k+l}(\frkg)}$ 
by the diagonal action of $\wh{\frkg}$ (Lemma \ref{lem:li:tens}).

\item
For a dg vertex algebra $V=(V^{\bl},d,\vac,T,Y)$, 
we define the \emph{opposite} dg vertex algebra $V^{\op}$ by 
\[
 V^{\op} := (V^{\bl},d,\vac,T^{\op},Y^{\op})
\] 
with $Y^{\op}(a,z) := Y(a,-z)$ and $T^\op:=-T$.
We have 
\[
 \gr^F(V^{\op}) \simeq (\gr^F V)^{\op}
\]
as dg vertex Poisson algebras.

\item
For a dgva object $(V,\mu)$ in $\dgVMod{V_k(\frkg)}$,
we have the new dgva object $(V^{\op},\mu^{\op})$
with $\mu^{\op}(a):=-\mu(a)$.
\end{itemize}

\begin{dfn}\label{dfn:ch:glue}
Let $(V,\mu_V)$ and $(V',\mu_{V'})$ be dgva objects 
in $\dgVMod{V_k(\frkg)}$ and in $\dgVMod{V_l(\frkg)}$ respectively.
We define a chiral momentum map $\mu: V^{\op} \otimes V' \to V_{k+l}(\frkg)$  
by $\mu(a \otimes b):=-\mu_V(a) + \mu_{V'}(b)$, and 
a dgva object $V' \wtc V$ in $\dgVMod{V_{k+l}(\frkg)}$ by
\[
 V' \wtc V := 
 \BRST(\wh{\frkg}_{k+l},V^{\op} \otimes V',\mu)
\]
and call it the \emph{chiral gluing}.
\end{dfn}

Recall that for a dg vertex algebra $V$, we denote by $R_V=F^0 V/F^1 V$ 
Zhu's $C_2$-algebra (Definition \ref{dfn:li:C2}). 
By Lemma \ref{lem:ch:grB=coB} and Definition \ref{dfn:co:glue} 
of the coisson gluing, we have:

\begin{prp}\label{prp:ch:main}
Under the same setting of Definition \ref{dfn:ch:glue}, 
we have the following quasi-isomorphism of dg vertex Poisson algebras:
\[
 \gr^F(V' \wtc V) \simeq (\gr^F V') \wtc (\gr^F V).
\]
On Zhu's $C_2$-algebras, we have 
\[
 R_{V'\wtc V} \simeq  R_{V'}  \wtc R_V. 
\]
\end{prp}

%%%%%%%%%%%%%%%%%%%%%%%%%%%%%%%%%%%%%%%%%%%%%%%%%%%%%%%%%%%%%%%%%%%%%%%%%%%%%%%%
%%%%%%%%%%%%%%%%%%%%%%%%%%%%%%%%%%%%%%%%%%%%%%%%%%%%%%%%%%%%%%%%%%%%%%%%%%%%%%%%
\subsection{The category $\chMT$}
\label{ss:ch:chMT}

Under the same setting as in \S \ref{ss:ch:gl},
we finally introduce the category $\chMT$, which is a vertex algebra analogue
of $\MT$ in \S \ref{ss:pr:MT} and $\coMT$ in \S \ref{sss:co:G}.
In the latter case, we considered a vertex Poisson object in 
$\eLQCoh{J_{\infty}(G)}(J_{\infty}(\frkg^*))$,
which satisfies some finite condition (Definition \ref{dfn:co:pao}).
Following \cite[\S 3, p.\ 10]{A18}, 
we introduce the corresponding category of vertex algebras.

%%%%%%%%%%%%%%%%%%%%%%%%%%%%%%%%%%%%%%%%%%%%%%%%%%%%%%%%%%%%%%%%%%%%%%%%%%%%%%%%
%\subsubsection{The Kazhdan-Lusztig category}
%\label{sss:ch:KL}
%
%In this part we introduce the \emph{Kazhdan-Lusztig category} $\dgKL_k$ 
%of certain dg modules over the universal affine vertex algebra $V_k(\frkg)$, 
%and give an equivalence of 
%the chiral derived Hamiltonian reduction and the BRST complex for those modules. 
%
%
Recall the equivalence of $V_k(\frkg)$-modules and 
smooth $\wh{\frkg}$-representations of level $k$ (Fact \ref{fct:li:sm}).
Thus on a $V_k(\frkg)$-module we can discuss the action of
$\frkg \subset \wh{\frkg}$ and $\frkg[[t]]t \subset \wh{\frkg}$.

\begin{dfn*}
Let $k \in \bbC$.
\begin{enumerate}[nosep]
\item 
We denote by $\dgKL_k(\frkg)$ the full subcategory of $\dgVMod{V_k(\frkg)}$
spanned by objects on which $\frkg[[t]]t$ acts locally nilpotently
and $\frkg$ acts locally finitely. 

\item
We denote by $\KL_k(\frkg)$ the full subcategory of $\dgKL_k(\frkg)$
spanned by objects concentrated in cohomological degree $0$. 
\item
We denote by $\KL^{\ord}_k(\frkg)$ the full subcategory of $\KL_k(\frkg)$
spanned by objects that are $\bbN$-graded (Definition \ref{dfn:li:gvmod}) 
and  each homogeneous subspaces are finite-dimensional.
\end{enumerate}
\end{dfn*}

\begin{rmk*}
\begin{enumerate}[nosep]
\item 
The subcategory $\KL_k(\frkg) \subset \dgKL_k(\frkg)$ spanned by 
those dg modules concentrated in degree $0$ was originally introduced
in \cite{A12} as the category of 
graded Harish-Chandra $(\wh{\frkg},G[[t]])$-modules of level $k$.
The category $\KL_k(\frkg)$ was used for showing the cohomology vanishing 
of BRST complex in \cite{A12},
and used in \cite{A18} to construct the genus zero chiral algebra 
$\bfV^{\calS}_{G,b}$ of class $\calS$.

\item
As noted in \cite[\S 3]{A18}, every object in $\KL_k(\frkg)$ is a colimit of 
a direct system of objects in $\KL^{\ord}_k(\frkg)$.
\end{enumerate}
\end{rmk*}

We immediately have:

\begin{lem}\label	{lem:ch:pao}
Let $V$ be a vertex algebra object in $\dgKL_k(\frkg)$.
Then $\gr^F$ is a Poisson algebra object in 
$\eLQCoh{J_{\infty}(G)}(J_{\infty}(\frkg^*))$,
and $R_V$ is a Poisson algebra object in $\eLQCoh{G}(\frkg^*)$,
\end{lem}

The family $\{\dgKL_k(\frkg) \mid k \in \bbC\}$ inherits the tensor structure of
$\{\dgVMod{V_k(\frkg)} \mid k \in \bbC\}$ in Lemma \ref{lem:li:tens}:
For $M \in \dgKL_k(\frkg)$ and $N \in \dgKL_k(\frkg)$, 
we have $M \otimes N \in \dgKL_{k+l}(\frkg)$ 
by the diagonal action of $\wh{\frkg}$.

Here is the definition of vertex analogue of the category $\MT$.

\begin{dfn}\label{dfn:ch:chMT}
We define the category $\chMT$ by the following description.
\begin{itemize}[nosep]
\item 
An object is a simply connected semi-simple algebraic group $G$ over $\bbC$.
We identify it with the associated Lie algebra $\frkg$.

\item
A morphisms from $\frkg_1$ to $\frkg_2$ is a vertex algebra object $(V,\mu_V)$
in $\dgVMod{\bigl(V_k(\frkg_1) \otimes V_l(\frkg_2)\bigr)}$ 
with some $k,l \in \bbC$ such that
$(V,\mu_V^1) \in \dgKL_k(\frkg_1)$ and $(V,\mu_V^2) \in  \dgKL_l(\frkg_2)$.
Here we defined 
$\mu_V^1: V_k(\frkg_1) \to V$, $\mu_V^1(a) := \mu_V(a \otimes \vac)$ and 
$\mu_V^2: V_l(\frkg_2) \to V$, $\mu_V^2(b) := \mu_V(\vac \otimes b)$.

\item
The composition of $(V,\mu_V) \in \Hom_{\chMT}(\frkg_1,\frkg_2)$ and 
$(V',\mu_{V'}) \in \Hom_{\chMT}(\frkg_2,\frkg_3)$ is given by
the chiral gluing $V' \wtc V$ in Definition \ref{dfn:ch:glue} 
where we regard $(V,\mu_V^2) \in \dgVMod{V_{l}(\frkg_2)}$
and $(V,\mu_{V'}^1) \in \dgVMod{V_{k'}(\frkg_2)}$.
\end{itemize}
\end{dfn}

By Proposition \ref{prp:ch:main} and Lemma \ref{lem:ch:pao} 
we have the main result:

\begin{thm}\label{thm:ch:main}
The functors $\gr^F:\dgVA \to \dgVP$, $R: \dgVA \to \dgpa$ 
and $R: \dgVP \to \dgpa$ give a commutative diagram 
\[
 \xymatrix{ \chMT \ar[r]^{\gr^F} \ar[d]_{R}& \coMT \ar[d]^{R^{\cois}} \\ 
              \MT \ar@{=}[r] & \MT}
\]
\end{thm}

\begin{rmk}\label{rmk:ch:glue}
Let us continue Remark \ref{rmk:co:glue} and 
give a connection to the argument in \cite[\S 3]{A18}.
Recall that in \S \ref{sss:co:G} we considered coisson gluing, 
which is a vertex Poisson analogue of composition of morphisms in $\MT$.
Here we consider a chiral analogue.

Let $(V,\mu_V)$ be a vertex algebra object in $\KL_k$.
Then the associated graded space $\gr^F V$ is a vertex Poisson algebra object 
in $\QCoh^{J_{\infty}(G)}(J_\infty(\frkg^*))$,
and Zhu's $C_2$-algebra $R_V$ is a Poisson algebra object in $\QCoh^{G}(\frkg)$.
Let $X=\Spec(R)$ be an affine Poisson scheme,
and assume the following conditions.
\begin{enumerate}[nosep,label=(\roman*)]
\item 
$R_V \simeq R$.
\item
The surjection $J_{\infty}(R_V) \surj \gr^F V$ of vertex Poisson algebras
is an isomorphism.
\end{enumerate}
If moreover $V$ is separated (Definition \ref{dfn:li:sep}),
then $V$ is called a \emph{strict chiral quantization} of $X$
\cite[Definition 2.1]{A18}.
We further assume the following conditions
considered in Remarks \ref{rmk:pr:glue} and \ref{rmk:co:glue}.
\begin{enumerate}[nosep,label=(\roman*)]
\setcounter{enumi}{2}
\item 
There is a closed subscheme $S \subset X$ 
such that the action map gives an isomorphism $G \times S \simto X$.
\item
The chiral momentum map 
$J_{\infty}(\mu_X)=(\gr^F \mu_V)^*: J_{\infty}(X) \to J_{\infty}(\frkg^*)$ is flat.
\end{enumerate}

Let $(V',\mu_{V'})$ be a vertex algebra object in $\KL_l$.
Then the tensor product $V^{\op} \otimes V'$ is a vertex algebra object 
in $\KL_{k+l}$ with chiral momentum map $\mu(a,b):=-\mu_V(a)+\mu_{V'}(b)$.

By the same argument in Remarks \ref{rmk:pr:glue} and \ref{rmk:co:glue},
we can then deduce an quasi-isomorphism 
\[
 \gr^F \BRST(\wh{\frkg}_{k+l},V^{\op} \otimes V',\mu) \simeq 
 \bigl((J_\infty(R))^{\op} \otimes (\gr^F V') \bigr) 
 \gq_{\tgr^F \mu} J_{\infty}(\frkg^*)
 \simeq \bbk[(\gr^F \mu)^{-1}(0) / \Delta(J_\infty(G))]
\] 
of cdgas, and as for the cohomology we have 
\[
 \gr^F H^{\sinf+\bl}(\wh{\frkg}_{k+l},V^{\op} \otimes V',\mu) \simeq  
 \bigl(\bbk[J_{\infty}(S)] \otimes_{J_{\infty}(\Sym(\frkg))} \gr^F V'\bigr)
 \otimes H^{\bl}(G,\bbC).
\]
%Combining with the cohomology vanishing shown in \cite[\S 4.6]{A15},
Thus we have 
\[
 \gr^F H^{\sinf+0}(\wh{\frkg}_{k+l},V^{\op} \otimes V',\mu) \simeq 
 \bbk[J_{\infty}(S)] \otimes_{J_{\infty}(\Sym(\frkg))} \gr^F V',
\]
%$H^{\bl}(\frkg,\bbC)$ denotes the Lie algebra cohomology of 
%the trivial representation of $\frkg$.
which recovers the formula in \cite[Theorem 3.1]{A18}.
\end{rmk}

%%%%%%%%%%%%%%%%%%%%%%%%%%%%%%%%%%%%%%%%%%%%%%%%%%%%%%%%%%%%%%%%%%%%%%%%%%%%%%%%
%%%%%%%%%%%%%%%%%%%%%%%%%%%%%%%%%%%%%%%%%%%%%%%%%%%%%%%%%%%%%%%%%%%%%%%%%%%%%%%%

\end{document}